\documentclass[3p]{elsarticle}

\usepackage{hyperref}
\usepackage{amsmath,amssymb,amsthm}
\usepackage{fixmath}
\usepackage{xcolor}

\journal{Physica D}

\biboptions{sort&compress}
\DeclareMathOperator{\Real}{Re}
\DeclareMathOperator{\Imag}{Im}
\DeclareMathOperator{\grad}{grad}
\DeclareMathOperator{\divr}{div}
\DeclareMathOperator{\vol}{vol}
\DeclareMathOperator{\dvol}{dvol}
\DeclareMathOperator{\spn}{span}
\DeclareMathOperator*{\argmax}{argmax}
\newcommand{\ii}{\mathrm{i}}
\newtheorem{thm}{Theorem}
\newtheorem{prop}[thm]{Proposition}
\newtheorem{lemma}[thm]{Lemma}
\newdefinition{cor}[thm]{Corollary}
\newdefinition{exmp}[thm]{Example}
\newdefinition{rk}[thm]{Remark}
\newdefinition{defn}[thm]{Definition}

\makeatother

\begin{document}

\begin{frontmatter}

\title{Extraction and Prediction of Coherent Patterns in Incompressible Flows through Space-Time Koopman Analysis}

\author[]{Dimitrios Giannakis\corref{mycorrespondingauthor}}
\ead{dimitris@cims.nyu.edu}
\cortext[mycorrespondingauthor]{Corresponding author}
\author[]{Suddhasattwa Das}

\address{Courant Institute of Mathematical Sciences, New York University, New York, NY 10012, USA}

\begin{abstract}
We develop methods for detecting and predicting the evolution of coherent spatiotemporal patterns in incompressible time-dependent fluid flows driven by ergodic dynamical systems. Our approach is based on representations of the generators of the Koopman and Perron-Frobenius groups of operators governing the evolution of observables and probability measures on Lagrangian tracers, respectively, in a smooth orthonormal basis learned from velocity field snapshots through the diffusion maps algorithm. These operators are defined on the product space between the state space of the fluid flow and the spatial domain in which the flow takes place, and as a result their eigenfunctions correspond to global space-time coherent patterns under a skew-product dynamical system. Moreover, using this data-driven representation of the generators in conjunction with Leja interpolation for matrix exponentiation, we construct model-free prediction schemes for the evolution of observables and probability densities defined on the tracers. We present applications to periodic Gaussian vortex flows and aperiodic flows generated by Lorenz 96 systems.
\end{abstract}

\begin{keyword}
  Koopman operators, Perron-Frobenius operators, Lagrangian coherent structures, kernel methods, diffusion maps, nonparametric prediction
\end{keyword}

\end{frontmatter}


\section{Introduction}

The formation of coherent structures is a ubiquitous feature of time-dependent flows in both natural and engineered systems, including coherent jets and vortices in planetary atmospheres, oceanic currents and eddies, and chemical mixing, among many other examples  \cite{Ottino89,SamelsonWiggins06,Haller15}. Objectively identifying and predicting these patterns has received considerable attention in the mathematical, physical, and engineering disciplines, resulting in a diverse range of techniques to achieve these goals. While virtually all such techniques utilize dynamical systems theory, they generally emphasize fairly distinct aspects of that theory, namely the geometric/state-space perspective \cite{HallerYuan00,ShaddenEtAl05,Haller11,SerraHaller16}, or the operator-theoretic/probabilistic perspective \cite{FroylandEtAl07,FroylandEtAl10b,FroylandEtAl10,Froyland13,BudisicMezic12} (though this dichotomy is not rigid as there are methods that employ aspects of both approaches; e.g., \cite{FroylandPadberg09,Froyland15,KarraschKeller16,BanischKoltai17,FroylandJunge17}). Among the many references in the literature, here we mention explicitly the paper of Liu and Haller \cite{LiuHaller04} on coherent patterns of diffusive tracers in time-dependent flows. They show that the associated advection-diffusion equation for the tracer concentration field admits a finite-dimensional invariant manifold embedded in its solution space and spanned by a finite set of time-dependent modes, previously identified by Pierrehumbert \cite{Pierrehumbert94} as strange eigenmodes. We also mention the recent work of Froyland and Koltai \cite{FroylandKoltai17} who develop a method for recovering coherent patterns in time-periodic flows through a transfer operator defined on an extended state space where the dynamics are autonomous. Both \cite{LiuHaller04,FroylandKoltai17} have common aspects with the techniques presented below.    

In this work, we study the problem of coherent pattern identification and prediction in time-dependent, incompressible flows having a skew-product structure. In particular, we consider that the time-dependence of the velocity field is itself the outcome of a dynamical system with ergodic properties---this allows us to take advantage of geometrical and operator-theoretic properties of that dynamical system that would otherwise not be present in flows with an arbitrary time dependence. In particular, we consider the product space between the state space for the velocity field dynamics (hereafter, $ A $) and  the physical domain (hereafter, $ X $) where tracer motion takes place. Intuitively, that product space, $ M = A \times X $, can be through of as a ``space-time manifold'' with the component $ A $ playing the role of time (as it governs the non-autonomous aspect of tracer dynamics), and $ X $ the role of space. On $ M $, the dynamics is autonomous and measure-preserving, though not necessarily ergodic, and has a natural skew-product structure owing to the fact that the state of the velocity field influences the dynamics of the tracers but not vice-versa. Moreover, associated with these dynamics are Koopman and Perron-Frobenius operators \cite{EisnerEtAl15} governing the evolution of observables and probability measures defined on the tracers. 

Our approach for identifying and predicting coherent spatiotemporal patterns is based on data-driven approximations of these operators and their generators constructed from time-ordered velocity field snapshots. In particular, for the purpose of coherent pattern identification, we solve the eigenvalue problem for the Koopman generator of this skew-product system with a small amount of diffusion added for regularization. This extends recently developed approximation techniques for Koopman operators of ergodic dynamical systems \cite{GiannakisEtAl15,Giannakis17,DasGiannakis17} to the setting of skew-product systems governing the evolution of Lagrangian tracers. In our approach, the regularized operator is elliptic on the product state space (cf.\ \cite{LiuHaller04,FroylandKoltai17}, who consider parabolic operators), and has only discrete spectrum \cite{FrankeEtAl10}; its eigenfunctions correspond to near-invariant or quasiperiodic observables of the skew-product system, and can be visualized as spatiotemporal patterns for a given trajectory on $ A $.  If the Koopman generator has discrete spectrum, then the eigenfunctions of the regularized generator behave as small-viscosity perturbations of these eigenfunctions. On the other hand, if the Koopman generator has continuous spectrum, the regularized eigenfunctions acquire a singular behavior on the diffusion regularization parameter \cite{ConstantinEtAl08}, and generally exhibit intricate spatial structures characteristic of strange eigenmodes. 

As in \cite{GiannakisEtAl15,Giannakis17,DasGiannakis17}, a key ingredient of our approach is to employ kernel algorithms developed for machine learning \cite{BelkinNiyogi03,CoifmanLafon06,VonLuxburgEtAl08,BerryHarlim16,BerrySauer16b} to build an orthonormal basis for the Hilbert space of the dynamical system without requiring a priori knowledge of the geometry of the state space or the governing equations. Moreover, in the skew-product setting of interest here, we do not require  availability of explicit tracer trajectories. This data-driven basis is chosen such that its elements are optimally smooth, in sense of extremizing a Dirichlet energy functional measuring roughness of functions. We also employ this functional to order Koopman eigenfunctions; thus, we select coherent patterns not on the basis of timescale (though at least some of the low-energy eigenfunctions usually capture slow timescales), but on the basis of small roughness, and thus amenability to robust approximation from data. 

For the purpose of prediction of observables and probability measures, we employ our data-driven representation of the generator in conjunction with matrix exponentiation algorithms based on Leja polynomial interpolation \cite{CaliariEtAl04,KandolfEtAl14} to approximate the semigroup action. This allows us to perform model-free statistical prediction in skew-product systems with high numerical stability over long stepsizes. We demonstrate with numerical experiments in periodic and non-periodic flows that the skill of these forecasts compares well against Monte Carlo forecasts with the perfect model, despite our models utilizing no prior information about the equations of motion, and without having access to explicit tracer trajectories.

The plan of this paper is as follows. In Section~\ref{secKoopman}, we present our operator-theoretic framework for identification and prediction of coherent patterns. In Section~\ref{secExamples}, we illustrate this framework in periodic flows with Gaussian streamfunctions where the matrix elements of the generator can be evaluated analytically   (i.e., without the need for a data-driven basis). In Section~\ref{secDataDrivenBasis}, we describe the construction of our data-driven basis and the implementation of the techniques from Section~\ref{secKoopman} in that basis. In Section~\ref{secL96}, we present applications to a class of aperiodic tracer flows driven by Lorenz 96 (L96) systems \cite{QiMajda16}. The paper ends in Section~\ref{secConclusions} with concluding remarks and perspectives on future work. Proofs of some of the Lemmas in the main text, details on numerical implementation, and a discussion on the spectral properties of one of the flows studied in Section~\ref{secExamples} are included in Appendices. Movies illustrating the spatiotemporal evolution of Koopman eigenfunctions, observables, and probability densities are provided as supplementary online material.

\section{\label{secKoopman}Operator-theoretic framework for skew-product systems}

\subsection{\label{secPrelim}Notation and preliminaries}

Consider a continuous-time ergodic dynamical system $ (A, \mathcal{ B }( A ), \Phi_t, \alpha ) $ operating in a closed (i.e., smooth, compact, orientable, and boundaryless) manifold  $ A $ equipped with its Borel $ \sigma $-algebra $ \mathcal{ B }( A ) $ under the smooth map $ \Phi_t : A \mapsto A $, $ t \in \mathbb{ R } $, preserving a smooth probability measure $ \alpha $. This map is generated by a complete vector field $ u : C^\infty( A ) \mapsto C^\infty( A ) $ such that $ u( f ) = \lim_{t \to 0} ( f \circ \Phi_t - f ) / t $, and the divergence $ \divr_\alpha u $ with respect to $ \alpha $ is identically zero. Consider also a closed Riemannian manifold $ ( X, g_X ) $ equipped with a smooth metric $ g_X $ and a smooth probability measure $ \xi $ on its Borel $ \sigma $-algebra $ \mathcal{ B }( X ) $. In what follows, $ (A, \mathcal{ B }( A ), \Phi_t, \alpha ) $ will be the dynamical system governing the evolution of a time-dependent fluid flow and $ (X, \mathcal{ B }( X ), g_X, \xi ) $ the physical space where this flow takes place. In particular, we consider that there exists a mapping $ F : A \mapsto \mathfrak{ X } $ sending $ A $ to the vector space $ \mathfrak{ X } $ of $ C^\infty $ vector fields on $ X $, divergence-free with respect to $ \xi $,  such that $ v\rvert_a = F( a ) $ is the velocity field corresponding to the state $ a \in A $, and $ \divr_\xi v\rvert_a$ vanishes for all $ a \in A$.  We equip this space with the Hodge inner product,  $ \langle \beta_1, \beta_2 \rangle_{g_X} = \int_X g_X( \beta_1, \beta_2 ) \, \dvol_{g_X} $, $ \beta_1, \beta_2 \in \mathfrak{ X } $. We also assume that $ F $ is an embedding of $ A $ into $ \mathfrak{ X } $ (i.e., it possesses a smooth inverse on $ F( A ) $), so that an observation $ v\lvert_a \in F( A ) $ provides complete information about the underlying state $ a \in A $. 

Since $ F  $ is an embedding, $ A $ inherits a smooth Riemannian metric $ g_A $ such that $ g_A( \beta_1, \beta_2 ) = \langle F_* \beta_1, F_* \beta_2 \rangle_{g_X} $, where $ F_* : TA \mapsto T \mathfrak{ X } \simeq \mathfrak{ X } $ is the pushforward map on tangent vectors associated with $ F $. Moreover, because $ F $ is smooth and $ A $ and $ X $ are compact, there exists a unique map $ \Psi : \mathbb{ R } \times A \times X  \mapsto X $ such that for every $ a\in A $ and $ x \in X $, $ \gamma_{a,x}( t ) = \Psi( t, a, x ) $ is a curve on $ X $ defined for all $ t \in \mathbb{ R } $, and $ F( \Phi_t( a ) ) $ is a tangent vector to that curve at the point $ \Psi( t, a, x ) $. The curve $ \gamma_{a,x} $ corresponds to the path of  a passive tracer released from $ x \in X $ and advected by the time-dependent velocity field $ F( \Phi_t( a ) ) $. In what follows, we use the shorthand notation $ \Psi_t = \Psi( t, \cdot, \cdot ) $. Note that $ \Psi_t $ satisfies the cocycle property, $\Psi_s( \Phi_t( a ), \Psi_t( a, x ) ) = \Psi_{s+t}( a,x ) $, for all $ a \in A $, $ x \in X $, and $ s, t \in \mathbb{ R } $. 

Let now $ \tau > 0 $ be a fixed sampling interval such that the discrete-time dynamical system $ ( A, \mathcal{ B }( A ), \hat \Phi_n, \alpha ) $ with $ \hat \Phi_n = \Phi_{n\tau} $, $ n \in \mathbb{ Z } $, is also ergodic. We consider that we have available a dataset $ \{ v_0, v_1, \ldots, v_{N-1} \} $ consisting of time-ordered observations $ v_n = F( a_n ) $ of the velocity field corresponding to the states $ a_n = \hat \Phi_n( a_0 ) $ with $ n \in \{ 0, 1, \ldots, N - 1 \} $. Given such a dataset, and assuming no availability of explicit tracer trajectories, or prior knowledge of the evolution law $ \Phi_t $ and the structure of the underlying state space $ A $, in what follows we present techniques to (i)  identify coherent spatiotemporal patterns associated with the motion of passive tracers in $ X $ (Section~\ref{secCoherent}); (ii) predict the evolution of observables and probability densities defined on the tracers (Section~\ref{secPrediction}).

\begin{exmp}
  \label{exStream}
  Let $ A = \mathbb{ T }^1 $ and $ \Phi_t $ be a rotation with frequency $ \omega $, i.e., $ \Phi_t( a ) = a + \omega t \mod 2 \pi $ (note that we abuse notation by using the symbol $ a $ to represent both a point in $ A $ and its canonical angle coordinate). Let also  $ X = \mathbb{ T }^2 $ be a periodic two-dimensional domain equipped with the canonical (flat) metric $ g_X = dx_1 \otimes dx_1 + dx_2 \otimes dx_2 $, where $ x_1 $ and $ x_2 $ are canonical angle coordinates on $ \mathbb{ T }^2 $.  We also consider that $ X $  is equipped with the normalized Haar measure $ \xi = \vol_{g_X} / ( 2 \pi )^ 2 = dx_1 \wedge dx_2 / ( 2 \pi )^2 $ such that $ \xi( X ) = 1 $. In this setting, any smooth function $ \zeta : A \mapsto C^\infty( X ) $ gives rise to a time-periodic incompressible velocity field with period $ 2 \pi / \omega $, given in the  $ \{ x_1, x_2 \} $ coordinates by  
  \begin{equation}
    \label{eqStream}
    v\rvert_a = F( a ) = v^{(1)}( a) \partial_1 + v^{(2)}(a ) \partial_2, \quad v^{(1)}( a ) = - \partial_2 \zeta( a ), \quad  v^{(2)}( a ) =  \partial_1 \zeta(a), 
  \end{equation}   
  where $ \partial_j = \frac{ \partial \; }{ \partial x_j } $. That is, $ \zeta $ is a time-periodic streamfunction, and $ \divr_\xi v = \partial_1 v^{(1)}  + \partial_2 v^{(2)} $ is identically zero. Note that~\eqref{eqStream} can be expressed in coordinate-free form using exterior calculus \citep[][Section~8.2]{AbrahamEtAl88}, namely $ v\rvert_a = ( \delta \star \zeta( a ) )^\sharp $, where $ \star $ is the Hodge-star operator mapping functions to 2-forms, $ \delta $  the codifferential mapping 2-forms to 1-forms, and  $ ( \cdot )^\sharp $ the Riemannian dual of 1-forms. It should also be noted that not every time-periodic, divergence-free velocity field in $ \mathfrak{X} $ is of the class in~\eqref{eqStream}; in particular, \eqref{eqStream} does not include harmonic vector fields with vanishing vector Laplacian, $ \upDelta_\mathfrak{X} v = 0 $. Harmonic vector fields cannot be expressed in terms of a continuous streamfunction, but in the case of the 2-torus domain $ X $ driven by the periodic dynamics on $ A $ they take the form of a state-dependent (hence, time-periodic) free stream $ \chi_1( a)  \partial_1 + \chi_2( a ) \partial_2 $, where $ \chi_1 $ and $ \chi_2 $ are real-valued functions in $ C^\infty( A ) $. For the class of vector fields in~\eqref{eqStream}, the Riemannian metric induced on $ A $ by the embedding $ F $ is given in the canonical angle coordinate of $ a $ of $ A $ by  $ g_A =  \tilde g_A \, da \otimes da $, where 
\begin{equation}
  \label{eqStreamMetric}
  \tilde g_A = \left \langle \frac{ \partial F }{ \partial a }, \frac{ \partial F }{ \partial a } \right \rangle_{g_X}  = \int_{x_1,x_2=0}^{2\pi} \left[ \left( \frac{ \partial^2 \zeta }{ \partial a \, \partial x_1 } \right)^2 + \left( \frac{ \partial^2 \zeta }{ \partial a \, \partial x_2 } \right)^2 \right] \, dx_1 \, dx_2.
\end{equation}
We will return to this class of time-periodic flows in Section~\ref{secExamples}, where harmonic vector fields will also appear after a Galilean transformation. This example can be generalized to nonperiodic incompressible flows on nonperiodic domains by appropriate modifications of $ A $ and $ X $. 
\end{exmp}

\begin{rk} \label{rkAssumptions} A number of the assumptions stated above can be potentially  relaxed. In particular, it should be possible to extend our framework from smooth manifolds to more general topological spaces by replacing Laplace-Beltrami operators and other operators that depend on a Riemannian geometry by appropriate Hilbert-Schmidt kernel integral operators \cite{DasGiannakis17}. Such an extension is beyond the scope of the present work, although in what follows we provide justification (Remark~\ref{rkEpsilon}) and numerical evidence (Section~\ref{secL96}) that our framework is applicable when $ A $ is not a smooth manifold. Moreover, the embedding assumption on $ F $ can be generically relaxed by performing delay-coordinate maps \citep[][]{SauerEtAl91} on the velocity field data, and using that data as representatives of the underlying dynamical states in $ A $.   
\end{rk}

\subsection{Extended state space and the associated Koopman and Perron-Frobenius operators}

Our approach for identifying and predicting coherent patterns relies on data-driven approximations of Koopman and Perron-Frobenius operators characterizing the evolution of observables and probability measures, respectively, on the product manifold $ M = A \times X $. Intuitively, we think of $ M $ as a ``space-time manifold'' with $ X $ playing the role of physical space where the motion of tracers takes place and $ A $ the role of ``time'' where the velocity field evolves. On $ M $, the dynamics is autonomous (though not necessarily ergodic), and is governed by the flow $ \Omega_t : M \mapsto M $, $ t \in \mathbb{ R } $, having the skew-product form $ \Omega_t( a, x ) = ( \Phi_t( a ), \Psi_t( a, x ) ) $. This flow preserves the product measure $ \mu = \alpha \otimes \xi $ defined on the Borel $ \sigma $-algebra $ \mathcal{ B }( M ) = \mathcal{B}(A) \otimes \mathcal{B}(X) $ on  $M $. In what follows, we will employ the usual notation $ \langle f_1, f_2 \rangle = \int_M f_1^* f_2 \, d\mu $ for the inner product between complex-valued functions $ f_1, f_2 \in L^2(M,\mu ) $, and use whenever convenient the abbreviation $ H = L^2( M, \mu ) $. We also abbreviate $ H_A = L^2( A, \alpha )$ and $ H_X = L^2( X, \xi ) $, and use  $ \langle \cdot, \cdot \rangle_{H_A} $ and $ \langle \cdot, \cdot \rangle_{H_X} $ to denote the inner products of these spaces (defined analogously to $ \langle \cdot, \cdot \rangle $), respectively.

Given an arbitrary smooth Riemannian metric $ g $ on $ M $, we use the notation $ \langle \beta_1, \beta_2 \rangle_g = \int_M g( \beta_1^*, \beta_2 ) \, d\mu $ to represent the weighted Hodge inner product associated with $ g $ and $ \mu $ on vector fields $ \beta_1, \beta_2 \in L^2 TM( g, \mu ) $. We also denote the canonical projection maps from $ M $ into its factors by $ \pi_A : M \mapsto A $ and $ \pi_X : M \mapsto X $. Due to the product structure of $ M $, we have the decomposition $ T M = T^A M \oplus T^X M $, $ T^A M = \ker \pi_{X*} $, $ T^X M = \ker \pi_{A*} $, where $ \pi_{A*} : TM \mapsto TA $ and $ \pi_{X*} : TM \mapsto TX $ are the pushforward maps on tangent vectors associated with $ \pi_A $ and $ \pi_X $, respectively. In particular, every tangent vector $ \beta \in T_{(a,x)} M $ can be uniquely decomposed as $ \beta = \beta^A + \beta^X $ with $ \beta^A \in T^A_a M $ and $ \beta^X \in T^X_x M $.   

Given a point $  (a, x ) \in M $, the curve $ \Gamma_{a,x}( t ) = \Omega_t( a, x ) $ describes the joint evolution of the base dynamics initialized at state $ a \in A $ and the passive tracer released at the point $ x \in X $ and advected by the time-dependent velocity field $ F( \Phi_t( a ) ) $. Moreover, tangent to the family of curves $ \{ \Gamma_{a,x} \}_{(a,x)\in M }$ is a vector field $ w : C^\infty( M ) \mapsto C^\infty( M ) $ such that $ w( f ) = \lim_{t\to0} ( f \circ \Omega_t - f ) / t $. Due to the skew-product structure of $ \Omega_t $, we have $ w = w^A + w^X  $ where $ w^A \in T^A M $ and $ w^X \in T^X M  $ are the canonical lifts of $ u $ and $ v $ on $ M $, respectively; that is, $  w^A( f )( a, x ) = \lim_{t\to0} ( f( \Phi_t( a ), x ) - f( a, x ) ) / t $ and $ w^X( f )( a, x ) = \lim_{t\to0}( f( a, \Psi_t( a, x ) ) - f( a, x ) ) / t $. Because $ \Omega_t $ preserves $ \mu $, $ w $ has vanishing divergence with respect to that measure, $ \divr_\mu w = 0 $. Intuitively, one can think of $ w $ as an analog of the material derivative $ \frac{ D\; }{ D t } $ in fluid dynamics in $ \mathbb{ R }^d $ measuring the rate of change of observables on fluid parcels.  Noting that $ \frac{ D\; }{ D t} = \frac{ \partial\; }{ \partial t } + \vec v \cdot \vec \nabla $ (where $ \vec v \cdot \vec \nabla $ is the directional derivative associated with a time-dependent vector field $ \vec v $ in $ \mathbb{ R }^d $), we can also identify $ w^A $ with $ \frac{ \partial\; }{ \partial t } $ and $ w^X $ with $ \vec v \cdot \vec \nabla $.     

Next, we introduce Koopman and Perron-Frobenius operators governing the evolution of observables and probability densities, respectively, on $ A $ and $M $. In the case of $ A $, we consider complex-valued observables in $  H_A $, and define as usual the group of Koopman operators $ U_t :  H_A \mapsto H_A $, $ t \geq 0 $,  acting on observables by composition with $ f $, i.e., $ U_t f = f \circ \Phi_t $  \citep[][]{EisnerEtAl15,Mezic05}. Since $ \alpha $ is an invariant measure, $ U_t $ is unitary, $ U^{-1}_t = U_t^* $, and $ U_t^* $ is a Perron-Frobenius (transfer) operator governing the evolution of complex measures on $ \mathcal{B}(A)  $ with densities in $ H_A $; i.e., if $ \nu $ is such a measure with density $ \rho = d\nu/d\alpha $, $ U_t^* \rho $ is the density of the measure $ \nu \circ \Phi_t^{-1} $ with respect to $ \alpha $. By Stone's theorem, the unitary group $  \{ U_t \}_{t\in\mathbb{R}} $ is generated by a skew-adjoint operator $ \tilde u : D( \tilde u ) \mapsto H_A $ with dense domain  $ D( \tilde u ) \subset H_A $, such that $ \tilde u( f ) = \lim_{t\to 0} ( U_t( f ) - f ) / t $, and $ \tilde u( f ) = u( f ) $ if $ f \in C^\infty( A ) $ (i.e., $ \tilde u $ is a skew-adjoint extension of $ u $ viewed as a map from $ C^\infty( A ) $ to $ H_A $). Moreover, there exists a unique spectral measure $ E_{\tilde u} : \mathcal{ B }( \mathbb{ R } ) \mapsto P(H_A) $, where  $ \mathcal{ B }(\mathbb{ R }) $ is the Borel $\sigma $-algebra on $ \mathbb{ R } $ and $ P(H_A ) $ the set of orthogonal projections on $H_A$, such that $ \tilde u =  \int_\mathbb{R } \ii \omega \, dE_{\tilde u}( \omega ) $ and $ U_t =  \int_\mathbb{R} e^{\ii \omega t } \, dE_{\tilde u}( \omega ) $. Similarly, the Perron-Frobenius group $ \{ U^*_t \}_{t\in\mathbb{R}} $ is generated by $ - \tilde u $, and we have $ U^*_t =  \int_\mathbb{R} e^{- \ii \omega t } \, dE_{\tilde u}( \omega ) $. 

We employ an analogous construction in the case of the dynamics on the product space $ M $. In this case, we consider complex-valued observables in the Hilbert space $H  $ associated with the invariant product measure $ \mu $. When studying individual dynamical trajectories on  $M $, we can think of functions $ f : M \mapsto \mathbb{ C } $ (and thus equivalence classes of functions in $ H $) as spatiotemporal patterns. Specifically, given an initial state $ a \in A $, the spatiotemporal pattern $ \tilde f_a( t, x ) $ associated with $ f  $ is given by $ \tilde f_a( t, x ) = f( \Phi_t a, x ) $. In this setting, the group of Koopman operators $ W_t :  H \mapsto H $, $ t\in \mathbb{ R }  $, acts on $ H $ via composition with the skew-product map,  $ W_t f = f \circ \Omega_t $, so that $ W_t f(a, x ) $ corresponds to the value of $ f $ at the point reached at time $ t $ by a passive tracer released at the point $ x \in X $ and advected by the time-dependent velocity field $ F( \Phi_t( a ) ) $. Similarly, the group of Perron-Frobenius operators $ W_t^* $ governs the evolution of densities of measures on  $ \mathcal{ B }( M ) $ under $ \Omega_t $, so that if $ \nu $ is a measure with density $ \rho = d\nu/d\mu \in H $, $ \rho_t = W_t^* \rho $ is the density of $ \nu_t = \nu \circ \Omega_t^{-1} $. In the case that $ \nu $ is a probability measure, $ W_t^* \rho $ corresponds to a time-dependent probability density characterizing uncertainty with respect to both tracer position in $ X $ and the underlying state of the velocity field in $ A $. Moreover, the function $ \sigma_t = \int_A \rho_t( a, \cdot ) \, d\alpha $ is the density of a time-dependent probability measure on $ \mathcal{B}(X) $ which corresponds to a marginalization of $ \nu_t $ over $ A $. Note that unlike $ \rho_t $, $ \sigma_t $ does not evolve autonomously. 

Consider now the generator of the unitary group $ \{ W_t \}_{t\in\mathbb{R}} $. In direct analogy with $ \tilde u $, the generator $ \tilde w : D( \tilde w ) \mapsto L^2( M, \mu ) $ of  this group is a skew-adjoint extension of the vector field $ w $ defined on a dense domain $ D( \tilde w ) \subset H $, and the skew-adjointness of $ \tilde w $ is a consequence of the fact that $ \Omega_t $ preserves $ \mu $.  By Stone's theorem, we have
\begin{equation}
  \label{eqStoneW}
  \tilde w = \int_\mathbb{ R } \ii \omega \, dE_{\tilde w}( \omega ), \quad W_t  = \int_\mathbb{R} e^{\ii\omega t} \, dE_{\tilde w}( \omega ), \quad W^*_t  = \int_\mathbb{R} e^{-\ii\omega t} \, dE_{\tilde w}( \omega ), 
\end{equation}
where $ E_{\tilde w} : \mathcal{ B } ( \mathbb{ R } ) \mapsto P( H ) $ is a unique spectral measure associated with $ \tilde w $ taking values on the set $ P( H ) $ of orthogonal projections on $ H $. Thus, if $ \tilde w $ is known, it is possible to compute $ W_t $ and $ W^*_t $, and hence predict the evolution of arbitrary observables and measures for the skew-product tracer system.  
    
\subsection{\label{secCoherent}Identification of coherent spatiotemporal patterns}

\subsubsection{\label{secRaw}Coherent patterns as eigenfunctions of the generator}

We identify coherent spatiotemporal patterns in time-dependent fluid flows through  approximate eigenfunctions of the Koopman generator $ \tilde w $ at small corresponding eigenvalues. To motivate our approach, consider the eigenvalue problem $ \tilde w( z ) = \lambda z $, and suppose that this problem has a nonconstant solution $ z \in C^\infty( M )  $ at eigenvalue $ \lambda = 0 $ (note that any $ \mu $-a.e.\ constant function is an eigenfunction of $  \tilde w $ at eigenvalue zero). Because $ W_t z = e^{t\lambda} z = z $ (by \eqref{eqStoneW}),  such an eigenfunction is preserved on Lagrangian tracers. Correspondingly, the level sets $ \zeta_c = \{ ( a, x ) \in M \mid z( a, x ) = c \} $ of $ z $ create an ergodic quotient \citep[][]{Mezic05,BudisicEtAl12} of $ M $; that is, a partition into codimension 1 invariant sets on which tracers on $ M $ are trapped. More specifically, if $ ( a, x ) \in \zeta_c $, then $ \Omega_t ( a, x ) \in \zeta_c $ for all $ t \in \mathbb{ R } $ since $ z( \Omega_t( a, x ) ) = z( a,x ) $. 

In the case that $ \tilde w $ has eigenfunctions with nonzero corresponding eigenvalues, then these eigenfunctions also provide a useful notion of coherent spatiotemporal patterns that vary periodically on the tracers. In particular, because $ \tilde w $ is skew-adjoint, its eigenvalues are purely imaginary, and $ \omega = \Imag( \lambda) $ measures an intrinsic oscillatory frequency associated with eigenfunction $ z $. It therefore follows from~\eqref{eqStoneW} that $ W_t z( a, x ) = e^{\ii \omega t} z( a, x ) $, or, equivalently, $ z \circ \Omega_t = R_t \circ z $, where $ R_t $ is the rotation map on the complex plane such that $ R_t( z ) = e^{\ii \omega t } z $. In this case, the sets $ \zeta_c $ are not invariant, but the collection $\{  \zeta_c \}_{c \in \mathbb{ C }} $ (which forms a partition of $ M $) is invariant in the sense that $ \Omega_t \zeta_c = \zeta_{\Phi_t c} $. Such a partition is known as a harmonic partition \citep[][]{BudisicEtAl12}.   

\begin{rk}[\label{rkGroup}Group structure of Koopman eigenvalues and eigenfunctions]  Given any two eigenfunctions $ z, z' \in C^\infty( M ) $ of $ \tilde w $ with corresponding eigenvalues $ \lambda $ and $ \lambda' $, it follows from the Leibniz rule that $ \tilde w( z z' ) = \tilde w( z ) z' + z \tilde w( z ) = ( \lambda + \lambda' ) z z' $, so that $ z z' $ is an eigenfunction at eigenvalue $ \lambda + \lambda' $. Thus, the sets of such eigenvalues and eigenfunctions form commutative groups under addition of complex numbers and multiplication of smooth functions, respectively. In particular, one can generate these sets from any maximal subset of rationally independent eigenvalues and their corresponding eigenfunctions. The number of elements of such sets is at most equal to the dimension of $ M $.    
\end{rk}

The preceding discussion suggests that a useful definition for coherent spatiotemporal patterns on $ M $ is through eigenfunctions of $ \tilde w $ at zero or small eigenvalue. However, despite its simplicity and attractive theoretical properties, this definition is likely to suffer from at least two major shortcomings, namely that (i) for systems of sufficient complexity (i.e., weak mixing systems) $ \tilde w $ will have no nonconstant eigenfunctions, and (ii) even if $ \tilde w $ has nonconstant eigenfunctions, the set of these eigenfunctions will generally include functions of high roughness which will adversely affect the conditioning of numerical schemes for the eigenvalue problem for $ \tilde w $. 
\begin{exmp}
  \label{exRot}To illustrate the latter point, suppose that the dynamics on $ A $ and $ X $ were both circle rotations with $ \Phi_t( a ) = a + t \mod 2 \pi $, $ \Psi_t( a, x ) = x + \beta t \mod 2 \pi  $, and $ \beta $ an irrational number. Then, $ M = A \times X $ is a 2-torus, and the dynamics $ \Omega_t $ is ergodic for the Haar measure. Moreover, the eigenvalues and eigenfunctions of $ \tilde w $ are given by $ \lambda_{ij} = \ii \omega_{ij} = \ii ( i + j \beta ) $ and $ z_{ij}( a, x ) = e^{\ii ( i a + j x ) } $, respectively, with $ i, j \in \mathbb{ Z } $. Because $ \beta $ is irrational, the set $ \{ \omega_{ij} \}_{i,j\in\mathbb{Z}} $ is dense in the real line; that is, given an eigenfrequency $ \omega_{ij} $ (including $ \omega_{00} = 0$) and a positive number $ \delta $, there exist infinitely many eigenfrequencies in the interval $ ( \omega_{ij} - \delta, \omega_{ij}  + \delta ) $. Given, however,  any smooth Riemannian metric $ g $ on $ M $, most of these frequencies will correspond to highly oscillatory eigenfunctions with large gradient. For instance, if $ g $ is the flat metric on $ \mathbb{ T }^2 $, then $ \lVert \grad_g z_{ij} \rVert_g = ( i^2 + j^2 )^{1/2} $, and while $ \omega_{ij} $ can be made arbitrarily close to zero using arbitrarily large (and opposite signed) $ i, j \in \mathbb{ Z } $, the gradient of the corresponding eigenfunctions will be arbitrarily large.   
\end{exmp}
The behavior in Example~\ref{exRot} is typical of other systems too, and renders the approximation of the eigenvalues and eigenfunctions of the raw generator by numerical algorithms problematic---in particular, the classical spectral Galerkin methods for eigenvalue problems on Hilbert spaces \cite{BabuskaOsborn91} require certain coercivity (ellipticity) properties to hold which are not satisfied by skew-adjoint operators such as $ \tilde w $. Thus, as done elsewhere in the Koopman and Perron-Frobenius operator literature \cite{DellnitzJunge99,Froyland13,GiannakisEtAl15,Giannakis17,DasGiannakis17,FroylandKoltai17}, we will regularize the generator by adding a small amount of diffusion that renders the spectrum discrete and its approximation by spectral Galerkin methods well-posed.  

\subsubsection{\label{secLB}Laplace-Beltrami operators on $ M $ and the associated eigenfunction bases}

We employ the framework developed in \cite{GiannakisEtAl15,Giannakis17,DasGiannakis17} which involves regularizing the generator by means of the Laplace-Beltrami operator $ \upDelta : C^\infty(M) \mapsto C^\infty( M) $ associated with a Riemannian metric $ h $ on $ M $ chosen such that it has compatible volume form to the invariant measure, $ \vol_h = \mu $. This operator is defined as $ \upDelta = - \divr_h \grad_h $, where $ \grad_h : C^\infty( M ) \mapsto TM $ and $ \divr_h : TM \mapsto C^\infty( M ) $ are the gradient and divergence operators associated with $ h $, respectively. Note that since $ \vol_h = \mu $, we have $ \divr_h = \divr_\mu $ and for any function $ f \in C^\infty( M ) $ and any vector field $ \beta \in TM $, $ \langle f, \divr_\mu \beta \rangle = - \langle \grad_h, \beta \rangle_h $. As a result, $ \upDelta $ is a symmetric operator, and it is a standard result in analysis on Riemannian manifolds \cite{Rosenberg97} that there exists an orthonormal basis of $ H $ consisting of eigenfunctions $ \{ \phi_k \}_{k=0}^\infty $ of (an appropriate self-adjoint extension of) $ \upDelta $ corresponding to the eigenvalues $ \{ \eta_k \}_{k=0}^\infty $, where $ 0 = \eta_0 < \eta_1 \leq \eta_2 \leq \cdots \nearrow \infty $. An important property of these eigenfunctions is that they are stationary points of the Rayleigh quotient $ R( f ) = E( f ) / \lVert f \rVert^ 2 $ associated with the Dirichlet energy functional $ E( f ) = \int_A \lVert \grad_h f \rVert^2 \, d\mu $, and $ R( \phi_k ) = E( \phi_k ) = \lambda_k $. Intuitively, the Dirichlet energy provides a measure of roughness of functions on $ M $ associated with the Riemannian metric $ h $ and the invariant measure $ \mu $. Fixing a parameter $ \ell > 0 $, the set $ \{ \phi_0, \ldots, \phi_{\ell-1} \} $ has minimal roughness with respect to $ E $ among all $ \ell$-element orthonormal sets in $ H  $. In what follows, we will use the eigenfunctions $ \{ \phi_k \} $, or (in Section~\ref{secDataDrivenBasis}) approximate eigenfunctions computed from data via kernel algorithms, to construct Galerkin approximation spaces for a regularized generator (to be defined in Section~\ref{secReg}).       

Let $ \sigma_A = d \alpha / \dvol_{g_A} \in C^\infty( A ) $ and $ \sigma_X = d \xi / \dvol_{g_X} \in C^\infty( X ) $ be the densities of the probability measures $ \alpha $ and $ \xi $ relative to the Riemannian measures associated $ g_A $ and $ g_X $, respectively. In particular, $ \sigma_A  $ represents the sampling density on $ A $ with respect to the invariant measure of the dynamics and the observation map $ F $. Let also $ m_A $ and $ m_X $ be the dimensions of $ A $ and $ X $, respectively. In the  setting of interest involving the $ m $-dimensional product space $ M = A \times X $ with $ m = m_A + m_X $, it is natural to construct the metric $ h $ as the sum $ h = \pi_{A*} h_A  +  \pi_{X*} h_X  $, where $ \pi_{A*} h_A $ and $ \pi_{X*} h_X $ are pullbacks of the conformally transformed metrics $ h_A = \sigma_A^{2/m_A} g_A $ and $ h_X = \sigma_X^{2/m_X} g_X $ on $ A $ and $ X $, respectively. By construction, $ \vol h_A = \alpha $ and $ \vol h_X = \xi $, and moreover for any $ f \in C^\infty( M  ) $ we have $ \grad_h f = \grad_h^A f + \grad_h^X f $, where $ \grad_h^A f $ and $ \grad_h^X f $ are unique vector fields in $ T^A M $ and $ T^X M $, respectively. With this choice of metric, the Laplace-Beltrami operator on $ M $ becomes $ \upDelta =  \upDelta'_A +  \upDelta'_X $, where $ \upDelta'_A = - \divr_\mu \grad_h^A $  and $ \upDelta'_X = - \divr_\mu \grad_h^X $. Note that $ \upDelta'_A $ and $ \upDelta'_X $ are lifts of the Laplace-Beltrami operators $ \upDelta_A : C^\infty( A ) \mapsto C^\infty( A ) $ and $ \upDelta_X : C^\infty( X ) \mapsto C^\infty( X ) $ associated with $ h_A $ and $ h_X $, respectively,  in the sense that $ \upDelta'_A ( f_A \circ \pi_A ) = ( \upDelta_A f_A ) \circ \pi_A  $ and $ \upDelta'_X ( f_X \circ \pi_X ) = ( \upDelta_X f_X ) \circ \pi_X $ for any $ f_A \in C^\infty( A ) $ and $ f_X \in C^\infty( X ) $.  Moreover, $ \upDelta'_A $ and $ \upDelta'_X $ commute, and the eigenfunction basis associated with $ \upDelta $ has the tensor product form $ \phi_k = ( \phi^A_{i(k)} \circ \pi_A ) ( \phi^X_{j(k)} \circ \pi_X ) $, where $ \phi^A_{i(k)} $ and $ \phi^X_{j(k)} $ are eigenfunctions of $ \upDelta_A $ and $ \upDelta_X $, respectively. Here, the index ordering $ i(k )$, $ j( k) $ is chosen such that $ i( 0 ) = j( 0 ) = 0 $ and the Dirichlet energies $ E( \phi_k ) $ form a non-decreasing sequence with $k $ in accordance with our convention. In particular, denoting the eigenvalues  corresponding to $ \phi^A_i $ and $ \phi^X_j $ by $ \eta_i^A $ and $ \eta_j^X $, respectively, we have $ E( \phi_k ) = \eta_k = \eta^A_{i(k)} + \eta^X_{j(k)} $, where $ \eta_k $ is the eigenvalue corresponding to $ \phi_k $. For later convenience, we quote here formulas relating the gradient, divergence, and Laplace-Beltrami operators associated with $ h_A $ to the corresponding operators with respect to the ambient space metric $ g_A$, viz., 
\begin{equation}
  \label{eqConf}
  \begin{gathered}
    \grad_{h_A} f = \sigma_A^{2/m_A} \grad_{g_A} f, \quad \divr_{h_A} \beta = \frac{ 1 }{ \sigma_A } \divr_{g_A} ( \sigma_A \beta ), \\
    \upDelta_A f = \sigma_A^{-2/m_A} \left[  \upDelta_{A,g} f - \left( 1 - \frac{ 2 }{ m } \right) g^{-1}( d\log\sigma_A, df ) \right], 
  \end{gathered}
\end{equation}
where $ g^{-1} $ is the ``inverse metric'' associated with $ g $, and $ f $ and $ \beta $ are an arbitrary smooth function and a vector field on $ A $, respectively.  

\begin{exmp}\label{exStream2}In the case of the periodic flows discussed in Example~\ref{exStream}, the sampling density and the conformally transformed metric on $ A $ are given by $ \sigma_A = 1 / ( 2 \pi \tilde g_A^{1/2} ) $ and $ h_A = 1/ ( 2 \pi )^2 da \otimes da $, where $ a  $ is a canonical angle coordinate on $ A $. Moreover, since $ X $ is equipped with the canonical flat metric, the density $ \sigma_X $ is uniform, $ \sigma_X = 1 / ( 2 \pi )^2 $, and the metric $ h_X $ is also flat; that is, $ h_X = ( dx_1 \otimes dx_1 + dx_2 \otimes dx_2 ) / ( 2 \pi )^2 $, where $ x_1 $ and $ x_2 $ are canonical angle coordinates on $ X $.  With these Riemannian metrics, the Laplace-Beltrami eigenfunctions are the canonical Fourier functions, $ \phi_i^A(a) = e^{\ii i a } $ and $ \phi^X_j( x ) = e^{\ii ( j_1 x_1 + j_2 x_2 ) } $, where $i $, $ j_1 $, and $ j_2 $ are integers, and $ j = ( j_1, j_2 ) $. The corresponding eigenvalues are $ \eta^A_i = i^2 $ and $ \eta^X_j = j_1^2 + j_2^2 $ so that $ \eta_k = i^2 + j_1^2 + j_2^2 $. 
\end{exmp}

Besides $ H $, in what follows we will make use of Sobolev spaces \cite{Rosenberg97} on $M $ of higher regularity. In particular, we consider the space $ H^1 := H^1( M, h ) $ defined the completion of $ C^\infty( M ) $ with respect to the norm $ \lVert f \rVert_{H^1} = \sqrt{ \langle f, f \rangle_1 }_{H^1}$ induced by the inner product $  \langle f_1, f_2 \rangle_{H^1} = \langle f_1, f_2 \rangle + \langle \grad_h f_1, \grad_h f_2 \rangle_{h} $ with $ f, f_1, f_2 \in C^\infty( M ) $ and $ \langle \grad_h f_1, \grad_h f_2 \rangle_{h} = \int_M h( \grad_h f_1, \grad_h f_2 ) \, d\mu $. Note that the eigenfunctions $ \phi_k $ of $ \upDelta $ have unbounded norm $ \lVert \phi_k \rVert_{H^1} = ( 1 + \eta_k )^{1/2} $ with $ k $, and therefore $ \{ \phi_k \} $ is not a basis of $ H^1( M, h ) $. On the other hand, the rescaled eigenfunctions $ \varphi_k  $ defined as $ \varphi_0 = \phi_0 $ and $ \varphi_{k>0} = \phi_k / \eta_k^{1/2} $ have bounded $ H^1 $ norm, $ \Vert \varphi_0 \rVert_{H^1} = 1 $, $ \lVert \varphi_{k>0} \rVert_{H^1} = ( 1 + \eta_k^{-1} )^{1/2} $,  and therefore $ \{ \varphi_k \}_{k=0}^\infty $ is a basis of $ H^1 $. Moreover, for a function $ f = \sum_{k=0}^\infty c_k \varphi_k $ expanded in this basis, the Dirichlet energy with respect to $ h $ can be conveniently computed from the $ \ell^2 $ norm of the expansion coefficient sequence excluding $ c_0$; i.e., $ E( f ) = \sum_{k=1}^\infty \lvert c_k \rvert^2 $. In Section~\ref{secReg}, we will employ the $ \{ \varphi_k \} $ basis for spectral Galerkin approximation of the eigenvalues and eigenfunctions of our regularized Koopman generator. We will also make occasional use of the Sobolev space $H^2 := H^2(M,h) $ defined as the completion of $ C^\infty( M ) $ with respect to the norm $ \langle f, f \rangle_{H^2}  = \sqrt{ \langle f, f \rangle_{H^2} }$ induced by the inner  product $ \langle f_1, f_2 \rangle_{H^2} = \langle f_1, f_2 \rangle_{H^1} + \int_M \upDelta f_1 \, \upDelta f_2 \, d\mu $.     
\subsubsection{\label{secReg}Coherent spatiotemporal patterns with regularization}
   
Following \cite{GiannakisEtAl15,Giannakis17,DasGiannakis17}, we employ the Laplace-Beltrami operator $ \upDelta $ introduced in Section~\ref{secLB} to construct regularized generators $ L_+ : C^\infty( M ) \mapsto H $ and $ L_- : C^\infty( M ) \mapsto H $ defined as
\begin{equation}
  \label{eqL}
  L_{\pm} = \pm w - \theta \upDelta, \quad  \theta > 0.
\end{equation}
 \begin{prop}
   \label{propL}
  The operators $ L_\pm $ have the following properties:
  
  (i) They are dissipative, $ \Real \langle f,  L_\pm f \rangle \leq 0 $, for all $ f \in C^\infty( M ) $, and therefore have dissipative closures, $ \bar L_\pm  $. 

  (ii) They form a dual pair, in the sense that $ L_\pm^* $ extends $ L_\mp $; i.e., $ L_+ \subset L_-^* $ and $  L_- \subset L_+^* $.

  (iii) There exists a maximal, closed, dissipative operator $  L : D(  L ) \mapsto H $ with a maximal, dissipative adjoint such that
  \begin{displaymath}
    L_+ \subset \bar L_+ \subseteq  L \subset L_-^*, \quad L_- \subset \bar L_- \subseteq  L^* \subset L_+^*.
  \end{displaymath}

  (iv) Given a function $ f \in H^2 $, we have $ \lVert \bar L_{\pm} f \rVert \leq C \lVert f \rVert_{H^2} $, where $ C $ is a constant independent of $  f$; thus, $  H^2 \subseteq D( \bar L_+ ) \subseteq D(  L ) $ and $ H^2 \subseteq D( \bar L_-) \subseteq D( L^*) $.
 
  (v) The spectra $ \sigma(   L ) $ and $ \sigma(  L^* ) $ are purely discrete; i.e., $ \lambda \in \sigma(   L ) $ if and only if $ \lambda $ is an eigenvalue of $  L $, and  $ \lambda' \in \sigma( L ^* ) $ if and only if $ \lambda' $ is an eigenvalue of $ L^* $. Moreover, $ \lambda^* \in \sigma(   L^* )  $ if and only if $ \lambda \in \sigma(  L ) $, and if $ z $ is an eigenfunction of $  L $ at eigenvalue $ \lambda $, then $ z^* $ is an eigenfunction of $ L^* $ at eigenvalue $ \lambda^* $.   
\end{prop}

A proof of Proposition~\ref{propL} can be found in~\ref{appPropL}. It follows from Proposition~\ref{propL}(iii) and a classical result due to Phillips~\cite{Phillips59} that  $  L $ is the generator of a strongly continuous contraction semigroup $ \{ S_t \}_{t\geq 0} $ on $ H $; that is, $ S_0 = I $, and for all $ s, t \geq 0 $,  $ S_s \circ S_t = S_{s+t} $, $ \lVert S_t \rVert \leq 1 $, and the map $ t \mapsto S_t $ is continuous in the strong operator topology. Similarly, $ L^* $ is the generator of the adjoint semigroup, $ \{ S^*_t \}_{t\geq 0} $, which is also contractive and strongly continuous.  Note that, in general, $ L $ and $ L^*$ are nonnormal operators,  and therefore are not guaranteed to be spectrally decomposable as $ \tilde w $ and $ \tilde w^* $ are. In Section~\ref{secPrediction} we will employ these semigroups for prediction of observables and probability densities on Lagrangian tracers, respectively, but for now we focus on using the spectral properties of $ L $ to recover coherent spatiotemporal patterns. In particular, the fact that  $ \sigma( L  ) $ consists only of eigenvalues, in conjunction with the discussion in Section~\ref{secRaw}, motivates the following definition of coherent spatiotemporal patterns for Lagrangian tracer systems driven by ergodic flows (and for more general skew-product systems):
\begin{defn}[Coherent spatiotemporal patterns]
  \label{defCoherent}
  For fixed $ \theta > 0 $, identify coherent spatiotemporal patterns $ \{ z_0, z_1, \ldots \} $ in $ H $ as the eigenfunctions of $ L $ with minimal Dirichlet energy with respect to the metric $ h $ introduced in Section~\ref{secLB}. That is, obtain $ z_k $ from the solutions of the problem $ L z_k = \lambda_k z_k $, $ \lVert z_k \rVert = 1 $, ordered such that $ 0 = E_0 < E_1 \leq E_2 \leq \cdots \nearrow \infty $, where the $ E_k $ are Dirichlet energies given by  
  \begin{displaymath}
    E_k = E( z_k ) = \int_M \lVert \grad_h z_k \rVert^2 \, d\mu = \sum_{i=0}^\infty \eta_i \lvert  \langle \phi_i, z_k \rangle \rvert^2.
  \end{displaymath} 
  The quantities $ \gamma_k = \Real \lambda_k $ and $ \omega_k = \Imag \lambda_k $ correspond to the growth rate (guaranteed non-positive by Proposition~\ref{propL}(iii)) and oscillatory frequency of eigenfunction $ z_k $.     
\end{defn}

Intuitively, for small $ \theta $ we think of the $ z_k $ from Definition~\ref{defCoherent} as approximate eigenfunctions of the raw generator $ \tilde w $. In the special case that $ \tilde w $ has a pure point spectrum, i.e., there exists a smooth orthonormal basis of $ H $ consisting of eigenfunctions $ \{ \tilde z_k \} $ of $ \tilde w $, that intuition can be justified using regular perturbation expansions \cite{Giannakis17}. In particular, it can be verified through heuristic asymptotics that for each $ \tilde z_k $ with corresponding eigenvalue $ \ii \tilde \omega_k $, there exists an eigenfunction $ z_k $ of $ L $ with corresponding eigenvalue $ \lambda_k = \gamma_k + \ii \omega_k $ and a smooth function $ z'_k $ orthogonal to $ \tilde z_k $ such that $ z_k = \tilde z_k + \theta z'_k + O( \theta^2 ) $, $ \gamma_k = - \theta E_k + O( \theta^2 ) $, and $ \omega_k = \tilde \omega_k + O( \theta^2 ) $. Thus, in systems with pure point spectrum, the main effect of diffusion in $ L $ is to suppress the pathological eigenfunctions of $ \tilde w $ with large Dirichlet energy (see Example~\ref{exRot}), and as $ \theta \to 0 $  the eigenvalues of $ L $ converge to the eigenvalues of $ \tilde w $. In fact, in systems with pure point spectrum one can additionally construct a metric $ h $ (using Takens delay embeddings) such that $ w $ and $ \Delta $ commute \cite{Giannakis17,DasGiannakis17}. In such cases, $ \tilde w $ and $ L $ have common eigenfunctions, and we have $ z_k = \tilde z_k $, $ \gamma_k = - \theta E_k $ and $ \omega_k = \tilde \omega_k $ without approximation. 

On the other hand, in systems with  mixed or continuous spectra (or rough Koopman eigenfunctions if the vector field $ w $ is non-smooth), the relationship between the spectra of $ \tilde w $ and $  L $ becomes significantly more complicated. At the very least, the fact that $  L $ always has eigenfunctions whereas  $ \tilde w $ may have no eigenfunctions (if the system is weak-mixing) raises questions about the physical significance of the coherent patterns defined according to Definition~\ref{defCoherent}. In fact, Constantin et al.~\cite{ConstantinEtAl08} show that even if $ \tilde w $ has eigenfunctions, but these eigenfunctions are not in $ H^1 $, a phenomenon known as dissipation enhancement occurs whereby adding $ O( \theta ) $ diffusion to $ w $ causes a greater than $ O( \theta ) $ change in the real part $ \gamma_k $ of the eigenvalues of $  L $ controlling the decay of the evolution semigroup $ \{ S_t \} $. More specifically, expressed in our notation, the results of \cite{ConstantinEtAl08} imply that if $ \tilde w $ has no eigenfunctions in $ H^1 $, then for any $ C > 0 $ there exists $ \theta_0 > 0 $ such that $ \lvert \gamma_k \rvert \geq C \theta $ for all $ \theta \leq \theta_0 $. In other words, phenomena such as dissipation enhancement suggest that in systems of sufficiently high complexity the choice of the diffusion operator in~\eqref{eqL} may influence significantly the properties of the coherent patterns recovered via eigenfunctions of $ L $. Here, besides this note for caution, we do not pursue the topic of the choice of $ \upDelta $  (which exists more broadly than our study of Lagrangian tracer systems) in the presence of mixed or continuous spectrum further. However, the experiments in Sections~\ref{secExamples} and~\ref{secL96} suggest our choice of $ \upDelta $ as the Laplace-Beltrami operator associated with the metric $ h $ leads to reasonable coherent patterns even in systems where the spectrum is not purely discrete.     

\subsubsection{\label{secGalerkin}Galerkin approximation}

We compute weak solutions to the eigenvalue problem for $L $ using a spectral Galerkin method. To pass to a weak (variational) formulation, we take the $ L^2( M, \mu ) $ inner product of both sides of the equation $ L z = \lambda z $, $ z \in C^\infty( M ) $, by an arbitrary test function $ \psi \in C^\infty( M ) $, integrate by parts the diffusion term, $ \langle \psi, \upDelta z \rangle = \langle \grad_h \psi, \grad_h z \rangle_h $, and introduce the sesquilinear forms 
\begin{equation}
  \label{eqSesquiAB}
  \begin{gathered}
    A( \psi, z ) = W( \psi, z )  - \theta E( \psi, z ), \quad W( \psi, z ) = \langle \psi, w( z ) \rangle \quad E( \psi, z ) = \langle \grad_h \psi, \grad_h z \rangle_h, \\
    B( \psi, z ) = \langle \psi, z \rangle. 
  \end{gathered}
\end{equation}
Note that $ E( \psi, z ) = E^*( z, \psi )$ is the Dirichlet form associated with the heat semigroup generated by $ \upDelta $, and that $ W(\psi, z ) = -W^*(z,\psi) $ by skew-symmetry of $ w $.  

By the Cauchy-Schwartz inequality and the fact that  $ \lVert \cdot \rVert \leq \lVert \cdot \rVert_{H^1} $,  these forms satisfy the bounds 
\begin{equation}
  \label{eqBounds}
  \lvert W( \psi, z ) \rvert \leq \lVert w \rVert_\infty \lVert \psi \rVert_{H^1} \lVert z \rVert_{H^1}, \quad E( \psi, z ) \leq \lVert \psi \rVert_{H^1} \lVert z \rVert_{H^1}, \quad B( \psi, z )  \leq \lVert \psi \rVert_{H^1} \lVert z \rVert_{H^1},
\end{equation}
where $ \lVert w \rVert_{\infty} = \max_{m\in M} \lVert w\rvert_m \rVert_h $. Therefore, $ A $ and $ B $ can be uniquely extended to bounded sesquilinear forms on $H^1  \times H^1 $. Moreover, one can verify that $ A $ has the coercivity property $ \Real A( z, z ) = -\theta E( z, z ) \leq - C  \lVert z\rVert_{H^1} $, where $ z $ is any function orthogonal to constant functions and $ C = \theta \eta_1  / ( 1 + \eta_1)^{-1} $. Together, the boundedness and coercivity of $ A $ and $B $ imply that the following variational eigenvalue problem is well posed \cite{BabuskaOsborn91}:  
\begin{defn}[variational formulation of eigenvalue problem for coherent patterns] \label{defEig} Find $ z \in H^1  $ and $ \lambda \in \mathbb{ C } $ such that for all $ \psi \in H^1 $ the equality $ A( \psi, z ) = \lambda B( \psi, z ) $ holds, where $ A $ and $B $ are the sesquilinear forms on $ H^1  \times H^1  $ defined in~\eqref{eqSesquiAB}. 
\end{defn} 

We compute approximate solutions of the eigenvalue problem in Definition~\ref{defEig} through a spectral Galerkin method expressed in finite-dimensional subspaces of $ H^1 $ spanned by the $ \{ \varphi_i \} $ basis introduced in Section~\ref{secLB}. In particular, introducing the $ \ell $-dimensional subspaces $ H^1_\ell = \spn\{ \varphi_0, \ldots, \phi_{\ell-1} \} \subset H^1 $, we seek solutions of the equation $ A( \psi, z ) = \lambda B( \psi, z ) $ for $ \lambda \in \mathbb{ C }$ and $ \psi, z \in H^1_\ell $. This is equivalent to solving the matrix generalized eigenvalue problem
\begin{equation}
  \label{eqEig}
  \boldsymbol{A} \vec c = \lambda \boldsymbol{B} \vec c,
\end{equation}  
where $ \vec c $ is an $ \ell $-dimensional complex column vector, and $ \boldsymbol{A} $ and $ \boldsymbol{B} $ are $ \ell \times \ell $ matrices with elements
\begin{gather*}
  A_{ij} = A( \varphi_i, \varphi_j ) = W_{ij} - \theta E_{ij}, \quad 
  B_{ij} = B( \varphi_i, \varphi_j ) = \langle \varphi_i, \varphi_j \rangle, \\
  W_{ij} = \langle \varphi_i, w( \varphi_j ) \rangle, \quad E_{ij} = E( \varphi_i, \varphi_j ) = \langle  \varphi_i, \varphi_j \rangle.
\end{gather*}
Given a generalized eigenvector $ \vec c = ( c_0,\ldots, c_{\ell-1} )^\top $ from~\eqref{eqEig}, the corresponding approximate eigenfunction of $ L $ is given by $ z = \sum_{i=0}^{\ell-1} c_i \varphi_i $. Moreover, as stated in Section~\ref{secLB}, the Dirichlet energy of this eigenfunction is given by $ E( z ) = \sum_{i=1}^{\ell-1} \lvert c_i \rvert^2 $. In accordance with Definition~\ref{defCoherent}, we order all such solutions $ ( \lambda_k, z_k ) $ in order of increasing Dirichlet energy $ E_k = E( z_k ) $, and normalize each solution to unit $ L^2 $ norm, $ \lVert z_k \rVert^2 = \lvert c_0 \rvert^2 + \sum_{i=1}^{\ell-1} \lvert c_i \rvert^2 / \eta_i = 1 $. Convergence of the approximate solutions to the solutions of the full eigenvalue problem in Definition~\ref{defEig} follows from spectral approximation theory for variational eigenvalue problems \cite{BabuskaOsborn91}. A discussion on the numerical solution of~\eqref{eqEig} can be found in \ref{appKoopEig}.  
 
\begin{rk}
  \label{rkDirichlet} Since $ W( f, f ) = 0 $ for every $ f \in H^1 $, we can determine the Dirichlet energies of the eigenfunctions $ z_k $ directly from the real parts of the corresponding eigenvalues, i.e., $ E_k = \gamma_k = \Real \lambda_k $. This is useful when solving~\eqref{eqEig} via iterative solvers with an option to compute subsets of eigenvalues with maximal real parts (e.g., the ARPACK library \cite{LehoucqEtAl98}, which is also utilized by our Matlab code). As discussed in~\ref{appKoopEig}, in practice we find that a more rapid numerical convergence occurs if we select eigenvalues with minimal moduli as opposed to maximal real parts. So long that a sufficiently large number of eigenfunctions is computed,  these two approaches should yield consistent results for the leading solutions of~\eqref{eqEig} ordered in order of increasing $ E_k $. Given the computational cost of solving~\eqref{eqEig} at high $ \ell $ (which is typical for the tensor product basis used in this work), hereafter we always compute eigenvalues and eigenfunctions using the minimal $ \lvert \lambda_k \rvert $ option. This approach carries a risk of failing to detect certain eigenvalues of high oscillatory frequency $ \omega_k $ but low Dirichlet energy, and indeed in Section~\ref{secExamples} this behavior will occur (although the missed eigenfunctions will turn out to be spatially constant on $ X $, and therefore trivial from the point of view of coherent spatiotemporal patterns).         
\end{rk}

Next, we comment on the numerical conditioning of our scheme. By construction of the $ \{ \varphi_i \} $ basis, $ [ E_{ij} ]_{ij} $ is an $ \ell \times \ell $ diagonal matrix with $ E_{00} = 0 $ and $ E_{ii} = 1 $ for $ i \geq 1 $. In particular, the operator norm of this matrix is independent of $ \ell $ (it is equal to 1 for $ \ell \geq 1 $). Moreover, it follows from the left-hand equation in~\eqref{eqBounds} and the fact that the $ \varphi_i $ have bounded $ H^1 $ norm that the norm of the $ \ell \times \ell $ matrix $ [ W_{ij} ]_{ij} $ can also be bounded by an $ \ell $-independent constant. These facts imply that the norm of $ \boldsymbol{ A } $ is bounded by an $ \ell $-independent constant. In the case of $ \boldsymbol{B} $, we can bound its condition number by $ C \ell^{2/m} $ where the constant $ C  $ is independent of $ \ell $, and $ m =\dim M $; this follows from the facts that $  \lVert \varphi_0 \rVert = 1 $, $ \lVert \varphi_{i>0} \rVert  = \eta^{-1/2}_i $, and Weyl's law $ \eta_i = O( i^{2/m} ) $ for the asymptotic growth of the Laplace-Beltrami eigenvalues $ \eta_i $ on a closed $ m $-dimensional manifold as $ i \to \infty $. Together, these properties of $ \boldsymbol{A} $ and $ \boldsymbol{B} $ ensure the good numerical conditioning of the eigenvalue problem in~\eqref{eqEig}. 

\subsection{\label{secPrediction}Prediction of observables and probability densities}

As stated in Section~\ref{secPrelim}, the Koopman and Perron-Frobenius groups of operators $ W_t $ and $ W_t^* $ respectively govern the evolution of observables and densities in the $ L^2 $ space associated with the skew-product tracer system. In this Section, we present a technique for approximating the action of these operators on observables and densities in an analytic or $ L^2 $ class. 

\subsubsection{\label{secBoundedObs}Computing the semigroup action for bounded analytic and $ L^2 $ observables}

Consider first the problem of computing $ W_t f $ for a general observable $ f \in H$. As in the case of the eigenvalue problem for coherent patterns (Section~\ref{secReg}), we relax this problem to the problem of computing the action $ S_t f $ of the contraction semigroup $ S_t $ generated by the regularized operator $ L $ in proposition~\ref{propL}. An advantage of using this semigroup is numerical stability due to the fact that $ L $ is dissipative; a disadvantage is that $ L $ is nonnormal  (unlike the raw generator $ \tilde w $), and therefore not necessarily decomposable via a spectral integral analogous to~\eqref{eqStoneW}. A difficulty common to both $ \tilde w $ and $ L $ is that they are unbounded operators, which means that  we cannot evaluate the corresponding $ W_t $ and $ S_t $ operators via series expansions for arbitrary observables in $ H $. To address this issue we restrict attention to a class of analytic observables with respect to $ L $ \citep[][Section~VIII.5]{ReedSimon80}. Specifically, introducing the space $ D^\infty( L ) = \bigcap_{n=0}^\infty D( L^n ) \subset H $, we consider the space $ B_c( L )  $, $ c \geq 0 $, consisting of all observables $ f \in D^\infty( L ) $ satisfying the bound $ \lVert L^n f \rVert \leq c_1 c^n \lVert f \rVert $ for all $ n \geq 0 $ and some constant $ c_1 >0 $. We also define $ B( L ) = \bigcup_{c\geq 0} B_c( L ) $. Observables in $ B( L )$ are known as bounded analytic vectors \cite{Nelson59,StochelSzafraniec97}. For such observables, the sum $ \sum_{n=0}^\infty t^n \lVert L^n f \rVert / n! $ converges for all $ t \geq 0$, and therefore the action of  the semigroup can be evaluated by means of the series expansion
\begin{equation}
  \label{eqSt}
  S_t f  = \sum_{n=0}^\infty \frac{ t^n }{ n! } L^n f. 
\end{equation} 

We now further restrict attention to bandlimited observables  in $ B( L ) $ which can be represented through finite expansions of the eigenfunctions $ z_k$ of $ L $. In particular, let $ Z_r = \{ z_0, z_1, \ldots, z_{r-1}\} $ be the $ r$-dimensional subspace of $H $ spanned by the leading $ r + 1 $ eigenfunctions of $ L $ (ordered in order of increasing Dirichlet energy in accordance with Definition~\ref{defCoherent}). Let also $ C_r = \max_{k\in\{0, \ldots, r\}} \lvert \lambda_k \rvert $. We then have:
\begin{lemma}
  \label{lemmaBL}
  (i) The subspace $ Z_r $ lies in $ B_{C_r}( L ) $.

  (ii) If the eigenfunctions $ z_0, z_1, \ldots $, form a basis of $  H $, and the corresponding eigenvalues $ \lambda_0, \lambda_1, \ldots $ have no accumulation points, then for every $ f \in B( L ) $ there exists $ r $ such that $ f $ lies in $ Z_r $.  
\end{lemma}
A proof of this Lemma can be found in~\ref{appBL}.                                                                                                                                                                            
Lemma~\ref{lemmaBL} implies that for  observables $ f \in Z_r$ we can evaluate the semigroup action $ S_t f $ via the series expansion in~\eqref{eqSt} for all $ t \geq 0 $; explicitly, we have
\begin{equation}
  \label{eqStFL}
  S_t f = \sum_{k=0}^{r-1} e^{t\lambda_k} c_k z_k, \quad f = \sum_{k=0}^{r-1} c_k z_k.
\end{equation}
Moreover, if the assumptions in part~(ii) of the Lemma hold, then~\eqref{eqStFL} holds for any observable in $ B( L ) $ (which is in turn a dense subspace of $ H $), again for for all $  t $. While we do not have results allowing us to determine when these assumptions hold for general $L $, there are specific cases where they are satisfied. In particular, if the Koopman generator $ \tilde w $ has pure point spectrum, then it can be arranged (e.g., using delay-coordinate maps \cite{Giannakis17,DasGiannakis17}) that $ w $ and  $ \upDelta $ commute; in that case, $ L $ has an orthonormal eigenbasis with no accumulation points in its eigenvalues (in fact, the eigenfunctions of $ L $ are also eigenfunctions of $ \tilde w$).   

Next, we consider how to compute or approximate~\eqref{eqSt} when $ f \in Z_r $ for some $ r $. First, it follows immediately from~\eqref{eqStFL} that if some eigenvalues and eigenfunctions  of $ L $ are available, say the first $ \ell  $, then one can approximate $ S_t f $,  $ f = \sum_{k=0}^{\ell-1} c_k z_k $, by $ S_{\ell,t}  f := \sum_{k=0}^{\ell-1} e^{t\lambda_k} c_k z_k $. This is equivalent to introducing the orthogonal projectors $ \Pi_\ell : H \mapsto Z_\ell $ and the spectrally truncated generator $ L_\ell = \Pi_\ell L \Pi_\ell $, and computing 
\begin{equation}
  \label{eqStFJ}
  S_{\ell,t} f  = \sum_{n=0}^\infty \frac{ t^n }{ n! } L_\ell^n f, 
\end{equation}
where it is obvious that $ S_{\ell,t} f = S_t f $ for all $ t \geq 0 $ if $ \ell \geq r $. 

An expansion analogous to~\eqref{eqStFJ} was used in \cite{Giannakis17} for prediction of observables of ergodic dynamical systems.   For the skew-product systems studied here, however, computing enough eigenfunctions of $ L $ so that~\eqref{eqStFJ} is sufficiently accurate may not be feasible (for instance, due to the large number of required basis elements in the tensor product basis $ \{ \varphi_k \}$). In such situations, one may opt instead to use orthogonal projections $ \tilde \Pi_\ell : H \mapsto H_\ell $ to the subspaces $ H_\ell =  \{ \phi_0, \ldots, \phi_{\ell-1} \} \subset H $ spanned by eigenfunctions of $ \upDelta $, which may be ``easier'' to compute. In particular, the eigenfunctions of $ \upDelta $ are also eigenfunctions of the corresponding heat kernel which is a compact operator \cite{Rosenberg97}; a fact which will  be useful when building data-driven bases in Section~\ref{secDataDrivenBasis} ahead. With this in mind, an alternative to~\eqref{eqStFJ} is to define $ \tilde L_\ell = \tilde \Pi_\ell L \tilde \Pi_\ell $ for some $ \ell \geq 0 $, and compute 
\begin{equation}
  \label{eqStFJ2}
  \tilde S_{\ell,t} f  = \sum_{n=0}^\infty \frac{ t^n }{ n! } \tilde L_\ell^n f, 
\end{equation}
where $ f $ can now be any observable in $H $ (as opposed to a bandlimited observable). As with~\eqref{eqStFJ}, at fixed $ \ell $, the sum over $ n $ in~\eqref{eqStFJ2} converges for all $ t \geq 0 $ since $ \tilde L_\ell $ has finite rank. However, unless rather restrictive conditions hold (e.g., that the eigenspaces of $ \upDelta $ are invariant under $ L $), the convergence of $ \tilde S_{\ell,t}f $ to $ S_t f $ is not guaranteed to be uniform in $ t $. Instead, classical results on semidiscrete approximation schemes for linear parabolic partial differential equations \cite{FujitaSuzuki91} lead to the weaker convergence result
\begin{equation}
  \label{eqBoundParabolic}
  \lVert  ( S_t - \tilde S_{\ell,t}  ) f \rVert \leq \frac{ C }{ t } \lVert ( I - \tilde \Pi_\ell ) f \rVert,
\end{equation}
where $ C $ is a constant and $ \lVert ( I - \tilde \Pi_\ell ) f \rVert $ converges to zero as $ \ell \to \infty $. Note that the weakness of the bound at short times stems from the fact that an arbitrary $ f \in H $ may not be in the domain of $ L $, but $ S_t f $ lies in $ D(L) $ for all $ t > 0 $. 

For the remainder of the paper, we restrict attention to approximations of $ S_t f $ via~\eqref{eqStFJ2} as opposed to~\eqref{eqStFJ}. This approximation can be evaluated by computing the matrix exponential of the $ \ell \times  \ell $ generator matrix $ \boldsymbol{L} $ with elements $ L_{ij} = \langle \phi_i, L \phi_j \rangle $, $ 0 \leq i,j, \leq \ell-1 $. Specifically, given $ f \in H $, we have $ \tilde S_{l,t} f = \tilde S_{l,t} \tilde \Pi_l f $, and expanding $ \tilde \Pi_\ell f = \sum_{k=0}^{\ell-1} b_k \phi_k $, $ b_k = \langle \phi_k, f \rangle $, it follows that $ \tilde S_{\ell,t} f = \sum_{k=0}^{\ell-1} c_k \phi_k $, where the expansion coefficients $ c_k $ can be computed in vector form using 
\begin{equation}
  \label{eqStMat} \vec c = e^{t \boldsymbol{L}} \vec b, \quad \vec b = ( b_0, \ldots, b_{\ell-1} )^\top, \quad \vec c = ( c_0, \ldots, c_{\ell-1} )^\top.   
\end{equation}
It should be noted that due to the tensor product nature of the Laplace-Beltrami eigenfunction basis for $ H $, the spectral truncation parameter $ \ell $ can be very large (e.g., in the experiments in Section~\ref{secL96}, $ \ell $ will be as large as $ 750 \times 65 \times 65 \simeq 3 \times 10^6 $); therefore, efficient evaluation of $ e^{t \boldsymbol{L} } \vec b $ becomes crucial.

\subsubsection{\label{secLeja}Approximation via Leja interpolation}

The problem of computing the action of matrix exponentials on vectors has been studied extensively  \cite{MolerVanLoan03,Higham08,CaliariEtAl14}. Here, we approximate the semigroup action in~\eqref{eqStMat} using the method of Caliari et al.~\cite{CaliariEtAl04} and  Kandolf et al.~\cite{KandolfEtAl14} which approximates the action of the matrix exponential on vectors via polynomial interpolation at Leja nodes. Specifically, the method in \cite{CaliariEtAl04,KandolfEtAl14} approximates $ e^{t\boldsymbol{L}} \vec b $ by $ p_d( t \boldsymbol{L} ) \vec b $ where $ p_d $ is a Newton interpolating polynomial of degree $ d $ associated with a sequence $ \{ \zeta_0, \ldots, \zeta_d \} $, $ \zeta_i \in \mathbb{ C }$, of Leja nodes in the complex plane. By employing an interpolating polynomial, the method does not require explicit evaluation of the matrix exponential; this is important from a computational efficiency standpoint since, for instance, $ e^{t\boldsymbol{L}} $ will generally be a full matrix even if $ \boldsymbol{L} $ is sparse. Moreover, the properties of Leja interpolation nodes ensure that the approximation remains well-conditioned at large $ d $. The properties of Leja sequences also allow for efficient iterative approximation refinement by increasing $ d $ until an error tolerance is met, using a stopping criterion derived from an a posteriori error estimate to terminate the iterations \cite{KandolfEtAl14}.      

The method in \cite{CaliariEtAl04,KandolfEtAl14} utilizes Leja nodes $ \zeta_i $ on the line segment $ K = ( \alpha + \beta ) / 2 + \ii [ -\gamma, \gamma ] $ in the complex plane, where $ \alpha $, $ \beta $, and $ \gamma $ are real numbers that determine the corners of a rectangle $ [ \alpha, \beta ] \times \ii [ -\gamma, \gamma ] \subset \mathbb{ C } $ bounding the spectrum of $ \boldsymbol{L} $. Caliari et al.\ \cite{CaliariEtAl04} identify this rectangle by applying Gershgorin's disk theorem to the symmetric and antisymmetric parts of the matrix $ \boldsymbol{L} $; we also follow a similar approach here. In particular, in the case of the generator matrix associated with the semigroup $ \tilde S_{\ell,t} $, the symmetric part $ \boldsymbol{L}^S = ( \boldsymbol{L} + \boldsymbol{L}^\top ) / 2 $ is a diagonal matrix with diagonal elements $ L^S_{ii} = - \theta \langle \phi_i, \upDelta \phi_i \rangle = - \theta \eta_i $ proportional to the Laplace-Beltrami eigenvalues (Dirichlet energies). Thus, we can set $ \alpha = L^S_{00} = 0 $ and $ \beta =  L^S_{\ell-1,\ell-1} = - \theta \eta_{\ell-1} $. To place a bound on the imaginary part of the spectrum, we apply Gershgorin's theorem to the antisymmetric part $ \boldsymbol{L}^A = ( \boldsymbol{L} - \boldsymbol{L}^{\top} ) / 2 $, where the matrix elements $ \boldsymbol{L}^A_{ij} = w_{ij} = \langle \phi_i, w( \phi_j ) \rangle $ are determined by the matrix representation of the Koopman generator $ w $ in the Laplace-Beltrami eigenfunction basis of $ H $. Applied to this matrix, Gershgorin's theorem gives $ \gamma = \max_{i\in\{0,\ldots,\ell-1\}} \sum_{j=0}^{\ell-1} \lvert w_{ij} \rvert $.   

Once the interval $ K $ has been established, the first Leja point $ \zeta_0 $ is selected arbitrarily in that interval, and the remaining points are computed iteratively, setting $ \zeta_j \in \argmax_{z\in K} \prod_{i=0}^{j-1} \lvert z - \zeta_i \rvert $ for $ j > 0 $. Associated with the Leja points are the divided differences $ r[ \zeta_i, \ldots, \zeta_j ] $, which can also be computed iteratively (and stably \cite{Caliari07}), using 
\begin{displaymath}
  r[ \zeta_i ] = e^{\zeta_i}, \quad r[ \zeta_i, \ldots, \zeta_j ] = \frac{r [\zeta_{i+1}, \ldots,\zeta_k] - r[\zeta_i, \ldots, \zeta_{j-1}]}{\zeta_j - \zeta_i}.
\end{displaymath}
The Newton interpolating polynomial $ p_d( t \boldsymbol{L} ) $ associated with the Leja nodes $ \{ \zeta_0, \ldots, \zeta_d \} $ is then given by 
\begin{displaymath}
  p_d( t \boldsymbol{L} ) = r[ \zeta_0 ] + \sum_{i=1}^d r[ \zeta_0, \ldots, \zeta_i ] \prod_{j=0}^{i-1} (  t \boldsymbol{L} -\zeta_j ).
\end{displaymath}
Besides the fact that the $ ( d + 1) $-th order approximation $ p_{d+1}( t \boldsymbol{L} ) \vec \vec b $ can be easily computed from $ p_d( t \boldsymbol{L} ) \vec b $, other attractive properties of the Leja method include good numerical conditioning and superlinear convergence to $ e^{t \boldsymbol{L} } \vec b $; additional details on these properties can be found in \cite{CaliariEtAl04,KandolfEtAl14}. Details on the numerical implementation used in the experiments in Sections~\ref{secExamples} and~\ref{secL96} ahead are included in~\ref{appNumLeja}. 

In applications, we are often interested in evaluating $ e^{t_n \boldsymbol{L}} \vec b $ at the times $ t_n = n \tilde \tau $ with $ n \in \{ 0, 1, \ldots \} $ and $ \tilde \tau > 0 $ a fixed sampling interval (not necessarily equal to the sampling interval $ \tau $ in the training data). In such cases, we compute the quantities $ \vec b_0, \vec b_1, \ldots $ iteratively using 
\begin{equation}
  \label{eqIterative}
  \vec b_{n+1} = \frac{ \vec b'_{n+1} }{ \lVert \vec b'_{n+1} \rVert }, \quad \vec b'_{n+1} = e^{\tilde \tau} \vec b_n, \quad \vec b_0 = \vec b.
\end{equation}    
In this procedure, the normalization by $ \lVert \vec b'_{n+1} \rVert $ ensures that the vectors $ b_n $ have constant norm, as would be expected for unitary evolution. In effect, this normalization step can be thought of as an inflation to compensate for contraction due to diffusion regularization and spectral truncation associated with the operator $ \tilde L_\ell $ in~\eqref{eqStFJ2}. In practice, we have observed that this step is not particularly important, but can sometimes lead to a modest increase of prediction skill.

\subsubsection{\label{secBoundedDen}Computing the semigroup action on probability densities}           

We now turn to the approximation of the Perron-Frobenius group action on probability measures with densities relative to the invariant measure in $ H $. As discussed in Section~\ref{secPrelim}, the evolution of such a density $ \rho $ is governed by the adjoint of the Koopman group,  $ W_t^* $ (i.e., the Perron-Frobenius group), and in order to approximate the action on this group on functions we proceed in an analogous way to the case of $ W_t $ in Sections~\ref{secBoundedObs} and~\ref{secLeja}. That is, we approximate $ W_t^* \rho $ by $ S_t^* \rho $, where $ S_t^* $ is the adjoint semigroup generated by $ L^* $, and  further approximate this quantity by projection onto the finite-dimensional subspaces $ H_\ell \subset H $, viz., $ S^*_{t} \rho  \approx \tilde S^*_{\ell,t} \rho = \sum_{n=0}^\infty \frac{ t^n }{ n! } \tilde L_\ell^{*n} \rho $, where $ \tilde L^*_\ell = \tilde \Pi_\ell L^* \tilde \Pi_\ell $. The error of this approximation can be estimated by an analogous expression to~\eqref{eqBoundParabolic}. Expanding $ \tilde \Pi_l \rho = \sum_{k=0}^{\ell-1} b_k \phi_k $ and $ \tilde S^*_{\ell,t} f = \sum_{k=0}^{\ell-1} c_k \phi_k $  with $ b_k = \langle \phi_k, \rho \rangle $ and $ c_k = \langle \phi_k, \tilde S^*_{\ell,t} \rho \rangle $, we obtain an analogous expression to~\eqref{eqStMat}, namely,  $ \vec c = e^{t \boldsymbol{L}^*} \vec b $, where $ \vec b = ( b_0, \ldots, b_{\ell-1} )^\top $ and $  \vec c = ( c_0, \ldots, c_{\ell-1} )^\top $ are $ \ell $-dimensional column vectors containing the expansion coefficients of $ \rho $ and $ \tilde S^*_{\ell,t} \rho $, respectively, and $ \boldsymbol{L}^* $ is the adjoint of the $ \ell \times  \ell $ generator matrix. We compute the action of the matrix exponential $ e^{t\boldsymbol{L}^*} $ on $ \vec b $ using Leja interpolation as described in Section~\ref{secLeja}.

\section{\label{secExamples}Semi-analytical examples}

In this Section, we apply the framework for coherent pattern identification and prediction presented in Section~\ref{secKoopman} to time-periodic incompressible flows on a doubly-periodic domain, $ X = \mathbb{ T }^2 $. According to  Examples~\ref{exStream} and~\ref{exStream2}, the dynamical state space for this class of flows is the circle, $ A = \mathbb{ T }^1 $, and as a result the product space $ M = A \times X $ is the 3-torus. Moreover, the Laplace-Beltrami eigenfunction basis of $ H $ associated with the metric $ h $ consists of the Fourier functions $ \phi_{ijk}( a, x ) = \phi_i^A( a ) \phi_{jk}^X( x ) $, where $ \phi_i^A( a ) = e^{\ii i a } $, $ \phi_{jk}^X( x ) = e^{\ii ( j x_1 + k x_2 ) } $,  $ i, j, k \in \mathbb{ Z } $, and $ a $ and $ x = ( x_1, x_2 ) $ are canonical angle coordinates for $ A $ and $ X $, respectively. The Laplace-Beltrami eigenvalue corresponding to  $ \phi_{ijk} $ is $ \eta_{ijk} = i^2 + j^2 + k^2 $. Note that in this Section we label the basis elements and the eigenvalues using three integer indices (as opposed to the single index used in Section~\ref{secLB}), as this notation is more convenient for expressing the matrix elements of the generator. Moreover, we use the symbols $ \ell_A $, $ \ell_{X_1} $, and $ \ell_{X_2} $ to represent the spectral truncation parameters for $ L^2( A, \alpha ) $ and each of the factors of $ L^2( X, \xi ) \simeq L^2( \mathbb{ T }^1 ) \otimes L^2( \mathbb{ T }^1 )$; that is, our finite-dimensional approximation spaces have dimension $\ell = ( 2 \ell_A + 1 ) \times ( 2 \ell_{X_1} + 1 ) \times ( 2 \ell_{X_2} + 1 ) $, and are spanned by $ \phi_{ijk} $ with $ -\ell_A \leq i \leq \ell_A $, $ - \ell_{X_1} \leq j \leq \ell_{X_1} $, and $ -\ell_{X_2} \leq k \leq \ell_{X_2} $, or, as appropriate, their normalized counterparts $ \varphi_{ijk} $ for the subspaces of $ H^1 $. Following Definition~\ref{defCoherent}, whenever we discuss numerical Koopman eigenvalues and eigenfunctions $ (\lambda_k, z_k )$, $ k \in \mathbb{ N } $, we order these eigenvalues and eigenfunctions in order of increasing Dirichlet energy $ E_k $, though as stated in Remark~\ref{rkDirichlet}, in practice we compute subsets of the spectrum consisting of eigenvalues of minimal modulus using iterative solvers. 

In the examples that follow, we consider the same periodic dynamics on $ A $ as in Example~\ref{exStream} (i.e., $ \Phi_t(a  ) = \omega t + a \mod 2 \pi$), and set the streamfunction $ \zeta $ to various linear combinations of circular Gaussian functions. With this choice, the matrix elements of the generator can evaluated in closed form using integral properties of modified Bessel functions. Thus, in this Section, all errors in the computation of eigenfunctions and evolution of observables and densities are caused by projection to finite-dimensional subspaces of $ H $, as opposed to ergodic averaging errors due to finite numbers of samples and/or errors associated with the construction of the Laplace-Beltrami eigenfunction basis. We will take up the problem of constructing data-driven approximations to the Laplace-Beltrami eigenfunction basis in Section~\ref{secDataDrivenBasis}.    

\subsection{\label{secMovingVortex}Moving Gaussian vortex in a periodic domain}

In our first example, we consider the streamfunction $ \zeta : A \mapsto C^\infty(X ) $ defined as
 \begin{equation}
  \label{eqZetaMoving}
  \zeta( a )(x) = e^{\kappa [ \cos( x_1 - a ) + \cos x_2 ] },  
\end{equation}
where $ \kappa $ is a positive parameter. Thus, for each $ a \in A $, $ \zeta( a ) $ corresponds (up to normalization) to a von Mises density on $ X = \mathbb{ T }^2 $, which is the analog of a Gaussian density in a periodic domain. The corresponding incompressible velocity field $ v\rvert_a $ from~\eqref{eqStream} generates a vortical flow centered at $ ( x_1, x_2 ) = ( a, 0 ) $, illustrated in Fig.~\ref{figPsiMoving}. Since $ \partial_{1} \zeta(a) $ and $ \partial_{2}(a) \zeta $ is zero when $ x_1 \in a + \{ 0, \pi \} $ and $ x_2 \in \{ 0, \pi \} $, the velocity field vanishes at the center of the vortex as well as at three points away from the vortex center, leading to a ring-like region around the center where the velocity field is concentrated. Note that the radial extent of the ring decreases with increasing $ \kappa $. Moreover, because the dynamics on $ A $ evolve according to $ \Phi_t( a ) = a + \omega t \mod 2 \pi $, the vortex moves without deformation at fixed speed $ \omega $ along the $ x_1 $ direction.

\begin{figure}
  \centering\includegraphics[width=.4\linewidth]{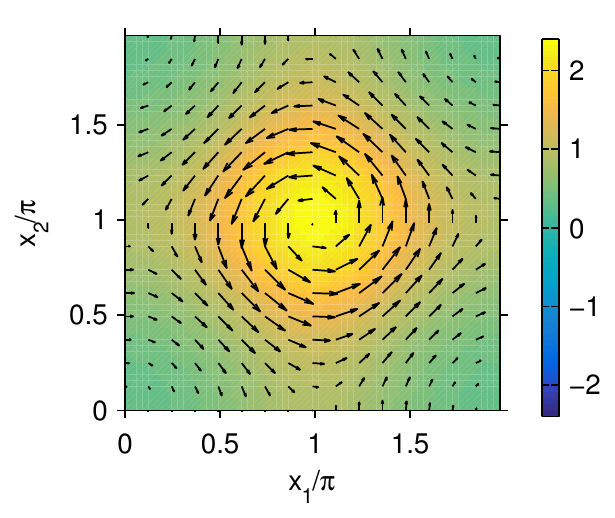} \includegraphics[width=.4\linewidth]{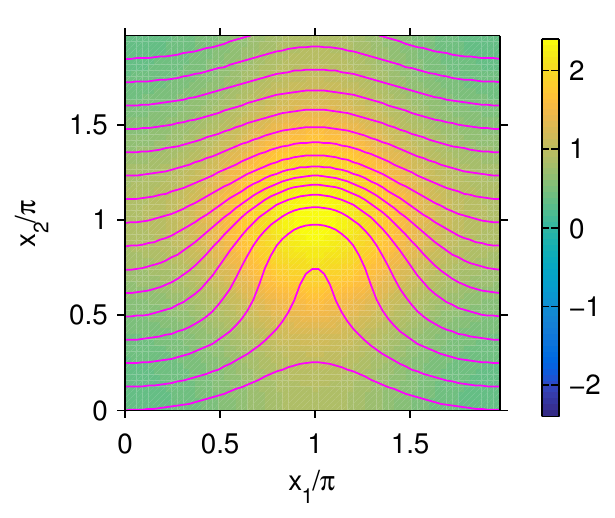}
  \caption{\label{figPsiMoving}(a) Streamfunction $ \zeta(a) $ (colors) and velocity field $ v\rvert_a $ (arrows) for the moving-vortex streamfunction in~\eqref{eqZetaMoving}  evaluated for $ \kappa = 0.5 $ and $ a = \pi $. (b) Streamfunction $ \zeta(a) $ (colors ) and representative streamlines $ s_y $ (lines) in the frame of reference comoving with the vortex center in (a) for $ \omega = 1 $. For clarity of visualization of the vortex center, the $ x_2 $ coordinate has been periodically shifted to $ x_2 + \pi $ in both (a) and (b).}  
\end{figure}

For our purposes, a particularly useful property of Gaussian streamfunctions such as~\eqref{eqZetaMoving} is that their integrals against Fourier functions can be evaluated in closed form with the help of the identity
\begin{equation}
  \label{eqBessel}
  \frac{ 1 }{ 2 \pi } \int_0^{2\pi} e^{\ii n \theta + \kappa \cos \theta } \, d \theta =  \frac{ I_{\lvert n \rvert }( \kappa ) }{ I_0( \kappa ) },
\end{equation}
where $ I_\nu $ is the modified Bessel function of the first kind and order $ \nu $. In particular, recall the decomposition $ w = w^A + w^X $ in Section~\ref{secPrelim} of the vector field $ w $ of the skew-product system  in terms of the vector fields $ w^A $ and $ w^X $ associated with the dynamics on $ A $ and $ X $, respectively. We have
\begin{equation}
  \label{eqWXMov0}
  \langle \phi_{ijk}, w^X( \phi_{lmn} ) \rangle = \langle \phi_{ijk}, \phi^A_l  v( \phi^X_{mn} ) \rangle = \langle \phi_{ijk}, \phi^A_l  ( \partial_{1} \zeta \, \partial_{2} \phi^X_{mn} - \partial_{2} \zeta \, \partial_{1} \phi^X_{mn}   )   \rangle, 
\end{equation} 
and using~\eqref{eqBessel} we arrive, after some algebra, at the result
\begin{align}
  \nonumber
  \langle \phi_{ijk}, w^X( \phi_{lmn} ) \rangle &= \frac{ \kappa }{ 2 I_0( \kappa )^2 } [ m I_{\lvert m - j \rvert}( \kappa ) ( I_{\lvert n-k - 1 \rvert}( \kappa ) - I_{\lvert n - k + 1 \rvert}(\kappa ))  \\   \label{eqWXMov} & \quad + n I_{\lvert n - k \rvert\rvert}( \kappa ) ( I_{\lvert m - j + 1 }( \kappa ) - I_{\lvert m - j - 1 \rvert}( \kappa ) ) ] \delta_{i+j,l+m}.  
\end{align}
In addition, we have
\begin{equation}
  \label{eqWAMov}
  \langle \phi_{ijk}, w^A( \phi_{lmn} ) \rangle = \langle \phi^A_i, u( \phi_l^A) \rangle \delta_{jm} \delta_{kn} = \ii \omega \delta_{il} \delta_{jm} \delta_{kn}.
\end{equation}
Together, \eqref{eqWXMov} and~\eqref{eqWAMov}, in conjunction with the fact that $ \langle \phi_{ijk}, \upDelta( \phi_{lmn} ) \rangle = \langle \grad_h \phi_{ijk}, \grad_h \phi_{lmn} \rangle_h = \eta_{ijk} \delta_{il} \delta_{jm} \delta_{kn} $ and the definitions of the normalized eigenfunction basis for $ H^1 $ (see Section~\ref{secLB}), are sufficient to evaluate all sesquilinear forms and operator matrix elements appearing in the schemes of Sections~\ref{secCoherent} and~\ref{secPrediction}. In what follows, we present results for the flow with $ \omega = 1 $ and $ \kappa = 0.5 $; with this choice of flow parameters the maximal velocity field norm is $ \simeq 0.3 $, and occurs at a radial distance of $ \simeq 0.6 \pi $ from the vortex center (all with respect to the canonical flat metric on $ \mathbb{ T}^2 $).   

\subsubsection{\label{secMovingVortexEig}Koopman eigenvalues and eigenfunctions}

Before presenting our results, it is useful to comment on the general spectral properties of the generator $ \tilde w $ for our choice of flow parameters. As discussed in~\ref{appMoving}, the spectrum of $ \tilde w $ always contains eigenvalues $ \ii \omega_k $ which can be grouped into (1) integer multiples of $ \omega $, i.e., $ \omega_k = k \omega $ with $ k \in \mathbb{ Z } $, corresponding to eigenfunctions $ z_k( a, x ) = e^{\ii k \omega a } $ that are constant on the spatial domain ($w^X( z_k ) = 0$); (2) zero eigenvalues corresponding to eigenfunctions constant on the streamlines of the velocity field in a comoving frame with the vortex center; (3) linear combinations of the eigenvalues in class~1 and~2 (i.e., integer multiples of $ \ii \omega $) corresponding to products of the respective eigenfunctions. In addition, the spectrum of $ \tilde w $ has a continuous component associated with functions in $ H$ that are nonconstant on the streamlines in class~2. 

Eigenfunctions of class~1 exist in all time-periodic flows with the associated skew-product dynamics of Section~\ref{secPrelim}; this is because $ w $ is trivially projectible under $ \pi_{A*} $. In particular, $ \pi_{A*} w $ is equal to the vector field $ u $ of the periodic dynamics on $ A $, and (a skew-adjoint extension of) the latter has pure point spectrum. Eigenfunctions of this class can therefore be expressed as pullbacks of eigenfunctions of the generator dynamics on $ A $; that is, we have $ z_k = z^A_k \circ \pi_A $, where $z^A_k(a) = e^{\ii ka} $ and $ u( z^A_k ) = \ii k \omega z_k^A $. We denote the closed subspace of $ H $ spanned by these eigenfunctions by $ \mathcal{ D }_A $ (here, the symbol $ \mathcal{ D} $ stands for discrete spectrum). Eigenfunctions of class~2 exist because $ M $ admits a submersion $ \pi_Y $ onto a circle $ Y $  where the image $ \pi_{Y*} w $ of the vector field is identically zero. This submersion can be intuitively thought of as a Galilean transformation to a frame of reference where the vortex is stationary, followed by projection transverse to the streamlines in that frame. As with the eigenfunctions of class~1, the eigenfunctions of class~2 lie in a closed subspace $ \mathcal{ D }_Y \subset H$ and can be expressed as pullbacks $ z_k = z_k^Y \circ \pi_Y $, where $ z_k^Y $ are orthonormal $ L^2 $ functions on $ Y $ equipped with the normalized Lebesgue measure. Eigenfunctions of class~3 are a direct consequence of the group structure of the Koopman eigenvalues and eigenfunctions stated in Remark~\ref{rkGroup}. 

Together, eigenfunctions in  classes~1--3 span the tensor  product space $ \mathcal{ D } = \mathcal{ D }_A \otimes \mathcal{ D }_Y $, which is closed and contains all eigenfunctions of $\tilde w $. Note that $ \mathcal{ D }_A $, $ \mathcal{ D }_Y $, and $ \mathcal{ D } $ are all invariant under the Koopman group $ W_t $. Thus, we have an invariant orthogonal decomposition $ H= \mathcal{ D } \otimes \mathcal{ D }^\perp $, where the  $ \mathcal{ D }^\perp $ is associated with the continuous spectrum of the generator.  The existence of the latter can be justified from the fact that the flow in the frame comoving with the center of the vortex is conjugate to a skew rotation on $ \mathbb{ T }^2 $, which is known to have continuous spectrum \cite{BroerTakens93,Mezic17}. Additional details on the properties of the spectrum of $ \tilde w $ can be found in~\ref{appMoving}. 

Figure~\ref{figZMoving} shows snapshots of representative numerical Koopman eigenfunctions for the moving vortex Gaussian flow, computed using the spectral truncation parameters $ \ell_A = \ell_{X_1} = \ell_{X_2} = 32 $ (amounting to a total of $ \ell = ( 2 \ell_A + 1 ) \times ( 2 \ell_{X_1} + 1 ) \times ( 2 \ell_{X_2} + 1 ) = \text{274,625} $ basis functions) and the diffusion regularization parameter $ \theta = 10^{-5} $. To obtain these results, we computed a total of 51 eigenfunctions using Matlab's \texttt{eigs} solver as described in~\ref{appKoopEig}, targeting eigenvalues of minimal modulus.  The eigenfunctions shown in Fig.~\ref{figZMoving}, $ z_2 $,  $ z_4 $,  $ z_6 $,  $ z_8 $,  $ z_{10} $,  $ z_{12} $, $ z_{29} $, and $ z_{33} $, are also visualized as spatiotemporal patterns in Movie~\href{http://cims.nyu.edu/~dimitris/files/tracers/movie1.mp4}{1}. The eigenvalues $ \lambda_k $ and Dirichlet energies $ E_k $ corresponding to these eigenfunctions are listed in Table~\ref{tableZMoving}. Note that the $( \lambda_k, z_k) $ come in complex-conjugate pairs, so we do not show them for consecutive $ k $.    

\begin{figure}
  \centering\includegraphics[width=\linewidth]{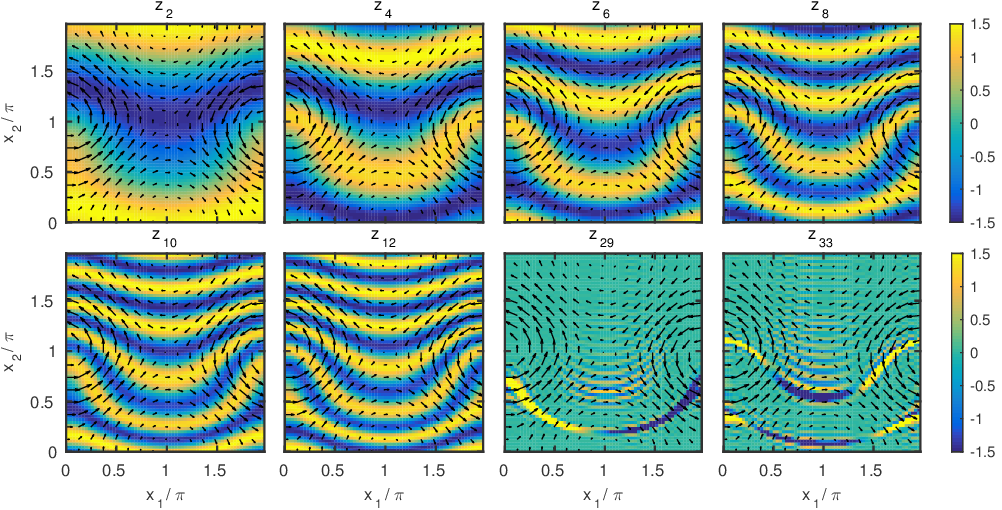}
  \caption{\label{figZMoving}Snapshots of the real parts of representative numerical Koopman eigenfunctions (colors) and the velocity field  $ v\rvert_a $ (arrows) for the moving-vortex flow with $ \omega = 1 $ and $ \kappa = 0.5 $. The diffusion regularization parameter used in this calculation is $ \theta = 10^{-5} $.  For clarity of visualization of the vortex center, the $ x_2 $ coordinate has been periodically shifted to $ x_2 + \pi $ in all panels.}  
\end{figure}

\begin{table}
  \centering\small
  \caption{\label{tableZMoving}Eigenvalues $ \lambda_k $ and Dirichlet energies $ E_k $ of the numerical Koopman eigenfunctions $ z_k $ for the moving vortex flow shown in Fig.~\ref{figZMoving}}
  \begin{tabular*}{.5\linewidth}{@{\extracolsep{\fill}}lll}
    \hline
    &  \multicolumn{1}{c}{$\lambda_k $} & \multicolumn{1}{c}{$ E_k $} \\
    \cline{2-3}
    $z_{2}$ & $ -1.48e-05 + 5.46e-19 \,\ii $ & $ 1.48e+00 $ \\ 
    $z_{4}$ & $ -5.64e-05 + 3.76e-17 \,\ii $ & $ 5.64e+00 $ \\ 
    $z_{6}$ & $ -1.26e-04 - 2.39e-17 \,\ii $ & $ 1.26e+01 $ \\ 
    $z_{8}$ & $ -2.23e-04 - 3.47e-16 \,\ii $ & $ 2.23e+01 $ \\ 
    $z_{10}$ & $ -3.48e-04 - 2.69e-16 \,\ii $ & $ 3.48e+01 $ \\ 
    $z_{12}$ & $ -5.00e-04 - 3.28e-16 \,\ii $ & $ 5.00e+01 $ \\ 
    $z_{29}$ & $ -2.82e-03 - 5.00e-01 \,\ii $ & $ 2.82e+02 $ \\ 
    $z_{33}$ & $ -3.45e-03 + 7.62e-01 \,\ii $ & $ 3.45e+02 $ \\ 
    \hline
  \end{tabular*}
\end{table}

The numerical eigenfunctions $ z_2 $, $ z_4 $, $ z_6 $,  $ z_8 $, $ z_{10}$, and $ z_{12}$,  are in good agreement with what expected from theory for the Koopman eigenfunctions in the subspace $ \mathcal{ D }_Y $. In particular, comparing Figs.~\ref{figPsiMoving}(b) and~\ref{figZMoving}, it is evident that the numerical eigenfunctions are to a good approximation constant on the streamlines in the frame comoving with the vortex center, and according to Table~\ref{tableZMoving} the imaginary parts of their corresponding eigenvalues are numerically close to zero. Moreover, the real parts of the eigenvalues are numerically close to $ - \theta E_k $; this is consistent with what expected from heuristic perturbation expansions of the eigenvalues and eigenfunctions of the operator $ L $ with respect to $ \theta $ \cite{Giannakis17}. 

Qualitatively, it appears that the numerical eigenfunctions in $ \mathcal{ D }_Y $ can be assigned a positive integer ``wavenumber'' $ \kappa $ equal to half of the number of zeros of their real and imaginary parts  along the line parallel to the $ x_2 $ coordinate line of $ X $ passing through the vortex center---as expected, the Dirichlet energy of these eigenfunctions is an increasing function of $ \kappa $. This behavior (in conjunction with the structure of $ L $ and the fact that the eigenfunctions in question  are constant on the streamlines) suggests that the numerical $ z_k $ in this class could be approximating pullbacks of eigenfunctions of a Laplace-Beltrami operator $ \upDelta_Y $ on the base space $ Y $ of the submersion $ \pi_Y $. That is, a reasonable hypothesis could be that for every $ z_k = z_k^Y \circ \pi_Y \in \mathcal{ D }_Y $ we have
\begin{displaymath}
  \lambda_k z_k = L ( z_k ) = - \theta \upDelta z_k = - \theta (\upDelta_Y z^Y_k) \circ \pi_Y. 
\end{displaymath}     
While such a scenario would be desirable, in the sense that the effect of diffusion would lead to no perturbation of the eigenfunctions of $L $ compared to the eigenfunctions of $ \tilde w $, fundamental results from the theory of harmonic maps on manifolds suggest that it is unlikely to hold exactly in practice. In particular, it is known that a necessary and sufficient condition for the relationship 
\begin{equation}
  \label{eqDeltaComm}
  \upDelta( f \circ \pi_Y ) = ( \upDelta_Y f ) \circ \pi_Y 
\end{equation}
to hold for arbitrary $ f \in C^\infty( Y ) $ is that the fibers of the submersion $ \pi_Y $ are minimal, i.e., they are surfaces of vanishing mean curvature \cite{Watson73,GoldbergIshihara78,GilkeyEtAl98}. Clearly, the latter is not the case given the nonconstant curvature of the streamlines in Fig.~\ref{figPsiMoving}(b), so we cannot expect~\eqref{eqDeltaComm} to hold for general $ f $. Nevertheless, despite that our diffusion operator is not tailored to the flow under study (in the sense of~\eqref{eqDeltaComm} failing to hold), the  numerical results in Fig.~\ref{figZMoving} and Table~\ref{tableZMoving} demonstrate that it is still possible to obtain high-quality numerical Koopman eigenfunctions via the scheme of Section~\ref{secCoherent}.

Next, we examine the effects of the continuous part of the spectrum of $ \tilde w $ to the numerically computed eigenfunctions of $ L $. As shown in Fig.~\ref{figZMoving} and Movie~\href{http://cims.nyu.edu/~dimitris/files/tracers/movie1.mp4}{1}, the numerical eigenfunctions of $ L $ include functions (e.g., $ z_{29} $ and $ z_{33} $ shown here) which are highly localized around individual streamlines in the comoving frame with the vortex center, and also exhibit abrupt changes along those streamlines. As a result, the numerical eigenfunctions of this class also have large Dirichlet energy (see Table~\ref{tableZMoving}). We believe that such eigenfunctions are remnants of the continuous spectrum of $ \tilde w $, which appear numerically due to finite spectral truncation in the Galerkin method for the eigenvalues and eigenvectors of $ L $. In particular, while $ \tilde w $ does not have nonconstant eigenfunctions on the streamlines (i.e., it does not have eigenfunctions in the Hilbert subspace $ \mathcal{ D}^\perp$), it is nevertheless possible to construct a Koopman eigenvalue problem on a suitable space of distributions \cite{Putinar97}, such that points in the continuous spectrum of $ \tilde w $ are eigenvalues corresponding to eigendistributions. Such eigendistributions are supported on sets of zero $ \mu $-measure (e.g., finite collections of streamlines), and therefore cannot be represented by $ L^2(M,\mu) $ functions, but they still can be approximated by $ L^2(M,\mu ) $ functions by molification. Such molified distributions $ \psi $ are not expected to solve the eigenvalue problem $ L \psi = \lambda \psi $ exactly, but for a given spectral truncation parameter $ \ell $ it should be possible to arrange that the residual $ ( L - \lambda  )\psi $ is $ L^2(M,\mu) $-orthogonal to the  finite-dimensional approximation space $ \spn \{ \varphi_0, \ldots, \varphi_{\ell-1} \}  \subset H^1 $, making $ \psi $ a solution of the numerical Koopman eigenvalue problem. With increasing $ \ell $, $ \psi $ would have to be concentrated on a set of increasingly small $ \mu $-measure (otherwise, the residual would fail to be orthogonal to $ \spn \{ \varphi_0, \ldots, \varphi_{\ell-1} \}$), leading to an unbounded increase of the Dirichlet energy $ E( \psi ) $ with $ \ell $. The presence of numerical eigenfunctions such as $ z_{29} $ and $ z_{33} $, which are highly concentrated near individual streamlines (or finite collections of streamlines) is consistent with this hypothesis. In separate calculations, we have also confirmed that increasing $ \ell $ indeed causes these eigenfunctions to become increasingly concentrated and appear deeper in the spectrum of $ L $ (ordered with respect to Dirichlet energy).      

Before closing this section, we note that the results presented above do not include eigenvalues and eigenfunctions of class~1. This is likely due to the fact that we have numerically computed only a small (51-element) subset of the spectrum targeting eigenvalues of minimal absolute value---as described in Remark~\ref{rkDirichlet}, this can cause us to miss eigenfunctions with low Dirichlet energy ($\Real\lambda_k$) but high oscillatory frequency ($\Imag\lambda_k$), such as the class~1 eigenfunctions associated with the periodic dynamics on $ A$. In separate calculations, we have confirmed that these eigenfunctions can be found by directly targeting that maximal $ \Real \lambda_k $ portion of the spectrum in \texttt{eigs}, at the expense of slower numerical convergence (see~\ref{appKoopEig}). Since the spatial patterns captured by these eigenfunctions are trivial we do not discuss them here. 

\subsubsection{\label{secMovingVortexPred}Prediction of observables and probability densities}

We now apply the techniques of Section~\ref{secPrediction} to perform prediction of observables and probability densities for the moving vortex flow, working with the same flow, spectral truncation, and diffusion regularization parameters as in Section~\ref{secMovingVortexEig}. In what follows, we focus on the $ L^2 $  observables   
\begin{equation}
  \label{eqF12}
  f_1( a, x ) = \phi_{010}(a,x) = e^{\ii x_1}, \quad f_2(a,x) = \phi_{001}(a,x) = e^{\ii x_2},
\end{equation}         
where $ x= ( x_1, x_2 )  $. Our interest in these observables stems from the fact that knowledge of $ W_t f_1 $ and $ W_t f_2 $ is sufficient to determine the position of Lagrangian tracers advected by the flow at forecast time $ t $. Specifically, using the notation $ ( x_1( t, a  ), x_2( t, a ) )$ for the canonical angle coordinates at time $ t $ of a tracer released at time $ t = 0 $ at the point $ x \in X $ when the state of the time-periodic velocity field is $ a \in A $, we have
\begin{displaymath}
  W_t f_1( a, x ) = e^{\ii x_1(t,a) }, \quad W_t f_2( a, x ) = e^{\ii x_2(t,a)},
\end{displaymath} 
and therefore we can recover $ x_1(t,a) $ and $ x_2( t, a) $ from $ \arg W_tf_1(x,a) $ and $ \arg W_t f_2(x, a ) $, respectively. 

Figure~\ref{figMovingX} shows snapshots of the evolution of $ ( x_1( t, a ), x_2( t ) ) $ obtained via the operator-theoretic approach of Section~\ref{secBoundedObs} and explicit solution of the ordinary differential equations (ODEs) governing the evolution of Lagrangian tracers under the time-periodic streamfunction in~\eqref{eqZetaMoving}. This evolution is also visualized more directly in Movie~\href{http://cims.nyu.edu/~dimitris/files/tracers/movie2.mp4}{2}. To obtain the operator-theoretic results we used Leja interpolation for matrix exponentiation as described in Section~\ref{secLeja} and~\ref{appNumLeja} with a relative error tolerance of $ 10^{-7} $ and forecast timestep $ \tilde \tau = 0.01 $. We integrated the tracer ODE system using Matlab's \texttt{ode45} solver with a $ 10^{-8} $ relative error tolerance, outputting the solution every $\tilde \tau $ time units. 

As is evident from both the operator-theoretic and ODE results the evolution of the $ x_1( a, t )  $ and $ x_2( a, t ) $ coordinates (hence, the $ W_t f_1 $ and $ W_t f_2 $ observables) is qualitatively different, with the former exhibiting significantly more mixing than the latter. This can be understood from the facts that  $ f_2 $ varies predominantly transverse to the streamlines in Fig.~\ref{figPsiMoving}(b), and thus projects more strongly to the discrete spectrum subspace $ \mathcal{ D }$, whereas $ f_1 $ projects more strongly to the continuous spectrum subspace $ \mathcal{ D }^\perp $ as it varies predominantly along the streamlines. More qualitatively, after each revolution of the vortex center around the periodic domain, a tracer will tend to experience comparable amounts of displacement in the positive and negative $x_2 $ directions, but may accrue a significant net displacement along the $ x_1 $ direction. The operator-theoretic model successfully captures this behavior for lead times up to $ t \simeq 40$ (i.e., $ \omega t / ( 2 \pi ) \sim 4 $ revolutions of the vortex in $ X $), but since it operates at a finite resolution (determined by the spectral truncation parameters $ \ell_A $, $ \ell_{X_1} $, and $ \ell_{X_2} $) and also has diffusion, it eventually develops biases as it fails to capture the increasingly small-scale variations developed by observables such as $ f_1 $. For such observables, the effect of diffusion in the semigroup $ S_t $ eventually dominates; the signature of this effect in Movie~\href{http://cims.nyu.edu/~dimitris/files/tracers/movie2.mp4}{2} is spurious high-velocity tracer motions leading to the formation of voids in $ X $ where tracers should be present. Clearly, the temporal extent of useful forecasts for an observable depends on both on how strongly it projects on $ \mathcal{ D }$ as opposed to $ \mathcal{ D }^\perp $, as well as the spectral truncation and diffusion regularization ($\theta$) used.

\begin{figure}
  \centering
  \includegraphics[width=.40\linewidth]{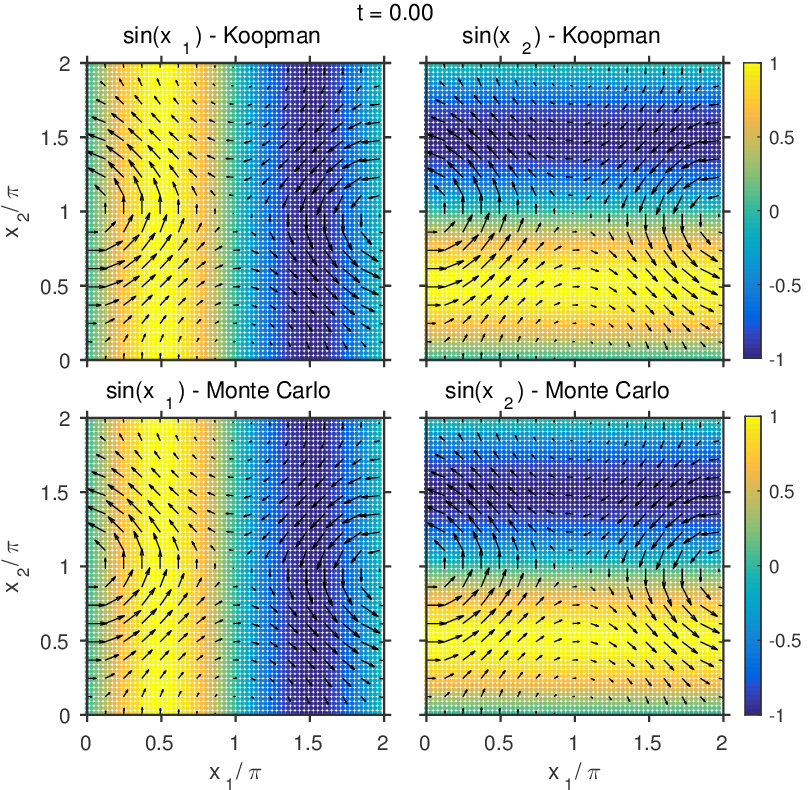}    \includegraphics[width=.40\linewidth]{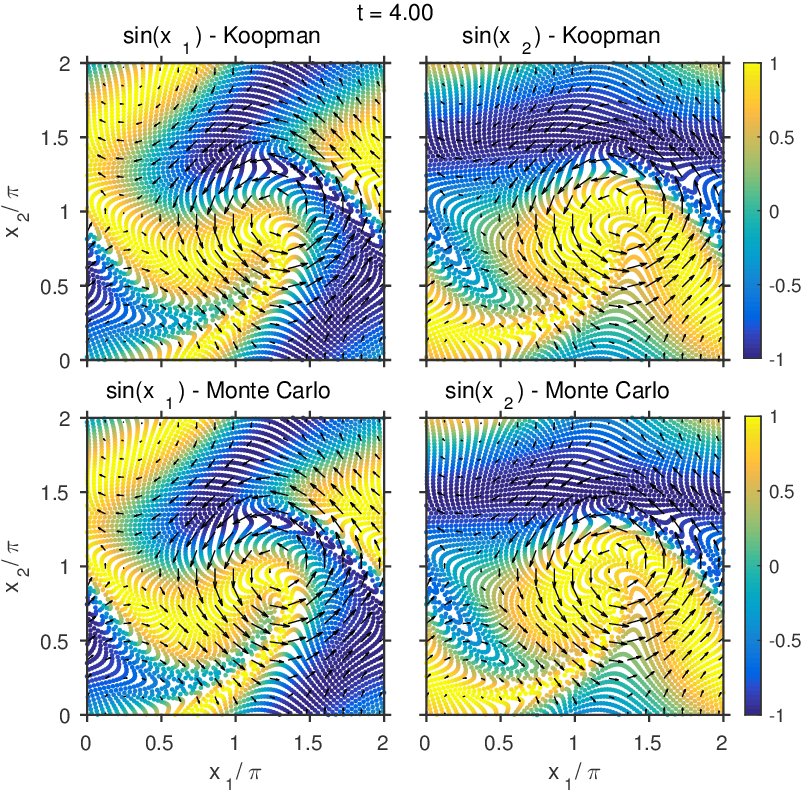}\\ 
  \includegraphics[width=.40\linewidth]{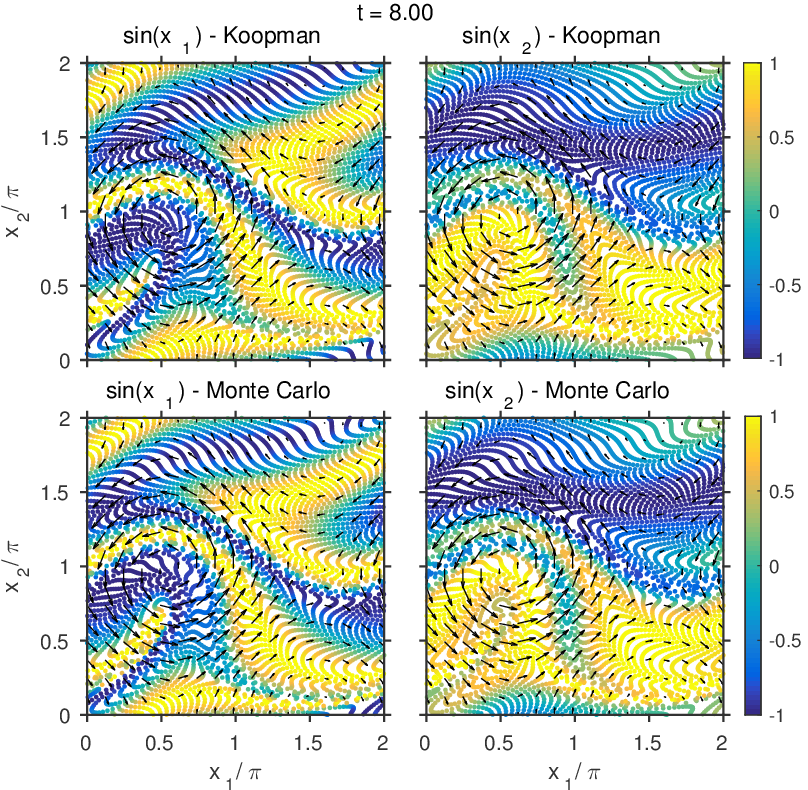}    \includegraphics[width=.40\linewidth]{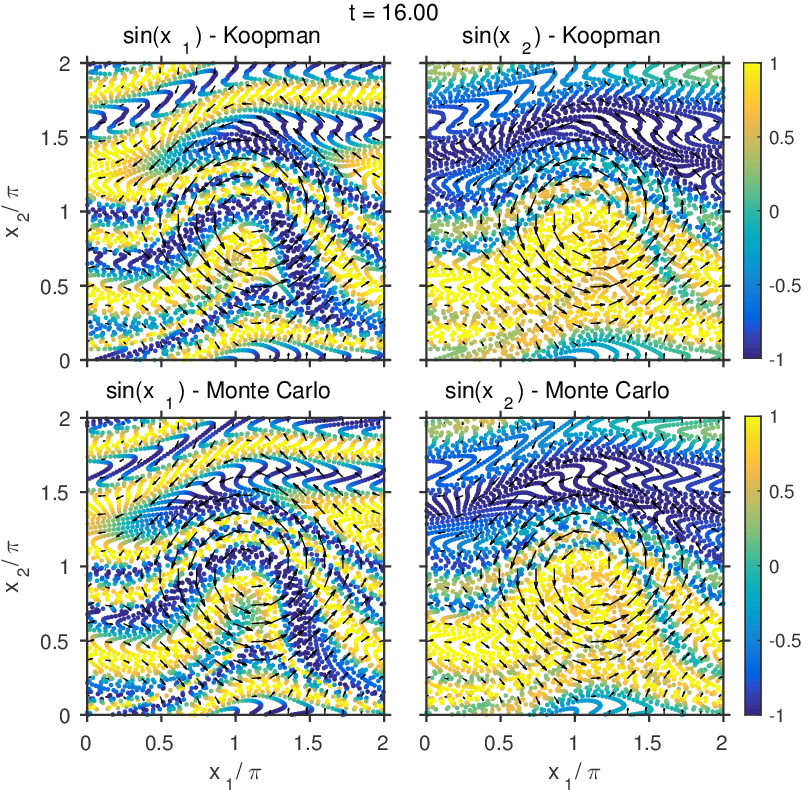}\\ 
 \includegraphics[width=.40\linewidth]{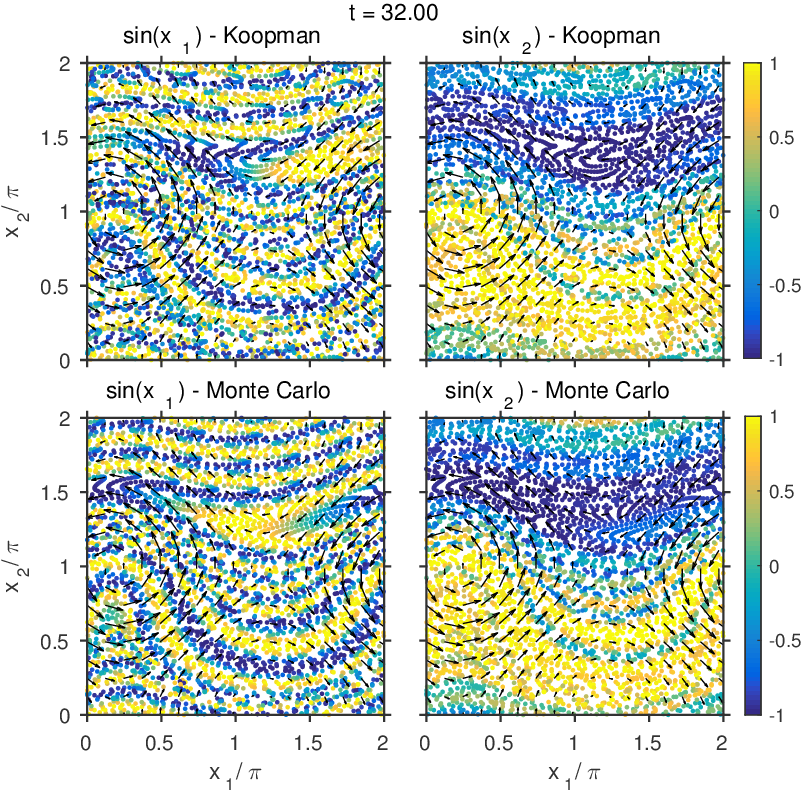}    \includegraphics[width=.40\linewidth]{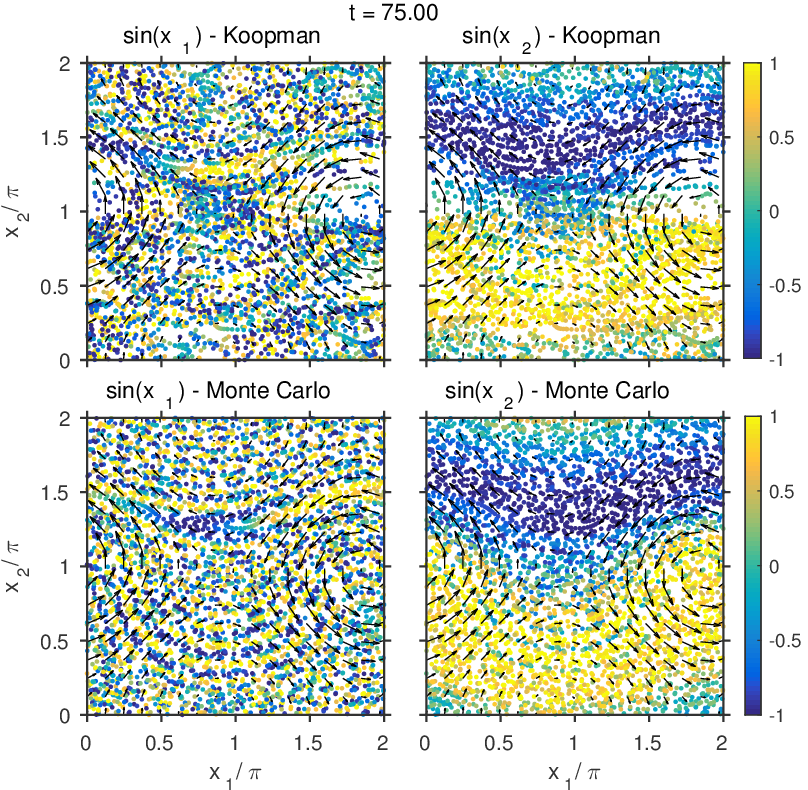} 

  \caption{\label{figMovingX}Snapshots of the evolution of the positions  $ ( x_1( t, a ), x_2( t,a ) ) \in X $ for an ensemble of Lagrangian tracers under the moving vortex flow as predicted by the operator-theoretic scheme in Sections~\ref{secBoundedObs} and~\ref{secLeja} and the perfect model for advection under the time-dependent streamfunction in~\eqref{eqZetaMoving}. The initial state of the velocity field in $A $ is  $ a = 0 $. The initial positions $ ( x_1, x_2 ) $ of the tracers are on a uniform square grid of spacing $ 2 \pi / ( 2 \ell_{X_1} + 1 ) = 2 \pi / ( 2 \ell_{X_2} + 1 ) = 2 \pi  / 65 $ along both the $ x_1 $ and $ x_2 $ coordinates. Colors show the sines of the initial $ x_1 $ and $ x_2 $ coordinates of the ensemble as an aide for visualizing the mixing properties of the flow. Arrows show the velocity field. For clarity of visualization of the vortex center, the $ x_2 $ coordinate has been periodically shifted to $ x_2 + \pi $ in all panels.}
\end{figure}

Next, we consider prediction of probability densities on tracers via the Perron-Frobenius operator $ W_t^* $ as described in Section~\ref{secBoundedDen}. Here, we perform prediction experiments initialized with a density $ \rho \in C^\infty(M) $ given by a product $ \rho(a,x) = \rho_A( a ) \rho_X( x ) $ of circular Gaussians, 
\begin{equation}
  \label{eqRhoAX}
  \rho_A( a ) = \frac{ 1 }{ I_0( \tilde \kappa ) } e^{\tilde \kappa \cos a}, \quad \rho_X( x ) = \frac{ 1 }{ I^2_0( \tilde \kappa ) } e^{\tilde \kappa [ \cos ( x_1 - \bar x_1 )  + \cos ( x_2 - \bar x_2 ) ] }.
\end{equation}
where $\tilde \kappa > 0 $ and $ \bar x_1, \bar x_2 \in [ 0, 2 \pi ) $. This density represents uncertainty in both the velocity field state in $ A $ and the tracer positions in $ X $. Note that $\rho $ attains its maximum value at the point $ ( 0, \bar x_1, \bar x_2 ) \in M $, and for increasing  $\tilde \kappa $ it becomes increasingly concentrated around that point. In what follows, we visualize the temporal evolution of $ \rho_t = U^*_t \rho $ through the marginal density $ \sigma_t = \int_A \rho_t( a, \cdot ) \, d \alpha $ on $ X $; the quantity $ \sigma_t( x ) $ is equal to the probability density to find a tracer at the point  $ x \in X $ at lead time $ t $ marginalized over all flow states $ a \in A $ given the initial density  $\rho $. We also examine the one-dimensional (1D) marginal densities $ \sigma_{1,t}( x_1) = \int_0^{2\pi} \sigma_t( x_1, x_2 ) \, dx_2 / ( 2 \pi ) $ and $ \sigma_{2,t}( x_2) = \int_0^{2\pi} \sigma_t( x_1, x_2 ) \, dx_1 / ( 2 \pi ) $. Clearly, the evolution of  $ \sigma_t $, $ \sigma_{1,t} $, and $ \sigma_{2,t} $ is neither Markovian nor unitary; that is, the $ L^2 $ norms of $ \sigma_t $, $ \sigma_{1,t} $, and $ \sigma_{2,t} $ with respect to the appropriate normalized Haar measures are generally not constant (though the corresponding $ L^1 $ norms are, of course, always equal to 1).        

Figure~\ref{figMovingRho} and Movie~\href{http://cims.nyu.edu/~dimitris/files/tracers/movie3.mp4}{3}  show the evolution of $ \sigma_t $, $ \sigma_{1,t} $ and $ \sigma_{2,t} $  for the initial probability density $ \rho $ from~\eqref{eqRhoAX} with $ \tilde \kappa  = 4 $ and $ ( \bar x_1, \bar x_2 ) = ( \pi, \pi/ 4) $. With these parameters, $ \sigma_t $ at $ t = 0 $ is concentrated upstream of the vortex center (for the state $ a = 0 \in A $ maximizing the initial density $ \rho_A $) along the $ x_1 $ direction, and has a positive offset along the $ x_2 $ direction (see top-left panel in Fig.~\ref{figMovingRho}). As a result, at $ t > 0$ the initially isotropic $ \sigma_t $ is swept and sheared by the moving vortex, producing an anisotropic density. This process is repeated with each revolution of the vortex center around the periodic domain, leading to a filamentary density with highly non-Gaussian features. As noted earlier, this flow produces significantly more mixing along the $ x_1 $ direction compared to the $ x_2 $ direction, and as a result, even after $ \sim 10 $ periods ($t=20\pi$) most of the probability mass remains concentrated in the $ x_2 > \pi $ portion of the domain.

\begin{figure}
  \centering\includegraphics[width=.38\linewidth]{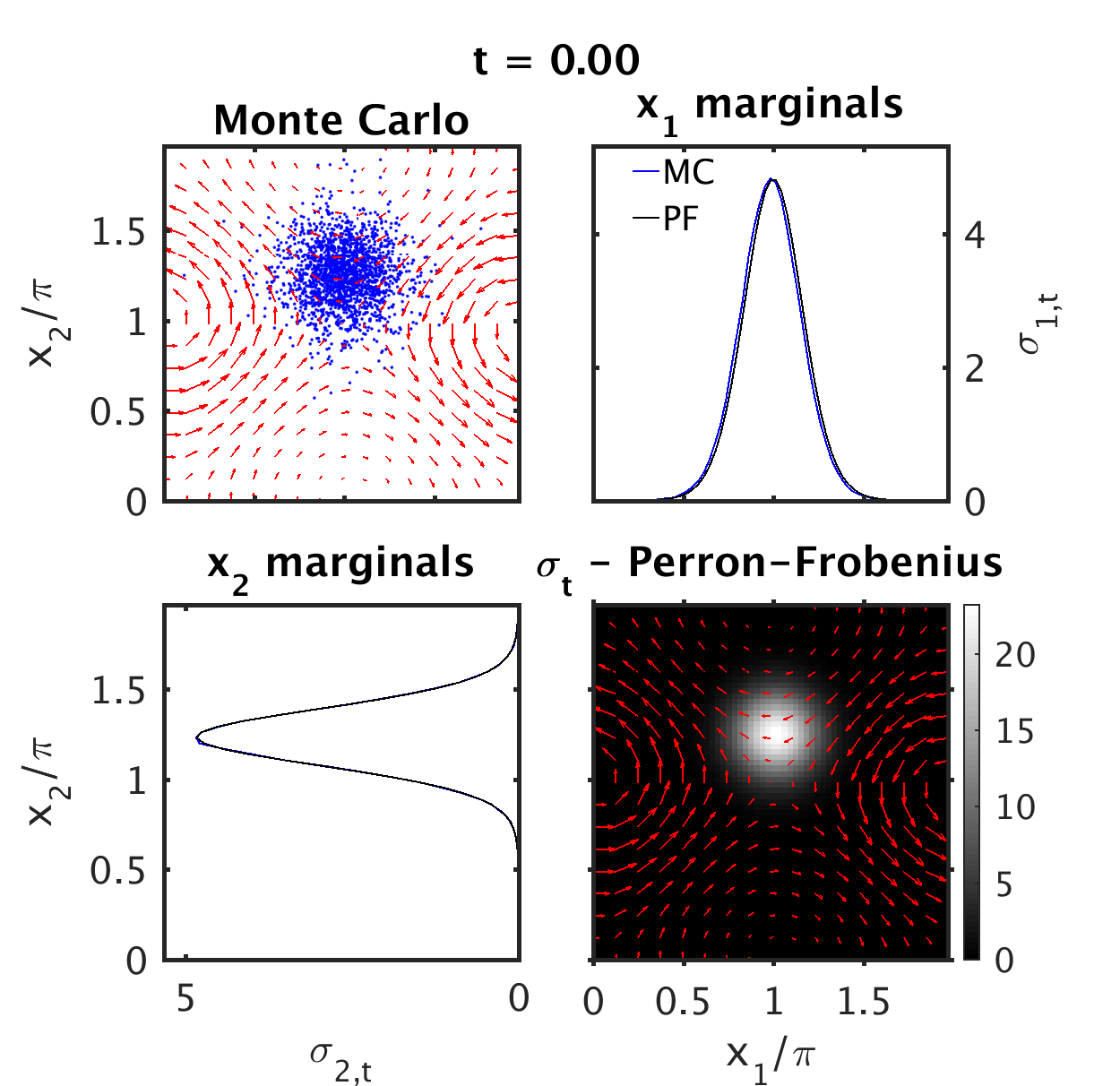}   \includegraphics[width=.38\linewidth]{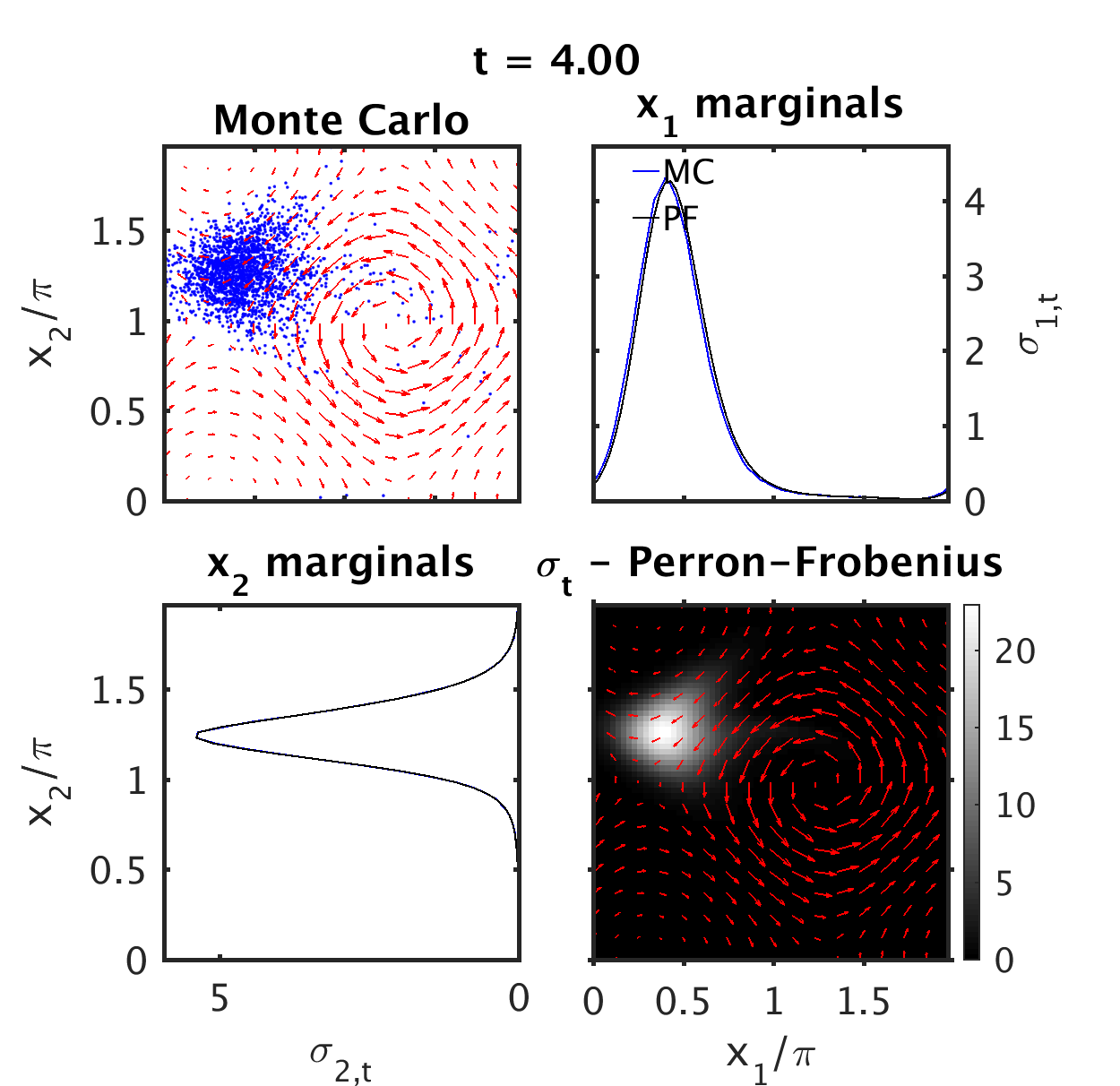} \\
\includegraphics[width=.38\linewidth]{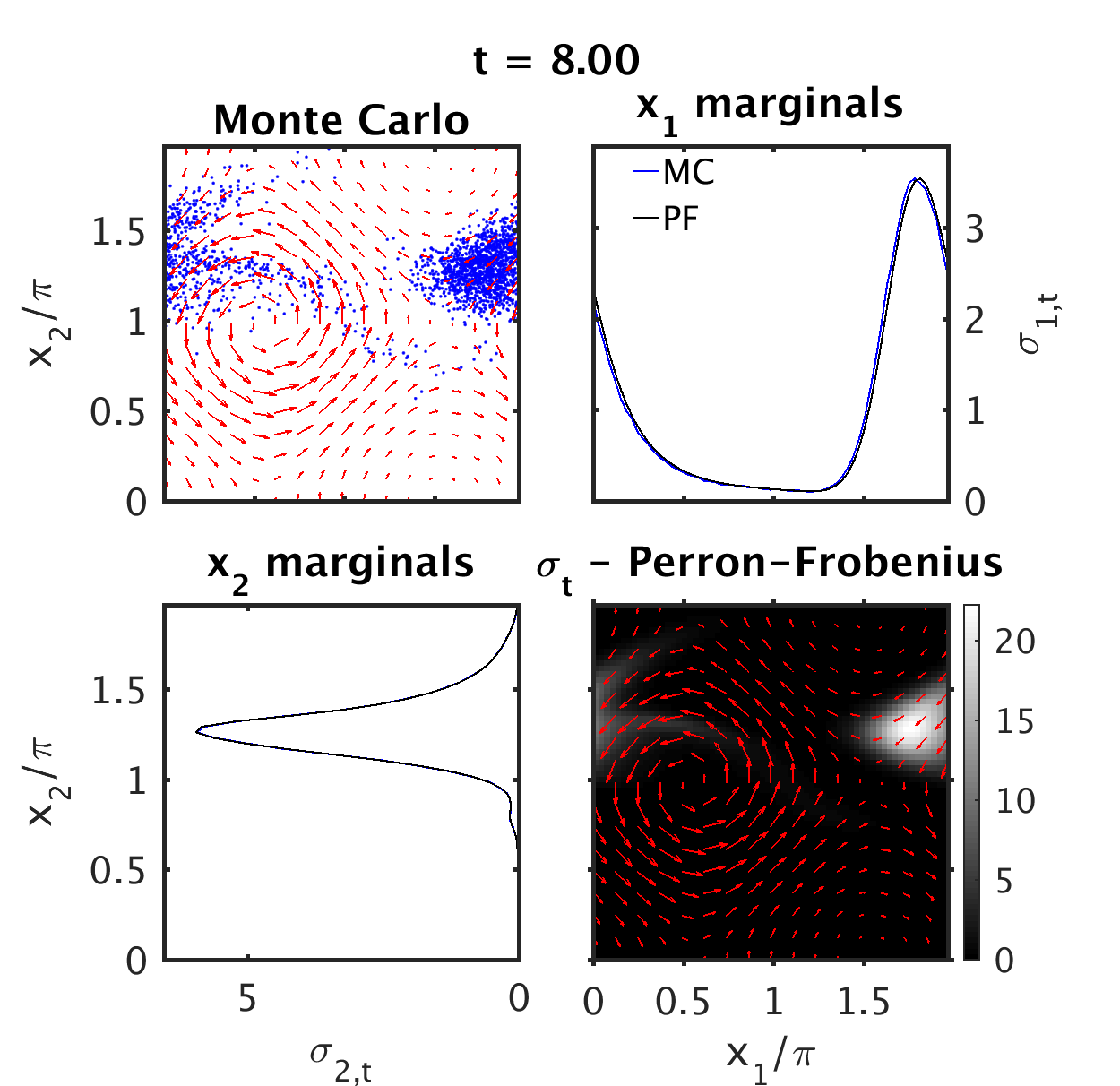}  \includegraphics[width=.38\linewidth]{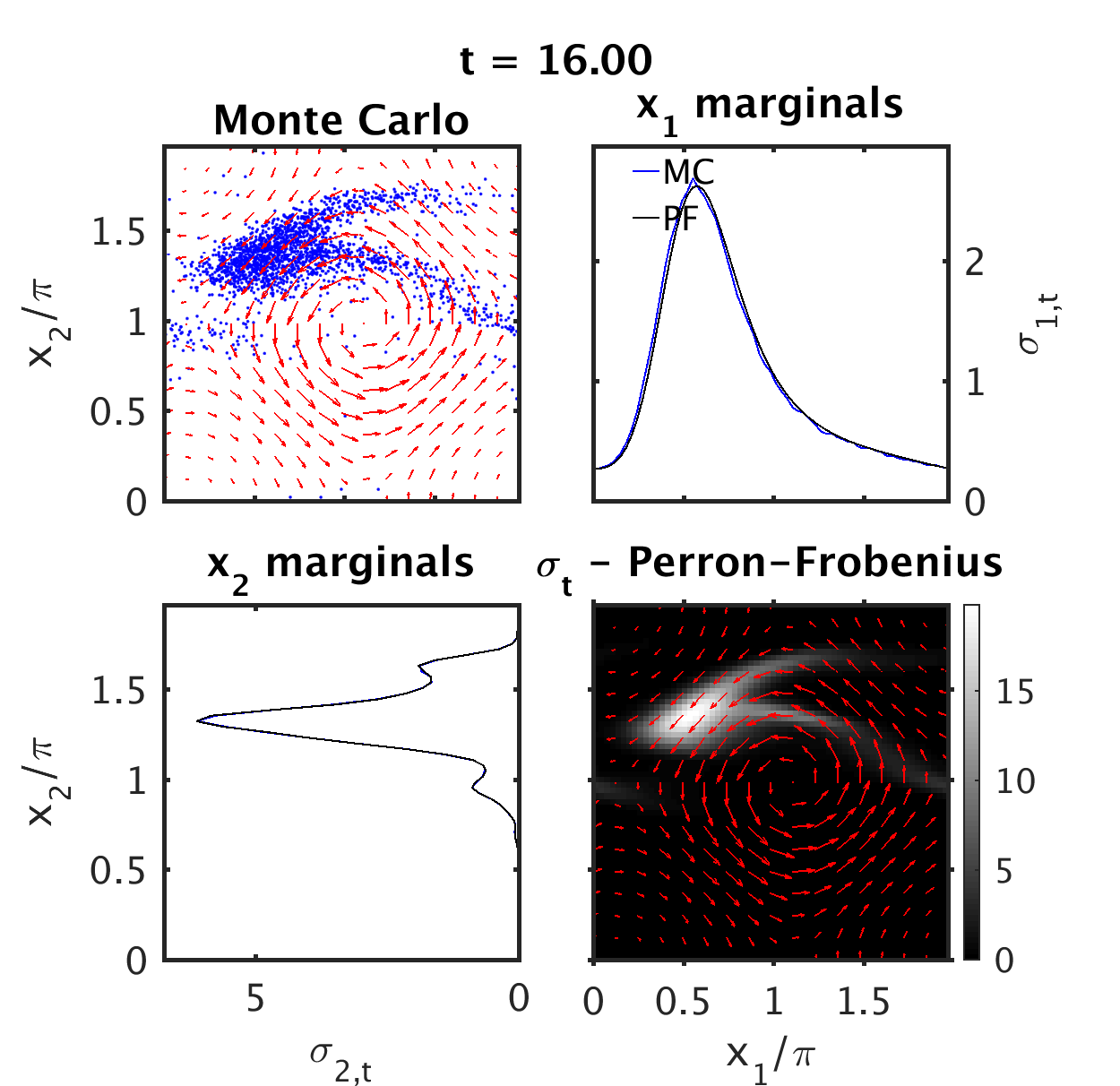} \\
\includegraphics[width=.38\linewidth]{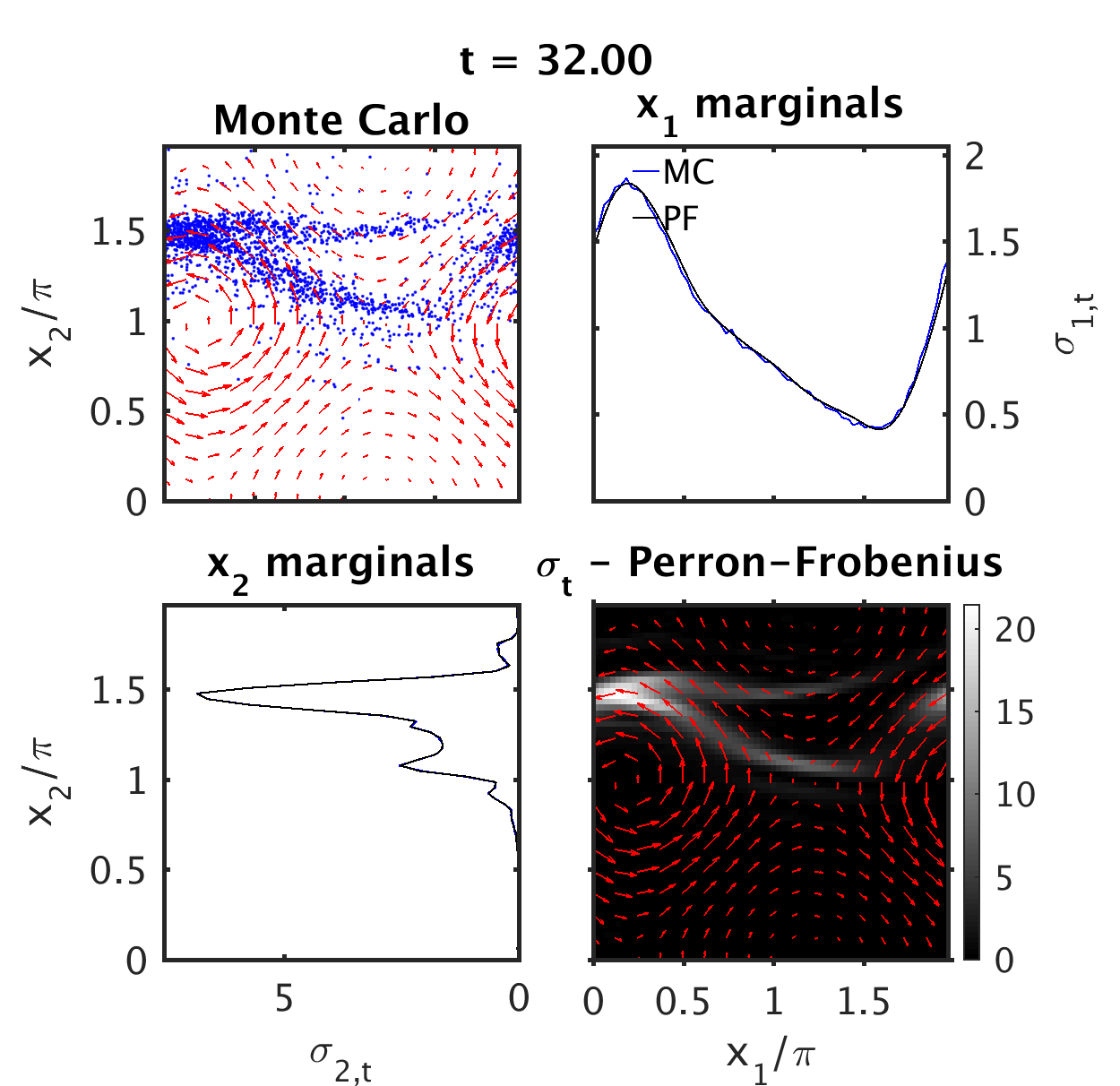}  \includegraphics[width=.38\linewidth]{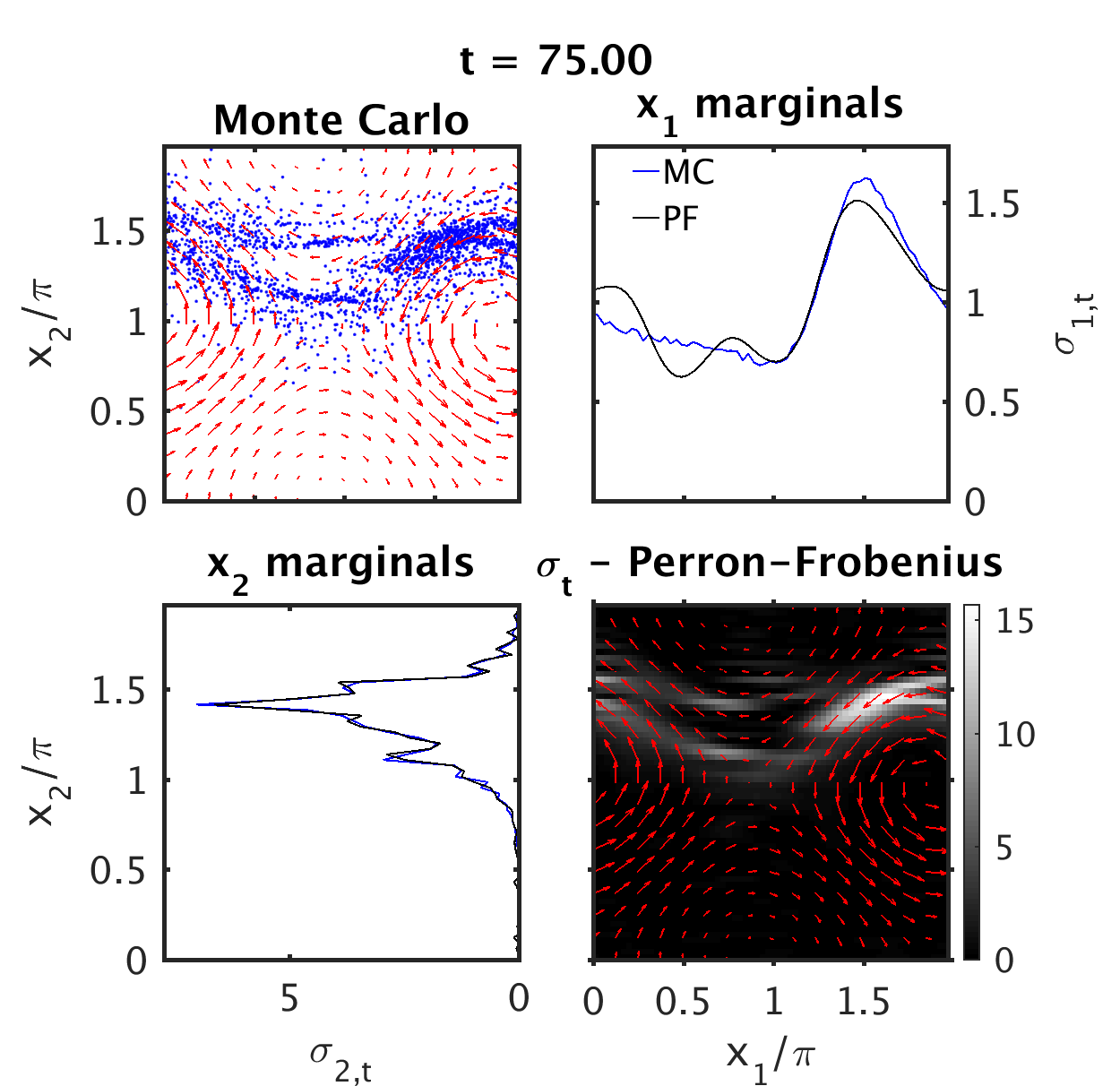}
  \caption{\label{figMovingRho}Snapshots of the temporal evolution of the marginal density $ \sigma_t $ under the moving-vortex flow determined through the operator-theoretic approach of Section~\ref{secBoundedDen} (grayscale colors) and a Monte Carlo simulation (blue dots) based on the full model for Lagrangian advection under the time-dependent streamfunction in~\eqref{eqZetaMoving}. Also shown are the 1D marginal densities $ \sigma_{1,t} $ and $\sigma_{2,t} $ determined from the operator-theoretic model (black lines) and the Monte Carlo simulation (blue lines). The size of the Monte Carlo ensemble is 50,000 with initial conditions drawn independently from the density $ \rho $ in~\eqref{eqRhoAX}, though for clarity of visualization only 2000 particles from this ensemble (drawn randomly) are shown in blue dots. The marginal densities $ \sigma_{1,t} $ and $ \sigma_{2,t} $ from the Monte Carlo ensemble were estimated by binning over uniform-sized bins of width $ 2 \pi / 65 $. The red arrows show the time-dependent velocity field evaluated for the state $ a = \Phi_t( 0 ) = \omega t \mod 2 \pi $; i.e., the mode of the probability density $ U_t^* \rho_A $ where $ \rho_A $ is the initial density on $ A $ from~\eqref{eqRhoAX}. }
\end{figure}

Overall, the results in Fig.~\ref{figMovingRho} and Movie~\href{http://cims.nyu.edu/~dimitris/files/tracers/movie3.mp4}{3} illustrate that, at least as far as the marginal densities are concerned, the data-driven operator-theoretic approach in Section~\ref{secBoundedDen} agrees well with the results from the Monte Carlo simulation based on the full model for Lagrangian tracer advection. In fact, the range of accurate forecasts from the data-driven model is longer than in the case of the observables $ f_1 $ and $f_2 $ characterizing the tracer positions (Fig.~\ref{figMovingX} and Movie~\href{http://dimitris.cims.nyu.edu/~dimitris/files/tracer/movie2.mp4}{2}). This is likely due to the fact that the marginal density $ \sigma_t $ is given by projection of the full density $ \rho_t $ to the subspace of $ H$ spanned by functions $ f $ that do not depend on state $ a \in A $ (i.e., $ w^A(f) = 0 $; see Section~\ref{secPrelim}), and moreover $ \sigma_{1,t} $ and $ \sigma_{2,t} $ involve additional projections onto subspaces of  $H $ spanned by $ x_ 2 $- and $ x_1 $-independent functions, respectively. Such projections are expected to cancel at least some of the errors that may be present in the full density $ \rho_t $, contributing to an increase of forecast skill for the marginal densities. In other words, these examples are a demonstration of the (perhaps, obvious) fact that for a fixed dynamical system and spectral truncation the prediction skill is observable dependent.        
  
\subsection{Switching Gaussian vortices}

Our second example uses the same torus domain as the moving-vortex example in Section~\ref{secMovingVortex}, but in this case we consider the streamfunction
\begin{equation}
  \label{eqZetaSwitching}
  \zeta( a )( x ) = C \cos( a ) e^{\kappa ( \cos x_1 + \cos x_2 ) }  + C \sin( a ) e^{\kappa  [  \cos( x_1 - \pi  ) + \cos x_2 ] },  
\end{equation}
where $ C $ and $ \kappa $ are positive vortex strength and concentration parameters, respectively. Coupled with the periodic dynamics on $ A $,  the streamfunction in~\eqref{eqZetaSwitching} describes a pair of vortices centered at $ ( x_1, x_2 ) = ( 0, 0 ) $ and $ ( x_1, x_2 ) = ( \pi, 0 ) $ in $ X $ (as usual, $ x_1 $ and $ x_2 $ represent canonical angle coordinates), and modulated in time by 90$^\circ $ out of phase sinusoids. Representative snapshots over half a period of this flow are shown in Fig.~\ref{figPsiSwitching}. There, it can be seen that an anticklockwise vortex centered at $ ( 0, 0 ) $ is present at $ a = 0 $, and this vortex is replaced by a vortex of the same sense centered at  $( \pi, 0 ) $ at $ a = \pi/ 2$, which is in turn replaced by an anticklockwise vortex centered at $ (0,0) $ at $ a = \pi $. In the intervening states $ a = \pi/4 $ and $ a = 3 \pi / 4 $, the flow has the structure of shear layers parallel to the $ x_1 $ and $ x_2 $ coordinate, respectively. This process is repeated over the interval $ a \in [ \pi, 2 \pi ) $ with a change of vortex sign. 

\begin{figure}
  \centering\includegraphics[width=\linewidth]{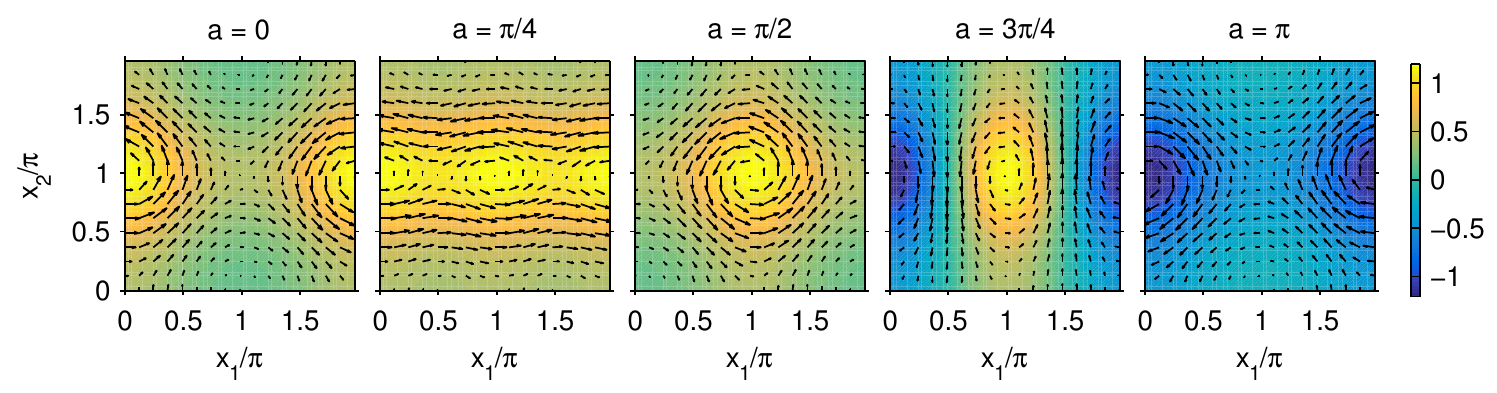}
  \caption{\label{figPsiSwitching}Streamfunction (colors) and velocity field (arrows) for the switching-vortex streamfunction in~\eqref{eqZetaMoving}  evaluated for $C = 2$, $ \kappa = 0.5 $,  and representative values of $ a $ in the interval $ [ 0, \pi ] $. For clarity of visualization of the center of the vortex, the $ x_2 $ coordinate has been periodically shifted to $ x_2 + \pi $.}  
\end{figure}

Loosely speaking, this ``switching vortex'' flow can be thought of as a continuous analog of a piecewise-constant stirring flow by point vortices \cite{Aref84} (also known as blinking vortex flow), which is known to exhibit chaotic advection. At the very least, one would expect the streamfunction in~\eqref{eqZetaSwitching} to produce more complex spectral behavior than the moving-vortex example in Section~\ref{secMovingVortex}, for the former is not decomposable via submersions associated with Galilean transformations as the latter is (see~\ref{appMoving}). In particular, as is evident from Fig.~\ref{figPsiSwitching}, the streamlines associated with~\eqref{eqZetaSwitching} undergo changes in topology as $ a $ varies, meaning that it is not possible to construct a submersion analogous to $ \Gamma $ in~\ref{appMoving} mapping the state-dependent velocity field on $ X $ to a steady velocity field with closed streamlines (the latter would imply the existence of non-trivial eigenfunctions of the generator $ \tilde w $ at eigenvalue zero). 

Qualitatively, the larger the parameter $ C $ in~\eqref{eqZetaSwitching} is, the more overturning we expect the tracers to undergo during one period of the dynamics on $ A $, leading to stronger mixing. In fact, while we do not have rigorous results to justify this assertion, it appears plausible that for sufficiently large $ C $ the only eigenfunctions of $\tilde w $ for the class of streamfunctions in~\eqref{eqZetaSwitching} are constant on $ X$. If this is indeed the case, then all numerical eigenfunctions we compute are ``genuinely approximate'' Koopman eigenfunctions, in the sense that they converge to eigenfunctions of the regularized generator $ L $ as the spectral order parameter $ \ell $ tends to infinity (such eigenfunctions always exist by Proposition~\ref{propL}(v)), but the eigenfunctions of $ L $ do not converge to eigenfunctions of $ \tilde w $ in the limit of vanishing diffusion regularization parameter $ \theta $. Thus, in such situations, diffusion regularization is expected to play an essential role for the recovery of coherent patterns via the approach of Section~\ref{secReg}.
  
Following a similar approach as in Section~\ref{secMovingVortex}, we compute the matrix elements of the generator in the Fourier basis $ \{ \phi_{ijk} \} $ of  $M $ using the integral identity for circular Gaussian integrals in~\eqref{eqBessel}. Specifically, we have
\begin{equation}
  \label{eqWXSwitch}
  \begin{aligned}
    \langle \phi_{ijk}, w^X( \phi_{lmn} ) \rangle &= \frac{ C \kappa }{ 4 I^2_0(\kappa) }( \beta_{jkmn} \delta_{i,l+1} + \beta^*_{jkmn} \delta_{i,l-1} ),  \\
      \beta_{jkmn} &= \left( 1 - \ii (-1)^{j-m} \right)   \left[ n I_{\lvert n-k\rvert}( \kappa ) \left( I_{\lvert m -j + 1 \rvert}( \kappa ) - I_{\lvert m - j - 1 \rvert}( \kappa ) \right) \right.\\
      &  \quad \left. - m I_{\lvert m -j \rvert}( \kappa ) \left(  I_{\lvert n - k + 1 \rvert} ( \kappa)  - I_{\lvert n - k - 1 \rvert }( \kappa ) \right) \right],
    \end{aligned}
\end{equation}
and the matrix elements $ \langle \phi_{ijk}, w^A( \phi_{lmn} ) \rangle $ are given  by~\eqref{eqWAMov} as in the moving-vortex flow. Using these results, we apply the operator-theoretic schemes of Section~\ref{secKoopman} using the same spectral truncation parameters $ \ell_A = \ell_{X_1} = \ell_{X_2} = 32 $ as in Section~\ref{secMovingVortex}. Throughout, we work with the flow frequency parameter $ \omega = 1 $.    

\subsubsection{Koopman eigenvalues and eigenfunctions}

Our objectives in this section are to study the properties of coherent spatiotemporal patterns recovered by numerical Koopman eigenfunctions for different values of the vortex strength parameter $ C $, as well as the dependence of these patterns to the diffusion regularization parameter $ \theta $. Figure~\ref{figZSwitching1} shows snapshots of representative eigenfunctions computed for the flow with the streamfunction parameters $ \kappa = 0.5 $ and  $ C = 2 $, and the diffusion regularization parameters $ \theta = 10^{-3} $ and $ 10^{-4} $. Videos showing these eigenfunctions are provided in Movies~\href{http://cims.nyu.edu/~dimitris/files/tracers/movie4.mp4}{4} and~\href{http://cims.nyu.edu/~dimitris/files/tracers/movie5.mp4}{5}, respectively; Table~\ref{tableZSwitching1} lists the corresponding eigenvalues and Dirichlet energies. As in Section~\ref{secMovingVortex}, we computed these results using Matlab's \texttt{eigs} solver, requesting 51 eigenvalues of minimal modulus. 

First, consider the results for $ \theta = 10^{-4} $. The numerical eigenfunctions in this case includes a family, which includes eigenfunctions $ z_1 $, $ z_3 $, $ z_5 $, $ z_7 $, $ z_9, $, $ z_{11} $, $ z_{13} $, $ z_{17} $, and $ z_{21} $ shown in Fig.~\ref{figZSwitching1} and Movie~\href{http://cims.nyu.edu/~dimitris/files/tracers/movie5.mp4}{5} which  are characterized by four globular patterns undergoing a sloshing motion as a result of advection by the time-dependent velocity field. According to the results in Table~\ref{tableZSwitching1}, the eigenvalues $ \lambda_k $ corresponding to these patterns have essentially vanishing imaginary part (that is, twelve orders of magnitude smaller than $ \Real \lambda_k $), indicating that, at least to the resolution afforded by our spectral truncation, these patterns are conserved on Lagrangian tracers; that is, the level sets of these eigenfunctions induce an ergodic partition on the state space $ M $. As expected, the eigenfunctions in this set  with higher Dirichlet energy exhibit increasingly smaller-scale oscillations, which are arranged in a concentric manner relative to the center of each cluster. In other words, the eigenfunctions in this family provide an increasingly fine partition of $M $ into quasi-invariant sets. Hereafter, we refer to eigenfunctions of this class as class~1 eigenfunctions.  

\begin{figure}
  \centering\includegraphics[width=.8\linewidth]{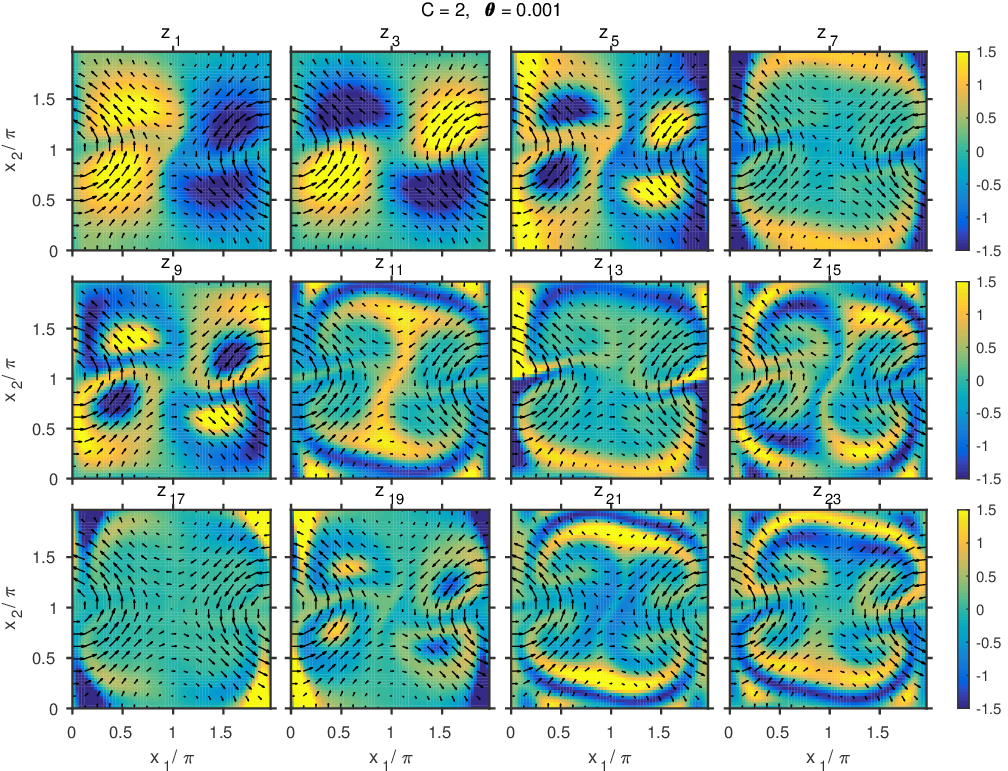}\\
  \includegraphics[width=.8\linewidth]{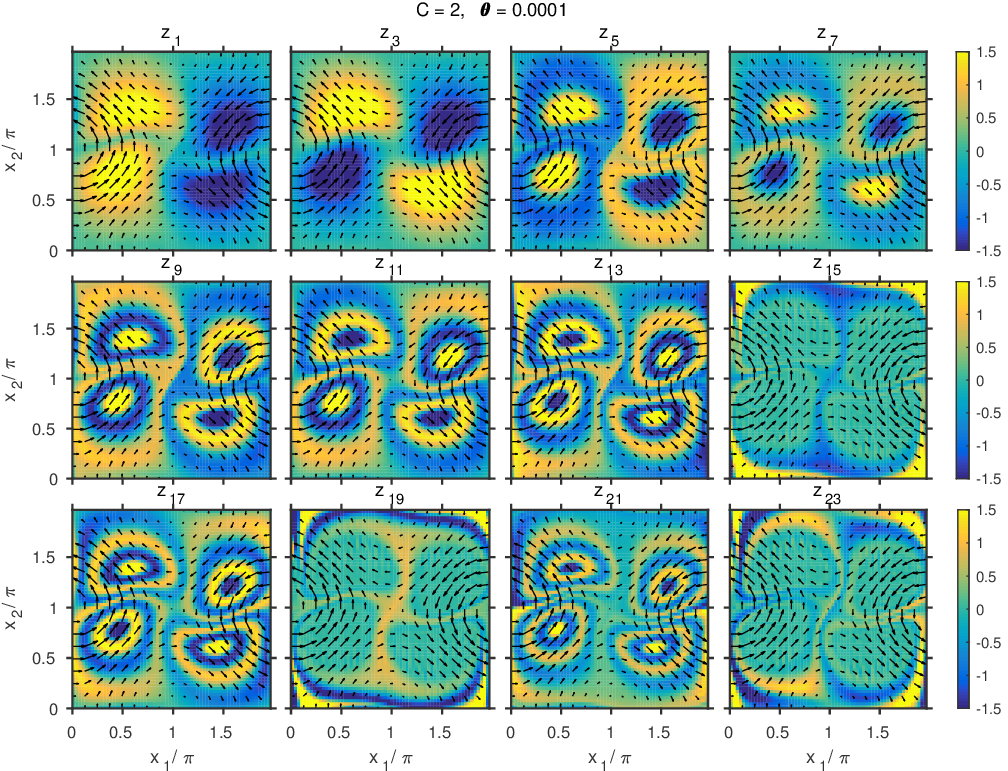}
  \caption{\label{figZSwitching1}Snapshots of the real parts of representative numerical Koopman eigenfunctions (colors) and the velocity field  $ v\rvert_a $ (arrows) for the switching-vortex flow with $ \omega = 1 $, $ \kappa = 0.5 $, and $  C = 2 $. Two sets of numerical eigenfunctions are shown, computed with the diffusion regularization parameter $ \theta = 10^{-3} $ (top) and $ \theta = 10^{-4} $ (bottom). For clarity of visualization of the vortex centers, the $ x_2 $ coordinate has been periodically shifted to $ x_2 + \pi $ in all panels.}  
\end{figure}

\begin{table}
  \centering\small
  \caption{\label{tableZSwitching1}Eigenvalues $ \lambda_k $ and Dirichlet energies $ E_k $ of the numerical Koopman eigenfunctions $ z_k $ for the switching-vortex flow with $ C = 2 $ shown in Fig.~\ref{figZSwitching1}.}
  \begin{tabular*}{\linewidth}{@{\extracolsep{\fill}}lllll}
    \hline
    & \multicolumn{2}{c}{$\theta = 0.001 $} & \multicolumn{2}{c}{$\theta=0.0001 $}\\
    \cline{2-3} \cline{4-5}
    & \multicolumn{1}{c}{$\lambda_k$} & \multicolumn{1}{c}{$E_k$} & \multicolumn{1}{c}{$\lambda_k$} & \multicolumn{1}{c}{$E_k$}\\
    $z_{1}$ & $ -2.28e-03 - 2.97e-17 \,\ii $ & $ 2.28e+00 $  &  $ -2.70e-04 - 7.16e-17 \,\ii $ & $ 2.70e+00 $ \\ 
    $z_{3}$ & $ -2.93e-03 - 5.16e-18 \,\ii $ & $ 2.93e+00 $  &  $ -2.94e-04 - 1.85e-16 \,\ii $ & $ 2.94e+00 $ \\ 
    $z_{5}$ & $ -1.11e-02 - 1.05e-17 \,\ii $ & $ 1.11e+01 $  &  $ -1.32e-03 + 4.11e-16 \,\ii $ & $ 1.32e+01 $ \\ 
    $z_{7}$ & $ -1.17e-02 + 1.68e-02 \,\ii $ & $ 1.17e+01 $  &  $ -1.45e-03 + 4.20e-16 \,\ii $ & $ 1.45e+01 $ \\ 
    $z_{9}$ & $ -1.41e-02 + 5.82e-17 \,\ii $ & $ 1.41e+01 $  &  $ -3.17e-03 + 8.79e-16 \,\ii $ & $ 3.17e+01 $ \\ 
    $z_{11}$ & $ -1.67e-02 + 3.54e-02 \,\ii $ & $ 1.67e+01 $  &  $ -3.49e-03 + 4.20e-16 \,\ii $ & $ 3.49e+01 $ \\ 
    $z_{13}$ & $ -2.27e-02 + 2.44e-02 \,\ii $ & $ 2.27e+01 $  &  $ -5.86e-03 + 6.70e-16 \,\ii $ & $ 5.86e+01 $ \\ 
    $z_{15}$ & $ -2.42e-02 - 4.33e-02 \,\ii $ & $ 2.42e+01 $  &  $ -6.28e-03 + 1.44e-02 \,\ii $ & $ 6.28e+01 $ \\ 
    $z_{17}$ & $ -2.61e-02 + 2.03e-02 \,\ii $ & $ 2.61e+01 $  &  $ -6.41e-03 - 4.72e-16 \,\ii $ & $ 6.41e+01 $ \\ 
    $z_{19}$ & $ -2.71e-02 - 3.62e-16 \,\ii $ & $ 2.71e+01 $  &  $ -9.04e-03 - 2.40e-02 \,\ii $ & $ 9.04e+01 $ \\ 
    $z_{21}$ & $ -2.73e-02 + 5.51e-02 \,\ii $ & $ 2.73e+01 $  &  $ -9.65e-03 + 2.07e-16 \,\ii $ & $ 9.65e+01 $ \\ 
    $z_{23}$ & $ -2.90e-02 + 4.87e-02 \,\ii $ & $ 2.90e+01 $  &  $ -1.01e-02 - 2.31e-02 \,\ii $ & $ 1.01e+02 $ \\
    \hline
  \end{tabular*}
\end{table}

In addition to the patterns described above, the numerical Koopman eigenfunctions include a family (hereafter, class~2), members of which are $ z_{15} $, $ z_{19} $, and $ z_{23} $ displayed in Fig.~\ref{figZSwitching1} and Movie~\href{http://cims.nyu.edu/~dimitris/files/tracers/movie5.mp4}{5}, which are concentrated in the separatrix regions between the four globular clusters of the previous family. The eigenvalues corresponding to these eigenfunctions have nonzero imaginary parts (e.g., $ \Imag \lambda_k = O( 10^{-2} ) $ for $ z_{15} $, $ z_{19} $, and $ z_{23} $), indicating that these eigenfunctions vary periodically in a Lagrangian frame following the tracers. The concentration of these eigenfunctions to subsets of $X $ of small Lebesgue measure and the sharp spatial gradients that they exhibit (see, e.g., $ z_{19} $ in Fig.~\ref{figZSwitching1} along the $ x_1 = x_2 $ diagonal near $ (0,0) $) are reminiscent of the behavior of numerical eigenfunctions $ z_{29} $ and $ z_{33} $ in Fig.~\ref{figZMoving} associated with the continuous spectrum of the moving vortex flow. 

Next, we examine the dependence of the patterns described above on the diffusion regularization parameter $ \theta $. Inspecting the results in Fig.~\ref{figZSwitching1} and Movie~\href{http://cims.nyu.edu/~dimitris/files/tracers/movie4.mp4}{4}, it is evident that the leading class~1 eigenfunctions identified for $ \theta = 10^{-4} $ are relatively robust under an increase of $ \theta $ from $ 10^{-4} $ to $ 10^{-3} $; for example,  eigenfunctions $ \{ z_1, z_3, z_5, z_9 \} $ have clear  $ \theta = 10^{-3} $ counterparts, $ \{ z_1, z_3, z_5, z_7 \} $. More quantitatively, the Dirichlet energies $ E_k $ of those eigenfunctions (see Table~\ref{tableZSwitching1}) do not change by more than $ \simeq 20\% $ (and in some cases this change is as little as $ \simeq 1\%) $. Together, these results suggest that class~1 eigenfunctions behave smoothly as $ \theta \to 0 $ and, correspondingly, that   they approximate true Koopman eigenfunctions. In contrast, class~2 eigenfunctions such as $ z_{15} $, $ z_{19} $, and $ z_{23} $ at $ \theta = 10^{-4} $ have a significantly more sensitive dependence on $ \theta $. That is, while analogs of some of these eigenfunctions can be identified in the $ \theta = 10^{-3} $ results (e.g., $z_{15} $ and $ z_{19} $ at $ \theta = 10^{-4} $ appear to be related to $ z_7 $ and $ z_{11} $ at $ \theta = 10^{-3} $, respectively), the spatial patterns of these eigenfunctions are visibly affected by diffusion (compare, e.g., $z_{19} $ at $ \theta = 10^{-4} $ with $ z_{11} $ at $ \theta=10^{-3} $). This also reflected in the corresponding Dirichlet energies which change by approximately a factor of five between the two cases.         
  
We now increase the vortex strength parameter to $ C = 4 $ and examine the resulting numerical eigenfunctions, using again the diffusion regularization parameter values $ \theta = 10^{-3} $ and $ 10^{-4} $. The resulting eigenfunctions and eigenvalues are shown in Fig.~\ref{figZSwitching2}, Movie~\href{http://cims.nyu.edu/~dimitris/files/tracers/movie6.mp4}{6} ($\theta=10^{-3}$), Movie~\href{http://cims.nyu.edu/~dimitris/files/tracers/movie7.mp4}{7} ($\theta=10^{-4}$), and Table~\ref{tableZSwitching2}. There, it is evident that in this flow with stronger stirring and overturning, the geometrically smooth patterns recovered in the $ C = 2 $ experiments are replaced by significantly more complex filamentary patterns. Nevertheless, the eigenfunctions still can be grouped into class~1 and class~2 families, the former characterized by approximately zero imaginary part of the corresponding eigenvalue $ \lambda_k $ and the latter consisting of eigenfunctions concentrated on subsets of the spatial domain $ X $ of small Lebesgue measure while having nonzero $ \Imag \lambda_k $. Examples of class~1 and class~2 eigenfunctions at $ \theta = 10^{-4} $ shown in Fig.~\ref{figZSwitching2} and Movie~\href{http://cims.nyu.edu/~dimitris/files/tracers/movie7.mp4}{7} are $\{ z_1, z_3, z_5 \} $ and $\{ z_7, z_9, z_{11} \} $, respectively.

\begin{figure}
  \centering 
      \includegraphics[width=.8\linewidth]{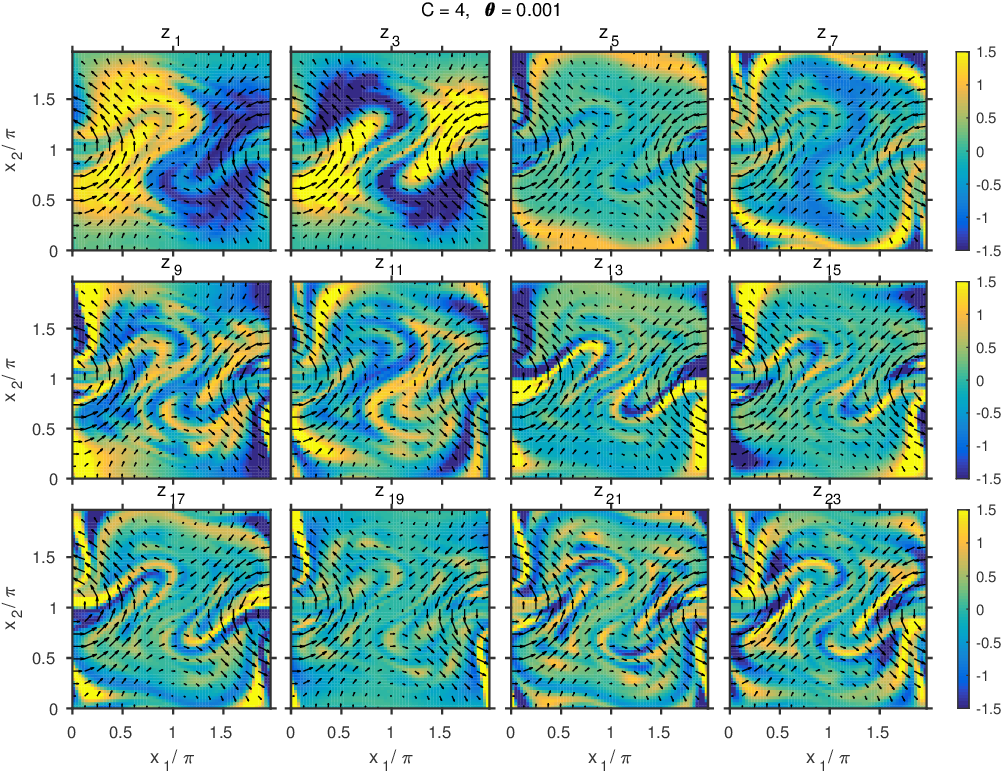}\\
    \includegraphics[width=.8\linewidth]{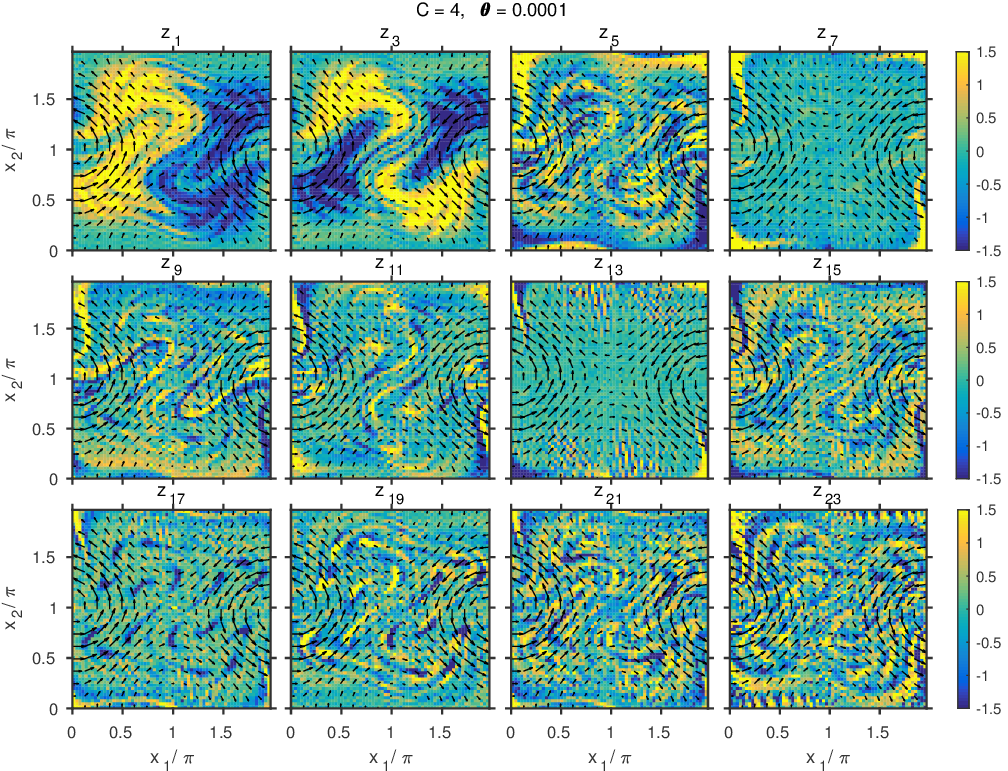}
  \caption{\label{figZSwitching2}As in Fig.~\ref{figZSwitching1}, but for the vortex strength parameter $ C = 4 $.}
\end{figure}

\begin{table}
  \centering\small
  \caption{\label{tableZSwitching2}Eigenvalues $ \lambda_k $ and Dirichlet energies $ E_k $ of the numerical Koopman eigenfunctions $ z_k $ for the switching-vortex flow with $ C = 4 $ shown in Fig.~\ref{figZSwitching2}.}
  \begin{tabular*}{\linewidth}{@{\extracolsep{\fill}}lllll}
    \hline
    & \multicolumn{2}{c}{$\theta = 0.001 $} & \multicolumn{2}{c}{$\theta=0.0001 $}\\
    \cline{2-3} \cline{4-5}
    & \multicolumn{1}{c}{$\lambda_k$} & \multicolumn{1}{c}{$E_k$} & \multicolumn{1}{c}{$\lambda_k$} & \multicolumn{1}{c}{$E_k$}\\
    $z_{1}$ & $ -9.71e-03 + 1.21e-15 \,\ii $ & $ 9.71e+00 $  &  $ -7.78e-03 - 3.94e-16 \,\ii $ & $ 7.78e+01 $ \\ 
    $z_{3}$ & $ -1.50e-02 - 5.63e-16 \,\ii $ & $ 1.50e+01 $  &  $ -9.31e-03 - 7.78e-16 \,\ii $ & $ 9.31e+01 $ \\ 
    $z_{5}$ & $ -4.20e-02 + 5.86e-02 \,\ii $ & $ 4.20e+01 $  &  $ -2.68e-02 + 4.47e-15 \,\ii $ & $ 2.68e+02 $ \\ 
    $z_{7}$ & $ -4.45e-02 - 1.13e-01 \,\ii $ & $ 4.45e+01 $  &  $ -2.83e-02 + 3.78e-02 \,\ii $ & $ 2.83e+02 $ \\ 
    $z_{9}$ & $ -5.35e-02 - 1.82e-15 \,\ii $ & $ 5.35e+01 $  &  $ -2.96e-02 + 4.75e-02 \,\ii $ & $ 2.96e+02 $ \\ 
    $z_{11}$ & $ -5.65e-02 - 1.32e-01 \,\ii $ & $ 5.65e+01 $  &  $ -3.19e-02 - 4.08e-02 \,\ii $ & $ 3.19e+02 $ \\ 
    $z_{13}$ & $ -5.67e-02 - 4.59e-02 \,\ii $ & $ 5.67e+01 $  &  $ -3.21e-02 - 2.84e-02 \,\ii $ & $ 3.21e+02 $ \\ 
    $z_{15}$ & $ -7.75e-02 + 4.81e-02 \,\ii $ & $ 7.75e+01 $  &  $ -3.24e-02 + 6.28e-15 \,\ii $ & $ 3.24e+02 $ \\ 
    $z_{17}$ & $ -7.83e-02 + 1.08e-01 \,\ii $ & $ 7.83e+01 $  &  $ -3.66e-02 + 5.38e-15 \,\ii $ & $ 3.66e+02 $ \\ 
    $z_{19}$ & $ -9.03e-02 - 1.12e-15 \,\ii $ & $ 9.03e+01 $  &  $ -3.91e-02 - 5.49e-15 \,\ii $ & $ 3.91e+02 $ \\ 
    $z_{21}$ & $ -9.15e-02 - 1.23e-01 \,\ii $ & $ 9.15e+01 $  &  $ -4.45e-02 + 9.24e-16 \,\ii $ & $ 4.45e+02 $ \\ 
    $z_{23}$ & $ -9.36e-02 + 1.16e-01 \,\ii $ & $ 9.36e+01 $  &  $ -4.88e-02 - 3.84e-15 \,\ii $ & $ 4.88e+02 $ \\
    \hline
  \end{tabular*}
\end{table}

Despite these similarities, the $ C = 2 $ and $ C = 4 $ results have a fundamental difference in that the former include a class of eigenfunctions (class~1) that are only weakly affected by the examined changes in $ \theta $, whereas in the latter case these changes in $ \theta $ impart a significant change to all of the eigenfunctions. For instance, when $ \theta $ is decreased from $ 10^{-3} $ to $ 10^{-4} $, the Dirichlet energies of class~1 eigenfunctions $ z_1 $ and $ z_3 $ increases by a factor of 8 and 6, respectively (see Table~\ref{tableZSwitching2}).  This implies that the eigenfunctions acquire increasingly small-scale features with decreasing $ \theta $, as is also evident in Fig.~\ref{figZSwitching2} and Movies~\href{http://cims.nyu.edu/~dimitris/files/tracers/movie6.mp4}{6} and~\href{http://cims.nyu.edu/~dimitris/files/tracers/movie7.mp4}{7}. In that regard, these eigenfunctions behave similarly to strange eigenmodes for diffusive Lagrangian tracers in time-periodic flows \cite{Pierrehumbert94}, obtained via Floquet analysis of an advection-diffusion operator in~\cite{LiuHaller04}.  

\subsubsection{Prediction of observables and probability densities} 

In this Section, we discuss prediction of observables and probability densities for the switching-vortex flow. Our experiments are analogous to those for the moving-vortex flow in Section~\ref{secMovingVortexPred}; that is, we consider prediction of the observables $ f_1 $ and $ f_2 $ in~\eqref{eqF12} characterizing the position of Lagrangian tracers, and also consider probability densities with the Gaussian initial conditions $ \rho$ in~\eqref{eqRhoAX}. Hereafter, we restrict attention to the flow with frequency $ \omega = 1 $ and vortex concentration and strength parameters $ \kappa = 0.5 $ and $ C = 4 $, respectively. Moreover, we work throughout with the diffusion regularization parameter $ \theta = 10^{-4} $. 

Figure~\ref{figSwitchingX} and Movie~\href{http://cims.nyu.edu/~dimitris/files/tracers/movie8.mp4}{8} show the evolution of the positions $ (x_1(t,a), x_2(t,a)) $ for an ensemble of tracers (initially arranged on a uniform square grid in $ X $) obtained via the operator-theoretic model from Section~\ref{secBoundedObs}, compared against the evolution obtained by explicit integration of the ODEs governing the advection of tracers in the switching-vortex flow. We used the same methods and error tolerance parameters to compute these results as in Section~\ref{secMovingVortexPred}; here, the forecast timestep is $ \tilde \tau = 0.025 $. In Fig.~\ref{figSwitchingX} and Movie~\href{http://cims.nyu.edu/~dimitris/files/tracers/movie8.mp4}{8} it can be seen that, at least over short to moderate lead times ($ t \lesssim 20$), the operator-theoretic model is able to predict the evolution of the tracer positions with comparable accuracy to the explicit ODE model, but eventually it develops biases due to the combined effects of diffusion and spectral truncation. Thus, the operator-theoretic model performs comparably to the moving-vortex experiment in Section~\ref{secMovingVortexPred}, but note that in this case the dynamics is more complex since mixing takes place with respect to both the $ x_1 $ and $x_2$ coordinates of the tracers.     
\begin{figure}
  \centering
  \includegraphics[width=.42\linewidth]{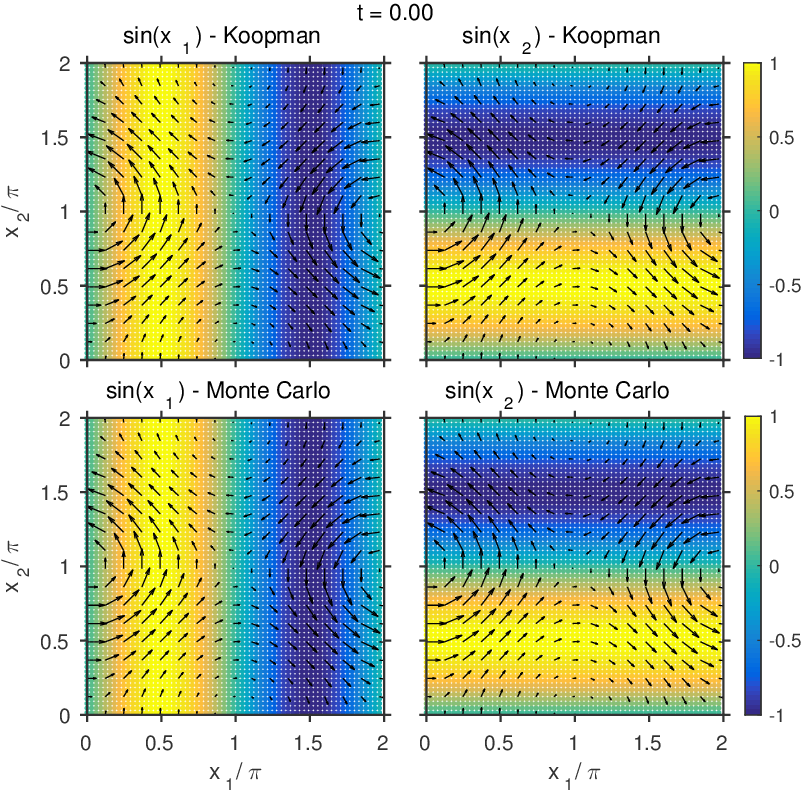}    \includegraphics[width=.42\linewidth]{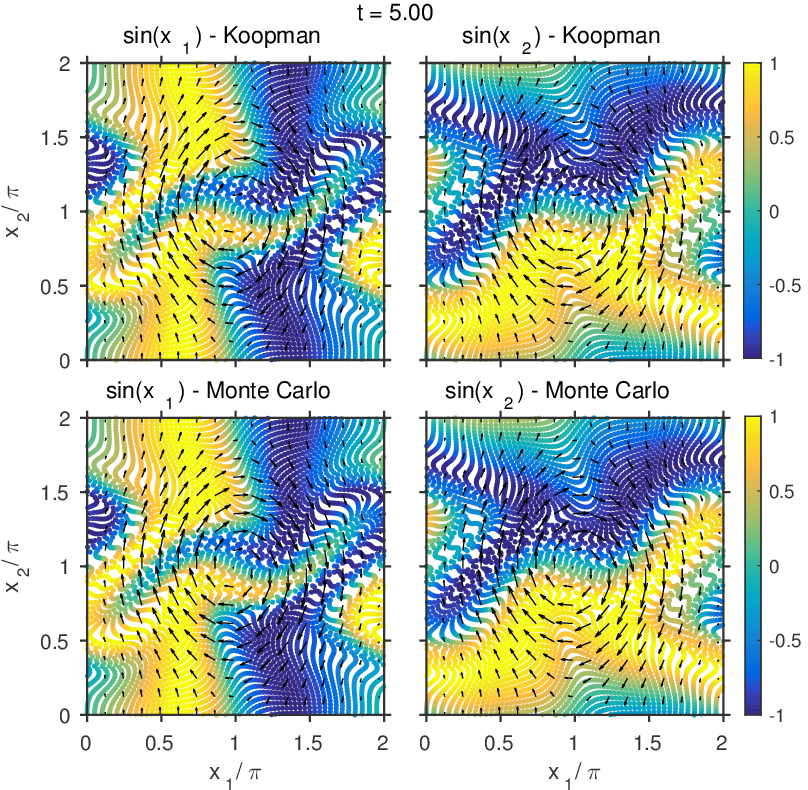}\\ 
  \includegraphics[width=.42\linewidth]{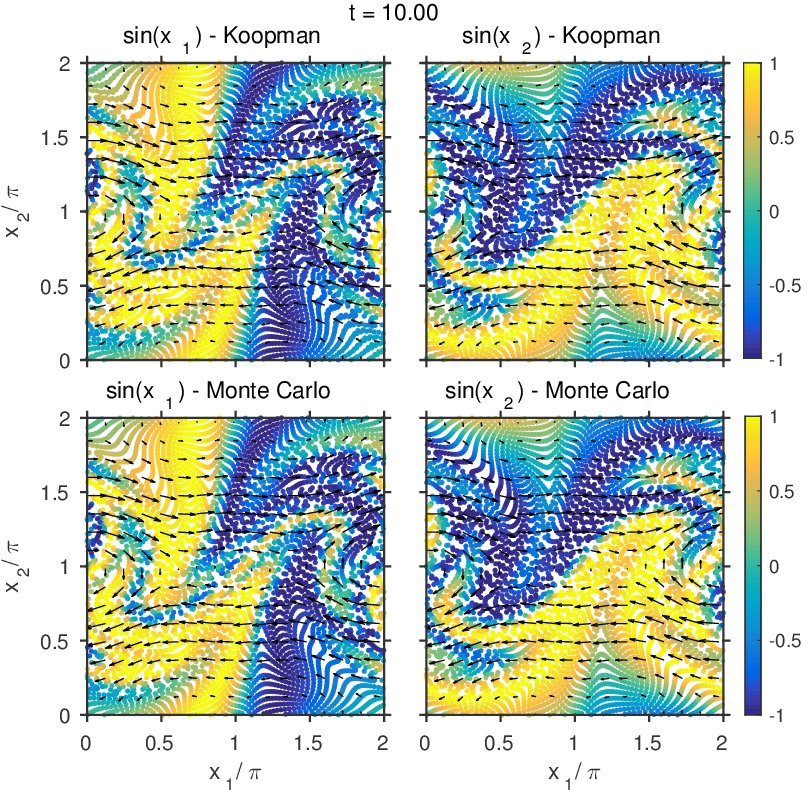}    \includegraphics[width=.42\linewidth]{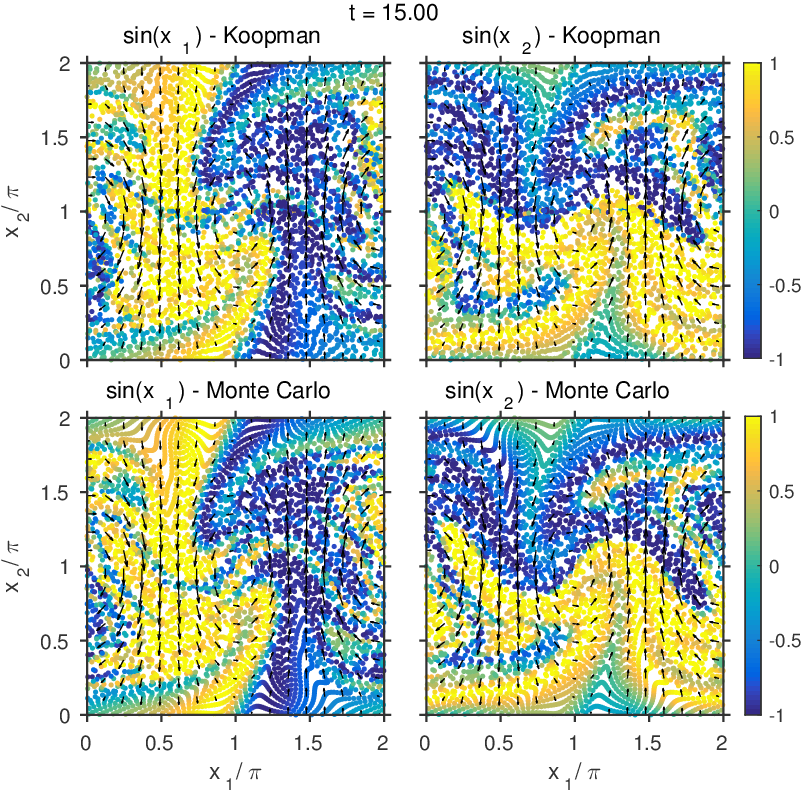}\\
  \includegraphics[width=.42\linewidth]{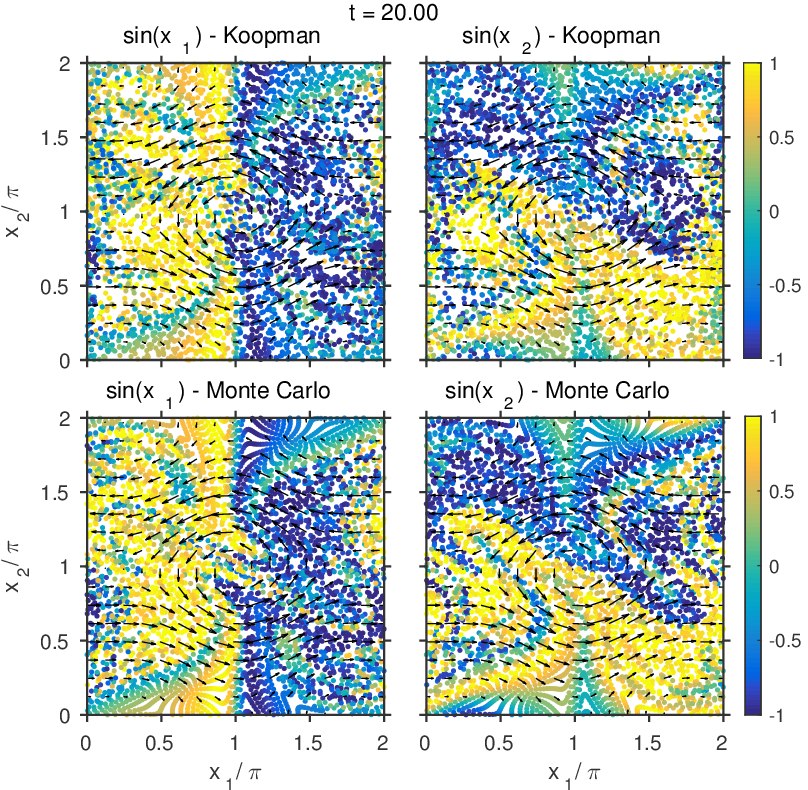} \   \includegraphics[width=.42\linewidth]{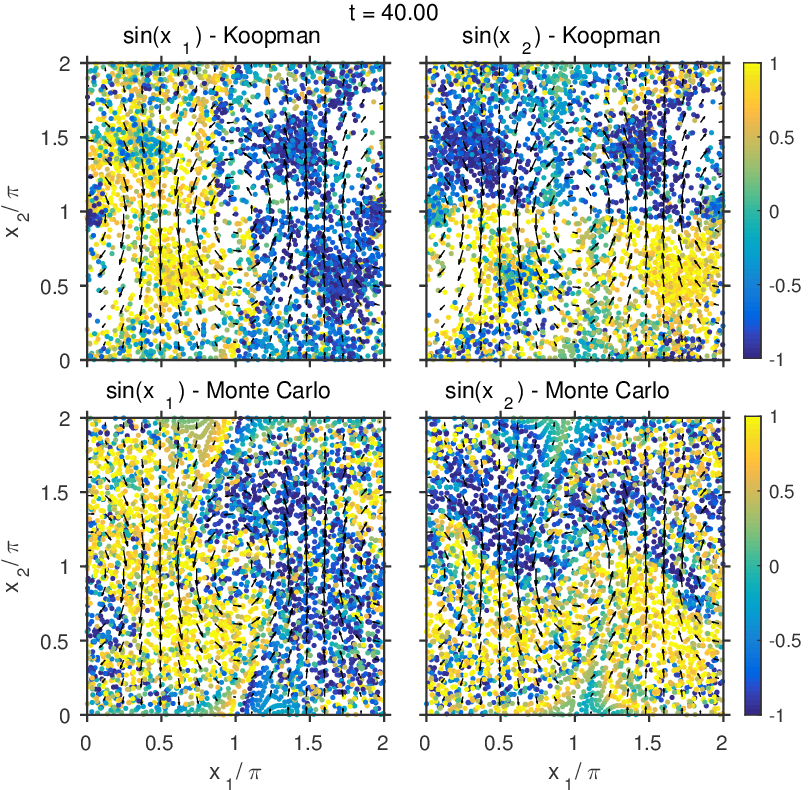}
  \caption{\label{figSwitchingX}As in Fig.~\ref{figMovingX}, but for the switching-vortex flow with the streamfunction in~\eqref{eqZetaSwitching} with $ \omega = 1 $, $ \kappa = 0.5 $, and $ C = 4$. The diffusion regularization parameter is $ \theta = 10^{-4} $.} 
\end{figure}

Figure~\ref{figSwitchingRho} and Movie~\href{http://cims.nyu.edu/~dimitris/files/tracers/movie9.mp4}{9} show the evolution of the marginal densities $ \sigma_t $, $ \sigma_{1,t} $, and $ \sigma_{2,t} $ (see Section~\ref{secMovingVortexPred}) over the forecast interval $ t \in [ 0, 50 ] $ obtained via the operator-theoretic model of Section~\ref{secBoundedDen} and a Monte Carlo ensemble of 300,000 particles. The initial density $ \rho $ on $ M $ is given by~\eqref{eqRhoAX}  with location and concentration parameters $ ( \bar x_1, \bar x_2 ) = ( \pi, 0 ) $ and $ \tilde \kappa = 3 $; as shown in Fig.~\ref{figSwitchingRho}, the corresponding initial marginal density $ \sigma_0 $ is concentrated to the right of the vortex center associated with the state $ a = 0 \in A $ (i.e., the mode of the density $ \rho_A( a ) $). At time $ t > 0 $, the initially radially symmetric density is sheared by the flow, and probability mass is gradually expelled from the vicinity of $ ( \pi, 0 ) $ where $ \sigma_0 $ is concentrated. During this process, $ \sigma_t $ and the 1D densities $ \sigma_{1,t} $ and $ \sigma_{2,t} $ become highly non-Gaussian. As is evident from Fig.~\ref{figSwitchingRho} and Movie~\href{http://cims.nyu.edu/~dimitris/files/tracers/movie9.mp4}{9}, the operator-theoretic model agrees well with the evolution of the marginal densities over the full forecast interval.   
     
\begin{figure}
   \centering
  \includegraphics[width=.42\linewidth]{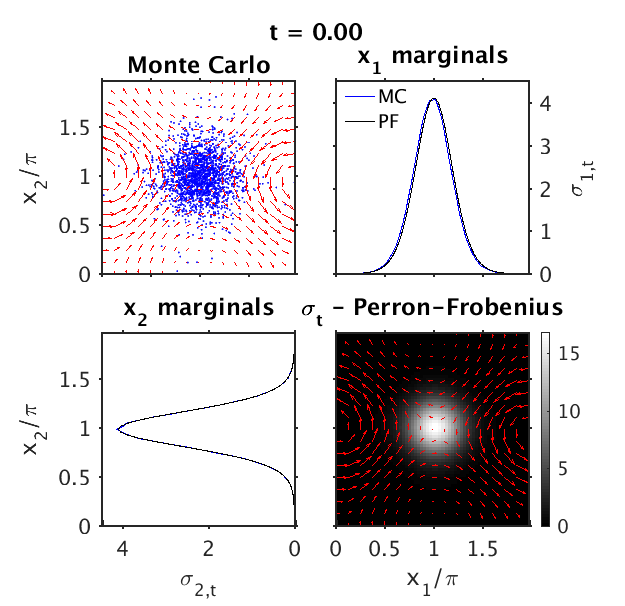}    \includegraphics[width=.42\linewidth]{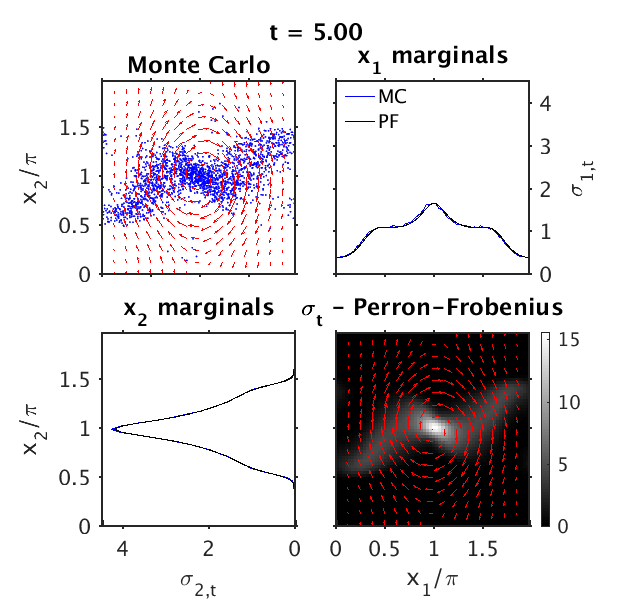} \\
  \includegraphics[width=.42\linewidth]{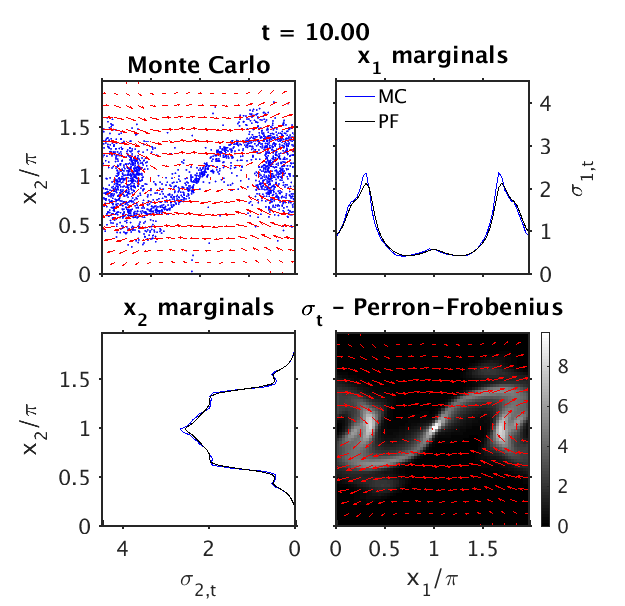}    \includegraphics[width=.42\linewidth]{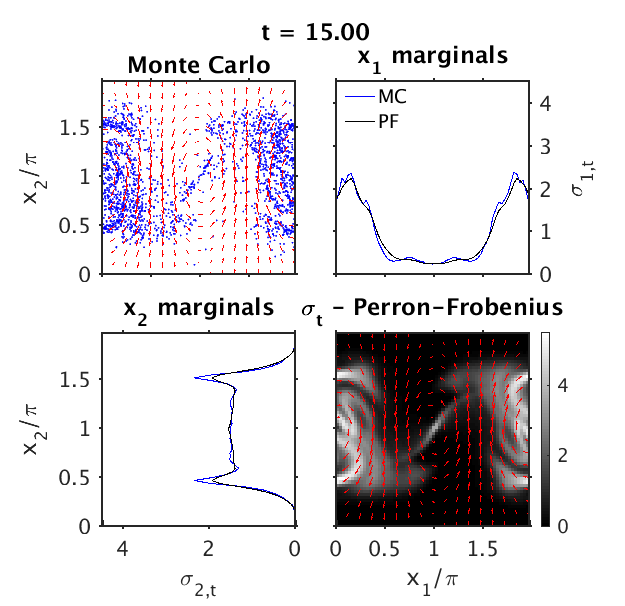} \\
  \includegraphics[width=.42\linewidth]{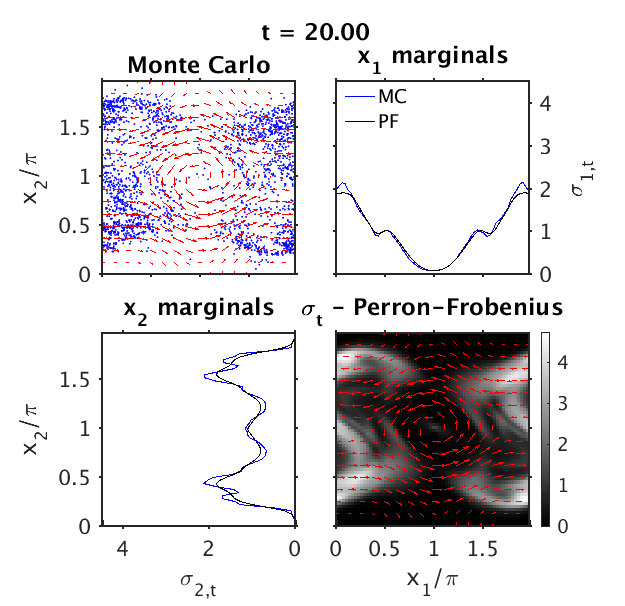}    \includegraphics[width=.42\linewidth]{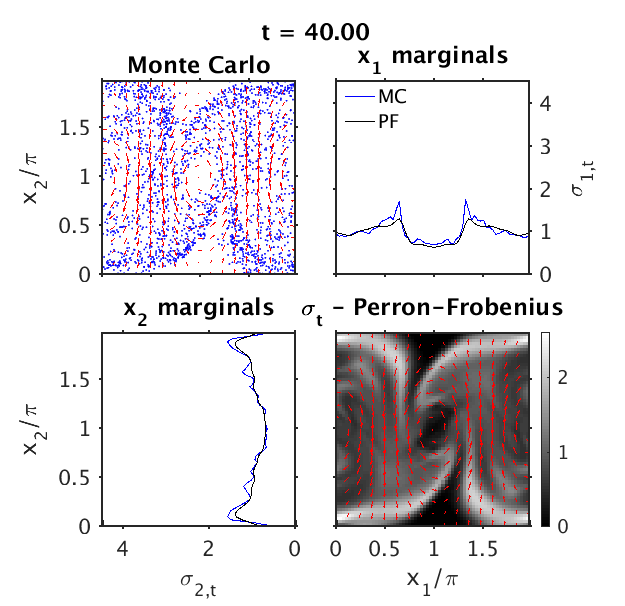} \\
  \caption{\label{figSwitchingRho}As in Fig.~\ref{figMovingRho}, but for the switching-vortex flow with the streamfunction in~\eqref{eqZetaSwitching} and $ \omega = 1 $, $ \kappa = 0.5 $, and $ C = 4$,  and the initial density from~\eqref{eqRhoAX} with $ \tilde \kappa = 3 $ and $ ( \bar x_1, \bar x_2 ) = (\pi, 0) $.} 
\end{figure}    

\section{\label{secDataDrivenBasis}Approximation in a data-driven basis}

Thus far, the development of our methods for coherent pattern extraction and nonparametric prediction, as well as the associated numerical experiments, were made under the strong assumption that orthonormal bases for the Hilbert spaces $ L^2( A, \alpha ) $ and $ L^2(X,\xi) $ associated with the state space $ A $ and the physical domain $ X $, respectively, are available. In this Section, we partially relax this assumption using kernel algorithms \cite{CoifmanLafon06,BerryHarlim16} to formulate the techniques of Section~\ref{secKoopman} in a basis of functions on $ A $ learned from the observed velocity field snapshots $ \{ v_0, \ldots, v_{N-1} \} $. In particular, our objectives are that (1) this basis approximates (in a sense that will be made precise below) the Laplace-Beltrami eigenfunction basis $ \{ \phi^A_0, \phi^A_1, \ldots \} $ of $ L^2(A,\alpha ) $ associated with the Riemannian metric $ h $ from Section~\ref{secLB}, and (2) the Dirichlet energies $E(\phi^A_k) = \eta_k^A $ of the basis functions can also be approximated.  

Our approach follows closely  \cite{GiannakisEtAl15,Giannakis15}, where an analogous data-driven basis was used for mode decomposition and prediction in ergodic dynamical systems. Here, we extend this framework to the setting of skew-product systems under the additional assumption that observations of the velocity field $ v\rvert_a $ driving the flow in $ X $ are available. Moreover, we introduce a more general framework for computing Dirichlet energies for the data-driven basis functions. It should be noted that in this work we do not address the question of constructing a data-driven basis of $ L^2(X,\xi) $, and we will continue to assume that the basis $ \{ \phi^X_{ij} \} $ is available. Nevertheless, building data-driven bases of $ L^2( X, \xi )$ from an arbitrary collection of points $ \{ x_0, x_1, \ldots, x_{\tilde N-1} \} $ in $ X $ should be possible through analogous kernel techniques to those used to build the basis for $ L^2(A, \alpha ) $, or via the variety of meshing techniques developed in finite element methods. It is also worthwhile noting that despite being formulated for finitely sampled data, the schemes for coherent pattern extraction and prediction introduced below are structurally very similar to those already presented in Section~\ref{secKoopman}. Thus, it should be possible for readers mainly interested in numerical results to skip ahead to Section~\ref{secL96} where we discuss experiments with aperiodic tracer flows driven by L96 systems.     

\subsection{\label{secKernel}Kernel function}

The starting point for the construction of our data-driven basis is a specification of a kernel function $ K: A \times A \mapsto \mathbb{ R } $ providing a measure of similarity of an arbitrary pair of states in $ A $ which is computable using observed velocity field data in $ \mathfrak{ X } $. Mathematically, we require that
\begin{enumerate}
  \item $ K $ is symmetric, i.e., $ K( a, b ) = K( b, a ) $ for any $ a, b \in A $;
  \item $ K $ is smooth (hence, bounded above by compactness) on $ A \times A $;
  \item $ K $ is bounded below, i.e., there exists a constant $ c > 0 $ such that $ K  \geq c $. 
\end{enumerate}     
In what follows, we work with the variable-bandwidth kernel introduced in~\cite{BerryHarlim16}, viz.
\begin{equation}
  \label{eqKVB}K_{\epsilon,N}( a, b ) = \exp\left( - \frac{ \lVert F( a ) - F( b ) \rVert^2 }{ \epsilon r_{\epsilon,N}( a ) r_{\epsilon,N}( b ) }\right),
\end{equation}
where $ \epsilon $ is a positive bandwidth parameter, and $ r_{\epsilon,N} $ a positive function in $ C^\infty( A ) $, bounded away from zero. In particular, the bandwidth function $ r_{\epsilon,N} $ is an estimate of the function $ r = \sigma_A^{-1/m_A} $, where $ \sigma_A $ is the density of the invariant measure $ \alpha $ relative to the volume of the ambient-space metric $ g_A $ introduced in Section~\ref{secLB}, and $ m_A $ the dimension of $ A $. This bandwidth function is constructed such that $ r_{\epsilon,N}( a ) $ converges to $ r( a ) $ almost surely with respect to the starting state $ a_0 \in A $ as $ N \to \infty $ and uniformly with respect to $ a \in A $ as $ \epsilon \to 0 $. Details on this construction can be found in \ref{appDensity}. With the definition in~\eqref{eqKVB}, the conditions on kernels stated above are satisfied by construction and the fact that $ A $ is compact. Moreover, by the pointwise ergodic theorem, as $N \to \infty $, $ K_{\epsilon,N} $ converges almost surely to the value $ K_\epsilon(a, b ) $ of the kernel
\begin{equation}
  \label{eqKVB2}K_{\epsilon}( a, b ) = \exp\left( - \frac{ \lVert F( a ) - F( b ) \rVert^2 }{ \epsilon r_{\epsilon}( a ) r_{\epsilon}( b ) }\right),
\end{equation}
which also meets conditions 1--3 stated above.    

\begin{rk}
  The role of the bandwidth functions $ r_{\epsilon,N} $ and $ r_\epsilon $ in~\eqref{eqKVB} and~\eqref{eqKVB2}, respectively, is to implement the conformal transformation leading to the Riemannian metric $ h $ (which has uniform volume form relative to the invariant measure; see Section~\ref{secLB}) from the ambient space metric $ g $. 
\end{rk}

\begin{rk} The analysis that follows relies on the assumption made in Section~\ref{secLB} that $ F $ is a smooth embedding of $ A $ into the space of velocity fields $ \mathfrak{ X } $. As stated in remark~\ref{rkAssumptions}, if that condition is not met then it is possible to replace $ F $ in~\eqref{eqKVB} (as well the definition of $ r_{\epsilon,N} $) by  the delay-coordinate map $ F_q : A \mapsto \mathfrak{ X }^q $ given by  
  \begin{displaymath}
    F_q( a ) = ( F( a ), F( \hat \Phi_{-1}( a ) ), \ldots, F( \hat \Phi_{-q}(a) ) ),
  \end{displaymath}
  where $ q $ is a positive integer parameter. According to the theory of delay-coordinate maps \cite{SauerEtAl91}, $F_q $ will generically be an embedding for sufficiently large $ q $ and under mild assumptions on $ \hat \Phi_n $ and $ F $. Moreover, as discussed in~\cite{Giannakis17}, performing delays can be beneficial even if $F  $ is already an embedding as it improves the noise robustness of the data-driven basis functions. Our scheme is also applicable when $ F $ takes values in other spaces than $ \mathfrak{ X } $ (e.g., if the fluid dynamical driving system has additional degrees of freedom such as temperature), so long as $ F $ and/or or $F_q $ are embeddings.     
\end{rk}

\subsection{\label{secKernelOp}Kernel integral operators}

Having specified an appropriate kernel, we now apply the normalization procedure introduced in the diffusion maps algorithm \cite{CoifmanLafon06} and further developed in~\cite{BerrySauer16b} to construct a kernel integral operator which can be used to approximate the eigenfunctions of the Laplace-Beltrami operator $ \upDelta_A $ associated with the Riemannian metric $ h $ and the associated Dirichlet energies. In what follows, we describe the construction of this integral operator and discuss  spectral convergence results justifying its utility in providing the data-driven basis used in our mode decomposition and prediction schemes. The discussion  below is mainly based on results established in a number of works in the literature, including  \cite{CoifmanLafon06,VonLuxburgEtAl08,BerryHarlim16,BerrySauer16,BerrySauer16b}. We believe that presenting this material here in a fairly self-contained manner is useful given that each of these works addresses a  fairly distinct aspect of the problem of interest here, but readers familiar with these topics may wish to skip to Section~\ref{secDirichlet}. 

We begin by establishing the relevant function spaces for our data-driven scheme. In particular, note that in practical applications with finitely many samples we cannot work with the Hilbert space $H_A = L^2(A,\alpha) $ directly. Instead, we consider the Hilbert space $ H_{A,N} = L^2(A,\alpha_N) $ of complex-valued functions associated with the sampling measure $ \alpha_N = \sum_{n=0}^{N-1} \delta_{a_n} / N $. That is, $ H_{A,N} $ consists of equivalence classes of functions $ f : A \mapsto \mathbb{ C } $ which have common values and are square-summable on the sampled states $ a_n $. Its associated inner product and norm are 
\begin{displaymath}
 \langle f_1, f_2 \rangle_{H_{A,N}} = \int_A f_1^* f_2 \, d\alpha_N = \frac{ 1 }{ N } \sum_{n=0}^{N-1} f_1^*( a_n ) f_2( a_n ), \quad \lVert f \rVert_{H_{A,N}} = \sqrt{\langle f, f \rangle_{H_{A,N}}}.
\end{displaymath}
Of course, $ H_{A,N}$ is isomorphic to $ \mathbb{ C }^N $, but we prefer to work with the former space to emphasize that its elements are equivalence classes of functions defined on the same underlying space $ A $ as in the case of $H_A $. Nevertheless, given $ f : A \mapsto \mathbb{ C } $, it is convenient to represent the corresponding equivalence class in $  H_{A,N} $ by the column vector $ \vec f = ( f_0, \ldots, f_{N-1} )^\top \in \mathbb{ C }^N $ storing the values $ f_n = f( a_n )  $ at the sampled states $ a_n $; we will use this notation whenever we wish to distinguish between functions on $ A $ and equivalence classes of functions in $ H_{A,N} $. Similarly, we will use the notation $ \overline{f} $ to represent the $ H_A $ equivalence class associated with $ f $. Clearly, we have $ \langle \vec f_1, \vec f_2 \rangle_{H_{A,N}} = \vec f_1 \cdot \vec f_2 / N $, where $ \cdot $ is the standard dot product for vectors in $ \mathbb{ C }^N $. Moreover, by the pointwise ergodic theorem, given $ f_1, f_2 : A \mapsto \mathbb{ C }$ such that $ \overline{f}_1 $ and $ \overline{f}_2 $ are both in $ H_A $, then for $ \alpha $-a.e.\ starting state $ a_0 \in A $, $ \vec f_1 $ and $ \vec f_2 $ are also in $ H_{A,N} $ and as $ N \to \infty $, $ \langle \vec f_1, \vec f_2 \rangle_{H_{A,N}} $ converges to $ \langle \overline f_1, \overline f_2 \rangle_{H_A} $.

Besides these Hilbert spaces, we will make use of the Banach space of continuous functions on $ A $ equipped with the uniform norm (denoted $ \lVert \cdot \rVert_\infty $). This space will provide a common space to define kernel integral operators associated with the finite dataset $ \{ a_n \}_{n=0}^{N-1} \subset A  $ and the whole state space $ A $, allowing us to take advantage of existing spectral convergence results in the machine learning and statistics literature \cite{VonLuxburgEtAl08}. It is then straightforward to lift those results to the Hilbert space setting of $ H_{A,N} $ and $ H_{A} $.     

With these definitions in place, consider the bounded integral operators $ G'_{\epsilon,N} : H_{A,N} \mapsto C(A) $ and $ G'_{\epsilon} : H_A \mapsto C(A) $ defined as 
\begin{gather*}
  G'_{\epsilon,N} \vec f = \frac{ 1 }{ \epsilon^{m_A/2} } \int_A K_{\epsilon,N}( \cdot, a ) \vec f( a ) \, d\alpha_N( a ) = \frac{ 1 }{ N \epsilon^{m_A/2}  } \sum_{n=0}^{N-1} K_{\epsilon,N}( \cdot, a_n ) f_n, \\ 
  G'_{\epsilon} \overline{ f } = \frac{ 1 }{ \epsilon^{m_A/2} } \int_A K_{\epsilon}( \cdot, a ) \overline{ f }( a ) \, d\alpha( a ).
\end{gather*}
Clearly, we can also define analogs $ G''_{\epsilon,N} : C( A ) \mapsto C( A ) $ and $ G_{\epsilon, N } : H_{A,N} \mapsto H_{A,N} $ of $ G'_{\epsilon,N} $ by $ G''_{\epsilon,N} = G'_{\epsilon,N} \circ \pi_N $ and $ G_{\epsilon,N} = \pi_N \circ G'_{\epsilon,N} $, where $ \pi_N : C( A ) \mapsto H_{A,N} $ is the canonical restriction map, $\pi_N f  = \vec f $. Although related by trivial restrictions, the  distinction between these operators will be useful below. Similarly, we introduce $ G''_\epsilon  = \iota \circ G'_\epsilon : C( A ) \mapsto C( A ) $ and $ G_\epsilon = G'_\epsilon \circ \iota : H_A \mapsto H_A $, where $ \iota : C( A ) \mapsto H $ is the canonical inclusion map, $ \iota f  = \overline{ f } $. By the assumptions on $ K_{\epsilon,N} $ and $K_\epsilon $ stated in Section~\ref{secKernel}, the range of both $ G''_{\epsilon,N} $ and $ G''_\epsilon $ is included in $ C^\infty( A) $. Note that $ G''_{\epsilon,N} f( a)  $ can be evaluated with no approximation at arbitrary $ a \in A $ given only the values $ f( a_n ) $ at the sampled states $ a_n $. Moreover, by the pointwise ergodic theorem, $ G''_{\epsilon,N} f $ converges to $ G''_\epsilon f $ in uniform norm for $ \alpha $-a.e.\ starting state $ a_0 $. In the case of $ G_{\epsilon,N} $ and $ G_\epsilon $, the symmetry of $ K_{\epsilon,N} $ and $ K_\epsilon $ implies that these operators are self-adjoint. In addition, we have: 

\begin{lemma}
  \label{lemmaG}
  The operators $ G''_{\epsilon,N} $ and $ G''_\epsilon $ are compact. 
\end{lemma}
A proof of Lemma~\ref{lemmaG} can be found in~\ref{appG}.

Next, using $ G'_{\epsilon,N} $ and $ G'_{\epsilon} $, respectively, we construct normalized operators $ P'_{\epsilon,N} : H_{A,N} \mapsto C( A  ) $ and $ P'_\epsilon : H_A \mapsto C(A)  $  by applying the normalization procedure in \cite{CoifmanLafon06,BerrySauer16}. Specifically, we set
\begin{equation}
  \label{eqPOp}
    P'_{\epsilon,N} f =\frac{ \tilde G'_{\epsilon, N } f }{ d_{\epsilon,N } }, \quad  d_{\epsilon,N} = \tilde G'_{\epsilon,N} 1_A, \quad \text{and} \quad
    P'_{\epsilon} f =\frac{ \tilde G'_{\epsilon } f }{ d_\epsilon }, \quad d_\epsilon = \tilde G'_\epsilon 1_A,  
\end{equation}    
where $ \tilde G'_{\epsilon,N} : H_{A,N} \mapsto C(A) $ and $ \tilde G'_\epsilon : C(A) \mapsto H_A $ are defined as
\begin{equation}
  \label{eqGTildeOp}
    \tilde G'_{\epsilon, N}  f = G'_{\epsilon,N}\left( \frac{ f }{ q_{\epsilon,N}} \right), \quad  q_{\epsilon,N} = G'_{\epsilon,N} 1_A, \quad \text{and} \quad
    \tilde G'_{\epsilon}  f = G'_{\epsilon}\left( \frac{ f }{ q_{\epsilon}} \right), \quad   q_{\epsilon} = G_{\epsilon} 1_A,
\end{equation}
and $ 1_A $ is the function equal to 1 at every point in $ A $. Both $ P'_{\epsilon,N} $ and $ P'_\epsilon $ can be alternatively be expressed as kernel integral operators, viz.
\begin{equation}
  \label{eqPOpKer}
\begin{gathered}
  P'_{\epsilon,N} \vec f = \int_A p_{\epsilon,N}( \cdot, a ) \vec f( a ) \, d\alpha_N( a ) = \frac{ 1 }{ N } \sum_{n=0}^{N-1} p_{\epsilon,N}( \cdot, a_n )  f_n, \quad p_{\epsilon,N}( b, a ) = \frac{ K_{\epsilon,N}(b,a ) }{ d_{\epsilon,N}( b ) q_{\epsilon,N}( a ) }, \\
  P'_{\epsilon} \overline f = \int_A p_{\epsilon}( \cdot, a ) \overline f( a ) \, d\alpha( a ), \quad p_\epsilon(b,a) = \frac{ K_\epsilon(b,a) }{ d_\epsilon(b) q_\epsilon(a) },
\end{gathered}
\end{equation}
where the kernels $ p_{\epsilon,N} $ and $ p_\epsilon $ satisfy 
\begin{equation}
  \label{eqPNorm}
  \int_A p_{\epsilon,N}( a, \cdot ) \, d\alpha_N = 1, \quad \int_A p_{\epsilon}( a, \cdot ) \, d\alpha = 1 
\end{equation}
for all $ a \in A $. By our conditions on kernels and compactness of $ A $, the functions $ q_{\epsilon,N} $, $ q_{\epsilon} $, $ d_{\epsilon,N} $, and $ d_\epsilon $ introduced above are all $ C^\infty $ positive functions bounded away from zero, and $ q_{\epsilon,N}( a ) \stackrel{\text{a.s.}}{\longrightarrow} q_\epsilon( a )$, $ d_{\epsilon,N}( a ) \stackrel{\text{a.s.}}{\longrightarrow} d_\epsilon( a ) $ for every $ a \in A $ by the pointwise ergodic theorem. As a result, $ p_{\epsilon,N} $ and $ p_\epsilon $ are smooth with respect to both of their arguments, positive, and bounded away from zero. In~\cite{BerrySauer16}, the normalization steps in~\eqref{eqPOp} and~\eqref{eqGTildeOp} are referred to as left and right normalization, respectively. 

\begin{rk} \label{rkQD}Intuitively, the quantity $ q_\epsilon $ can be thought of as an estimate (up to a proportionality constant) of the density of the invariant measure $ \alpha $ relative to the volume form of a Riemannian metric on $ A $ that depends on $ K_\epsilon $. For the choice of kernel in~\eqref{eqKVB}, that metric is in fact the desired metric $ h $, whose volume form is equal to the invariant measure, and one can show that $ \lim_{\epsilon\to0} q_\epsilon( a ) = q $ for every $ a \in A $, where $ q $ is a positive constant. As a result, we have $ \lim_{\epsilon\to0} d_\epsilon = 1 $. Note that despite that $ q_\epsilon $ and $ d_\epsilon $ are constants in the limit, their  behavior at $ O( \epsilon ) $ is also important (e.g., for Lemma~\ref{lemmaP}(ii) ahead), and at that order they are non-constant on $ A $. 
\end{rk}

Next, through an appropriate use of restriction/inclusion operators, we  define operators $ P_{\epsilon,N} : H_{A,N} \mapsto H_{A,N} $ and $ P''_{\epsilon,N} : C(A) \mapsto C(A)  $ having the same kernel $ p_{\epsilon,N} $ as $ P'_{\epsilon,N} $, and similarly we  define $ P_{\epsilon,N} : H_{A} \mapsto C( A ) $ and $ P''_{\epsilon,N} : H_{A} \mapsto H_{A} $ having the same kernel $ p_\epsilon $ as $ P_{\epsilon} $. Note that by smoothness of $ p_{\epsilon,N} $ and $ p_\epsilon $, both $ P_{\epsilon,N} $ and $ P_\epsilon $ are Hilbert-Schmidt (hence compact) kernel integral operators. In particular, \eqref{eqPNorm} and the fact that $ p_{\epsilon,N} $ and $ p_\epsilon $ are bounded away from zero imply imply that $ P_{\epsilon, N} $ and $ P_\epsilon $ generate ergodic Markov semigroups, respectively $ \{ I, P_{\epsilon,N}, P^2_{\epsilon,N}, \ldots \} $ and $ \{ I, P_\epsilon, P_\epsilon^2, \ldots \} $, and thus each has a simple eigenvalue at 1 corresponding to a constant eigenfunction. The following Lemma lists four key properties of the averaging operators introduced thus far. 

\begin{lemma}
  \label{lemmaP}
  (i) For $ \alpha $-a.e.\ starting point $ a_0 $ in the training data, $ P''_{\epsilon,N} $ converges as $ N \to \infty $ to $ P''_\epsilon $ in $ \lVert \cdot \rVert_{\infty} $ norm.
  
  (ii) For every $ f \in C^\infty( A ) $ and $ \epsilon > 0 $, the mapping $ \epsilon \mapsto P''_\epsilon f $ is $ C^\infty $; in particular, for every $ a \in A $ we have
  \begin{equation}
    \label{eqPEpsilon}
    P''_\epsilon f( a ) = f( a ) - \frac{ \epsilon }{ 8 } \upDelta_A f( a ) + O( \epsilon^2 ),
  \end{equation}
  where the estimate holds uniformly on $ A $.

  (iii) For every $ s \geq 0 $, $ P^{s/\epsilon}_\epsilon $ converges as $ \epsilon \to 0 $ to the heat operator $ \mathcal{ P }_s = e^{-s\upDelta_A} : H_A \mapsto H_A $, $ s \geq 0 $, in $ \lVert \cdot \rVert_{H_A} $ norm. 
  
  (iv) $ P_{\epsilon,N} $ and $ P_{\epsilon} $ are similar to the self-adjoint Hilbert-Schmidt operators $ \hat P_{\epsilon,N} : H_{A,N} \mapsto H_{A,N} $ and $ \hat P_\epsilon : H_A \mapsto H_A $, respectively, where
  \begin{equation}
    \label{eqGBar}
    \begin{gathered}
      \hat P_{\epsilon,N} \vec f =  \int_A \hat p_{\epsilon,N}( \cdot, a ) \vec f( a ) \, d\alpha_N( a ), \quad \hat p_{\epsilon,N}( b, a ) = \frac{ K_{\epsilon,N}(b,a ) }{ [ \epsilon^{m_A} \tilde d_{\epsilon,N}( b ) \tilde d_{\epsilon,N}( a ) ]^{1/2} }, \quad \tilde d_{\epsilon,N} = q_{\epsilon,N} d_{\epsilon,N}, \\
      \hat P_{\epsilon} \bar f = \int_A \hat p_{\epsilon}( \cdot, a ) \bar f( a ) \, d\alpha( a ), \quad \hat p_\epsilon(b,a) = \frac{ K_\epsilon(b,a) }{ [ \epsilon^{m_A} \tilde d_\epsilon(b) \tilde d_\epsilon(a) ]^{1/2} }, \quad  \quad \tilde d_{\epsilon} = q_{\epsilon} d_\epsilon.
    \end{gathered}
\end{equation} 
Moreover, we have $ \lVert \hat P^{s/\epsilon}_{\epsilon} - \mathcal{ P }_s \rVert_{H_A} \to 0 $ as $ \epsilon \to 0 $. 
\end{lemma}
A proof of this Lemma can be found in~\ref{appP}.

Let now $ \{ \psi^{(\epsilon,N)}_0, \ldots, \psi^{(\epsilon,N)}_{N-1} \} $ and $ \{ \psi_0^{(\epsilon)}, \psi_1^{(\epsilon)}, \ldots \} $ be orthonormal bases of $ H_{A,N} $ and $ H_A $, respectively, consisting of eigenfunctions of $ \hat P_{\epsilon,N} $ and $ \hat P_\epsilon $. Let also  $ \Lambda_0^{(\epsilon,N)}, \ldots, \Lambda_{N-1}^{(\epsilon,N)} $ and $ \Lambda_0^{(\epsilon)}, \Lambda_1^{(\epsilon)}, \ldots $, respectively, be the corresponding (real) eigenvalues. Since $ \hat P_{\epsilon,N} $ is similar to $ P_{\epsilon,N} $ (resp.\ $ \hat P_\epsilon $ is similar to $ P_\epsilon $) and $ P_{\epsilon,N} $ (resp.\ $ P_{\epsilon} $) is an ergodic  Markov operator, it follows that the eigenvalues admit the ordering $ 1 = \Lambda^{(\epsilon,N)}_0 > \Lambda^{(\epsilon,N)}_1 \geq \cdots \geq \Lambda^{(\epsilon,N)}_{N-1}  $ (resp.\ $ 1 = \Lambda_0^{(\epsilon)} > \Lambda^{(\epsilon)}_1 \geq  \Lambda^{(\epsilon)}_{2} \geq \cdots $). Moreover, one can verify that eigenfunctions of $ P_{\epsilon,N} $ and  $ P_\epsilon $ corresponding to these eigenvalues are given by 
\begin{equation}
  \label{eqPsiO}
  \phi_k^{(\epsilon,N)} =  \beta_{\epsilon,N}^{-1/2} \psi_k^{(\epsilon,N)}, \quad \phi^{(\epsilon)}_k =  \beta_\epsilon^{-1/2} \psi_k^{(\epsilon)}, 
\end{equation}
respectively, where $ \beta_{\epsilon,N} $ and $ \beta_\epsilon $ are the densities of the stationary measures of the Markov semigroups generated by $ P_{\epsilon,N} $ (resp.\ $ P_\epsilon $) relative to $ \alpha_N $ (resp.\ $ \alpha$). Explicitly, we have $ \beta_{\epsilon,N} = \hat d_{\epsilon,N} / \int_A \hat d_{\epsilon,N} \, d\alpha_N $, $ P^*_{\epsilon,N} \beta_{\epsilon,N} = \beta_{\epsilon,N} $, and $ \int_A \beta_{\epsilon,N} \, d\alpha_N = 1 $, where $ \hat d_{\epsilon,N} =  d_{\epsilon,N}/ q_{\epsilon,N} $ is a positive smooth function, bounded away from zero. Analogous formulas hold  for $ \beta_\epsilon $ with $  \hat d_\epsilon = d_\epsilon / q_\epsilon $.  Note  that while $ \{ \phi_0^{(\epsilon,N)}, \ldots, \phi_{N-1}^{(\epsilon,N)} \} $ and $\{  \phi_0^{(\epsilon)}, \phi_1^{(\epsilon)}, \ldots \} $ are bases of $ H_{A,N} $ and $ H_A $, respectively, they are generally not orthonormal (although, as we will see in Corollary~\ref{corEig} ahead, the lack of orthonormality vanishes in a suitable limit $ N \to \infty $ and $ \epsilon \to 0 $). For later convenience, we introduce the kernel integral operators $ \hat P'_{\epsilon,N} : H_{A,N} \mapsto C( A ) $ and $ \hat P'_\epsilon : H_A \mapsto C( A) $ which are defined using the same kernels as in the definitions of $ \hat P_{\epsilon,N} $ and $ \hat P_\epsilon $, respectively. 

By smoothness of $ p_\epsilon $ and compactness of $ A $, $ P_{\epsilon,N} $ and $ P_\epsilon $ are Hilbert-Schmidt integral operators, and therefore compact. Moreover, by Lemma~\ref{lemmaG}, $ P''_{\epsilon,N} $, and $ P''_\epsilon $ are also compact. (In fact, $ P_{\epsilon,N} $ and $ P''_{\epsilon,N} $ have even finite rank.) As a result, the eigenvalues of $ P_{\epsilon,N} $, $ P_\epsilon $, $ P''_{\epsilon,N} $, and $ P''_\epsilon $ can only accumulate at zero, and their nonzero eigenvalues have finite multiplicities. For such eigenvalues, convergence in operator norm implies convergence of eigenvalues and eigenfunctions corresponding to simple eigenvalues. For non-simple eigenvalues, convergence in operator norm implies convergence of projection operators to the corresponding eigenspaces \cite{VonLuxburgEtAl08}. In particular, we have:

\begin{lemma}
  \label{lemmaEig}
  
  (i) Let $ \phi \in C( A) $ be an eigenfunction of $ P''_{\epsilon,N} $ (resp.\ $ P''_\epsilon$) at eigenvalue $ \Lambda $. Then, $ \vec \phi = \pi_N \phi $ (resp.\ $ \overline\phi= \iota\phi$) is an eigenfunction of $ P_{\epsilon,N} $ (resp.\ $ P_\epsilon$), also at eigenvalue $ \Lambda $. 
  
  (ii) Let $ \vec \phi \in H_{A,N} $ be an eigenfunction of $ P_{\epsilon,N} $ at nonzero eigenvalue $ \Lambda $. Then $ \phi = \Lambda^{-1} P'_{\epsilon,N} \vec\phi$ is an eigenfunction of $ P''_{\epsilon,N} $, also at eigenvalue $ \Lambda $. The analogous result holds for eigenfunctions of $ P_\epsilon $ and $P''_\epsilon $. 
  
  (iii) As $ N \to \infty $, the eigenvalue $ \Lambda_k^{(\epsilon,N)} $ of $ P_{\epsilon,N} $ converges $ \alpha $-a.s. to the eigenvalue $ \Lambda_k^{(\epsilon)} $ of $ P_\epsilon $, assuming that the latter is nonzero. Moreover, for each eigenfunction $ \phi_k^{(\epsilon)} $ of $ P_\epsilon $ corresponding to $ \Lambda_k^{(\epsilon)} \neq 0 $, there exist eigenfunctions $ \phi_k^{(\epsilon,N)} $ of $ P_{\epsilon,N} $   such that $ P'_{\epsilon,N}\phi_k^{(\epsilon,N)} / \Lambda_k^{(\epsilon,N)}\stackrel{\text{a.s.}}{\longrightarrow}  P'_{\epsilon}\phi_k^{(\epsilon)} / \Lambda_k^{(\epsilon)} $, where the limit is taken in uniform norm. 
  
  (iv) The eigenvalues $ \Lambda_k^{(\epsilon)}$ of $ P_\epsilon $ satisfy  
  \begin{equation}
    \label{eqLimLambda}
    \lim_{\epsilon \to 0} \frac{  - \log \Lambda_k^{(\epsilon)} }{ \epsilon } = \eta^A_k,
  \end{equation}
  where $ \eta^A_k $ is the $ k $-th eigenvalue of $ \upDelta_A $. Moreover, for each eigenfunction $ \phi_k^A $ of $ \Delta_A $ at eigenvalue $ \eta^A_k $ there exist eigenfunctions $ \phi_k^{(\epsilon)} $ of $ P_\epsilon $ such that $ \phi_k^{(\epsilon)} \to \phi_k $ in $ L^2 $ norm.

  (v) The analogous results to (i)--(iv) hold for the self-adjoint operators $\hat P_{\epsilon,N} $ and $ \hat P_\epsilon $.
\end{lemma}

A proof of Lemma~\ref{lemmaEig} is included in~\ref{appEig}. Together, the results in Lemma~\ref{lemmaEig} imply that we can approximate the eigenvalues and eigenfunctions of $ \upDelta_A $ using the eigenvalues and eigenfunctions of either $ P_{\epsilon,N} $ or $ \hat P_{\epsilon,N} $ (both of which are accessible from data),  in the following sense.

\begin{cor}
  \label{corEig}
  For $ \alpha $-a.e.\ starting state $ a_0 $ in the training data, 
  \begin{displaymath}
    \lim_{\epsilon\to 0} \lim_{N\to\infty} \frac{ - \log \Lambda_k^{(\epsilon,N)} }{ \epsilon } = \eta^A_k. 
  \end{displaymath}
  Moreover, there exist eigenfunctions $ \phi_k^{(\epsilon,N)} \in H_{A,N} $ of $ P_{\epsilon,N} $ (resp.\ eigenfunctions $ \psi_k^{(\epsilon,N)} $ of $ \hat P_{\epsilon,N} $) that approximate the eigenfunctions $ \phi^A_k \in H_A $ of $ \upDelta_A $, in the sense that 
  \begin{displaymath}
    \lim_{\epsilon \to 0 } \lim_{N\to\infty} \frac{ 1 }{ \Lambda_k^{(\epsilon,N)} } P'_{\epsilon,N }\phi_k^{(\epsilon,N)} \to \tilde \phi^A_k, \quad \lim_{\epsilon \to 0 } \lim_{N\to\infty} \frac{ 1 }{ \Lambda_k^{(\epsilon,N)} } \hat P'_{\epsilon,N} \psi_k^{(\epsilon,N)} \to \tilde \phi^A_k,
  \end{displaymath}
  in uniform norm, where $ \tilde \phi^A_k \in C^\infty( A ) $ is the unique smooth representative of $ \phi^A_k \in H_A $. In particular, for any fixed state $ a_n \in \{ a_0, \ldots, a_{N-1} \} $, the values $ \phi_k^{(\epsilon,N)}( a_n ) $ and $ \psi_k^{(\epsilon,N)}(a_n) $ converge $ \alpha $-a.s.\ to $ \tilde\phi_k^A(a_n) $. 
\end{cor}

Details on the numerical solution of the eigenvalue problems for $ P_{\epsilon,N} $ and $ \hat P_{\epsilon,N} $ are provided in~\ref{appNumEig}.

\subsection{\label{secDirichlet}Data-driven bases and the associated Dirichlet energies}    

Having established the spectral convergence results of Section~\ref{secKernelOp} (in particular, Corollary~\ref{corEig}), we are ready to specify the data-driven bases and their associated Dirichlet energies to be used in the mode decomposition and forecasting techniques of Sections~\ref{secCoherent} and~\ref{secPrediction}, respectively. In essence, given a fixed dataset of $ N $ samples, this procedure entails three steps, namely (1) selection of an appropriate value $ \epsilon( N ) $ of the kernel bandwidth parameter $ \epsilon $; (2) specification of basis functions for the finite-dimensional space $ H_{A,N} $; (3) an assignment of Dirichlet energies for the basis functions and construction of an associated Sobolev space $ H^1_{A,\epsilon,N}$ (needed for the variational formulation of the Koopman eigenvalue problem in Section~\ref{secReg}). We now describe these steps. 

\subsubsection{\label{secTuning}Kernel bandwidth selection}
   
We select $ \epsilon $ using the method developed in \cite{CoifmanEtAl08,BerryHarlim16}, which was also employed in \cite{BerryEtAl15,GiannakisEtAl15,Giannakis17,DasGiannakis17}. The method tunes $ \epsilon $ by examining the behavior of the function $ S( \epsilon ) = \int_{A\times A} K_\epsilon \, d\alpha \otimes d\alpha $, which is approximated in the case of finite data by $ S_N( \epsilon ) = \sum_{i,j=0}^{N-1} K_\epsilon( a_i, a_j ) / N^2 $. As $ \epsilon \searrow 0 $ and $ \epsilon \to \infty $, that function tends to the constant values 0 and 1, respectively. Moreover, as shown in \cite{CoifmanEtAl08,BerryHarlim16}, there exists a regime of intermediate  $ \epsilon $ values where $ S( \epsilon ) $ exhibits the scaling $ S( \epsilon ) \simeq C ( 2 \pi \epsilon )^{m_A/2} $, where $ C $ is a constant independent of the intrinsic dimension $m_A $ of $ A $. In particular, in that regime, $ T(\epsilon) := \frac{ d\log S( \epsilon ) }{ d \log \epsilon } $ is approximately equal to $ m_A / 2 $. In \cite{BerryHarlim16}, the quantity $ T( \epsilon ) $ is interpreted as a ``resolution'' of the kernel, and they suggest choosing $ \epsilon $ as the maximizer of that function, yielding also an estimate of $ m_A $. In what follows, we select $ \epsilon $ using that criterion, approximating $ T( \epsilon ) $ through finite differences of $ S_N( \epsilon ) $ on a logarithmic $ \epsilon $ grid. Further details on this procedure can be found in \cite{BerryEtAl15,Giannakis17}.

\subsubsection{\label{secBasis}Basis functions}

In light of the results of Section~\ref{secKernelOp}, obvious candidates for data-driven basis functions of $ H_{A,N} $ are the eigenfunctions of $ P_{\epsilon,N} $ and $ \hat P_{\epsilon,N} $, i.e., $ \{ \phi_0^{(\epsilon,N)}, \ldots, \phi_{N-1}^{(\epsilon,N)} \} $ and $ \{ \psi_0^{(\epsilon,N)}, \ldots, \psi_{N-1}^{(\epsilon,N)} \} $, respectively. According to Corollary~\ref{corEig}, both of these choices lead to consistent approximations of the Laplace-Beltrami eigenfunctions $ \phi^A_k $ in the limit of infinitely many samples and vanishing bandwidth parameter, but away from that limit they each have their advantages and disadvantages. 

More specifically, the main advantage of the $ \{\psi_k^{(\epsilon,N)} \}$ basis is that it is orthonormal on $H_{A,N} $. Its main disadvantage is that it is not an eigenbasis of a Markov operator, and as a result, estimating the associated Dirichlet energies using the eigenvalues $ \Lambda_k^{(\epsilon,N)} $ from diffusion maps would lead to certain biases (e.g., $ \psi_0^{(\epsilon,N)} $ corresponding to eigenvalue $ \Lambda_0^{(\epsilon,N)} = 1 $ would be assigned zero Dirichlet energy while being non-constant; the latter follows from~\eqref{eqPsiO} and the fact that $ \phi_0^{(\epsilon,N)} $ is constant whereas $  \beta_{\epsilon,N}$ is non-constant). Working with the $ \{ \phi_k^{(\epsilon,N)} \} $ basis avoids this issue, but only at the expense of losing orthogonality of the basis functions. 

One way of alleviating the latter issue is to observe that the $ \phi_k^{(\epsilon,N)} $ are orthonormal on the Hilbert space $ H_{A,\epsilon,N} $ associated with the stationary distribution of the associated Markov semigroup $ P_{\epsilon,N} $ (see Section~\ref{secKernelOp}). That is, the   $ \phi_k^{(\epsilon,N)} $ are orthonormal with respect to the inner product  
\begin{equation}
  \label{eqBeta}
  \langle f_1, f_2 \rangle_{H_{A,\epsilon,N}} = \int_A f_1^* f_2  \, d\alpha_{\epsilon,N}, \quad  \alpha_{\epsilon,N} = \frac{ 1 }{ N } \sum_{i=1}^N \beta_{\epsilon,N}( a_n ) \delta_{a_n}, 
\end{equation} 
where $ \beta_{\epsilon,N} $ is the stationary density of $ P_{\epsilon,N} $ introduced in Section~\ref{secKernelOp}. This suggests that we can formulate our data-driven techniques in this space instead of $ H_{A,N} $. However, in doing so one must remember that taking the limit $ N \to \infty $ at fixed $ \epsilon > 0 $ will not lead to convergence to the inner product of the Hilbert space $ H_A $ associated with the invariant measure $\alpha$, but rather to the inner product $ \langle f_1, f_2 \rangle_{H_{A,\epsilon}} = \int_A  f_1^* f_2  \beta_\epsilon \, d\alpha $ of a different Hilbert space, $ H_{A,\epsilon} $, associated with the stationary density $ \beta_\epsilon $  of $ P_\epsilon $ (in particular, the Koopman group is non-unitary in this Hilbert space). However, the limit $ N \to \infty $ can be taken in conjunction with a decreasing function $ \epsilon( N ) $ (e.g., via the tuning procedure in Section~\ref{secTuning}), so that  $ \langle f_1, f_2 \rangle_{H_{A,\epsilon(N)}} $ converges to  $ \langle f_1, f_2 \rangle_{H_{A}} $. In particular, in this limit, the functions $ \beta_{\epsilon,N} $ and $ \beta_{\epsilon} $ both converge to $ 1_A $ (see also Remark~\ref{rkQD}).  

Due to the important role that Dirichlet energies play in our schemes, in the numerical experiments presented in Section~\ref{secL96} ahead, we always work with the $ \{ \phi_k^{(\epsilon,N)} \} $ basis and the weighted inner product space $ H_{A,\epsilon,N} $. However, in practice we find that the range of values of $ \beta_{\epsilon,N} $ is usually not too large (e.g., in the applications of Section~\ref{secL96} it does not vary by more than $ \sim 10\% $), and therefore comparable results can also be obtained using the orthonormal $  \{ \psi_k^{(\epsilon,N)} \} $ basis of $ H_{A,N}$.

 For later convenience, we note that since $ H_{A,\epsilon,N} $ is finite-dimensional, the pairwise products $ \phi_i^{(\epsilon,N)} \phi_j^{(\epsilon,N)} $ of the basis elements are also in $ H_{A,\epsilon,N} $. Thus, we can write $ \phi_i^{(\epsilon,N)} \phi_j^{(\epsilon,N)} = \sum_{k=0}^{N-1} c_{ijk} \phi_k^{(\epsilon,N)}$, where the coefficients
\begin{equation}
  \label{eqStructureConstants}
  c_{ijk} = \langle \phi_k^{(\epsilon,N)}, \phi^{(\epsilon,N)}_i \phi^{(\epsilon,N)}_j \rangle_{H_{A,\epsilon,N}} 
\end{equation}     
can be thought of as ``structure constants'' for the algebra of equivalence classes of functions $ H_{A,\epsilon,N} $. Since the $ \phi^{(\epsilon,N)}_k $ are all real, the $ c_{ijk} $ are invariant under arbitrary permutations of $ i, j, k $. Moreover, because $ \phi^{(\epsilon,N)}_0 $ is a constant, we have $ c_{0jk} = \delta_{jk} $. Analogous coefficients can also be defined for the $ \{ \phi^{(\epsilon)}_k \} $ basis of $ H_{A,\epsilon} $ by continuity of the $ \phi^{(\epsilon_k)} $, but note that  $H_{A,\epsilon} $ is not closed under multiplication.  

\subsubsection{\label{secDirichletE}Dirichlet energies}

We now discuss how to construct data-driven analogs of the Dirichlet form and the associated Laplace-Beltrami operator needed to implement the variatonal formulation of the eigenvalue problem for coherent patterns in Section~\ref{secGalerkin}. First, before stating definitions for these quantities suppose that we have at our disposal quantities $ \eta^{(\epsilon,N)}_0, \ldots, \eta^{(\epsilon,N)}_{N-1} $ and  $ \eta_0^{(\epsilon,N)}, \eta_1^{(\epsilon,N)}, \ldots $ approximating the Laplace-Beltrami eigenvalues $ \eta_k^A $. In particular, we assume that the $ \eta_k^{(\epsilon,N)} $ and $ \eta_k^{(\epsilon)} $ are ordered in a non-decreasing sequence, that $ \eta_0^{(\epsilon,N)} = \eta_0^{(\epsilon)} = 0 $, and that all other $ \eta_k^{(\epsilon,N)} $ and $ \eta_k^{(\epsilon)} $ are positive and finite. We also assume that $ \eta_k^{(\epsilon,N)} \stackrel{\text{a.s.}}{\longrightarrow} \eta_k^{(\epsilon)}$ for fixed $k $ and $ \epsilon $ by ergodicity. Then, for $ p\in \{ 0,1,\ldots \} $, we define the Sobolev spaces
\begin{gather*}
  H^p_{A,N,\epsilon} = \left\{ \vec f = \sum_{k=0}^{N-1} c_k \phi_k^{(\epsilon,N)} \in H_{A,N,\epsilon}: \sum_{j=0}^p \sum_{k=0}^{N-1} \left( \eta_k^{(\epsilon,N)} \right)^j\lvert c_k\rvert^2 < \infty \right\}, \\
  H^p_{A,\epsilon} = \left\{ \bar f = \sum_{k=0}^\infty c_k \phi_k^{(\epsilon)}\in H_{A,\epsilon}: \sum_{j=0}^p \sum_{k=0}^{\infty} \left( \eta_k^{(\epsilon)} \right)^j\lvert c_k \rvert^2 < \infty \right\}.
\end{gather*}          
Clearly, $ H^p_{A,N,\epsilon} = H_{A,N,\epsilon} $ since these spaces are finite-dimensional (though we prefer to maintain their distinction for consistency with the infinite-dimensional case), but $ H^p_{A,\epsilon} $ are strict subspaces of $ H_{A,\epsilon} $. However, we have:

\begin{lemma}
  \label{lemmaDense}The spaces $ H^p_{A,\epsilon} $ are dense in $ H_{A,\epsilon} $. 
\end{lemma}
\begin{proof}
  Let $ f $ be an arbitrary function in $ H_{A,\epsilon}$, and let $ \delta $ be any positive number. Then, there exists an integer $ n_\delta $ such that for all $ n \geq n_\delta $ the function $ f_n = \sum_{k=0}^{n} \langle \phi_k^{(\epsilon)}, f \rangle \phi_k^{(\epsilon)} $ satisfies $ \lVert f - f_n \rVert < \delta $, and moreover  $  \sum_{j=0}^p \sum_{k=0}^{\infty} ( \eta_k^{(\epsilon)} )^j\lvert \langle \phi_k^{(\epsilon)}, f_n \rangle_{H_{A,\epsilon}} \rvert^2 = \sum_{j=0}^p \sum_{k=0}^{n} ( \eta_k^{(\epsilon)} )^j\lvert \langle \phi_k^{(\epsilon)}, f_n \rangle_{H_{A,\epsilon}} \rvert^2 $ is finite so $ f_n \in H^p_{A,\epsilon} $. 
\end{proof}

Let now $ \vec f_j = \sum_{k=0}^{N-1} c_{kj} \phi_k^{(\epsilon,N)} $ and $ \bar f_j = \sum_{k=0}^{\infty} \tilde c_{kj} \phi_k^{(\epsilon,N)} $  be arbitrary functions in $ H^p_{A,\epsilon,N} $ and $ H^p_{A,\epsilon} $, respectively. As usual, we equip these spaces with the inner products  
\begin{displaymath}
  \langle \vec f_1, \vec f_2 \rangle_{H^p_{A,\epsilon,N}} = \sum_{j=0}^p \sum_{k=0}^{N-1}  \left( \eta_k^{(\epsilon,N)} \right)^j c_{k1}^* c_{k2},  \quad \langle \bar f_1, \bar f_2 \rangle_{H^p_{A,\epsilon}} = \sum_{j=0}^p \sum_{k=0}^{\infty}  \left( \eta_k^{(\epsilon)} \right)^j  \tilde c_{k1}^* \tilde c_{k2}, 
\end{displaymath} 
and the norms $ \lVert \vec f_1 \rVert_{H^p_{A,\epsilon,N}} = \sqrt{ \langle \vec f_1, \vec f_1 \rangle_{H^p_{A,\epsilon,N}} }$ and  $ \lVert \bar f_1 \rVert_{H^p_{A,\epsilon}} = \sqrt{ \langle \bar f_1, \bar f_1 \rangle_{H^p_{A,\epsilon}} }$. In addition, for $ p = 1 $, we define the sesquilinear forms $ E_{A,\epsilon,N} : H^1_{A,\epsilon,N} \times H^1_{A,\epsilon,N} \mapsto \mathbb{ C }$ and $ E_{A,\epsilon} : H^1_{A,\epsilon} \times H^1_{A,\epsilon} \mapsto \mathbb{ C }$,
\begin{equation}
  \label{eqEEpsilon}
  E_{A,\epsilon,N}( \vec f_1, \vec f_2 ) = \sum_{k=0}^{N-1} \eta_k^{(\epsilon,N)}c^*_{1k} c_{2k}, \quad E_{A,\epsilon}( \bar f_1, \bar f_2 ) = \sum_{k=0}^{\infty} \eta_k^{(\epsilon)} \tilde c^*_{1k} \tilde c_{2k}, 
\end{equation}
and for $ p = 2 $ we define the operators $ \upDelta_{\epsilon,N} : H^2_{A,N,\epsilon} \mapsto H_{A,N,\epsilon} $ and $ \upDelta_{A,\epsilon} : H^2_{A,\epsilon} \mapsto H_{A,\epsilon} $ with
\begin{equation}
  \label{eqDeltaEpsilon}
  \upDelta_{A,\epsilon,N} \vec f_1 = \sum_{k=0}^{N-1} \eta_k^{(\epsilon,N)} c_{1k} \phi_k^{(\epsilon,N)}, \quad  \upDelta_{A,\epsilon} \bar f_1 = \sum_{k=0}^{\infty} \eta_k^{(\epsilon)} \tilde c_{1k} \phi_k^{(\epsilon)}.
\end{equation} 
These forms and operators have the following properties: 
\begin{prop}
  \label{propDelta}
  (i) $ E_{A,\epsilon,N}$ and $ E_{\epsilon} $ obey the bounds 
\begin{displaymath}
  E_{A,\epsilon,N}( \vec f_1, \vec f_2 ) \leq \lVert \vec f_1 \rVert_{H^1_{A,\epsilon,N}}  \lVert \vec f_2 \rVert_{H^1_{A,\epsilon,N}}, \quad E_{A,\epsilon}( \bar f_1, \bar f_2 ) \leq \lVert \bar f_1 \rVert_{H^1_{A,\epsilon}}  \lVert \bar f_2 \rVert_{H^1_{A,\epsilon}}, 
\end{displaymath}
respectively.

  (ii) $ \upDelta_{A,\epsilon,N}$ and $ \upDelta_{A,\epsilon} $ are self-adjoint. Moreover, $ -\upDelta_{A,\epsilon,N} $ and $ - \upDelta_{A,\epsilon} $ are dissipative; i.e., 
  \begin{displaymath}
    \Real \langle \vec f, -\upDelta_{A,\epsilon,N} \vec f \rangle_{H_{A,\epsilon,N}} \leq 0, \quad \Real \langle \bar f, -\upDelta_{A,\epsilon} \bar f \rangle_{H_{A,\epsilon}} \leq 0.
  \end{displaymath}
\end{prop}

A proof of this Proposition can be found in~\ref{appPropDelta}. In Section~\ref{secDataDrivenM} ahead, $ E_{A,\epsilon,N} $ and $ E_{A,\epsilon} $ will be used to construct data-driven analogs of the Dirichlet form $ E $ employed in the Koopman eigenvalue problem for coherent patterns in Section~\ref{secGalerkin}. It is worthwhile noting that Proposition~\ref{propDelta}(ii) in conjunction with \cite{Phillips59} implies that $ \upDelta_{A,\epsilon,N} $ and $ \upDelta_{A,\epsilon} $ (or their maximal self-adjoint extensions) are generators of strongly-continuous contraction semigroups, $ \{ \mathcal{ P }_{\epsilon,N,s} \}_{s\geq 0} $ and  $ \{ \mathcal{ P }_{\epsilon,s} \}_{s\geq 0} $, respectively, with $ \mathcal{ P }_{\epsilon,N,s} = e^{-s\upDelta_{A,\epsilon,N}} $ and $ \mathcal{ P }_{\epsilon,s} = e^{-s \upDelta_{A,\epsilon}} $. For appropriate choices of $ \eta_k^{(\epsilon,N)} $ and $ \eta_k^{(\epsilon)}$ (e.g.,~\eqref{eqEtaEpsN} below), these semigroups approximate the heat semigroup $ \{ \mathcal{ P }_s \}_{s \geq 0} $, $ \mathcal{ P }_s = e^{-s \upDelta_A } $,  associated with the Laplace-Beltrami operator $ \upDelta_A $. 

We now return to the question of defining $ \eta_k^{(\epsilon,N)} $ and $ \eta_k^{(\epsilon)} $. As with the definition of the data-driven basis functions in Section~\ref{secBasis}, here too there are options. Below, we list three possibilities for $ \eta_k^{(\epsilon,N)} $, 
\begin{equation}
  \label{eqEtaEpsN}
  \eta_k^{(\epsilon,N)} = \frac{  1 - \Lambda_k^{(\epsilon,N)} }{  \epsilon }, \quad \eta_k^{(\epsilon,N)} = - \frac{ \log \Lambda_k^{(\epsilon,N)} }{  \epsilon }, \quad
  \eta_k^{(\epsilon,N)} = \frac{  \left( \Lambda_k^{(\epsilon,N)} \right)^{-1} - 1 }{ \epsilon },
\end{equation} 
which we refer to as option~1, 2, and 3 respectively. By Lemma~\ref{lemmaEig}(iii), as $ N \to \infty $, these quantities converge to 
\begin{equation}
  \label{eqEtaEps}
  \eta_k^{(\epsilon)} = \frac{  1 - \Lambda_k^{(\epsilon)} }{  \epsilon }, \quad \eta_k^{(\epsilon)} = - \frac{ \log \Lambda_k^{(\epsilon)} }{  \epsilon }, \quad
  \eta_k^{(\epsilon)} = \frac{  \left( \Lambda_k^{(\epsilon)} \right)^{-1} - 1 }{ \epsilon },
\end{equation}
respectively. 

\begin{rk}
  In our definitions of $ \eta_k^{(\epsilon,N)} $ and $ \eta_k^{(\epsilon)} $ via options~2 and~3, we have tacitly assumed that the $ \Lambda_k^{(\epsilon,N)}  $ and $ \Lambda_k^{(\epsilon)} $ are all positive, or, equivalently, that $ P_{\epsilon,N} $ and  $ P_\epsilon $ are positive operators. In the event that these operators have negative eigenvalues, one can modify the definitions of $ \eta_k^{(\epsilon,N)} $ and $ \eta_k^{(\epsilon)} $ for $ k > k^*$ (with $ k^* $ the index of the smallest positive eigenvalue) to a positive, increasing sequence with no accumulation points. For example, one could set $ \eta^{(\epsilon)}_{k>k^*} = \eta_{k^*}^{(\epsilon)} ( k / k^* )^{2/m_A} $, which is consistent with Weyl's law for the asymptotic growth of Laplace-Beltrami eigenvalues on compact Riemannian manifolds. It should be noted that this modification is mainly formal since for any finite spectral truncation parameter $\ell_A $ used in our schemes there exist $ N_0 $ and $ \epsilon_0 $ such that, for all $ N > N_0 $ and $ \epsilon < \epsilon_0 $,  $ \Lambda_{\ell_A-1}^{(\epsilon,N)} $ and $ \Lambda_{\ell_A-1}^{(\epsilon)} $ are positive. For the rest of the paper, we will  treat the $ \Lambda_k^{(\epsilon,N)} $ and $ \Lambda_k^{(\epsilon)} $ as positive eigenvalues. 
\end{rk}

With option~1, the $ \eta_k^{(\epsilon,N)} $ are equal to eigenvalues of the operator $  ( I - P_{\epsilon,N} ) / \epsilon $, i.e., the normalized graph Laplacian widely used in machine learning \cite{BelkinNiyogi03,CoifmanLafon06,VonLuxburgEtAl08,BerrySauer16}. It is a well-known fact that for an appropriate decreasing sequence $ \epsilon( N ) $, the $ \eta_k^{(\epsilon(N),N)} $ from this option converge to $ \eta_k^A $; that is,  $ \eta_k^{(\epsilon(N),N)} $ are consistent estimators of the Laplace-Beltrami eigenvalues. At fixed $ \epsilon $, the limit quantities $ \eta_k^{(\epsilon)} $, and hence the corresponding operator $ ( I - P_\epsilon ) / \epsilon $, are bounded (in particular, $ \eta_k^{(\epsilon)} \leq 2 / \epsilon $ since $ P_\epsilon $ is Markov). The $ \eta_k^{(\epsilon,N)} $ from option~2 also converge to the true Laplace-Beltrami eigenvalues for a decreasing sequence $ \epsilon( N) $ in accordance with Corollary~\ref{corEig}, but notice that since $ \lim_{k\to\infty}\Lambda^{(\epsilon)}_k = 0 $, in this case $ \eta_k^{(\epsilon)} $ and $ \upDelta_{A,\epsilon} $ become unbounded as $ N \to \infty $. Finally, the $ \eta_k^{(\epsilon,N)} $ from option~3 converge to $ \eta_k^A $ (again, for an appropriate decreasing sequence $ \epsilon( N ) $) under the additional assumption that the remainder $ R_k( \epsilon ) $ in the Taylor expansion $  - \log \Lambda_k^{(\epsilon)} =  ( \Lambda_k^{(\epsilon)} ) ^{-1} - 1 + R_k( \epsilon ) $ is asymptotically smaller than $ O( \epsilon ) $ as $ \epsilon \to 0 $. At finite $ N $ and nonzero $ \epsilon $ we have $( \Lambda_k^{(\epsilon,N)} )^{-1}  - 1 \geq \log \Lambda_k^{(\epsilon,N)} $, and therefore the Dirichlet energies from this option dominate those from option~2. Moreover the ratio $ ( (\Lambda_k^{(\epsilon,N)})^{-1} - 1 ) / ( - \log \Lambda_k^{(\epsilon,N)} ) $ increases with $ k $.

Note now that each of the definitions of $ \eta_k^{(\epsilon,N)} $ in~\eqref{eqEtaEpsN} will lead to a different regularized Koopman eigenvalue problem for coherent spatiotemporal patterns (after employing these quantities to define Dirichlet forms for functions on $ M $, as described in Section~\ref{secDataDrivenM} below).  At a minimum, whatever the specific choice of $ \eta_k^{(\epsilon,N)} $ (and hence $ \eta_k^{(\epsilon)} $) is, the associated Sobolev space $ H^1_\epsilon $ should have the property that the sesquilinear form $\langle \cdot, u( \cdot ) \rangle $ associated with the generator of the Koopman group on $ A $ is bounded on  $ H^1_{A,\epsilon} \times H^1_{A,\epsilon} $  (see Section~\ref{secGalerkin}), at least in the limit $ \epsilon \to 0 $. While all of the definitions for $ \eta_k^{(\epsilon,N)} $ in~\eqref{eqEtaEpsN} meet that asymptotic requirement, they will of course behave differently at finite $ N $ and nonzero $ \epsilon $.  In particular, in light of the above arguments, we expect the $ \eta_k^{(\epsilon,N)} $ from option~3 to provide the strongest and most selective regularization among the three options, followed by option~2, and then option~1. Indeed, in experiments we found that the Dirichlet forms from option~3 generally led to higher-quality numerical coherent patterns, so in what follows we nominally work with that option.

\subsection{\label{secDataDrivenM}Data-driven techniques for skew-product systems}

We now construct analogs of the schemes of coherent pattern extraction and nonparametric prediction of Section~\ref{secKoopman} using the basis functions and Dirichlet forms from Section~\ref{secDirichlet}. Throughout this Section, we use the symbols $ \phi_k^{A,\epsilon,N} $ and $ \eta_k^{A,\epsilon,N} $ in place of $ \phi_k^{(\epsilon,N)} $ and $ \eta_k^{(\epsilon,N)} $, respectively, in order to avoid notational ambiguities when building tensor product bases on the product space $ M = A \times X $. 

\subsubsection{\label{secApproximateKoop}Approximate Koopman generator on the Hilbert space associated with the sampling measure}

As stated in Section~\ref{secKernelOp}, in practical applications involving finite numbers of samples we do not have access to the Hilbert space $H_A $ where the Koopman group $ \{ U_t \} $ is naturally defined, and we work instead with Hilbert spaces such as $ H_{A,N} $ and $ H_{A,\epsilon,N} $ associated with the sampling measure $ \alpha_N$. While the latter two spaces form natural settings for defining kernel integral operators and their associated eigenfunctions and Dirichlet forms, they are unfortunately not suitable for defining a Koopman operators analogous to $ U_t $ via composition with the flow map. In essence, this is because the dynamical flow $ \Phi_t $ does not preserve null sets with respect to  $ \alpha_N $; that is, $ \alpha_N( S ) = 0 $ for some measurable set $ S \subset A $ does not imply that $ \alpha_N( \Phi_t^{-1}( S ) ) $ vanishes too. More specifically, Koopman operators on $ H_{A,N} $ and $ H_{A,\epsilon,N} $ would have to act on equivalence classes $ \vec f $ of functions  that have common values on the sampled states $ a_n $, but given two representatives $ f_1 $ and $ f_2 $ in $ \vec f $ the functions $ f_1 \circ \Phi_t $ and $ f_2 \circ \Phi_t $ could lie in different equivalence classes, resulting in ``$U_t \vec f = \vec f \circ \Phi_t $'' being ill-defined.  

Despite this issue, it is nevertheless possible to define operators on $ H_{A,N} $ and $ H_{A,\epsilon,N} $ that approximate $ U_t $ and the generator $ \tilde u $ without explicit use of the flow map. Here, we focus on approximations of the generator $ \tilde u $ that appears in the schemes of Section~\ref{secKoopman}. As a concrete example of such an approximation, consider the operator $ u_{N,\tau} : H_{A,\epsilon,N} \mapsto H_{A,\epsilon,N} $ (here, $ \tau $ is the sampling interval; see Section~\ref{secPrelim}), whose action on $ \vec f = ( f_0, f_1, \ldots, f_{N-1} ) $ is given by $ u_{N,\tau} \vec f = \vec g = ( g_0, \ldots, g_{N-1} ) $ with 
\begin{equation}
  \label{eqFD}
  g_0 = g_{N-1} = 0 \quad \text{and} \quad g_n = ( f_{n+1} - f_{n-1} ) / ( 2 \tau ), \quad  1 \leq n \leq N - 2. 
\end{equation}
Intuitively, $ u_{N,\tau} $, can be thought of as a (second-order, central) finite-difference approximation of $ \tilde u( f ) = \lim_{t\to0}( U_t f - f ) / t $ for $ f \in D( \tilde u ) $. Such finite-difference approximations were previously used in \cite{Giannakis15} to construct $ \tilde u$-dependent kernels and in \cite{GiannakisEtAl15,Giannakis17,DasGiannakis17} for approximating Koopman generators as done here.  

The following Lemma (proved in~\ref{appUConsistency}) establishes a consistency property of the sesquilinear form associated with the approximate generator, evaluated on the data-driven basis elements:

\begin{lemma}
  \label{lemmaUConsistency}
  Let $ \phi_i^{A,\epsilon,N}, \phi_j^{A,\epsilon,N} \in H_{A,\epsilon,N} $ be any two eigenfunctions of $ P_{\epsilon,N} $ from Section~\ref{secKernelOp} converging, in the sense of Corollary~\ref{corEig}, to eigenfunctions $ \phi_i^A, \phi_j^A \in H_A $ of $ \upDelta_A $. Then,
  \begin{equation}
    \label{eqUConsistency}
    \lim_{\tau \to 0} \lim_{\epsilon\to 0} \lim_{N\to\infty} \langle \phi_i^{A,\epsilon,N}, u_{N,\tau}( \phi_j^{A,\epsilon,N} ) \rangle_{H_{A,\epsilon,N}} = \langle \phi_i^A, \tilde u( \phi^A_j ) \rangle_{H_A}. 
  \end{equation}
  Moreover, the order of the limits $ \epsilon \to 0 $ and $ \tau \to 0 $ can be interchanged.  
\end{lemma}         
As we discuss below, the consistency results in Corollary~\ref{corEig} and Lemma~\ref{lemmaUConsistency} are sufficient to establish consistency of the techniques presented in Section~\ref{secKoopman} in the presence of errors due to finite training datasets and discrete-time sampling.      

\begin{rk} Of course, besides the second-order, central finite difference scheme in~\eqref{eqFD}, other approaches could be employed to construct the approximate generator $ u_{N,\tau} $, including higher-order finite differences or interpolation techniques. For instance, a straightforward improvement of~\eqref{eqFD} would be to replace  $ g_0 = 0 $ and $ g_{N-1} = 0 $ with values obtained from linear extrapolation of $ ( g_1, g_2 ) $ and $(  g_{N-3}, g_{N-2} ) $, respectively, though this modification would have only minor consequences at large $ N $.      
\end{rk}

\subsubsection{\label{secDataDrivenEig}Koopman eigenvalue problem for coherent spatiotemporal patterns}

Recall that we assume that the Laplace-Beltrami eigenfunction basis $ \{ \phi_0^X, \phi^X_1, \ldots \} $ for the Hilbert space $H_X $ of functions on the spatial domain and the associated eigenvalues $ \{ \eta_0^X, \eta_1^X, \ldots \} $ are given. Thus, in order to formulate an analog of the Galerkin method in Section~\ref{secGalerkin} utilizing the data-driven constructions of Sections~\ref{secKernel}--\ref{secApproximateKoop}, it suffices to consider the Hilbert space $ H_{\epsilon,N} = \{ f : M \mapsto \mathbb{ C } : \int_M \lvert f \rvert^2 \, d\mu_{\epsilon,N} < \infty \} $ associated with the product measure $ \mu_{\epsilon,N}  = \alpha_{\epsilon,N} \otimes \xi $. Our orthonormal basis for this space has a tensor product form, consisting of the functions $ \phi^{\epsilon, N}_{ij} = \phi_i^{A,\epsilon,N} \phi_j^X $, where $ i \in \{ 0, 1, \ldots, N-1 \} $ and $ j \in \{ 0, 1, \ldots \} $. We also define $ \eta^{\epsilon,N}_{ij} = \eta^{A,\epsilon,N}_i + \eta_j^X $, the associated Sobolev spaces 
\begin{displaymath}
  H^p_{\epsilon,N} = \left\{ f = \sum_{i=0}^{N-1} \sum_{j=0}^\infty c_{ij} \phi_{ij}^{\epsilon,N} : \sum_{l=0}^p \sum_{i=0}^{N-1} \sum_{j=0}^\infty ( \eta_{ij}^{\epsilon,N} )^l \lvert c_{ij} \rvert^2 < \infty \right\}
\end{displaymath}
with $ \langle f_1, f_2 \rangle_{H^1_{\epsilon,N}} = \sum_{p=0}^1 \sum_{ij} ( \eta^{\epsilon,N}_{ij} )^p c^*_{1ij} c_{2ij} $, $ f_k = \sum_{i=0}^{N-1} \sum_{j=0}^\infty c_{kij} \phi^{\epsilon,N}_{ij} $,   the Dirichlet form $ E_{\epsilon,N} : H^1_{\epsilon,N} \times H^1_{\epsilon,N} \mapsto \mathbb{ C } $ with $ E_{\epsilon,N}( f_1, f_2 ) = \sum_{i=0}^{N-1} \sum_{j=0}^\infty \eta_{ij}^{\epsilon,N}c^*_{1ij} c_{2ij}$, and the Laplace-Beltrami operator $ \upDelta_{\epsilon,N} : H^2_{\epsilon,N} \mapsto H_{\epsilon,N} $ with $ \upDelta_{\epsilon,N} f_1 = \sum_{i=0}^{N-1} \sum_{j=0}^\infty \eta_{ij}^{\epsilon,N} c_{1ij} $. As in Section~\ref{secLB}, we order the basis elements $ \phi_{ij}^{\epsilon,N} $ in order of increasing Dirichlet energy $ E_{\epsilon,N}( \phi_{ij}^{A,\epsilon,N} ) = \eta_{ij}^{\epsilon,N } $, and when convenient we use single-index notation, $ \phi_k^{\epsilon,N} = \phi_{i(k),j(k)}^{\epsilon,N} $ and $ \eta_k^{\epsilon,N} = \eta_{i(k)j(k)}^{\epsilon,N} $, $ k \in \{ 0, 1, 2, \ldots \} $, consistent with that ordering. We also introduce the normalized eigenfunctions $ \varphi_{0}^{\epsilon,N} = \phi_{0}^{\epsilon,N} $ and $ \varphi_{k>0}^{\epsilon,N} = \phi_k^{\epsilon,N} / ( \eta_k^{\epsilon,N} )^{1/2} $, which form an orthogonal basis of $ H^1_{\epsilon,N} $ with $ \lVert \varphi_0 \rVert_{H^1_{\epsilon,N}} = 1 $ and $ \lVert \varphi_k^{\epsilon,N} \rVert_{H^1_{\epsilon,N}} = [ 1 + ( \eta_k^{\epsilon,N} )^{-1} ]^{1/2}$. Note that as with the corresponding basis of $ H^1 $ in Section~\ref{secLB}, we have $ E_{\epsilon,N}( \varphi^{\epsilon,N}_i, \varphi^{\epsilon,N}_j ) = \delta_{ij} $ if $ i, j > 0 $ and $ E_{\epsilon,N}( \varphi^{\epsilon,N}_0, \varphi_j^{\epsilon,N} ) = 0 $. 

It follows by construction of the $  \{ \phi_k^{\epsilon,N} \} $ basis that a consistency result analogous to Corollary~\ref{corEig} holds; that is, as $ N \to \infty $ and $ \epsilon \to 0 $, for every eigenfunction $ \phi_k $ of the Laplace-Beltrami operator $ \upDelta $ there exist  $ \phi_k^{\epsilon,N} $ converging to $ \phi_k $ in the sense of Corollary~\ref{corEig}. In the case of the $ \eta^{\epsilon,N}_k $, if these quantities are computed using options~1 and~2 in~\eqref{eqEtaEpsN}, then Corollary~\ref{corEig} implies that they converge to the Laplace-Beltrami eigenvalues $ \eta_k $. If they are computed using option~3 (i.e., our nominal choice), then the $ \eta^{\epsilon,N} $ may either (1) converge to $ \eta_k $, (2) converge to finite quantities $ \tilde \eta_k \geq \eta_k $, (3) diverge. In addition to the first of these cases, case~2 is also  sufficient for the well-posedness of the Koopman eigenvalue problem introduced below. In particular, associated with $ \tilde \eta_k $ is a Sobolev space $ \tilde H^1 $ on $ M $ (defined in the obvious way), which is dense in $ H $ and satisfies $ \lVert f \rVert_{\tilde H^1 } \geq \lVert f \rVert_{H^1} $. This space can be used as an alternative to $H^1 $ in the variational eigenvalue problem in Definition~\ref{defEig}. In what follows, we will use the symbols $ \tilde \varphi_i $ and $ \tilde E : \tilde H^1 \times \tilde H^1 \mapsto \mathbb{ C } $ to represent the normalized eigenfunctions and Dirichlet form, respectively, for  $\tilde H^1 $. Then, for any $ \varphi_i^{\epsilon,N} $ and $ \varphi_j^{\epsilon,N} $ converging to $ \tilde \varphi_i $ and $ \tilde \varphi_j $, respectively,  
\begin{equation}
  \label{eqEConsistency}
  \lim_{\epsilon\to 0} \lim_{N\to\infty} E_{\epsilon,N}( \varphi_i^{\epsilon,N}, \varphi^{\epsilon,N}_j ) = \tilde E( \tilde \varphi_i, \tilde \varphi_j ).
\end{equation}

Next, using the results of Section~\ref{secApproximateKoop}, we introduce an operator on $ H^1_{\epsilon,N} $ approximating the generator  $\tilde w $ of the skew-product dynamical system on $ M $. Specifically, we define $ w_{N,\tau} : H^1_{\epsilon,N} \mapsto H_{\epsilon,N} $,  $ w_{N,\tau} = w_{N,\tau}^A + w_N^X $, where $ w_{N,\tau}^A $ is the natural lift of $ u_{N,\tau} $ on $ H_{\epsilon,N} $ (that is, $ w_{N,\tau}^A( f_A f_X ) = u_{N,\tau}( f_A ) f_X $ where $ f_A \in H_{A,\epsilon,N} $ and $ f_X \in H_X $), and $ w_N^X $ is  the lift of the state-dependent vector field $ v\rvert_a $ on $ X $  as in Section~\ref{secPrelim}. This operator has the analogous consistency property to that established for $  u_{N,\tau} $ in Lemma~\ref{lemmaUConsistency}, i.e., 
\begin{equation}
  \label{eqWConsistency}
  \lim_{\tau \to 0 } \lim_{\epsilon \to 0} \lim_{N\to\infty} \langle \phi_i^{\epsilon,N}, w_{N,\tau}( \phi_j^{\epsilon,N} ) \rangle_{H_{\epsilon,N}} = \langle \phi_i, \tilde w( \phi_j ) \rangle.  
\end{equation}    
Moreover, for all Dirichlet energy options in~\eqref{eqEtaEpsN} we have 
\begin{displaymath}
  \lim_{\tau \to 0 } \lim_{\epsilon \to 0} \lim_{N\to\infty} \langle \varphi_i^{\epsilon,N}, w_{N,\tau}( \varphi_j^{\epsilon,N} ) \rangle_{H_{\epsilon,N}} = \langle \tilde \varphi_i, w( \tilde\varphi_j ) \rangle.
\end{displaymath}
In the case of options~1 and~2 we additionally have $ \langle \tilde \varphi_i, w( \tilde\varphi_j ) \rangle_H = \langle \varphi_i, w( \varphi_j ) \rangle_H $, and thus we obtain a consistent approximation of $ W( \varphi_i, \varphi_j ) $, where $ W $ is the sesquilinear form on $ H^1 \times H^1 $ associated with $ \tilde w $ (see~\eqref{eqSesquiAB}). In the case of option~3, that approximation may not be consistent, but so long as the $ \tilde \eta_k $ are finite, we nevertheless have 
\begin{equation}
  \label{eqWBound}
  \lvert \langle \tilde\varphi_i, \tilde w( \tilde \varphi_j ) \rangle \rvert \leq \lVert w \rVert_\infty \lVert \tilde \varphi_i \rVert_{H^1} \lVert \tilde \varphi_j \rVert_{H^1} \leq  \lVert w \rVert_\infty \lVert \tilde \varphi_i \rVert_{\tilde H^1} \lVert \tilde \varphi_j \rVert_{\tilde H^1} \leq \lVert w \Vert_{\infty} ( 1 + \tilde \eta_1^{-1} ).
\end{equation}
Thus, $ W $ can also be defined as a sesquilinear form on $\tilde H^1 \times \tilde H^1 $.    

\begin{rk}
  \label{rkEpsilon}One of the properties that must necessarily hold in order for the variational eigenvalue problems employed in this work to be well posed is that $ W( \psi, z ) = \langle \psi, \tilde w( z ) \rangle $ is bounded on the domain of the Dirichlet form used for regularization (e.g., $ E $ and $ \tilde E $ with domains $ H^1 \times H^1 $ and $ \tilde H^1 \times \tilde H^1 $, respectively). Depending on how the  $ \eta_k^{(\epsilon)} $ are defined, it may be possible to satisfy this condition without taking the limit $ \epsilon \to 0 $. For instance, the $ \eta_k^{(\epsilon)} $ from options~2 and~3 in~\eqref{eqEtaEps} are unbounded for all $ \epsilon > 0 $, so it is possible that $ W $ is bounded on $ H^1_\epsilon \times H^1_\epsilon $; the order-1 Sobolev space associated with the $ \eta_k^{(\epsilon)} $. If it indeed holds, this property would suggest that our schemes are applicable for more general state spaces $ A $ than smooth manifolds (since $ \epsilon \to 0 $ limits may be problematic on such spaces), including non-smooth compact topological attractors and/or compact supports of physical measures. Addressing these issues (which arise more broadly than the skew-product systems of interest here) is beyond the scope of this work, but in Section~\ref{secL96} we provide numerical evidence that our approach may indeed be applicable if $ A  $ is not smooth.       
\end{rk}

Fixing now the diffusion regularization parameter $ \theta \geq 0 $ and the spectral truncation parameters $ \ell_A \leq N- 1 $, $ \ell_X $, and $ \ell = \ell_A \times ( 2 \ell_X + 1 ) $, we introduce the $\ell $-dimensional Galerkin approximation spaces $ H^1_{\epsilon,N,\ell} = \spn\{ \varphi^{\epsilon,N}_0, \ldots, \varphi^{\epsilon,N}_{\ell-1} \} \subset H^1_{\epsilon,N} $ and consider the following variational eigenvalue problem, which is analogous to the eigenvalue problem in Definition~\ref{defEig}:
\begin{defn}[data-driven eigenvalue problem for coherent patterns] \label{defEigDat} Find $ z \in  H_{\epsilon,N,\ell}^1  $ and $ \lambda \in \mathbb{ C } $ such that for all $ \psi \in  H_{\epsilon,N,\ell}^1 $ the equality $ A_{\epsilon,N,\tau}( \psi, z ) = \lambda B_{\epsilon,N} ( \psi, z ) $ holds, where $ A_{\epsilon,N,\tau} $ and $B_{\epsilon,N} $ are the sesquilinear forms on $ H_{\epsilon,N}^1  \times H_{\epsilon,N}^1  $, given by 
  \begin{displaymath}
    A_{\epsilon,N,\tau}( \psi, z ) = W_{\epsilon,N,\tau}( \psi, z )  - \theta E_{\epsilon,N}( \psi, z ), \quad W_{\epsilon,N,\tau}( \psi, z ) = \langle \psi, w_{N,\tau}( z ) \rangle_{H_{\epsilon,N}}, \quad B_{\epsilon,N}( \psi, z ) = \langle \psi, z \rangle_{H_{\epsilon,N}}. 
  \end{displaymath}
\end{defn} 

Solving the eigenvalue problem in Definition~\ref{defEigDat} is equivalent to solving an $ \ell \times \ell $ matrix generalized eigenvalue problem with an entirely analogous structure to~\eqref{eqEig}. Moreover, it follows from~\eqref{eqEConsistency} and~\eqref{eqWConsistency} that for any $\ell $, the Galerkin method in Definition~\ref{defEigDat} converges as $ N \to \infty $ and $ \epsilon \to 0 $ to a Galerkin method for the variational eigenvalue problem on $ \tilde H^1 \times \tilde H^1 $ given by $ \tilde A( \psi, z ) = \lambda B( \psi, z ) $,  $ \tilde A( \psi, z ) =  W( \psi, z ) - \theta \tilde E( \psi, z ) $, restricted to the $\ell $-dimensional approximation space. By the properties of $ W $ (in particular,~\eqref{eqWBound}) and $ \tilde E$, this problem is well-posed and the Galerkin approximation converges as $ \ell \to \infty $. If we have in addition $ \tilde \eta_k = \eta_k $ (which is necessarily the case for options~1 and~2 in~\eqref{eqEtaEpsN} but not for option~3), then it follows that $ \tilde A( \psi, z ) = A( \psi, z ) $, and the Galerkin method in Definition~\ref{defEigDat} converges to the one in Definition~\ref{defEig}. It should be noted that $ A $ and $ \tilde A $ essentially differ only with respect to diffusion regularization, and are both expected to yield adequate approximate Koopman eigenfunctions at small values of the regularization parameter $ \theta $  (at least when $ \tilde w $ has discrete or mixed spectrum). As stated in Section~\ref{secDirichletE}, in practical  applications we find that the diffusion regularization from option~3 leads to a modest improvement in the quality of  numerical Koopman eigenfunctions. Explicit formulas for the matrices appearing in the generalized eigenvalue problem from Definition~\ref{defEigDat} in the case of a doubly-periodic spatial domain, $ X = \mathbb{ T}^2 $, can be found in~\ref{appGenT2}.

\subsubsection{\label{secDataDrivenPred}Prediction of observables and probability densities}

In this section, we introduce prediction schemes for observables and probability densities analogous to those in Section~\ref{secPrediction}, but formulated in the data-driven basis $ \{ \phi_k^{\epsilon,N} \} $ of  $ H_{\epsilon,N} $. In what follows, we consider a prediction observable $ f : M \mapsto \mathbb{ C } $ such that $ \bar f $ lies in $ H $; this ensures that for $ \alpha $-a.e.\  starting state $ a_0  $ in the training data, $ \vec f $ lies in $ H_{\epsilon,N} $. It then follows from the convergence of the data-driven eigenfunctions $ \phi_k^{A,\epsilon,N} $ to the continuous functions $ \tilde \phi_k^A $ established in Corollary~\ref{corEig}, that the expansion coefficients $ c_k^{\epsilon,N} = \langle \phi_k^{\epsilon,N}, \vec f \rangle_{H_{\epsilon,N}} $ converge to $ c_k = \langle \phi_k,  \bar f  \rangle $ for all $ k $; specifically, 
\begin{equation}
  \label{eqCConv}
  \lim_{\epsilon \to 0} \lim_{N\to\infty} c_k^{\epsilon,N} = c_k. 
\end{equation}

Let now $ L_{\epsilon,N,\tau}  : H_{\epsilon,N} \mapsto H_{\epsilon,N} $ be the data-driven regularized generator defined via $ L_{\epsilon,N,\tau } \vec f = w_{N,\tau} \vec f - \theta \upDelta_{\epsilon,N} f $, where $ \theta \geq 0 $, and $ w_{N,\tau} $ and $ \upDelta_{\epsilon,N} $ are the data-driven generator and Laplace-Beltrami operators from Section~\ref{secDataDrivenEig}, respectively. Let also $ \tilde \Pi_{\epsilon,N,\ell} $ be orthogonal projectors from $ H_{\epsilon,N} $ to $ H_{\epsilon,N,\ell} = \spn \{ \phi^{\epsilon,N}_0, \ldots, \phi^{\epsilon,N}_{\ell-1} \} $, and define (cf.~\eqref{eqStFJ2})
\begin{equation}
  \label{eqSemigroupPhi}
  \tilde S_{\epsilon,N,\tau,\ell,t} \vec f = \sum_{n=0}^\infty \frac{ t^n }{ n! } \tilde L^n_{\epsilon,N,\tau,l} \vec f, \quad \tilde L_{\epsilon,N,\tau,\ell} = \tilde \Pi_{\epsilon,N,\ell} L_{\epsilon,N,\tau} \tilde \Pi_{\epsilon,N,\ell}.
\end{equation}
Then, assuming that the eigenvalues $ \eta_k^{\epsilon,N} $ of $ \upDelta_{\epsilon,N} f $ converge to the eigenvalues $ \eta_k $ of $ \upDelta $, we have 
\begin{equation}
  \label{eqLCConv}
  \lim_{\tau \to 0} \lim_{\epsilon \to 0} \lim_{N\to\infty} \langle \phi_j^{\epsilon,N}, L_{\epsilon,N,\tau} \phi_k^{\epsilon,N} \rangle_{H_{\epsilon,N}} = \langle \phi_j^{\epsilon,N}, L \phi_k \rangle.       
\end{equation}
Together, \eqref{eqCConv} and~\eqref{eqLCConv} lead to the following consistency result:

\begin{lemma}
  \label{lemmaSConsistency}
  If  the eigenvalues $ \eta_k^{\epsilon,N} $ converge to $ \eta_k $, then for any observable $ f : M \mapsto \mathbb{ C}  $ such that $ \bar f \in H $ and $ \alpha $-a.e.\ starting state $ a_0 $ in the training data, 
\begin{displaymath}
  \lim_{\tau \to 0 }\lim_{\epsilon\to 0} \lim_{N\to \infty} \tilde S_{N,\epsilon,\tau,\ell,t} \vec f =  \tilde S_{\ell,t} \bar f.
\end{displaymath}  
\end{lemma}

Now, as with the eigenvalue problem in Section~\ref{secDataDrivenEig}, if the  $ \eta_k^{\epsilon,N} $ do not converge to $ \eta_k $ as $ N \to \infty $ and $ \epsilon \to 0 $, but instead converge to finite quantities $  \tilde \eta_k \neq \eta_k $, then these quantities are bounded below by $ \eta_k $,  and the matrix elements $  \langle \phi_j^{\epsilon,N}, \tilde L_{\epsilon,N,\tau} \phi_k^{\epsilon,N} \rangle_{H_{\epsilon,N}} $ converge to the matrix elements $ \langle \phi_j,  \tilde L \phi_k \rangle $ of $  \tilde L = w - \theta  \tilde \upDelta $, where $  \tilde \upDelta $ is a symmetric, dissipative operator such that $  \tilde \upDelta f = \sum_{k=0}^\infty \langle \phi_k, f \rangle \tilde \eta_k \phi_k $. An analogous consistency result to that in Lemma~\ref{lemmaSConsistency} then holds for the contraction semigroup generated by $ \tilde L $, and this semigroup can also be used as an approximation of the Koopman group $ W_t $. 

Following the approach of Section~\ref{secBoundedDen}, we approximate the action of the Perron-Frobenius group on probability densities $ \rho : M \mapsto \mathbb{ C } $  with $ \bar \rho \in H $ by evaluating
\begin{equation}
  \label{eqAdjSemigroupPhi}
  \tilde S^*_{\epsilon,N,\tau,\ell,t} \vec \rho = \sum_{n=0}^\infty \frac{ t^n }{ n! } \tilde L^{*n}_{\epsilon,N,\tau,\ell} \vec \rho, \quad \tilde L^*_{\epsilon,N,\tau,\ell} = \tilde \Pi_{\epsilon,N,\ell} L^*_{\epsilon,N,\tau} \tilde \Pi_{\epsilon,N,\ell}.
\end{equation}    
This approximation has an analogous consistency property to that in Lemma~\ref{lemmaSConsistency} if the $ \eta_k^{\epsilon,N} $ converge to $ \eta_k$; otherwise, if the $ \eta_k^{\epsilon,N} $ converge to $ \tilde \eta_k < \infty $,  $ \tilde S^*_{\epsilon,N,\tau,\ell,t}  $ approximates the Perron-Frobenius group $ W_t^* $ through a different semigroup than $ S^*_t $. 

As with the schemes of Sections~\ref{secPrediction}, we evaluate~\eqref{eqSemigroupPhi} and~\eqref{eqAdjSemigroupPhi} by forming the $ \ell  \times \ell  $ generator matrices $ \boldsymbol{L} = [ \langle \phi_i^{\epsilon,N}, L_{\epsilon,N,\tau} \phi_j^{\epsilon,N} \rangle ]_{ij} $ and $ \boldsymbol{ L }^* $, and computing $ e^{t \boldsymbol{L}} \vec b $ and $ e^{t \boldsymbol{L^*}} \vec b' $, respectively, via Leja interpolation, where $ \vec b = ( b_0, \ldots, b_{\ell-1} )^\top $, $ b_i = \langle \phi_i^{\epsilon,N}, f \rangle_{H_{\epsilon,N} } $, and $ \vec b' = ( b'_0, \ldots, b'_{\ell-1} )^\top $, $  b'_i = \langle \phi_i^{\epsilon,N}, \rho \rangle_{H_{\epsilon,N} } $, are the expansion coefficients of $ f $ and $ \rho $ in the $ \{ \phi_i^{\epsilon,N} \} $ basis of $ H_{\epsilon,N} $, respectively. Explicit formulas for the elements of $ \boldsymbol{ L } $ and $ \boldsymbol{ L }^* $ in the case $ X = \mathbb{ T }^2 $ are provided in~\ref{appGenT2}. 

\section{\label{secL96}Demonstration in flows driven by Lorenz~96 systems}

\subsection{\label{secL96Model}Model description}

In this Section, we apply the methods presented in Section~\ref{secDataDrivenBasis} to a class of time-dependent incompressible flows introduced by Qi and Majda \cite{QiMajda16}. In this class of flows the spatial domain is doubly periodic, i.e., $ X = \mathbb{ T }^2 $ as in the examples of Section~\ref{secExamples}. Moreover, in canonical angle coordinates $\{  x_1, x_2 \} $ the velocity field has the general form
\begin{displaymath}
  v \rvert_a = v^{(1)}( a ) \, \partial_1 + v^{(2)}( a ) \, \partial_2,  
\end{displaymath}   
where $ a $ is the state in $ A $ (to be specified below), and $ v^{(1)}, v^{(2)} $ are continuous functions on $ A $ taking values in $ C^\infty( X) $ and satisfying $ \partial_{1} v^{(1)}( a ) = \partial_{2} v^{(2)}( a ) = 0 $ for every $ a \in A $. Thus, by construction, $ v\rvert_a $ has vanishing divergence with respect to the normalized Haar measure $ \xi $ on $ X $ for all $  a \in A $. Based on earlier work by Majda et al.\ \cite{MajdaKramer99,MajdaGershgorin10}, in \cite{QiMajda16}  they choose the functions $ v^{(1)}( a ) $ and $ v^{(2)}( a ) $ so as to produce a flow involving a large-scale (spatially homogeneous) cross-sweep flow along the $ x_1 $ direction and a spatially inhomogeneous shear along the $ x_2 $ direction. Specifically, 
\begin{equation}
  \label{eqL96V}
  v^{(1)}( a ) = \hat s_0( a) \phi^X_{00}, \quad v^{(2)}( a ) = \left( \sum_{q=-J}^{-1} + \sum_{q=1}^J \right) \hat s_q( a ) \phi^X_{q0},  
\end{equation}     
where $ J $ is a positive integer, $ \hat s_{-J}, \ldots, \hat s_J $ are complex-valued functions on $ A $, and $ \phi^X_{ij}( x_1, x_2 ) = e^{\ii(ix_1 + jx_2)} $, $ i,j \in \mathbb{ Z } $, are the Fourier functions on $ X $. Note that~\eqref{eqL96V} determines the Fourier expansion coefficients $ \hat v_{qr}^{(1)} $ and $ \hat v_{qr}^{(2)} $ in the generator formula for arbitrary state-dependent velocity fields in~\eqref{eqVFourier}; that is, 
\begin{displaymath}
  \hat v_{qr}^{(1)} = \hat s_0 \delta_{q0} \delta_{r0}, \quad \hat v_{qr}^{(2)} = 
  \begin{cases}
    \hat s_q \delta_{r0}, & q \in \{ -J, \ldots,-1,1, \ldots, J\}, \\
    0, & \text{otherwise}.
  \end{cases}
\end{displaymath}   

The model in~\eqref{eqL96V} was introduced as a low-order model for passive tracers in a jet; correspondingly, the dynamics driving the evolution of the Fourier coefficients $ \hat s_j $ was chosen so as to mimic some of the key features of turbulent tracer advection such as intermittency and non-Gaussian statistics. In particular, in \cite{QiMajda16}  the evolution of the Fourier coefficients is driven by a L96 system with $ 2 J + 1 $ degrees of freedom, which we take here to be $ \dot s = \vec u( s ) $, 
with $ s = ( s_0, \ldots, s_{2J} ) \in \mathbb{ R }^{2J+1} $, $ \vec u( s ) = ( u_1( s ), \ldots, u_{2J}(s) ) $, and
\begin{equation}
  \label{eqL96}
    u_j( s ) = ( s_{j+1} - s_{j-2} ) s_{j-1} - s_j + F_{\text{L96}}, \quad s_{-2} = s_{2J-1}, \quad s_{-1} = s_{2J}, \quad s_{2J+1}=s_0, \quad F_{\text{L96}} \geq 0.
\end{equation}
This system was introduced by Lorenz in 1996 \cite{Lorenz96} as a low-order model for atmospheric flow at a constant latitude circle, with the variables $ s_j $ representing the zonal (west--east) component of the velocity field at $ 2 J + 1 $ zonally equispaced gridpoints. 

Formally, we assume that for each parameter choice  $ ( J, F_\text{L96} ) $ used below, the L96 system admits an invariant ergodic measure $ \alpha $ supported on a compact (invariant) set $ A \subset \mathbb{ R }^{2J +1} $, which is physical with respect to the Lebesgue measure of $ \mathbb{ R }^{2J+1} $. That is, given solutions $ s( 0 ), s( \tau ), s( 2 \tau ), \ldots $ of the L96 system sampled discretely in time at the sampling interval $ \tau $, the set
\begin{displaymath}
  B( \alpha ) = \left \{ s( 0 ) \in \mathbb{ R }^{2J+1} : \lim_{N\to\infty} \frac{ 1 }{ N } \sum_{n=0}^{N-1} f(  s( n \tau ) ) = \int_A f \, d\alpha, \; f \in C(\mathbb{ R }^{2J+1} ) \right \}, 
\end{displaymath}
called the basin of $ \alpha $, is assumed to have positive Lebesgue measure in $ \mathbb{ R }^{2J+1} $. Note that unlike the Lorenz 63 model, where this property is known to hold \cite{LuzzattoEtAl05}, to our knowledge the existence of physical measures for L96 systems is an open problem. 

Under these assumptions, we can construct a skew-product flow on $ M = A \times X $ governing the evolution of Lagrangian tracers as described in Section~\ref{secPrelim}, and approximate the associated Koopman generator from long-time integrations of the L96 driving system. In particular, the assumed compactness of $ A $ ensures that the kernel integral operators $ P_{\epsilon,N} $ and $ P_\epsilon $ from Section~\ref{secDataDrivenBasis} used to build our data-driven basis are well defined and compact. However, $ A $ is not expected to be smooth, meaning that the various $ \epsilon \to 0 $ limits taken in Section~\ref{secDataDrivenBasis} to establish convergence of our data-driven basis functions and  Dirichlet energies to corresponding quantities associated with heat kernels on smooth manifolds may not exist. Nevertheless, according to Remark~\ref{rkEpsilon}, our schemes for coherent pattern identification and prediction may still converge away from that limit.   
  
Under the assumption that such a set $ A \subset \mathbb{ R }^{2J+1} $ exists, we choose it as the state space of the ergodic system driving the non-autonomous flow on $ X $. Moreover, following \cite{QiMajda16}, we define $ \hat s_r \in C( A ) $ such that for $ a = ( s_0, \ldots, s_{2J} ) \in A $,  
\begin{equation}
  \label{eqL96Fourier}
  \hat s_q( a ) = \frac{ 1 }{ 2 J + 1 } \sum_{j=0}^{2J} e^{-2 \pi \ii q j / ( 2 J + 1 )} s_j(a).
\end{equation} 
Thus, the functions $ \hat s_q $ used to define the state-dependent velocity field in~\eqref{eqL96V} are obtained by taking the discrete Fourier transform of the L96 variables $ s_j $. Taking also the generator $ u $ of the dynamics on $ A $ to be the restriction of the L96 vector field $ \vec u $ from~\eqref{eqL96} completes the specification of the skew-product system on $ M $ with vector field $ w = u + v $ and invariant measure $ \mu = \alpha \otimes \xi $. Note that the specification of $ v|_a  $ via~\eqref{eqL96V} and~\eqref{eqL96} also defines the observation map $ F : A \mapsto \mathfrak{ X } $ from Section~\ref{secPrelim}; i.e., $ v \rvert_a = F( a ) $. By standard properties of Fourier transforms, we have $ \lVert F( a ) - F( a' ) \rVert_{\mathfrak{X}} = \lVert a - a'  \rVert_{\mathbb{R}^{2J+1}} $ so that computing pairwise distances between velocity field snapshots is equivalent to computing pairwise distances between L96 state vectors. 

As discussed in~\ref{appGenT2}, given a data-driven basis $ \{ \phi_i^{\epsilon,N} \} $ of $ H_{A,\epsilon,N} $ (computed as described in Section~\ref{secL96Basis} ahead) and the corresponding triple product coefficients $ c_{ijk} $ from~\eqref{eqStructureConstants}, the matrix elements of the generator $ \tilde w_{N,\tau} = \tilde w^A_{N,\tau} + \tilde w^X_N $ appearing in the schemes of Section~\ref{secDataDrivenBasis} can be computed through~\eqref{eqWAT2} and~\eqref{eqWXT2Trunc}. In the case of the L96-driven flow in~\eqref{eqL96V}, the latter equation becomes 
\begin{equation}
  \label{eqWXL96Trunc}
  \langle \phi^{\epsilon,N}_{ijk}, \tilde w_{N,\ell_v}^X( \phi^{\epsilon,N}_{lmn} ) \rangle_{H_{\epsilon,N}} = \sum_{p=0}^{\ell_v-1} \ii c_{ilp} \left( m \hat s_{p0} \delta_{jm} \delta_{kn} + \hat s_{p,j-m} \delta_{kn} \right),
\end{equation}
where
\begin{displaymath}
  \hat s_{pq} = 
  \begin{cases}
    \langle \phi_q^{\epsilon,N}, \hat s_q \rangle_{H_{A,\epsilon,N}}, & \lvert q \rvert \leq J,\\
    0, & \text{otherwise},
  \end{cases}
\end{displaymath} 
and $ \ell_v $ is a spectral truncation parameter for the state-dependent velocity field $ v\rvert_a $. In what follows, we set $ \ell_v = \ell_A $.

\subsection{\label{secL96Basis}Dynamical regimes and data-driven basis functions}

The L96 model in~\eqref{eqL96} includes a number of important features of atmospheric dynamics, namely energy-preserving nonlinear advection ($( s_{j+1} - s_{j-2} ) s_{j-1}$), damping ($-s_j$), and external forcing ($F_\text{L96}$). It is well known  that as $ F_\text{L96} $ increases, the system transitions through different dynamical regimes ranging from steady, periodic, quasiperiodic, weakly chaotic, strongly chaotic, and fully turbulent \cite{MajdaEtAl05}. Here, we work with the parameter values $ J = 20 $ (corresponding to 41 degrees of freedom) and $ F_{\text{L96}} = 4 $ or $ F_\text{L96}= 5 $; representative solutions for these regimes, computed ``on the attractor'', and the root mean square (RMS) amplitudes of the Fourier modes $ \hat s_j $   are shown in Figs.~\ref{figL96} and~\ref{figL96Rec}, respectively.  Published numerical results on the L96 model \cite{KarimiPaul10} indicate that the case $ J = 20 $, $ F_\text{L96} = 5 $ corresponds to a weakly chaotic regime with Lyapunov attractor dimension $ \simeq 17 $. We were not able to find analogous results for the $ J = 20 $, $ F_\text{L96} = 4 $ case, but based on available numerical results for nearby regimes \cite{VanKekemSterk17} and qualitative inspection of numerical integrations, it is likely that this regime exhibits quasiperiodic or low-dimensional chaotic behavior (referred to hereafter as quasiperiodic behavior). In particular, it can be seen from Figs.~\ref{figL96} and~\ref{figL96Rec} that the $ F_\text{L96} = 4 $ regime is dominated by wavenumber 7 waves with negative phase velocity, whose duration varies aperiodically. Wavenumber 7 waves are also clearly visible in the $ F_\text{L96} = 5 $ regime, but now the system also exhibits pronounced intermittent excursions from coherent wavenumber~7 activity.    

\begin{figure}
  \centering\includegraphics[scale=.8]{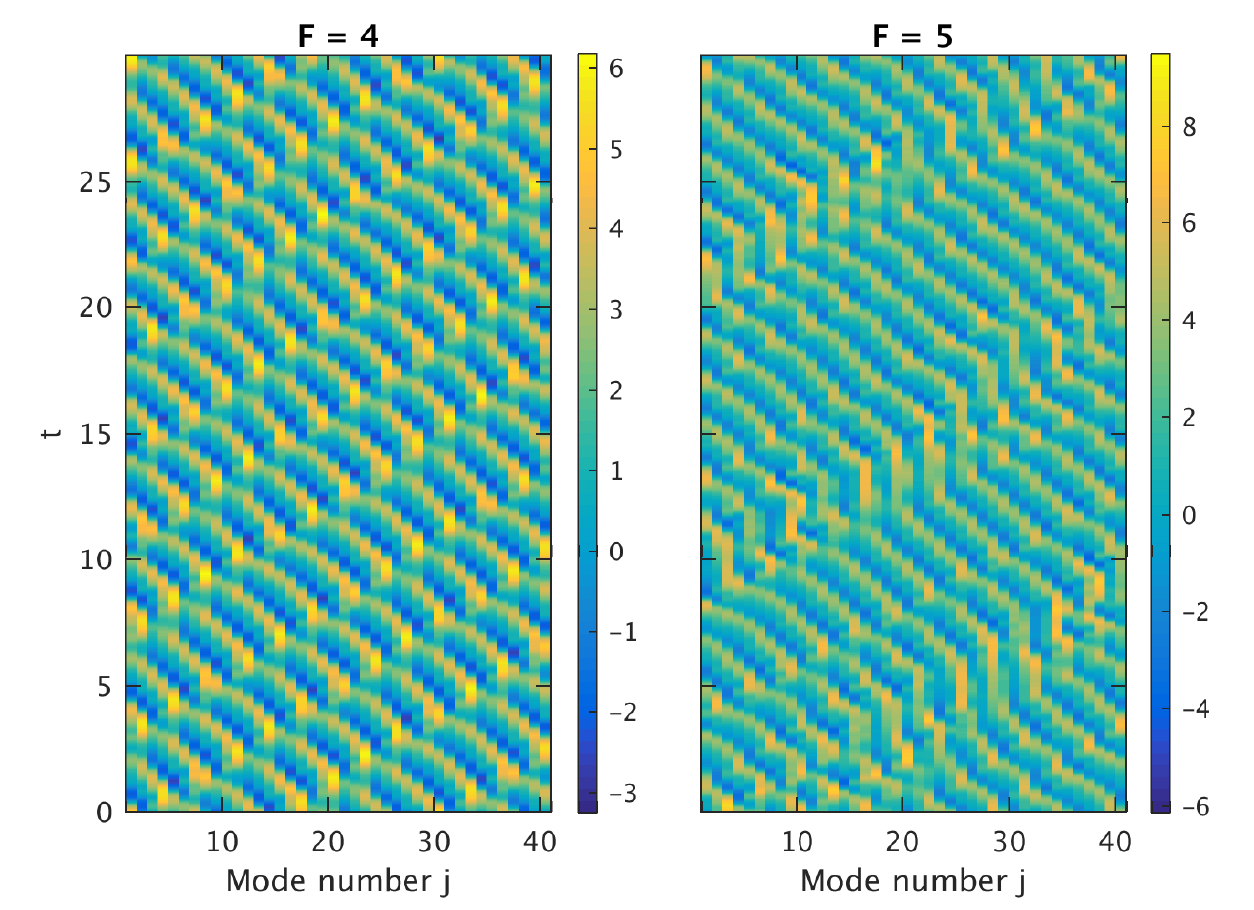}
  \caption{\label{figL96}Evolution of the L96 systems with $ F_\text{L96} = 4 $ and 5 for an interval of 30 time units.}
\end{figure}

\begin{figure}
  \centering\includegraphics{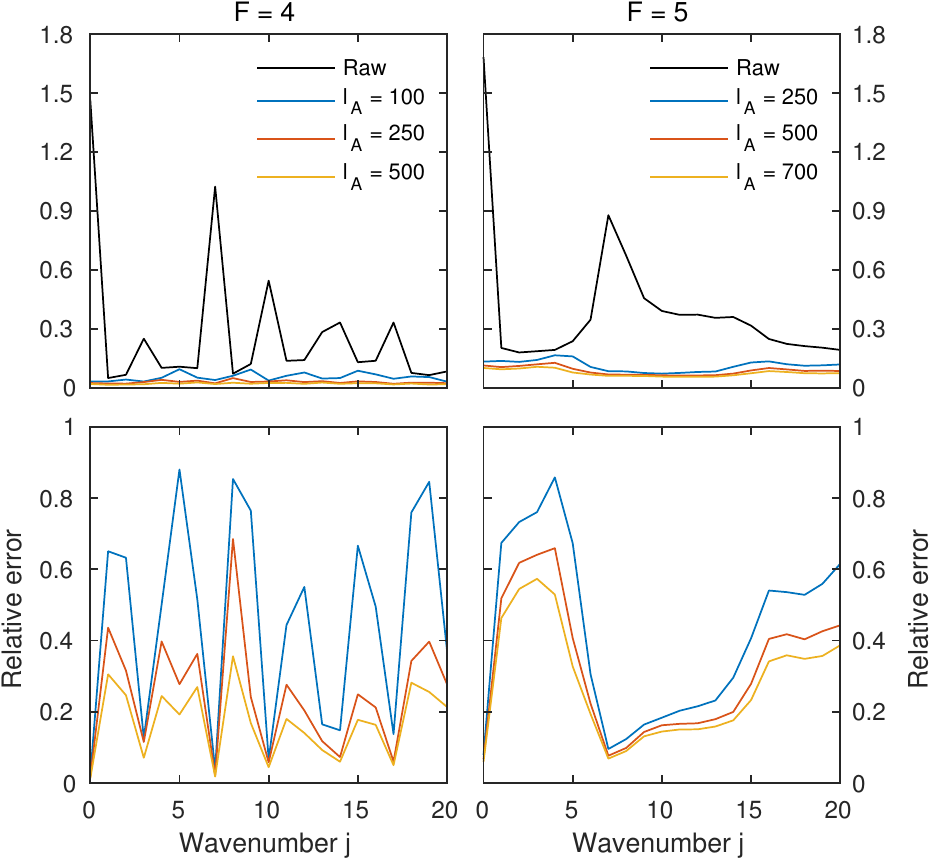}
  \caption{\label{figL96Rec}Top panels: RMS amplitude (black lines) and reconstruction error (colored lines) for the Fourier modes of the L96 systems with $ F_\text{L96} = 4 $ and 5, computed for representative values of the spectral truncation parameter $ \ell_A $. Bottom panels: Relative reconstruction error for the same $ \ell_A $ values.  Since the L96 modes $ s_j $ are real only the Fourier modes corresponding to nonnegative wavenumbers are shown.}
\end{figure}

Numerically, we generate training data from our chosen dynamical regimes by integrating the L96 model (using Matlab's \texttt{ode45} solver) with initial conditions $ s(0) = ( 1, 0, \ldots, 0 ) $. Waiting for transients to decay over a long spinup time (approximately 5,000 natural time units), we sample the solution $ s( t ) $ $ N $ times every $ \tau = 0.01 $ time units after the spinup period, where $ N = \text{128,000} $ and $ \text{512,000} $ for $ F_\text{L96} = 4 $ and $ 5 $, respectively (Figs.~\ref{figL96} and~\ref{figL96Rec} were produced using those samples). Thus, taking into account the assumption on physical measures made in Section~\ref{secL96Model}, we are effectively assuming that for both of the above choices of $ ( J, F_\text{L96} ) $ the point $ ( 1, 0, \ldots, 0 ) \in \mathbb{ R }^{2J+1} $, lies in the basin of some ergodic probability measure $ \alpha $ supported on an invariant compact set $ A $. In what follows, we use the shorthand notation $ a_n = s( n \tau ) $, $ n \in \{ 0, \ldots, N-1\} $ for the L96 training states for notational compatibility with the rest of the paper, but note that we do not require that the $a_n $ lie exactly on $ A $.   

Using this training data, we compute data-driven basis functions $ \phi_k^{A,\epsilon,N} $  and their associated eigenvalues $ \Lambda_k^{A,\epsilon,N} $ as described in Sections~\ref{secKernel} and~\ref{secKernelOp}. Representative eigenfunctions form these calculations, plotted as time series for the same portions of the training datasets as those shown in Fig.~\ref{figL96}, are displayed in Fig.~\ref{figL96Phi}. There, it can be seen that in both the $ F_\text{L96} = 4 $ and $ F_\text{L96} = 5 $ cases the eigenfunctions capture a range of timescales generated by the L96 system. As discussed in Section~\ref{secTuning}, our procedure for tuning the kernel bandwidth also yields an effective intrinsic dimension for $ A $; this was found to be 4.1 and 6.4, respectively for  $ F_\text{L96} = 4 $ and 5, respectively. Note that the effective dimension for $ F_\text{L96} = 5 $ is smaller than the corresponding Lyapunov attractor dimension reported in \cite{KarimiPaul10} by a factor of 2.5. This is not too surprising given that if $ A $ is indeed $ \simeq 17 $-dimensional our 512,000 sample dataset only provides a coarse sampling of that set. 

\begin{figure}
  \centering\includegraphics{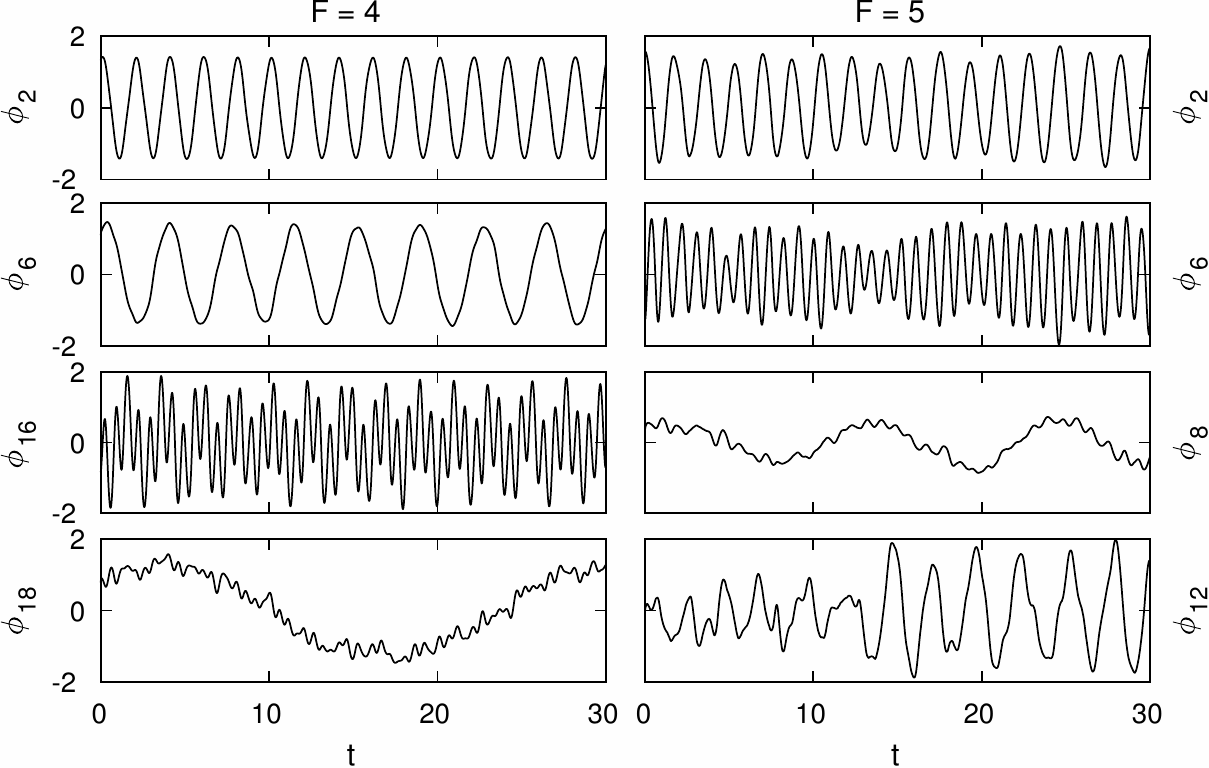}
  \caption{\label{figL96Phi}Time series of representative data-driven basis functions $ \phi_k^{A,\epsilon,N} $ for the L96 systems with $ F_\text{L96} = 4 $ and 5.} 
\end{figure}   

Next, we examine how many basis functions $ \phi_k^{A,\epsilon,N} $ are needed to accurately reconstruct the Fourier modes $ \hat s_j $ in~\eqref{eqL96Fourier} (and hence the state-dependent velocity field $ v\rvert_a$). Figure~\ref{figL96} shows the absolute and relative RMS reconstruction error for the Fourier modes as a function of the spectral truncation parameter $ \ell_A $, computed as $ \delta_{j,\ell_A} = \lVert \hat s_j - \hat s_{j,\ell_A} \rVert_{H_{A,\epsilon,N}} $ and $ \delta_{j,\ell_A} / \rVert \hat s_j \rVert_{H_{A,\epsilon,N}} $, respectively, where $ \hat s_{j,\ell_A} = \sum_{k=0}^{\ell_A-1} \langle \phi_k^{A,\epsilon,N}, \hat s_j \rangle_{H_{A,\epsilon,N}} $. Two features that stand out in these results are that (1) for both $ F_\text{L96}= 4 $ and $ F_\text{L96} = 5 $, the high-energy Fourier modes require fewer basis functions for the same amount of reconstruction error than the low-energy modes, and (2) as expected from the more chaotic nature of the dynamics in that regime, the Fourier modes for $ F_\text{L96}= 5 $ require a significantly larger number of basis functions for accurate reconstruction. The large number of basis functions required for accurate reconstruction of $ v\rvert_a $ becomes an especially pertinent issue associated with computational cost since eventually we will be working with the tensor product basis $ \phi^{\epsilon,N}_{ijk} $ on $M $, where, e.g.,  for $ \ell_A \sim 250 $ and Fourier resolution $ \ell_{X_1} = \ell_{X_2} \sim 32 $ the number of basis functions becomes $ \ell_A ( 2\ell_{X_1} + 1 )^2 \sim 10^6 $. As discussed in~\ref{appNumerical}, this issue can be mitigated in numerical implementations through the use of code that computes the action of the various operator matrices in our schemes on vectors without explicit formation of the matrices themselves and/or code that exploits any sparsity that these matrices might possess. Indeed, it is evident from~\eqref{eqWXL96Trunc} and~\eqref{eqWAT2} that the generator matrices associated with the L96-driven system are sparse, and the code used to obtain the results presented below takes advantage of that sparsity (see, in particular, \ref{appNumLeja}),   

\subsection{\label{secL96Coherent}Coherent spatiotemporal patterns}

We begin by considering the $ F_\text{L96} = 4 $ system. Figure~\ref{figL96Z} shows snapshots of the real parts of approximate Koopman eigenfunctions computed for this system through the variatonal eigenvalue problem in Definition~\ref{defEigDat}. In this case, we used the  spectral truncation and diffusion regularization parameters $ \ell_A = 250 $, $ \ell_{X_1} =\ell_{X_2} = 32 $ and $ \theta = 0.001 $. For this choice of spectral truncation parameters the  total number of basis functions in the tensor product basis is $ \ell = \ell_A ( 2 \ell_{X_1} + 1)^2 = \text{1,056,250} $. Note that the number of Fourier basis functions in each dimension on $  X $, $ 2 \ell_X + 1 = 65 $, is greater than the number $ 2 J + 1 = 41 $ of L96 modes; this is important since the generator of the dynamics on $  M $ can  generate smaller scales on observables than those present in the Fourier coefficients on the velocity field $ v\rvert_a $. The dynamical evolution of the numerical Koopman eigenfunctions in Fig.~\ref{figL96Z} can also be visualized more directly as a video in Movie~\href{http://cims.nyu.edu/~dimitris/files/tracers/movie10.mp4}{10}. Table~\ref{tableZL96} lists the corresponding eigenvalues and Dirichlet energies. 

\begin{figure}
  \centering\includegraphics[width=.8\linewidth]{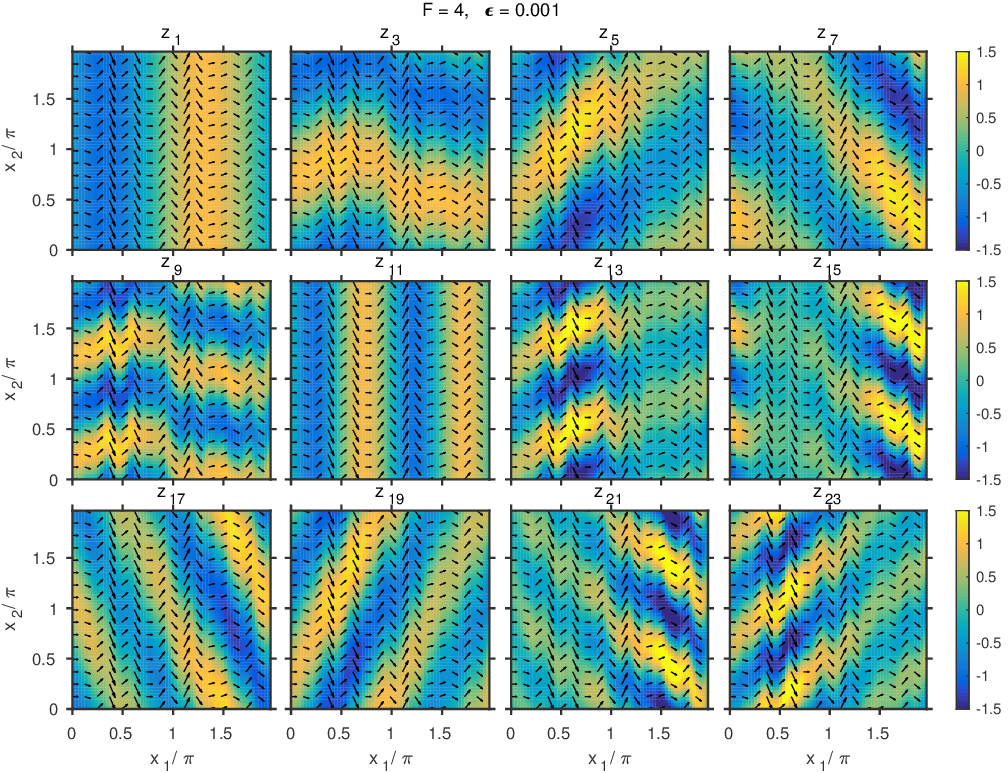}\\
  \includegraphics[width=.8\linewidth]{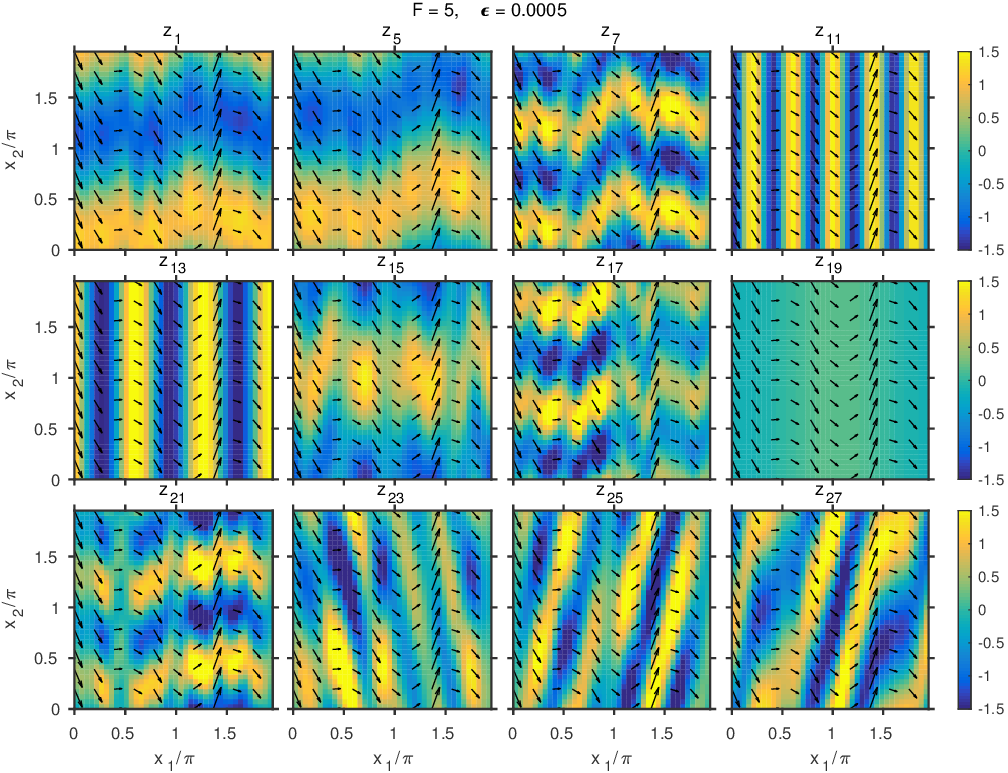}
  \caption{\label{figL96Z}Snapshots of the real parts of representative numerical Koopman eigenfunctions (colors) and the velocity field $ v\rvert_a $ (arrows) for the L96-driven flows with $ F_\text{L96} = 4 $ and 5. The numerical eigenfunctions were computed with the diffusion regularization parameters $ \theta = 10^{-3} $ and $ 5 \times 10^{-4} $, respectively. In both cases, the data-driven Dirichlet energies were computed via option~3 in~\eqref{eqEtaEpsN}.}
\end{figure}

\begin{table}
  \centering\small
  \caption{\label{tableZL96}Eigenvalues $ \lambda_k $ and Dirichlet energies $ E_k $ of the numerical Koopman eigenfunctions $ z_k $ for the L96-driven flows shown in Fig.~\ref{figL96Z}}
  \begin{tabular*}{\linewidth}{@{\extracolsep{\fill}}llllll}
    \hline
    \multicolumn{3}{c}{$F_\text{L96}=4 $, $\theta = 0.001 $} & \multicolumn{3}{c}{$F_\text{L96} = 5 $, $\theta=0.0005 $}\\
    \cline{1-3} \cline{4-6}
    & \multicolumn{1}{c}{$\lambda_k$} & \multicolumn{1}{c}{$E_k$} & & \multicolumn{1}{c}{$\lambda_k$} & \multicolumn{1}{c}{$E_k$}\\
    $ z_{1} $ & $ -5.67e-03 + 1.44e-02 \,\ii $ & $ 5.60e+00 $  &  $ z_{1} $ & $ -5.94e-03 + 1.51e-04 \,\ii $ & $ 1.18e+01 $ \\ 
    $ z_{3} $ & $ -7.21e-03 - 2.08e-04 \,\ii $ & $ 7.18e+00 $  &  $ z_{5} $ & $ -1.63e-02 - 1.75e-17 \,\ii $ & $ 1.98e+01 $ \\ 
    $ z_{5} $ & $ -1.20e-02 + 1.62e-02 \,\ii $ & $ 1.19e+01 $  &  $ z_{7} $ & $ -2.84e-02 - 4.51e-04 \,\ii $ & $ 3.17e+01 $ \\ 
    $ z_{7} $ & $ -1.22e-02 - 1.67e-02 \,\ii $ & $ 1.21e+01 $  &  $ z_{11} $ & $ -3.44e-02 - 5.87e-03 \,\ii $ & $ 5.69e+01 $ \\ 
    $ z_{9} $ & $ -2.84e-02 - 1.71e-03 \,\ii $ & $ 2.83e+01 $  &  $ z_{13} $ & $ -3.58e-02 - 3.18e-03 \,\ii $ & $ 6.76e+01 $ \\ 
    $ z_{11} $ & $ -2.97e-02 - 2.55e-02 \,\ii $ & $ 2.96e+01 $  &  $ z_{15} $ & $ -3.58e-02 - 1.19e-02 \,\ii $ & $ 7.06e+01 $ \\ 
    $ z_{13} $ & $ -3.08e-02 + 2.02e-02 \,\ii $ & $ 3.07e+01 $  &  $ z_{17} $ & $ -3.72e-02 + 9.11e-04 \,\ii $ & $ 7.19e+01 $ \\ 
    $ z_{15} $ & $ -3.32e-02 + 2.56e-02 \,\ii $ & $ 3.31e+01 $  &  $ z_{19} $ & $ -3.76e-02 + 2.35e-03 \,\ii $ & $ 7.20e+01 $ \\ 
    $ z_{17} $ & $ -3.42e-02 - 2.51e-02 \,\ii $ & $ 3.40e+01 $  &  $ z_{21} $ & $ -3.77e-02 - 1.15e-02 \,\ii $ & $ 7.41e+01 $ \\ 
    $ z_{19} $ & $ -3.45e-02 - 2.60e-02 \,\ii $ & $ 3.44e+01 $  &  $ z_{23} $ & $ -3.84e-02 + 1.02e-02 \,\ii $ & $ 7.49e+01 $ \\ 
    $ z_{21} $ & $ -4.68e-02 - 1.83e-02 \,\ii $ & $ 4.67e+01 $  &  $ z_{25} $ & $ -3.83e-02 - 1.81e-02 \,\ii $ & $ 7.53e+01 $ \\ 
    $ z_{23} $ & $ -5.16e-02 - 2.85e-02 \,\ii $ & $ 5.15e+01 $  &  $ z_{27} $ & $ -3.89e-02 + 2.31e-02 \,\ii $ & $ 7.70e+01 $ \\ 
    \hline
  \end{tabular*}
\end{table}

In general, the numerical eigenfunctions for $ F_\text{L96}= 4 $ have the structure of propagating patterns, advected along the positive $ x_1 $ direction by the spatially uniform cross-sweep flow $ v^{(1)} $ in~\eqref{eqL96V}. Qualitatively, we can identify four distinct classes of eigenfunctions, namely (1) eigenfunctions such as $ z_1 $ and $ z_{11} $ that do not depend on the $x_2 $ coordinate; (2) eigenfunctions such as $ z_3 $ and $ z_5 $ that depend on $ x_2 $ but are approximately constant on the instantaneous streamlines of the flow in $ X $; (3) eigenfunctions such as $ z_5 $, $z_7 $, $z_{17} $, and $ z_{19} $ that depend on both $ x_1 $ and $ x_2 $ and whose level sets form approximately straight lines in the $ (x_1, x_2) $ plane; (4) eigenfunctions such as $ z_{13} $, $ z_{15} $, $ z_{21} $, and $ z_{23} $ that have the structure of propagating wavepackets. Note that eigenfunctions of class~2 have small values of $ \Imag \lambda_k $ ($O(10^{-4})$ and $O(10^{-3}) $ for $ z_3 $ and $ z_5 $, respectively) indicating that the level sets of these functions are nearly invariant under the skew-product flow. On the other hand, the rest of the $ F_\text{L96}=4 $ eigenfunctions shown here have $ \Imag\lambda_k = O(10^{-2} )$, indicating that they vary on the tracers more strongly, albeit still at a slow timescale compared to the natural timescale of the flow.  Unlike the moving-vortex example in Section~\ref{secMovingVortex} (see also \ref{appMoving}), in the case of the L96-driven system we do not have a detailed justification of the properties of the numerical Koopman eigenfunctions. Nevertheless, at least in some cases, we can provide partial justification provided that certain simplifying assumptions are made. 

First, consider the eigenfunctions of class~1. Since the velocity field component $ v^{(1)}(a) $ is uniform on $ X $ for all $ a \in A $, the vector field $ w $ is projectible under the map $ \pi : M \mapsto A \times \mathbb{ T }^1 $ such that $ \pi( a, x_1, x_2 ) = ( a, x_1 ) $. In particular, we have $ \pi_* w = u + \hat s_0 \partial_1 $. Assume now that (1) $ \hat s_0 $ is constant, and (2) the generator $ \tilde u $ on $ A$ has eigenfunctions. The validity of assumption~1 can be tested from the time series of $ \hat s_0 $ in the training data; in the $ F_\text{L96} = 4 $ regime we find the mean and standard deviation of $ \hat s_0 $ to be 1.46 and 0.038, respectively, indicating that $ \hat s_0 $ can indeed be treated as a constant to a good approximation. Assumption~2 is supported from the quasiperiodic nature of the dynamics in the $ F_\text{L9} = 4 $ regime (see Fig.~\ref{figL96}). Under these assumptions, the product $ \tilde z = z_A \phi_k $ of any eigenfunction $ z_A $ of $ u $ at frequency $ \omega $ with any Fourier function $ \phi_k $ at wavenumber $ k \in \mathbb{ Z } $ will be an eigenfunction of $ \pi_* w $ at eigenvalue $ \ii( \omega + k \hat s_0 ) $. Moreover, the pullback $ z = \tilde z \circ \pi $ of that eigenfunction onto $ M $ will be an eigenfunction of $ w $ at the same eigenvalue having constant values at the fibers of $ \pi $ (i.e., constant $ x_2 $ hypersurfaces), as observed in Fig.~\ref{figL96Z} and Movie~\href{http://cims.nyu.edu/~dimitris/files/tracers/movie10.mp4}{10}.

Next, to interpret the eigenfunctions of class~2, observe that if $ v^{(1)}(  a ) $ is nonzero (as is the case for all $ a $ in the $ F_\text{L96} = 4 $ regime), the instantaneous streamlines of time-dependent flows of the class~\eqref{eqL96V} are closed. In particular, since $ v^{(1)}( a ) $ is non-vanishing and independent of $ x_2 $, the streamlines at fixed $ a $ are solutions of the ODE
\begin{displaymath}
  \frac{ dx_2 }{ dx_1 } = \frac{ v^{(2)}( a, x_1 ) }{ v^{(1)}( a ) },
\end{displaymath}   
and since $ v^{(2)}( a, x_1 ) $ has no wavenumber zero Fourier component, integration of the right-hand side shows that $ x_2( 0 ) = x_2( 2 \pi ) $, which implies that the streamlines are closed. Thus, as done in the moving-vortex example (see~\ref{appMoving}), we can consider a projection onto a circle $ Y \subset  X $ transverse to the streamlines. In particular, we can define $ \pi : M \mapsto A \times Y $ such that $ \pi( a, x_1, x_2 ) = ( a, y ) $ where $ y $ is the intercept of the streamline at state $ a \in A $ passing through the point $ ( x_1, x_2 ) \in X $ with the circle $ x_1 = 0 $. Since $ v \rvert_a $ is everywhere tangent to the streamlines by definition, the $ w^X $ component of the generator vanishes under $ \pi_* $, and we have $ \pi_* w = \pi_* w^A = u \oplus 0 $. Thus, every $L^2 $ function $ \tilde z $ on $ A \times Y $ which is constant on $ A$ is an eigenfunction of $ \pi_* w $ at eigenvalue zero, and that function pulls back to an eigenfunction $ z = \tilde z \circ \pi $ of $ \tilde w $, also at eigenvalue zero, which is constant on the streamlines of the L96-driven flow. The numerical eigenfunctions of class~2 described above are consistent with this behavior. 

While we do not have a justification of the properties of class~3 and~4 eigenfunctions, it is worthwhile noting that the spatially localized character of the latter suggests that they are associated with the continuous spectrum of the generator (cf.\ Section~\ref{secMovingVortex} and~\ref{appMoving}). 

We now consider the regime $ F_\text{L96} = 5 $. For the reasons discussed in~\ref{appKoopEig}, our current code and available computational resources did not allow us to work with as high values of the spectral truncation parameter $ \ell_A $ as the results in Fig.~\ref{figL96Rec} indicate are required for accurate reconstruction of the state-dependent velocity field $ v\rvert_a $. Thus, aiming for a compromise between resolution on $ A $ and $ X$, we decreased $ \ell_{X_1} $ and $ \ell_{X_2} $ to 20, which allowed us to increase $ \ell_A $ to 500. With these parameter choices, the number of Fourier modes per dimension of $ X $ is equal to the number of degrees of freedom $ 2 J + 1 $ of the L96 system; that is, we are fully resolving the velocity field but not any smaller scales that could develop as a result of the action of that velocity field on tracers. Moreover, the total number of basis functions is $ \ell = \ell_A ( 2 \ell_X + 1 )^2 = \text{840,500} $. Note that even though $ \ell $ in this case is smaller than in the $ F_\text{L96} = 4 $ experiments, the memory requirements are actually higher due to the presence of $ \ell_A \times \ell_A $-sized blocks of nonzero elements in the generator matrix $ \boldsymbol{ A } $; see~\eqref{eqWXL96Trunc}.            

Figure~\ref{figL96Z}, Movie~\href{http://cims.nyu.edu/~dimitris/files/tracers/movie11.mp4}{11} and Table~\ref{tableZL96} show representative Koopman eigenfunctions and eigenvalues computed with these parameters for the $ F_\text{L96} = 5 $ regime. The compromises we had to make on spectral resolution somewhat diminish our confidence in these results, but nevertheless some of he qualitative features identified in the $ F_\text{L96} = 4 $ experiments are also present here. In particular, the leading $ F_\text{L96} = 5 $ eigenfunction, $ z_1 $, exhibits the characteristics of class~2 eigenfunctions discussed above; i.e., its level sets tend to follow the instantaneous streamlines of the flow (though with discrepancies at times, likely due to spectral truncation) and the corresponding eigenfrequency is small, $ \Imag \lambda_1 = O( 10^{-4} ) $. 

Eigenfunctions broadly resembling the class~1 eigenfunctions at $ F_\text{L96}= 4 $ are also present at $ F_\text{L96} = 5 $ (e.g., $ z_{11} $, $ z_{13}$, and $ z_{19}$), although they differ notably from the previous regime in that they display significant amplitude modulations. This is actually not too surprising, for our interpretation of class~1 eigenfunctions relied on the simplifying assumptions that the cross-sweep flow component $ v^{(1)} $ is constant and that the generator of the dynamics on $ A $ has eigenfunctions. At $ F_\text{L96} = 5$, the mean and standard deviation of $ \hat s_0 $ are 1.67 and 0.17, respectively; i.e., the strength of fluctuations relative to the mean is about an order of magnitude stronger than in the $ F_\text{L96} = 4 $ case. Moreover, the assumption that $ \tilde u $ has eigenfunctions becomes more questionable at $ F_\text{L96} = 5 $. Thus, while we cannot exclude the possibility that the amplitude fluctuations in $ z_{11} $, $ z_{13} $, and $ z_{19} $ are due to spectral truncation, our results are consistent with these functions being associated with the continuous spectrum of the $ F_\text{L96} = 5 $ system, exhibiting amplitude modulations as a result of localization on $ A $.    
        
\subsection{Prediction of observables and probability densities}

We now examine prediction of observables and probability densities using the methods of Section~\ref{secDataDrivenPred}. In what follows, we discuss a suite of numerical experiments that are analogous to those performed for the vortex flow models in Section~\ref{secExamples}; namely, we consider prediction of observables representing the position of tracers in $ X $ and probability densities with a circular Gaussian structure at forecast initialization. 

Since the data-driven basis function $ \phi_0^{A,\epsilon,N} $ is constant by construction (see Section~\ref{secKernelOp}), we can define $ f_1 = \phi^{\epsilon,N}_{010} = \phi^X_{10} $ and $ f_2 = \phi^{\epsilon,N}_{001} = \phi^X_{01} $ in direct analogy with the corresponding observables in~\eqref{eqF12}, and use the semigroup action $ \tilde S_{\epsilon,N,\tau,\ell,t} f_j( a, x )  $ from~\eqref{eqSemigroupPhi} to approximate the Koopman group action $ W_t f_j( a, x ) $ (recall that the latter gives the position at lead time $ t $ of a tracer released from point $ x \in X $ when the L96 state is $ a $). Figures~\ref{figL96X_F4} and~\ref{figL96X} show  the evolution of tracer positions for the systems with $ F_\text{L96} = 4 $ and $ F_\text{L96}=5 $, respectively, obtained via the operator-theoretic approach of Section~\ref{secDataDrivenPred} and explicit calculation by numerical ODE integration of the time-dependent flow in~\eqref{eqL96V}. These results were computed for initial L96 states in the training data; that is, in the operator-theoretic and ODE-based experiments we compute $ \tilde S_{\epsilon,N,\tau,\ell,t} f_j( a, x ) $ and $ \Psi_t(a, x ) $ for starting states $ a $ in the training data, $ \{ a_n \}_{n=0}^{N-1} $. This avoids out-of-sample extension errors that would occur if $ a \notin \{ a_n \}_{n=0}^{N-1} $, so that the model error generated by the operator theoretic model is entirely due to diffusion regularization and spectral truncation. The evolution depicted in Figs.~\ref{figL96X_F4} and~\ref{figL96X} is also more directly visualized as videos in Movies~\href{http://cims.nyu.edu/~dimitris/files/tracers/movie12.mp4}{12} and~\href{http://cims.nyu.edu/~dimitris/files/tracers/movie13.mp4}{13}, respectively. In the $ F_\text{L96} = 4 $ case, we show results for 1,024 timesteps, i.e., forecast times up to 10.24 natural time units. In the $ F_\text{L96} = 5 $ case, we were only able to compute for 256 time steps with the computational resources available to us.   
   
\begin{figure}
  \centering
  \includegraphics[width=.4\linewidth]{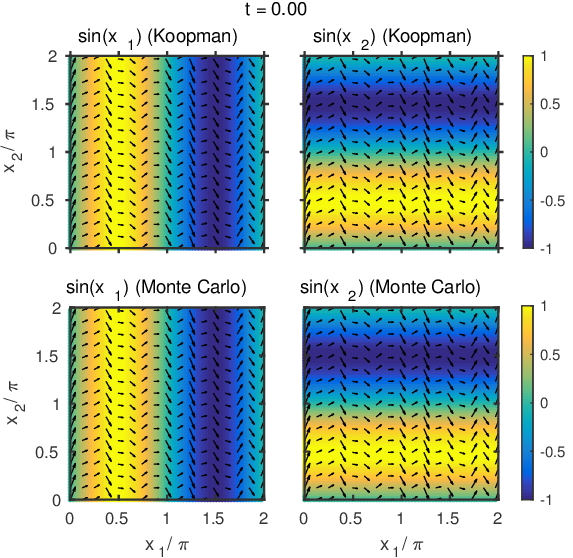}    \includegraphics[width=.4\linewidth]{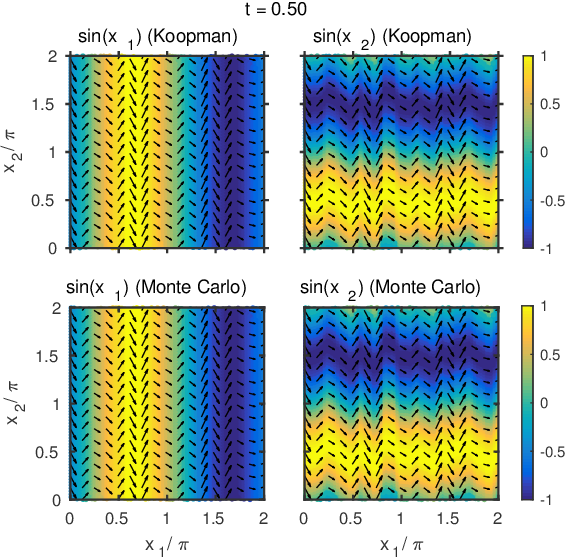}\\ 
  \includegraphics[width=.4\linewidth]{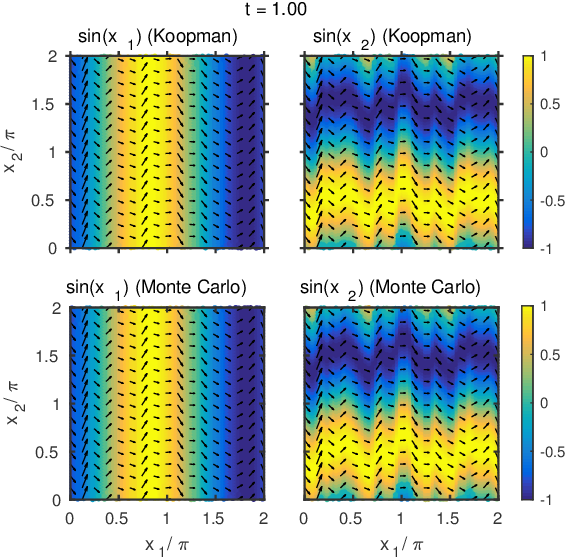}    \includegraphics[width=.4\linewidth]{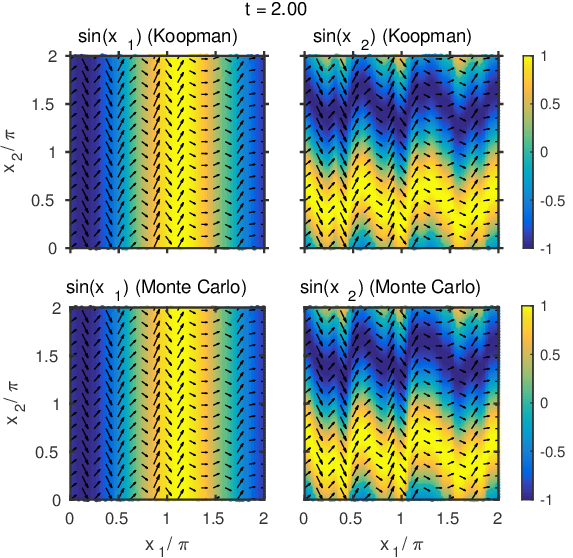}\\
\includegraphics[width=.4\linewidth]{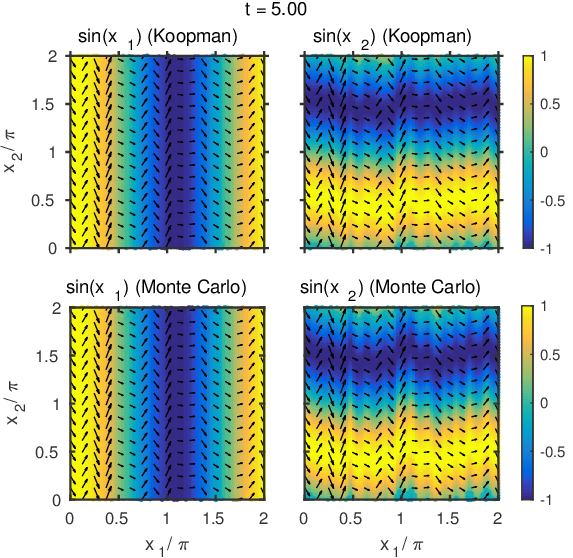} \   \includegraphics[width=.4\linewidth]{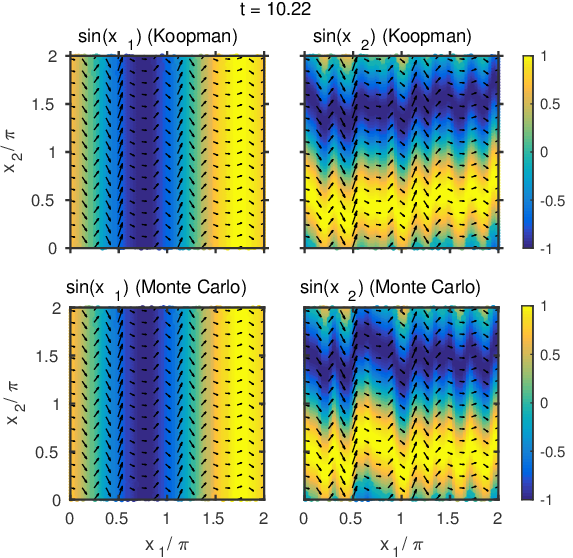}

  \caption{\label{figL96X_F4}Snapshots of the evolution of the positions  $ ( x_1( t, a ), x_2( t,a ) ) \in X $ for an ensemble of Lagrangian tracers under the L96-driven flow with $ F_\text{L96} = 4$, computed via the data-driven approach of Section~\ref{secDataDrivenPred} and via explicit ODE integration of the full model for tracer advection. The initial state of the L96 system is state $ a_0 $ in the training dataset. The initial positions $ ( x_1, x_2 ) $ of the tracers are on a uniform square grid of spacing $ 2 \pi / ( 2 \ell_{X_1} + 1 ) = 2 \pi / ( 2 \ell_{X_2} + 1 ) = 2 \pi  / 65 $ along both the $ x_1 $ and $ x_2 $ coordinates. Colors show the sines of the initial $ x_1 $ and $ x_2 $ coordinates as an aide for visualizing the displacement of the tracers. Arrows show the velocity field. The diffusion regularization parameter is $ \theta = 10^{-5} $, and the spectral truncation parameter $\ell_A$ is equal to 250.} 
\end{figure}
  
\begin{figure}
  \centering
  \includegraphics[width=.4\linewidth]{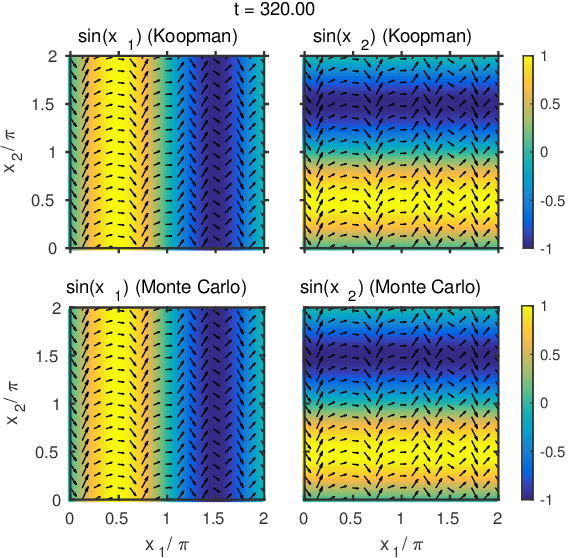}    \includegraphics[width=.4\linewidth]{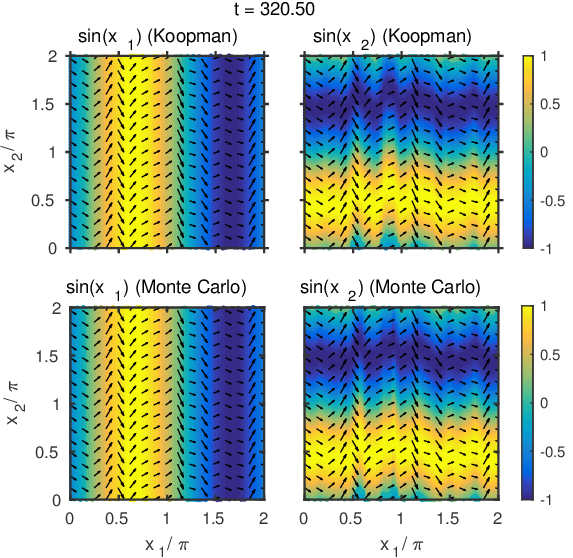}\\ 
  \includegraphics[width=.4\linewidth]{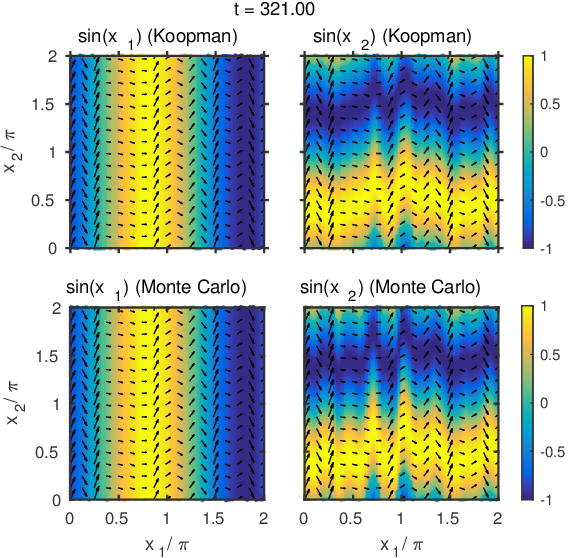}    \includegraphics[width=.4\linewidth]{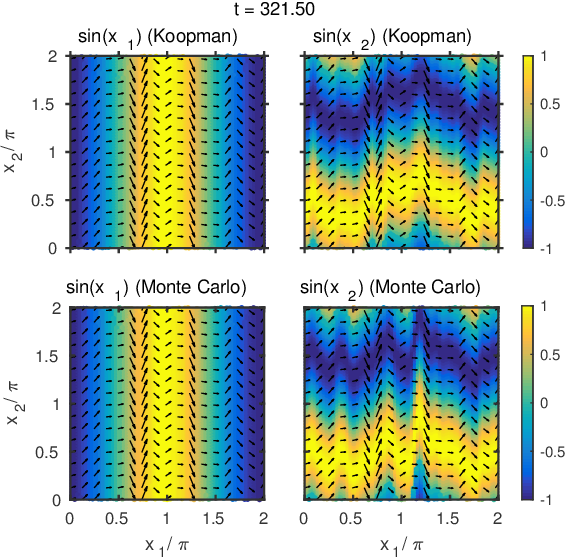}\\
\includegraphics[width=.4\linewidth]{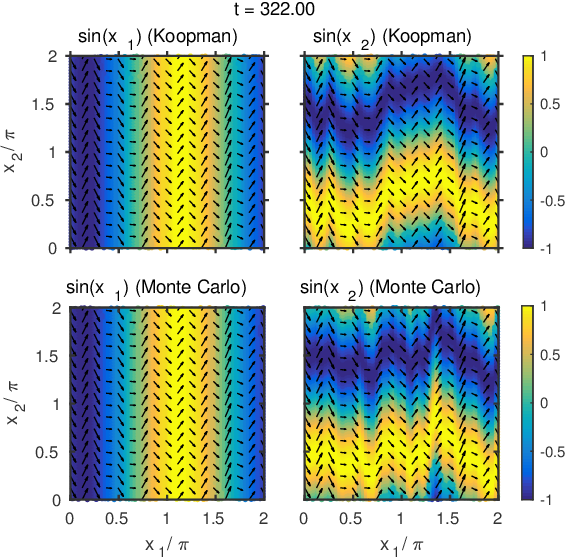} \   \includegraphics[width=.4\linewidth]{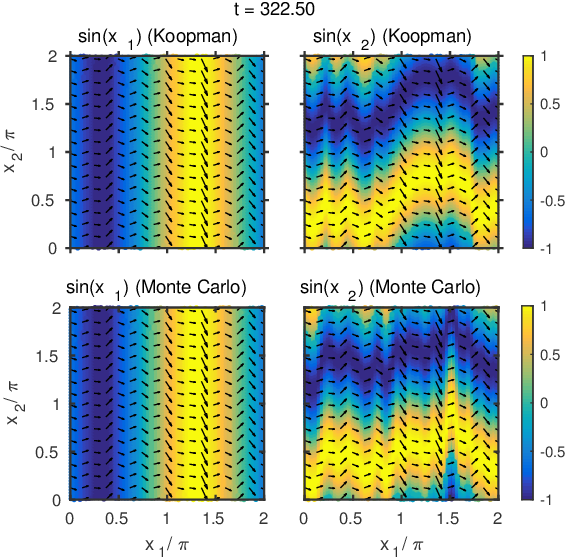}

  \caption{\label{figL96X}As in Fig.~\ref{figL96X_F4}, but for the L96 system with $ F_\text{L96} = 5 $. In this case, the initial state of the L96 system is state $ a_\text{32,000} $ in the training dataset (the corresponding starting time is $ t = 320$). The diffusion regularization parameter is $ \theta = 10^{-5} $, and the spectral truncation parameter $\ell_A$ is equal to 750.} 
\end{figure}

As expected from the structure of the velocity field in~\eqref{eqL96V}, in both the $ F_\text{L96}= 4 $ and $ F_\text{L96} = 5 $ cases the system generates a spatially uniform (but time-dependent) tracer flow along the $ x_1 $ direction and spatially nonuniform displacements along the $ x_2 $ direction, creating traveling-wave-like disturbances in observables $ f_1 $ and $ f_2$. As expected, the spatiotemporal structure of these disturbances is more complex in the $ F_\text{L96} = 5 $ case, but even at $ F_\text{L96} = 4 $ the system generates patterns at a broader range of spatial scales than that suggested by the dominant Fourier modes of the velocity field (see Fig.~\ref{figL96Rec}). In general, over the examined time interval, the results from the operator-theoretic model and the the Monte Carlo simulation are in very good agreement for $ F_\text{L96} = 4 $, with only minor discrepancies becoming noticeable towards the end of the interval ($ t \simeq 10$). The results from the two methods are also in reasonably good agreement in the $ F_\text{L96} =5 $ experiments, although there the operator theoretic model fails to capture some of the small-scale spatial patterns developed by the true model, even at times as early as $ t = 2$. These discrepancies are likely due to both spectral truncation (notice  in Fig.~\ref{figL96Rec} that even at $ \ell_A = 750 $ there are significant reconstruction errors in the velocity field) and chaotic dynamics of the L96 system. Nevertheless, the results in Fig.~\ref{figL96X} and Movie~\href{http://cims.nyu.edu/~dimitris/files/tracers/movie13.mp4}{13} indicate that, at least over short to moderate lead times, the operator-theoretic model is able to perform useful forecasts of tracer positions.

Next, we consider prediction of probability densities. As stated above, we set the initial density $ \rho_X $ on $ X $ to a circular Gaussian, given by the same formula as~\eqref{eqRhoAX}. In the experiments that follow, we work throughout with the location and concentration parameters $ ( \bar x_1, \bar x_2) = ( \pi, \pi ) $ and $ \kappa = 3 $, respectively. To construct an initial density $ \rho_A $ on $ A $ we fix an arbitrary state $ a \in B(\alpha) $, and define $ \rho_A( a_n ) = p_{\epsilon,N}( a, a_n ) $, where $ p_{\epsilon,N} $ is the Markov kernel from~\eqref{eqPOpKer}. In what follows, we choose $ a $ to be a state ``on the attractor''  but outside of the training dataset $ \{ a_n \}_{n=0}^{N-1} $. In particular, we compute the solution $ s( t ) $ of the L96 system beyond the latest time $ t = ( N - 1 ) \tau $ in the training dataset, and set $  a = s(2 N \tau ) $. Taking the product $ \rho(a_n, x ) = \rho_A(a_n) \rho_X( x ) $ then leads to our initial density in $ H_{\epsilon,N} $. Note that being an element of $ H_{\epsilon,N} $, $ \rho( a, x )$ is unspecified for $ a \notin \{ a_n \} $, but its expansion coefficients in the $ \phi_k^{\epsilon,N} $ basis depend only on the values $ \rho(a_n,x)$. Given these expansion coefficients, we approximate the evolution of the initial density $ \rho $ through the semigroup action $ \rho_t = \tilde S^*_{\epsilon,N,\tau,\ell,t} \rho $ in~\eqref{eqAdjSemigroupPhi}. As in Section~\ref{secExamples}, we visualize the evolution of $ \rho_t $ through the corresponding  marginal densities $ \sigma_t $, $ \sigma_{1,t} $, and $ \sigma_{2,t} $.    

Figures~\ref{figL96_F4_Rho} and~\ref{figL96Rho} show the evolution of these marginal densities under the L96-driven flows with $ F_\text{L96} = 4 $ and $ F_\text{L96} = 5$, respectively, computed using the operator-theoretic approach of Section~\ref{secDataDrivenPred} and explicit Monte Carlo simulation for ensembles of 300,000 tracers drawn independently from $ \rho $. The spectral truncation parameters and simulation time intervals are identical to the corresponding values used for prediction of the $ f_1 $ and $ f_2 $  observables for $ F_\text{L96} = 4 $ and $ F_\text{L96} = 5 $ discussed above. A more direct visualization of the evolution in Figs.~\ref{figL96_F4_Rho} and~\ref{figL96Rho} is provided in Movies~\href{http://cims.nyu.edu/~dimitris/files/tracers/movie14.mp4}{14} and~\href{http://cims.nyu.edu/~dimitris/files/tracers/movie15.mp4}{15}, respectively. 
 
\begin{figure}
  \centering\includegraphics[width=.38\linewidth]{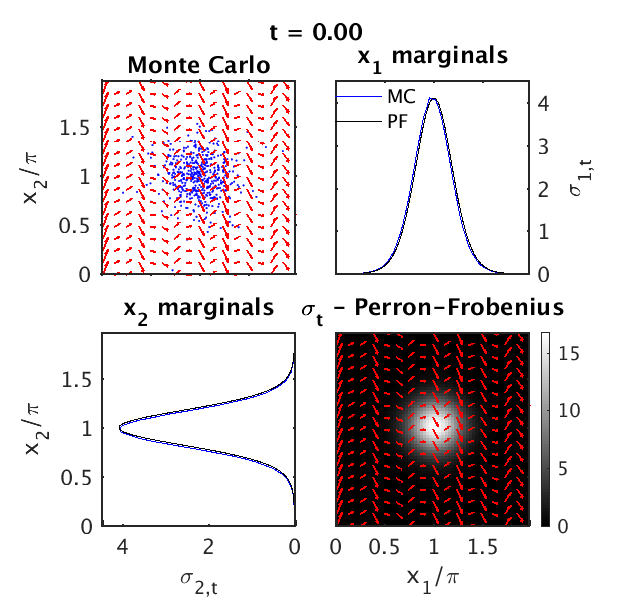}   \includegraphics[width=.38\linewidth]{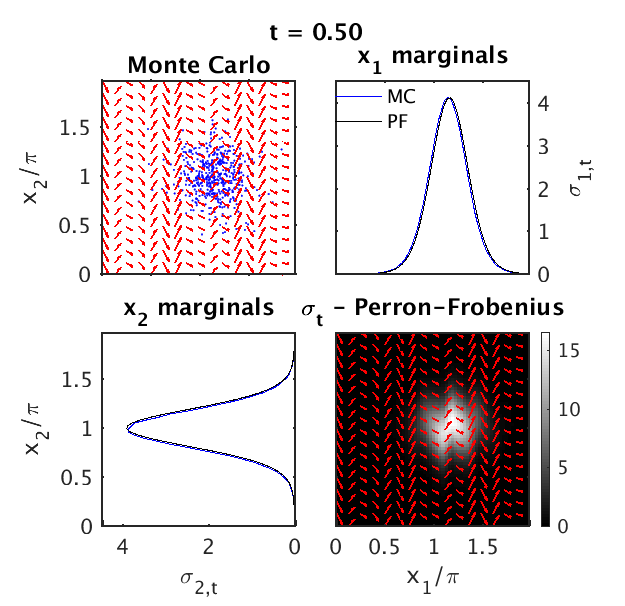} \\
\includegraphics[width=.38\linewidth]{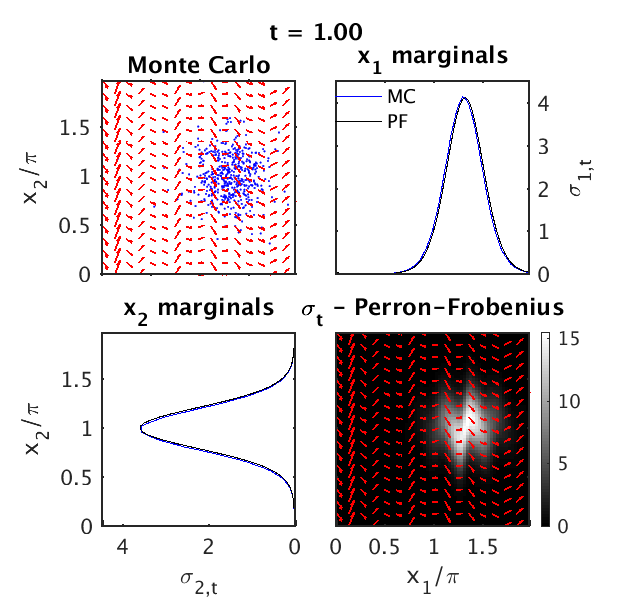}  \includegraphics[width=.38\linewidth]{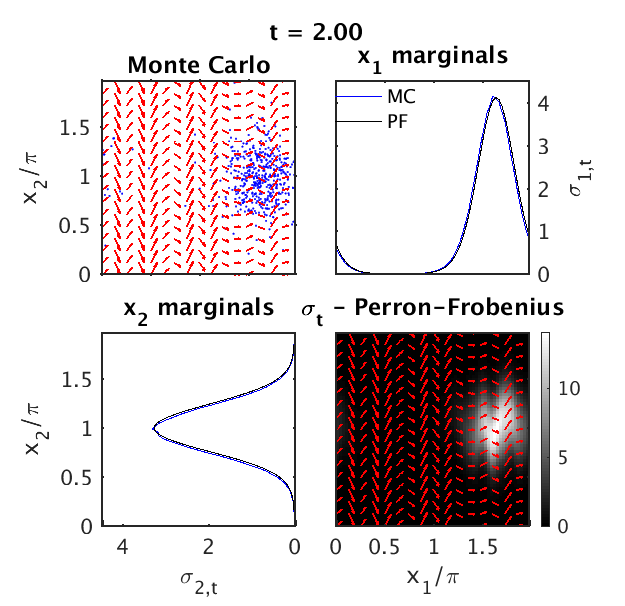} \\
\includegraphics[width=.38\linewidth]{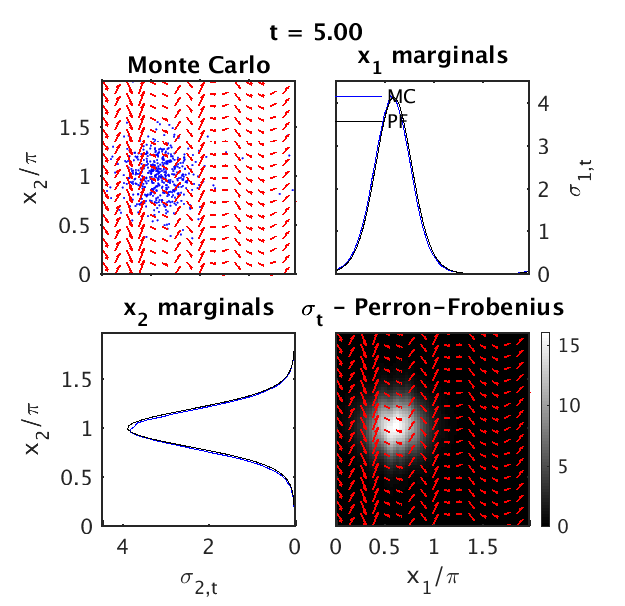}  \includegraphics[width=.38\linewidth]{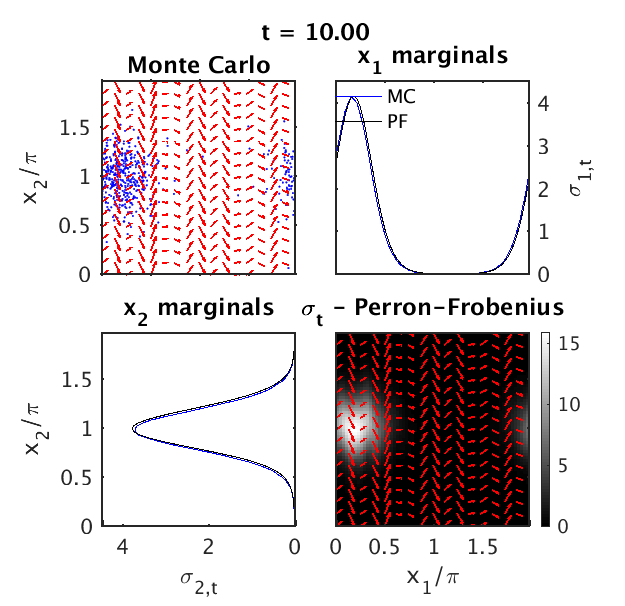}
  \caption{\label{figL96_F4_Rho}Snapshots of the temporal evolution of the marginal density $ \sigma_t $ under the L96-driven flow with $ F_\text{L96} = 4 $  through the data-driven scheme in Section~\ref{secDataDrivenPred} (grayscale colors) with $ \theta = 10^{-5} $ and $ \ell_A = 250 $, and a Monte Carlo simulation (blue dots) based on the full model for Lagrangian advection under the time-dependent flow in~\eqref{eqL96V}. Also shown are the one-dimensional marginal densities $ \sigma_{1,t} $ and $\sigma_{2,t} $ determined from the data-driven model (black lines) and the Monte Carlo simulation (blue lines). The density estimates from the Monte Carlo ensemble were determined in the same manner as in Fig.~\ref{figMovingRho}. For reference, the red arrows show the time-dependent velocity field evaluated for the L96 state $ s( t ) $ such that $ s( 0 ) = a $ is the state used to construct the initial density $ \rho_A $.}
\end{figure}

\begin{figure}
  \centering\includegraphics[width=.42\linewidth]{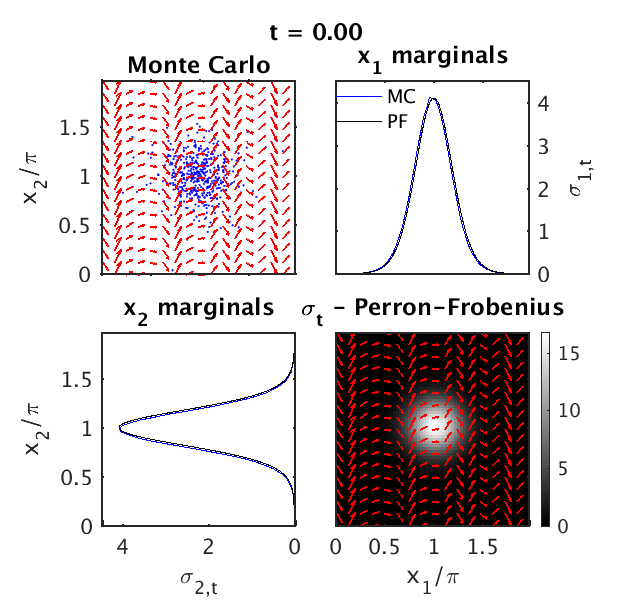}   \includegraphics[width=.42\linewidth]{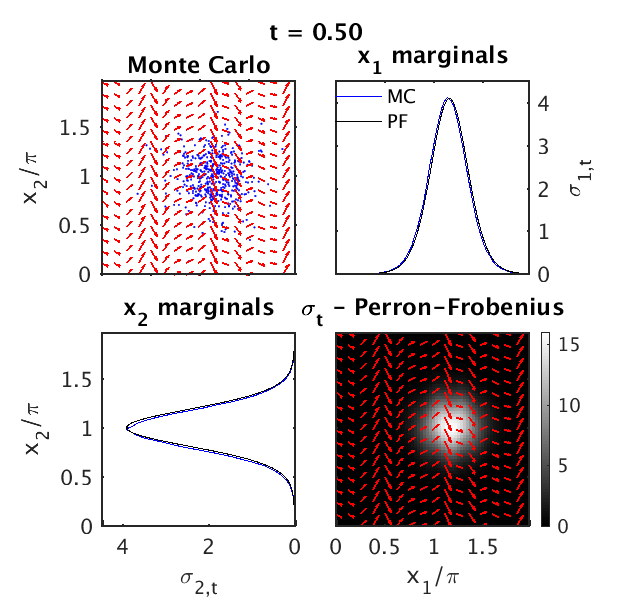} \\
\includegraphics[width=.42\linewidth]{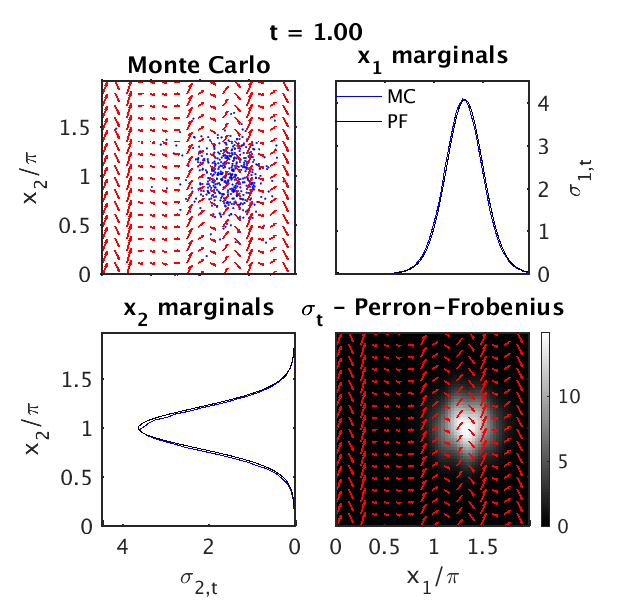}  \includegraphics[width=.42\linewidth]{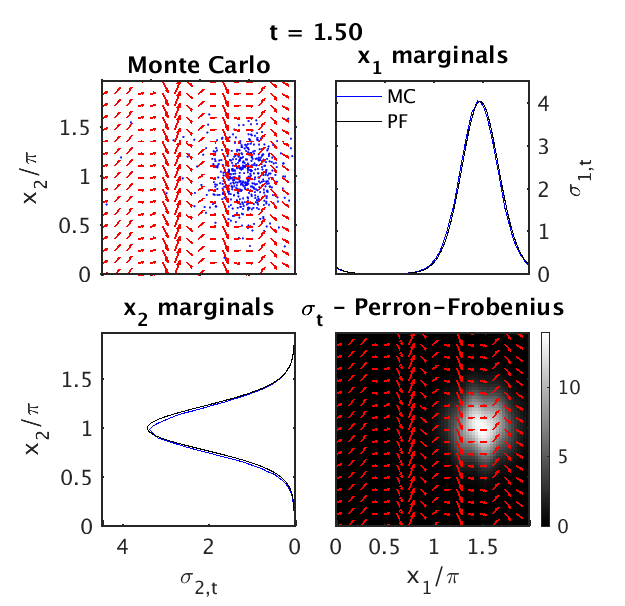} \\
\includegraphics[width=.42\linewidth]{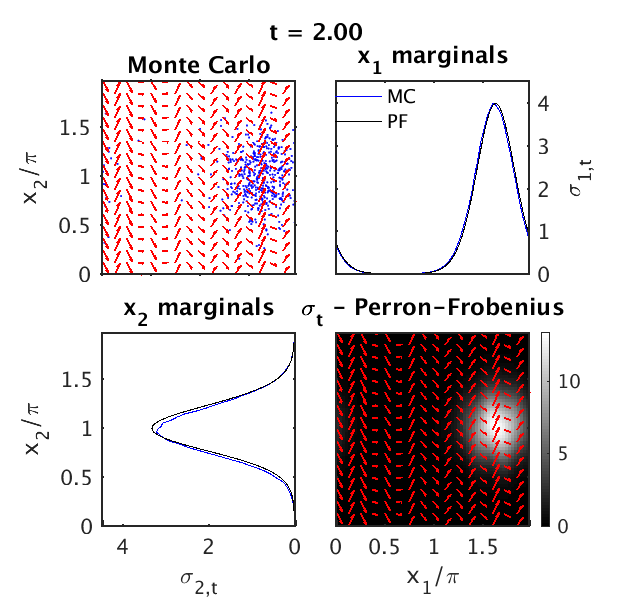}  \includegraphics[width=.42\linewidth]{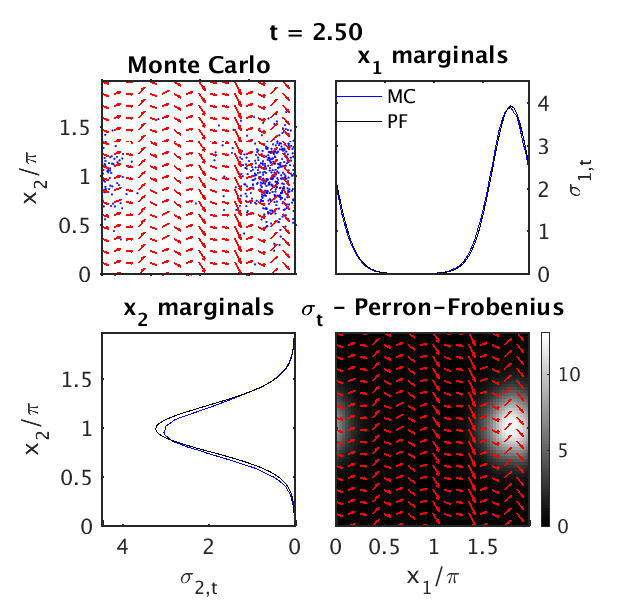}
  \caption{\label{figL96Rho}As in Fig.~\ref{figL96_F4_Rho}, but for the flow with $ F_\text{L96} = 5 $ and the spectral truncation parameter $ \ell_A = 750 $.}
\end{figure}

In both the $ F_\text{L96}= 4 $ and $ F_\text{L96} = 5 $ cases, the initially globular density $ \sigma_t $ is distorted by the velocity field fluctuations in the $ x_2 $ direction while simultaneously being advected along the positive $ x_1 $ direction by the spatially uniform cross-sweep flow. The spatial distortions of the initial density are more pronounced at $ F_\text{L96} = 4 $, likely due to the more pronounced spectral peaks of the velocity field in that regime. In both cases, the 1D marginal density $ \sigma_{1,t} $ is advected along the $ x_1 $ direction with little change of shape, whereas $ \sigma_{2,t} $ undergoes oscillations in its variance due to the velocity field fluctuations in the $ x_2 $ direction. In general, the agreement between the operator-theoretic and Monte Carlo results is very good for $ F_\text{L96} = 4 $. At $ F_\text{L96} = 5 $, the agreement is also reasonably good (at least over the examined lead times), but biases in $ \sigma_{2,t} $ start becoming noticeable at $ t = 2$; as with prediction of $ f_2 $, this is at least partly caused by spectral truncation.    

\section{\label{secConclusions}Concluding remarks}

In this paper, we have developed a family of techniques for coherent pattern identification and statistical prediction in measure-preserving, skew-product dynamical systems, with a particular emphasis on systems of Lagrangian tracers evolving under time-dependent fluid flows driven by ergodic dynamical systems. Central to these techniques is a data-driven representation of the generator of the Koopman and Perron-Frobenius groups of operators on an extended state space (as recently done in \cite{FroylandKoltai17} in the case of periodic flows), given by the product of the state space $ A $  of the dynamical system driving the flow and the physical domain $ X $ in which the flow takes place. In particular, our approach is to represent the generator of the associated skew-product system (governing the joint evolution of the velocity field and the tracers) in a smooth basis of the $ L^2 $ space of $ A \times X $, constructed as a tensor product of bases for the $ L^2 $ spaces on  $ A $ and $ X $. Importantly, the basis for $ L^2( A ) $ is built from Laplace-Beltrami eigenfunctions learned from time-ordered velocity field snapshots using kernel algorithms for machine learning \cite{CoifmanLafon06,BerryHarlim16}, requiring no a priori knowledge of the geometry or the dynamics on $ A $, or availability of explicit tracer trajectories. 

With the availability of this basis, we constructed a Galerkin method for the eigenvalue problem of a regularized generator for the skew-product system having only point spectrum \cite{FrankeEtAl10}, and used the eigenfunctions of this operator to define coherent spatiotemporal patterns. Following earlier work \cite{GiannakisEtAl15,Giannakis17}, our approach has been to select eigenfunctions with minimal roughness; the latter measured through a Dirichlet energy functional accessible to data-driven approximation through the same kernel algorithms as those used to build the basis. Taking the exponential of the generator and its adjoint, we also constructed representations of the Koopman and Perron-Frobenius operators governing the evolution of observables and probability densities of the skew-product system, and used these representations in model-free statistical prediction schemes. Here, we computed the action of the generator exponential on functions by means of Leja polynomial interpolation algorithms \cite{CaliariEtAl04,KandolfEtAl14}, used in exponential integrators for ODEs. 

We demonstrated the utility of these techniques in periodic flows with Gaussian streamfunctions on two-dimensional periodic domains and in aperiodic flows driven by Lorenz 96 systems \cite{QiMajda16} in weakly chaotic regimes. The periodic-flow examples where chosen so as to have mixed spectra and generator matrices computable in closed form, allowing us to assess the properties of these schemes in the absence of sampling errors. We saw that in these flows the numerically computed eigenfunctions of the generator provide useful notions of coherent spatiotemporal patterns, with properties consistent with those expected from theory (when such results are available), and in some cases providing numerical access to the continuous spectrum through patterns behaving as ``strange eigenmodes'' \cite{LiuHaller04,Pierrehumbert94}. We also saw that our methods for prediction of observables and probability densities perform comparably to explicit integration of the full model, including in their ability to reproduce highly non-Gaussian probability density functions. In the case of L96-driven flows, our numerical results provide evidence that our framework is also applicable when the state space of the driving system is not a smooth manifold, but rather a more general compact invariant set of a dynamical system. Here, a major challenge was that the L96 driving system can require large numbers of basis functions in order to accurately represent the generator, especially in chaotic regimes with many positive Lyapunov exponents. Nevertheless, we found that our approach yields reasonably good results, both in terms of coherent pattern detection and prediction of observables and densities.  

Two aspects of this work that warrant future study and potential improvement are (1) the way that diffusion regularization interacts with the skew-adjoint Koopman generator in the presence of continuous spectrum, and (2) the high computational cost associated with the tensor product basis for the Hilbert spaces $ H $ encountered in skew-product systems. The first issue was addressed in a specific case in \cite{Giannakis17,DasGiannakis17} where it was shown that through an appropriate use of delay-coordinate maps one can construct a data-driven diffusion operator the commutes with the Koopman generator on the discrete subspace of $ H $. However, that work did not address how objects such as strange eigenmodes behave when the diffusion operators and the Koopman generator do not commute (which, to our knowledge, is an open problem in PDE theory). Possible modifications of our approach to address the second issue would be to replace the tensor product basis with a basis constructed directly on the extended state space of the skew-product system, possibly by using explicit tracer trajectories. Alternatively, one could seek efficiency improvements by employing multi-resolution bases \cite{CoifmanMaggioni06,AllardEtAl12}, or through sparse representations of the generator \cite{BruntonEtAl16}. We plan to pursue these topics in future work.

\section*{Acknowledgments}

Dimitrios Giannakis acknowledges support by  ONR YIP grant N00014-16-1-2649, NSF grant DMS-1521775, and DARPA grant HR0011-16-C-0116. Suddhasattwa Das is supported as a postdoctoral research fellow from the first grant. The authors wish to thank Pierre Germain and Igor Mezi\'c for stimulating conversations. 
 
\appendix 

\section{\label{appTechnical}Technical results}

\subsection{\label{appPropL}Proof of Proposition~\ref{propL}}

  (i) That $ L_{\pm} $ are dissipative follows immediately from the skew-symmetry of $ w $ and positivity of $ \upDelta $. The latter imply that for any $ f \in C^\infty(M) $,  $ \langle f, w( f ) \rangle = 0 $ and $ \langle f, \upDelta f \rangle \geq 0 $,  leading to $ \Real \langle f, L_\pm f  \rangle = - \theta \langle f, \upDelta f \rangle \leq 0 $. It is a classical result \cite{Phillips59} that every densely-defined dissipative operator is closable and and its closure is dissipative. 

(ii) Let $ f $ and $ g $ be arbitrary functions in $ C^\infty( M )$. It follows by integration  by parts in conjunction with the skew-symmetry of $ w $ and the symmetry of $ \upDelta $ that $ \langle f, L_+ g \rangle = \langle L_- f, g \rangle $. Since this relationship holds for any $ g \in C^\infty( M ) $, we can conclude that $ f \in D( L_+^* ) $ and $ L_+^* f = L_- f $. It is also straightforward to check that $ L_+ $ is not closed. Together, these results in conjunction with the fact that $ L^*_- $ is closed imply that $ L_+ \subset L_-^* $. The result $ L_- \subset L_+^* $ follows similarly. 

(iii) The existence of $ L $ follows from~\citep[][Chapter~IV, Proposition~4.2]{SzNagyEtAl10} in conjunction with the fact that $ L_+ $ is dissipative (in their terminology, accretive). That $ L $ is maximally dissipative is equivalent to $ L^* $ also being maximally dissipative \cite{Phillips59,FirschbacherEtAl16}, leading to the second claim.

(iv) First, observe that for any function $ h \in C^\infty( M ) $ we have 
\begin{equation}
  \label{eqLCInf}
   \lVert L_+ h \rVert \leq \lVert w( h ) \rVert + \theta \lVert \upDelta h \rVert \leq ( \lVert w \rVert_\infty + \theta ) \lVert h \rVert_{H^2} = C \lVert h \rVert_{H^2},  
\end{equation}
where $ C = \lVert w \rVert_\infty + \theta $, $ \lVert w \rVert_\infty = \max_{m\in M} \lVert w \rvert_m \rVert_h $, and we have used the fact that $ \lVert \cdot \rVert \leq \lVert \cdot \rVert_{H^2} $. Next, let $ f $ be any function in $ H^2 $, and consider a $ C^\infty $ Cauchy sequence $ f_n $ converging to $ f $ in $ H^2 $ norm. Equation~\eqref{eqLCInf}, in conjunction with the fact that $ f_n $ is Cauchy in $ H^2 $, implies that $ g_n = L_+ f_n  = \bar L_+ f_n  $ is Cauchy in $ H $. Therefore, since $ \bar L_+ $ is closed, $ ( f_n, g_n ) $ converges to a point $ ( f, g ) $ in the graph of $ \bar L_+ $, where $ g = \bar L_+f = \lim_{n\to \infty  } g_n  $. Hence, we can conclude that 
\begin{displaymath}
 \lVert \bar L_+ f \rVert = \lim_{n\to\infty} \lVert L_+ f_n \rVert \leq C \lim_{n\to\infty} \lVert f_n \rVert_{H^2} = C \lVert f \rVert_{H^2},
\end{displaymath} 
as claimed. The result for $ \bar L_- $ follows similarly.
  
  (v) A proof of the claim can be found in \cite[][Proposition~1]{FrankeEtAl10}. \qedhere

\subsection{\label{appBL}Proof of Lemma~\ref{lemmaBL}}

(i) Given $ f \in Z_r $, there exist unique coefficients $ c_0, \ldots, c_{r-1} $ such that $ f = \sum_{k=0}^{r-1} c_k z_k $. For later convenience, let $ \vec c = ( c_0, \ldots, c_{r-1} ) \in \mathbb{ C }^{r} $, and define the norm $ \lVert \vec c \rVert_L = \sum_{j,k=0}^{r-1} c_j^* G_{jk} c_k $ with
\begin{equation}
  \label{eqGrammZ}
  G_{jk} = \langle z_j, z_k \rangle,
\end{equation}
 so that $ \lVert f \rVert = \lVert \vec c \rVert_L $. By equivalence of norms on finite-dimensional vector spaces, there exist constants $ K_1 $ and $ K_2 $ such that $ K^{-1}_1 \lVert c \rVert \leq  \lVert c \rVert_L \leq K_2 \lVert \vec c \rVert $, where $ \lVert \vec c \rVert $ is the canonical 2-norm on $ \mathbb{ C }^{r} $. Thus, for $ n \geq 0 $  we have
\begin{align*}
  \lVert L^n f \rVert &= \left \lVert \sum_{k=0}^{r-1} c_k L^n z_k \right\rVert 
  = \left \lVert \sum_{k=0}^{r-1} c_k \lambda_k^n z_k \right \rVert 
  = \left \lVert \overrightarrow{[ c_k \lambda_k^n ]_k} \right \rVert_L \\
  & \leq K_2 \left\lVert \overrightarrow{[ c_k \lambda_k^n ]_k} \right\rVert 
  \leq K_2 C^n_{r} \lVert \vec c \rVert 
  \leq K_2 K_1 C^n_{r} \lVert \vec c \rVert_L 
  = K_1 K_2 C^n_{r} \lVert f \rVert,
\end{align*}
from which it follows that $ f \in B_{C_r}( L ) $.

(ii) Since $\{ z_0, z_1, \ldots \} $ is assumed to be basis of $ H $, for any $ f \in B( L ) $ there exist unique coefficients $ c_0, c_1, \ldots $ such that $ f = \sum_{k=0}^\infty c_k z_k $. Moreover, there exist constants $ K  $ and $ C $ such that $ \lVert L^n f \rVert \leq K C^n \lVert f \rVert $ for any $n \geq 0 $. This implies that $ \left \lVert \sum_{k=0}^\infty c_k \lambda_k^n z_k \right \rVert \leq K C^n \lVert f \rVert $ for any $ n $, but because $ L $ is unbounded and $ \lambda_0, \lambda_1, \ldots $ is assumed to have no accumulation points, this bound cannot hold unless there exists $ r $ such that $ c_k = 0 $ for all $ k \geq r $. Thus, it follows that $ f \in Z_r $, as desired.   

\subsection{\label{appG}Proof of Lemma~\ref{lemmaG}}

That $G''_{\epsilon,N}$ is compact follows immediately from the fact that it has finite rank. Showing that $G''_{\epsilon}$ is compact is equivalent to showing that for any sequence $f_n\in C(A) $ with bounded uniform norm $ \lVert f_n \rVert_\infty $, the sequence $G''_{\epsilon}f_n$ has a limit point in the uniform-norm topology. Since the underlying space $A$ is compact, this will follow from the Arzela-Ascoli theorem, provided that it can be shown that $Gf_n$ is equicontinuous and uniformly bounded. Indeed, uniform boundedness and equicontinuity follow from the inequalities
\begin{gather*}
  \left\lvert G''_{\epsilon}f(a) \right\rvert = \left\lvert \int_A K_{\epsilon}(a,b)f(b)\, d\alpha(b) \right\rvert \leq \lVert K_\epsilon\rVert_{H\times H}\lVert f\rVert_{\infty}, \\
  \left\lvert G''_{\epsilon}f(a) - G''_{\epsilon}f(b) \right\rvert  \leq \lVert K_\epsilon(a,\cdot) - K_\epsilon(b,\cdot)\rVert_{H}\lVert f\rVert_{\infty},
\end{gather*}
respectively, proving the Lemma. 

\subsection{\label{appP}Proof of Lemma~\ref{lemmaP}}

(i) The claim follows from Proposition~11 in \cite{VonLuxburgEtAl08}. 

  (ii) The smoothness of the mapping $ \epsilon \mapsto P''_\epsilon f $ follows from the smoothness of $ p_\epsilon $. In particular, using the results in Appendices A.4 and A.5 in~\cite{BerryHarlim16} and the relation between the Laplace-Beltrami operator associated with $ h_A $ and the ambient space metric $ g_A $ in~\eqref{eqConf}, it is possible to derive the uniform asymptotic approximation
  \begin{displaymath}
    \epsilon^{-m_A/2} G''_\epsilon f( a ) = c_0 f( a ) + \epsilon c_2 \left( \tilde c_2( a ) f( a ) + \upDelta_A f( a ) \right) + O( \epsilon^2 ), 
  \end{displaymath} 
  where $ f \in C^3( A) $,  
\begin{displaymath}
  c_0 = \int_{\mathbb{ R}^{m_A}} e^{-\lVert x \rVert^2} \, dx = \pi^{m_A/2}, \quad c_2 = \int_{\mathbb{R}^{m_A}} e^{-\lVert x \rVert^2} \lVert x \rVert^2 \, dx = \frac{ \pi^{m_A/2} }{  8 } 
\end{displaymath}
 are Gaussian integrals, and $ \tilde c_2 $ is a smooth function. Using this expansion in the normalization steps in~\eqref{eqPOp} and~\eqref{eqGTildeOp} (as done in~\cite{CoifmanLafon06,BerrySauer16}) leads to the desired result. In particular, note that the term $ \epsilon c_2 \tilde c_2 f $ cancels as a result of normalization. 

  (iii) The claim was proved in~\cite{CoifmanLafon06} in the case of radial (fixed-bandwidth) Gaussian kernels. The corresponding result for the variable-bandwidth kernels in~\eqref{eqKVB} can be obtained following the same approach, employing the asymptotic expansion in~\eqref{eqPEpsilon} in place of the corresponding expansion for radial Gaussian kernels as necessary.  

(iv) Introducing the multiplication operators $ D_{\epsilon, N} : H_{A,N} \mapsto H_{A,N} $ and $ D_\epsilon : H_A \mapsto H_A $ with $ D_{\epsilon,N} \vec f = \hat d_{\epsilon,N} \vec f $, $ D_\epsilon \bar f = \hat d_\epsilon \bar f $, and $  \hat d_{\epsilon,N}  = d_{\epsilon,N}/ q_{\epsilon,N} $, $  \hat d_{\epsilon}  = d_{\epsilon}/ q_{\epsilon} $, it follows immediately that $ P_{\epsilon,N } = D^{-1/2}_{\epsilon,N} \hat P_{\epsilon, N} D^{1/2}_{\epsilon,N}$ and  $ P_{\epsilon } = D^{-1/2}_{\epsilon} \hat P_{\epsilon} D^{1/2}_{\epsilon} $. Note that we can also relate $ \hat P_{\epsilon,N} $ and $ \hat P_\epsilon $ to the corresponding unnormalized operators $ G_{\epsilon,N} $ and $ G_\epsilon $ via 
\begin{gather*}
  \hat P_{\epsilon,N} \vec f = \frac{ 1 }{ \tilde d^{1/2}_{\epsilon,N} } G_{\epsilon, N} \left( \frac{ \vec f }{ \tilde d^{1/2}_{\epsilon, N} } \right), \quad  \hat P_{\epsilon} \bar f = \frac{ 1 }{ \tilde d^{1/2}_{\epsilon} } G_{\epsilon} \left( \frac{ \bar f }{ \tilde d^{1/2}_{\epsilon} } \right). 
\end{gather*}
To establish convergence of $ \hat P_\epsilon^{s/\epsilon} $ to $ \mathcal{ P }_s $, note first that, as $ \epsilon \to 0 $, $ D_\epsilon $ converges in operator norm to the multiplication operator $ D $ associated with the positive, constant function $ \hat d = \lim_{\epsilon\to0} d_{\epsilon}/ q_{\epsilon} = 1/q $. Correspondingly, $ D_\epsilon^{1/2} $ and $ D_\epsilon^{-1/2} $ converge to $ D^{1/2} $ and $ D^{-1/2} $, respectively. We thus have $ \hat P_\epsilon^{s/\epsilon} = D_\epsilon^{1/2} P_\epsilon^{s/\epsilon} D_\epsilon^{-1/2} $, and
\begin{align*}
  \lVert \hat P_\epsilon^{s/\epsilon} - \mathcal{ P }_s \rVert & \leq   \lVert P_\epsilon^{s/\epsilon} - \mathcal{ P }_s \rVert  + \lVert D_\epsilon^{1/2} \rVert \lVert [ \mathcal{ P}_s, D_\epsilon^{-1/2} ] \rVert \\
  & =   \lVert P_\epsilon^{s/\epsilon} - \mathcal{ P }_s \rVert  + \lVert D_\epsilon^{1/2} \rVert \lVert [ \mathcal{ P}_s, D_\epsilon^{-1/2} - D^{-1/2} ] \rVert \\
  &  \leq   \lVert P_\epsilon^{s/\epsilon} - \mathcal{ P }_s \rVert  + 2 \lVert D_\epsilon^{1/2} \rVert \lVert \mathcal{ P}_s \rVert \lVert D_\epsilon^{-1/2} - D^{-1/2} \rVert,
\end{align*}
where $ [ \cdot, \cdot ] $ denotes the commutator of operators. As $ \epsilon \to 0 $, both terms in the right-hand side of the last inequality vanish. 

\subsection{\label{appEig}Proof of Lemma~\ref{lemmaEig}}

(i, ii) The claims are obvious.

(iii) Since $ \Lambda^{(\epsilon)}_k $ is nonzero, it follows from part~(ii) that it is an eigenvalue of $ P''_\epsilon $, and $\tilde \phi_k^{(\epsilon)} =  P'_\epsilon \phi_k^{(\epsilon)} / \Lambda_k^{(\epsilon)} $ is an eigenfunction. As shown in \cite{VonLuxburgEtAl08}, because $ P''_{\epsilon,N} $ converges (in this case, $ \alpha $-a.s.)  to $ P''_\epsilon $ (Lemma~\ref{lemmaP}(i)), there exist eigenvalues $ \tilde \Lambda_k^{(\epsilon,N)} $ of $ P''_{\epsilon,N} $ and corresponding eigenfunctions $ \tilde \phi_k^{(\epsilon,N)} $ converging as $ N \to \infty $ to $ \Lambda_k^{(\epsilon)} $ and $ \phi_k^{(\epsilon)} $, respectively. It then follows from part~(i) that $ \tilde \Lambda^{(\epsilon,N)}_k $ and $ \phi_k^{(\epsilon,N)} = \pi_N \tilde \phi^{(\epsilon,N)}_k$ are also eigenvalues and eigenfunctions of $ P_{\epsilon,N} $, respectively. Thus, $ \Lambda_k^{(\epsilon,N)} = \tilde \Lambda_k^{(\epsilon,N)} $ converges to $ \Lambda_k^{(\epsilon)}$, as desired. Moreover, since $ \Lambda_k^{(\epsilon)} > 0 $, there exists a stage $ N_0 $ such that for all $ N > N_0 $ we have $ \Lambda_k^{(\epsilon,N)} > 0 $, and therefore (by part~(ii)) $ \tilde \phi_k^{(\epsilon,N)} = P'_{\epsilon,N} \phi_k^{(\epsilon,N)} / \Lambda_k^{(\epsilon,N)} $. The fact that $ \tilde \phi_k^{(\epsilon,N)} \stackrel{\text{a.s.}}{\longrightarrow}   \tilde \phi_k^{(\epsilon)} $ in uniform norm completes the proof of the claim.          

(iv) Since $ P_\epsilon^{s/\epsilon} $ converges to the heat operator $ \mathcal{ P }_s $ (Lemma~\ref{lemmaP}(iv)), we have  $ ( \Lambda_k^{(\epsilon)} )^{s/\epsilon} \to \exp( -s \eta_k^A ) $, which proves~\eqref{eqLimLambda}. Similarly, the existence of a family of eigenfunctions $ \phi_k^{(\epsilon)} $ converging to $ \phi_k $ follows from convergence of $ P_\epsilon^{s/\epsilon} $ to $ \mathcal{ P }_s $.  

(v) The results follow by obvious modifications of the arguments in the proof of parts (i)--(iv).

\subsection{\label{appPropDelta}Proof of Proposition~\ref{propDelta}}

  (i) As in the case of the Dirichlet form $ E $ in Section~\ref{secGalerkin}, the claim follows from the Cauchy-Schwartz inequality and the facts that $ \lVert  \cdot \rVert_{H_{A,\epsilon,N}} \leq \lVert  \cdot \rVert_{H^1_{A,\epsilon,N}} $ and $ \lVert  \cdot \rVert_{H_{A,\epsilon}} \leq \lVert  \cdot \rVert_{H^1_{A,\epsilon}} $. 

  (ii) In the case of $ \upDelta_{A.\epsilon,N} $, the claim follows immediately from the fact that the operator is symmetric and $ H_{A,\epsilon,N}$ is finite-dimensional. In the  case of $ \upDelta_{A,\epsilon} $, since the operator is symmetric and densely defined (by Lemma~\ref{lemmaDense}), $ \upDelta^*_\epsilon $ is an extension of $ \upDelta_\epsilon $. It therefore suffices to show that $ \upDelta_\epsilon $ is also an extension of $ \upDelta_\epsilon^* $, i.e., $ D( \upDelta_\epsilon^*) \subseteq D( \upDelta_\epsilon) = H^2_{A,\epsilon} $. In particular, it suffices to show that for every $ f_1,f_2 \in H^2_{A,\epsilon} $ there exists $ g \in H_{A,\epsilon} $ such that $ \langle f_1, \upDelta_{A,\epsilon} f_2 \rangle_{H_{A,\epsilon}} = \langle g, f_2 \rangle_{H_{A,\epsilon}} $. The last equation is satisfied for $ g = \sum_{k=0}^\infty \eta_k^{(\epsilon)} \tilde c_{k1} \phi_k^{(\epsilon)} $, proving that $ \upDelta_{A,\epsilon } $ is self-adjoint. The fact that $ - \upDelta_{A,\epsilon,N} $ and $ -\upDelta_{A,\epsilon}$ are dissipative is obvious from the definition of these operators in~\eqref{eqDeltaEpsilon}.

\subsection{\label{appUConsistency}Proof of Lemma~\ref{lemmaUConsistency}}

First, note that $ \phi_i^A $ and $ \phi^A_j $ are $ L^2 $ equivalence classes of $ C^\infty $ functions, so they lie in the domain of $ \tilde u $ and the right-hand side of~\eqref{eqUConsistency} is finite. Next, consider the bounded operator $ u_\tau : H_A \mapsto H_A $ defined via $ u_\tau( f ) = ( U_\tau f - U_{-\tau} f ) / ( 2 \tau ) $, and the $ C^\infty $ functions $ \tilde \phi_k^{A,\epsilon,N} = P'_{\epsilon,N} \phi_k^{A,\epsilon,N} / \Lambda_k^{\epsilon,N} $ and $ \tilde \phi_k^{A,\epsilon} = \lim_{N\to\infty} \tilde \phi_k^{A,\epsilon,N} $  (the existence of $ \tilde \phi_k^{A,\epsilon} $ follows from Lemma~\ref{lemmaEig}). Then, using Corollary~\ref{corEig},  we obtain
  \begin{align*}
    \lim_{N\to\infty}  \langle \phi_i^{A,\epsilon,N}, u_{N,\tau}( \phi_j^{A,\epsilon,N} ) \rangle_{H_{A,\epsilon,N}} & =     \lim_{N\to\infty} \sum_{n=1}^{N-2} \tilde \phi_i^{(A,\epsilon,N)}( a_n ) \frac{ \tilde \phi_j^{(A, \epsilon,N)}( a_{n+1} ) -  \tilde \phi_j^{(A, \epsilon,N)}( a_{n+1} ) }{ 2 \tau } \beta_{\epsilon,N} \\
    & = \langle \phi_i^{A,\epsilon}, \tilde u_{\tau}( \phi_j^{A,\epsilon} ) \rangle_{H_{A,\epsilon}}.
  \end{align*}
  Since $ \phi_j^{\epsilon,N} $ and $ \phi_j^A $ are both in the domain of $ \tilde u $, and the inner product weight function $ \beta_{\epsilon} $ for $ H_{A,\epsilon} $  satisfies $ \lim_{\epsilon\to 0} \beta_{\epsilon} = 1_A $ (see Section~\ref{secKernelOp}), we have
\begin{displaymath}
  \lim_{\tau\to0} \lim_{\epsilon\to0} \langle \phi_i^{A,\epsilon}, \tilde u_{\tau}( \phi_j^{A,\epsilon} ) \rangle_{H_{A,\epsilon}} =  \lim_{\epsilon\to0} \lim_{\tau\to0} \langle \phi_i^{A,\epsilon}, \tilde u_{\tau}( \phi_j^{A,\epsilon} ) \rangle_{H_{A,\epsilon}} = \langle \phi_i^A, \tilde u_\tau( \phi_j^A ) \rangle_{H_A},
\end{displaymath}
leading to the desired result.

\subsection{\label{appGenT2}Explicit formulas for the generator for periodic spatial domains}

Since all the numerical experiments in this paper (Sections~\ref{secExamples} and~\ref{secL96}) were performed on a doubly-periodic spatial domain, $ X = \mathbb{ T }^2 $, equipped with the normalized Haar measure $ \xi $, in this Appendix we provide explicit formulas for the matrix representation of the generator $ w_{N,\tau} $ in this domain for arbitrary state-dependent incompressible velocity fields $ v\rvert_a $. Throughout, we work in the tensor product basis $ \{ \phi_{ijk}^{\epsilon,N} \} $, $ \phi_{ijk}^{\epsilon,N} = \phi_i^{A,\epsilon,N} \phi^X_{ij} $, with $ \phi_i^{A,\epsilon,N} $, $ i \in \{ 0, \ldots, N-1 \} $, set to the data-driven basis functions from Section~\ref{secDataDrivenBasis}, and $ \phi^X_{jk} $ set to the Fourier basis, $ \phi^X_{jk}(x_1, x_2) =  e^{\ii ( j x_1 + k x_2 ) } $, $ j, k \in \mathbb{ Z } $. 

A special property of $ \mathbb{ T }^2 $ (with obvious generalizations to tori of other dimensions) is that the canonical basis vector fields $ \partial_1 = \frac{ \partial\; }{ \partial x_1} $ and  $ \partial_2 = \frac{ \partial\; }{ \partial x_2} $ span the tangent space $ T_x \mathbb{ T }^2 $ at every $ x \in X $. Thus, we can expand  $ v\rvert_a $, $ a \in A $, as
\begin{equation}
  \label{eqVFourier}
  v\lvert_a = \sum_{q,r=-\infty}^\infty  \phi_{qr}^X \left( \hat v^{(1)}_{qr}(a) \partial_1 + \hat v^{(2)}_{qr}(a) \partial_2 \right),
\end{equation}    
where $ \hat v^{(1)}_{qr} $ and $ \hat v^{(2)}_{qr} $ are smooth, complex-valued functions on $ A $. The corresponding equivalence classes $ \vec v^{(1)}_{qr} $ and $ \vec v^{(2)}_{qr} $ in $ H_{A,\epsilon,N} $ can be further expanded as
\begin{displaymath}
  \vec v^{(1)}_{qr} = \sum_{p=0}^{N-1} v^{(1)}_{pqr} \phi_p^{A,\epsilon,N}, \quad  \vec v^{(2)}_{pqr} = \sum_{p=0}^{N-1} v^{(2)}_{pqr} \phi_p^{A,\epsilon,N},  
\end{displaymath} 
with $ v^{(1)}_{pqr} = \langle \phi_p^{A,\epsilon,N}, \hat v_{qr}^{(1)} \rangle $ and $ v^{(2)}_{pqr} = \langle \phi_p^{A,\epsilon,N}, \hat v_{qr}^{(2)} \rangle $. 

Next, we make use of an important property of the Fourier basis, namely that $ \phi^X_{jk} $ is an eigenfunction of both $ \partial_1 $ and $ \partial_2 $ with $ \partial_1 \phi^X_{jk} = \ii j \phi_{jk}^X $ and $ \partial_2 \phi^X_{jk} = \ii k \phi_{jk}^X $. Using this property in conjunction with the expansion in~\eqref{eqVFourier}, we obtain (cf.~\eqref{eqWXMov0}) 
\begin{align}
  \nonumber
  \langle \phi^{\epsilon,N}_{ijk}, \tilde w^X( \phi^{\epsilon,N}_{lmn} ) \rangle_{H_{\epsilon,N}} &= \sum_{p=0}^{N-1} \sum_{q,r=-\infty}^\infty \left( \hat v^{(1)}_{pqr} \langle \phi^{\epsilon,N}_{ijk}, \phi^{\epsilon,N}_{pqr} \, \partial_1 \phi^{\epsilon,N}_{lmn} \rangle + \hat v^{(2)}_{pqr} \langle \phi^{\epsilon,N}_{ijk}, \phi^{\epsilon,N}_{pqr} \, \partial_2 \phi^{\epsilon,N}_{lmn} \rangle \right) \\
  \nonumber &= \sum_{p=0}^{N-1} \sum_{q,r=-\infty}^\infty \ii ( m v^{(1)} _{pqr}  + n v^{(2)}_{pqr} ) c_{ilp} \delta_{j,q+m} \delta_{k,r+n} \\
    \label{eqWXT2} &= \sum_{p=0}^{N-1} \ii ( m v^{(1)}_{p,j-m,k-n} + nv^{(2)}_{p,j-m,k-n} ) c_{ilp}, 
\end{align}
where the coefficients $ c_{ilp} $ are given in~\eqref{eqStructureConstants}. Moreover, we have (cf.~\eqref{eqWAMov})
\begin{equation}
  \label{eqWAT2}
  \langle \phi^{\epsilon,N}_{ijk}, \tilde w_{N,\tau}^A( \phi^{\epsilon,N}_{lmn} ) \rangle_{H_{\epsilon,N}} =\langle \phi_i^{(\epsilon,N)}, u_{N,\tau}(\phi_l^{(\epsilon,N)} ) \rangle_{H_{A,\epsilon,N}} \delta_{jm} \delta_{kn},
\end{equation}
where the quantities $ \langle \phi_i^{(\epsilon,N)}, u_{N,\tau}(\phi_l^{(\epsilon,N)} ) \rangle_{H_{A,\epsilon,N}} $ are evaluated by finite differences using~\eqref{eqFD}. 

Equations~\eqref{eqWXT2} and~\eqref{eqWAT2}, together with the definition of the $ \{ \varphi_{ijk}^{\epsilon,N} \} $ basis of $ H^1_{\epsilon,N} $ and the facts that $ E_{\epsilon,N}( \varphi_{ijk}^{\epsilon,N}, \varphi_{lmn}^{\epsilon,N} ) = \delta_{il} \delta_{jm} \delta_{kn} $ and $ B_{\epsilon,N}( \varphi_{ijk}^{\epsilon,N}, \varphi_{lmn}^{\epsilon,N} ) =  \delta_{il} \delta_{jm} \delta_{kn} / \eta_{ijk}^{\epsilon,N} $ are in principle sufficient to form the matrices appearing in the eigenvalue problem in Definition~\ref{defEigDat}. Equations~\eqref{eqWXMov} and~\eqref{eqWAMov} in conjunction with $ \langle \phi_{ijk}^{\epsilon,N}, \upDelta_{\epsilon,N}\phi_{lmn}^{\epsilon,N} \rangle = \eta_{ijk}^{\epsilon,N} \delta_{il} \delta_{jm} \delta_{kn} $ are also in principle sufficient to compute the generator matrices $ \boldsymbol{ L } $ and $ \boldsymbol{ L}^* $ employed in the data-driven prediction schemes for observables and densities in Section~\ref{secDataDrivenPred}. In practice, however, evaluation of \eqref{eqWXT2} is likely to be unwieldy since it requires computation of $ N $ eigenfunctions $ \phi_i^{A,\epsilon} $, which may not be feasible at large $ N $. To address this issue, we take advantage of the expansion of $ v\lvert_a $ in~\eqref{eqVFourier} in terms of the globally defined vector fields $ \partial_1 $ and $ \partial_2 $, defining
\begin{displaymath}
  v_{\ell_v} \lvert_a = \sum_{q,r=-\infty}^\infty  \phi_{qr}^X \left( \hat v^{(1)}_{\ell_v, qr}(a) \partial_1 + \hat v^{(2)}_{\ell_v,qr}(a) \partial_2 \right), \quad \vec v^{(1)}_{\ell_v,qr} = \sum_{p=0}^{\ell_v-1} v^{(1)}_{pqr} \phi_q^{A,\epsilon,N}, \quad  \vec v^{(2)}_{\ell_v,pqr} = \sum_{p=0}^{\ell_v-1} v^{(2)}_{pqr} \phi_p^{A,\epsilon,N},  
\end{displaymath}
where $ \ell_v \leq N $ is a spectral truncation parameter (usually, $ \ell_v \ll N$). Using this smoothed version of the velocity field, we compute
\begin{equation}
  \label{eqWXT2Trunc}
\langle \phi^{\epsilon,N}_{ijk}, \tilde w_{\ell_v}^X( \phi^{\epsilon,N}_{lmn} ) \rangle_{H_{\epsilon,N}} = \sum_{p=0}^{\ell_v-1} \ii ( m v^{(1)}_{p,j-m,k-n} + nv^{(2)}_{p,j-m,k-n} ) c_{ilp},
\end{equation}
and implement the schemes of Section~\ref{secDataDrivenBasis} using these quantities instead of~\eqref{eqWXT2}. Note that this approximation does not affect the asymptotic consistency of our schemes, so long as an appropriate increasing sequence of $ \ell_v $ is employed; e.g., $ \ell_v = \ell_A $. All of the numerical experiments of Section~\ref{secL96} were performed with the latter choice for $ \ell_v $.

\section{\label{appNumerical}Numerical implementation}

All numerical experiments presented in this paper were carried out in Matlab running on a medium-size Linux cluster with ``big memory'' capabilities (up to 1.5 TiB of memory per compute node). In the case of the L96-driven experiments of Section~\ref{secL96}, we also interface with a Fortran~77 code (using Matlab's MEX framework) to carry out matrix-vector products involving the Koopman and Perron-Frobenius generators used in Leja interpolation. It should be noted that while big-memory capabilities are highly convenient for numerical implementation of our schemes, in many cases they are not essential. In particular, it is frequently the case that matrix-vector products involving the generator can be evaluated with tolerable computational overhead without explicit formation of the generator matrix itself; our Fortran-based implementation for the L96-driven system is an example of this approach.

\subsection{\label{appNumLeja}Leja interpolation}

We have implemented the data-driven prediction schemes of Sections~\ref{secPrediction} and~\ref{secDataDrivenPred} using a publicly available implementation of the Leja method in Matlab authored by Caliari and Kandolf.\footnote{Available at \url{https://www.mathworks.com/matlabcentral/fileexchange/44039-matrix-exponential-times-a-vector}.} In addition to the Leja interpolation procedure, this Matlab code also performs a subdivision step, expressing $ \vec c = e^{t \boldsymbol{L}} \vec b $ as $ \vec c = ( e^{t \boldsymbol{L}/s} )^s \vec b $ for a positive integer parameter $ s $, and repeatedly performing Leja interpolation to approximate each of the substeps $ \vec b^{(j+1)} = e^{t \boldsymbol{L}/s}  \vec c^{(j-1)} $ for $ j \in \{ 0, \ldots, s-1 \} $, $ \vec b^{(0)} = \vec b $, and $ \vec b^{(s)} = \vec c $. This subdivision procedure is introduced in order to overcome the so-called ``hump problem'' \cite{MolerVanLoan03}, which is a transient error growth with the polynomial order of approximation $ d $ sometimes observed at large $ t $. The code uses an automatic procedure to determine the number of substeps $ s $, as well as $ d $ in each substep, needed to achieve a desired accuracy.  

As stated above, in the case of the L96-driven flow, we interface the code of Caliari and Kandolf with a Fortran code to evaluate matrix-vector products involving the generator matrix $ \boldsymbol{L} $ (computed via~\eqref{eqWXL96Trunc} and~\eqref{eqWAT2}). Due to the sparse structure of that matrix,  it is possible to evaluate matrix-vector products $ \boldsymbol{L} \vec b $ without explicit formation of $ \boldsymbol{L} $, at the expense of a moderate computation cost overhead. This is particularly advantageous for reducing memory use at large values of the spectral truncation parameter $ \ell $, but typically requires coding the evaluation of $ \boldsymbol{L} \vec b $ using nested \texttt{for} loops which are inefficient in an array-oriented language such as Matlab. On the other hand, these loops can be efficiently coded in a compiled language such as Fortran using multithreading and other high-performance features, allowing scalability of the method to large $ \ell $. A similar approach can be applied in more general settings to take advantage of any special structure that the generator matrix might possess.

\subsection{\label{appDensity}Density estimation}

Following \cite{BerryHarlim16,BerryEtAl15,GiannakisEtAl15,Giannakis17}, we compute the kernel bandwidth functions $ r_{\epsilon,N} $ in~\eqref{eqKVB} using kernel density estimation. First, we introduce an integer-valued function $ I : \mathfrak{ X } \times \{ 0, \ldots, N- 1\} \mapsto \{ 0, \ldots, N-1\} $ such that $ I( v, j ) $ is equal to the time index $ i $ of the $ j $-th nearest point in the training dataset $ \{ v_i \}_{i=0}^{N-1} $ to an arbitrary data point $ v \in \mathfrak{ X}  $ measured with respect to the $ L^2 $ norm $ \lVert v - v_j \rVert_{g_X} $. Fixing a positive integer parameter $ k_\text{nn} $, we then define the bandwidth function $ \tilde r_N : A \mapsto \mathbb{ R }_+ $ such that $ \tilde r_N^2( a ) = \sum_{j=2}^{k_\text{nn}} \lVert F( a ) - v_{I(F(a),j)} \rVert^2 / ( k_\text{nn} -1 )  $ and the kernel $ \tilde K_\epsilon : A \times A\mapsto \mathbb{ R }_+ $ given by 
\begin{displaymath}
  \tilde K_{\epsilon,N}( a, b ) = \exp\left( - \frac{ \lVert F( a ) - F( b ) \rVert^2 }{ \epsilon \tilde r_{\epsilon,N}( a ) \tilde r_{\epsilon,N}( b ) }\right).
\end{displaymath}
Using this kernel, we select a value for the bandwidth parameter $ \epsilon $ and compute an estimate $ m_{A,\epsilon,N} $ of the manifold dimension $ m_A $ using the approach outlined in Section~\ref{secTuning}. We then compute the function $   \sigma_{A,\epsilon,N} : A \mapsto \mathbb{ R }_+ $ given by 
\begin{displaymath}
  \sigma_{A,\epsilon,N}(a) =  \frac{ 1 }{ N ( \pi \epsilon  \tilde r^2_N( a ) )^{\frac{m_{A,\epsilon,N}}{2}} } \sum_{n=0}^{N-1} \tilde K_\epsilon( a, a_n ),
\end{displaymath}
 and define the bandwidth function $ r_{\epsilon,N} $ used in the variable-bandwidth kernel $ K_{\epsilon,N} $ in~\eqref{eqKVB} as $ r_{\epsilon,N} = \sigma_{A,\epsilon,N}^{-1/m_{A,\epsilon,N}} $. Note that the optimal bandwidth parameter $ \epsilon $ and the dimension estimate $ m_{A,\epsilon,N} $ from $ K_\epsilon $ are generally not equal to the corresponding values computed using $ K_{\epsilon,N} $, though the dimension estimates from the two approaches are generally in reasonably good agreement. Throughout this paper we work with the neighborhood size $ k_\text{nn} = 8 $ which was also used in \cite{BerryEtAl15,GiannakisEtAl15,Giannakis17}. 

By compactness of $ A $ and smoothness of $ \tilde K_\epsilon $, $ \sigma_{A,\epsilon,N} $ and $ r_{\epsilon,N} $ are both smooth functions for all $N $.  Moreover, as $ N \to \infty $,  $ \sigma_{A,\epsilon,N}$ and $ r_{\epsilon,N} $ converge  pointwise and  $ \alpha $-a.s. to smooth functions $ \sigma_{A,\epsilon} $ and $ r_\epsilon $,  respectively, and by properties of Gaussian integrals on compact manifolds, $ \sigma_{A,\epsilon} $ converges as $ \epsilon \to 0 $ to the sampling density $ \sigma_A $, as desired. Furthermore, the error $ \sigma_{A,\epsilon} - \sigma_A $ is uniformly $ O( \epsilon ) $; the latter is a necessary condition for Lemma~\ref{lemmaP}(ii) to hold \cite{BerryHarlim16}. Note that besides the density estimation procedure outlined above any other technique achieving this level of accuracy can be employed in the definition of the bandwidth functions $ r_{\epsilon,N} $ in~\eqref{eqKVB}.

\subsection{\label{appNumEig}Kernel eigenvalue problems}

The eigenvalue problems for the finite-rank operators $ P_{\epsilon,N} $ and $ \hat P_{\epsilon,N} $ are equivalent to the eigenvalue problems for the $ N \times N $ matrices $ \boldsymbol{P} = [ N^{-1} p_{\epsilon,N}( a_m, a_n ) ]_{m,n} $ and  $ \hat{\boldsymbol{P}} = [ N^{-1} \hat p_{\epsilon,N}( a_m, a_n ) ]_{m,n} $; that is, in the notation of Section~\ref{secKernelOp}, we have
\begin{displaymath}
  \boldsymbol{P} \vec \phi_k^{(\epsilon,N)} = \Lambda_k^{(\epsilon,N)} \vec\phi_k^{(\epsilon,N)}, \quad \hat{\boldsymbol{P}} \vec \psi_k^{(\epsilon,N)} = \Lambda_k^{(\epsilon,N)} \vec\psi_k^{(\epsilon,N)}.
\end{displaymath}   
Numerically, it is generally preferable to solve the eigenvalue problem for $\hat{\boldsymbol{P}} $ to take advantage of the symmetry of that matrix, and compute the eigenvectors of $ \boldsymbol{P} $ via $ \vec \phi_k^{(\epsilon,N)} = \boldsymbol{ B}^{1/2} \vec \psi_k^{(\epsilon,N)} $, where $ \boldsymbol{ B } $ is the $ N \times N $ diagonal matrix with diagonal elements $ B_{nn} = \beta_{\epsilon,N}(a_n) $. The numerical results presented in this paper were obtained via this approach, using Matlab's \texttt{eigs} solver for sparse arrays. For symmetric matrices such as $ \hat{\boldsymbol{P}} $, this solver utilizes an implicitly restarted Lanczos method implemented in the ARPACK library \cite{LehoucqEtAl98}.  

As is customary, we take advantage of the exponential decay of the kernel $ K_{\epsilon,N} $ and sparsify $ \hat{\boldsymbol{P}} $. Specifically, we start by computing the kernel values $ K_{\epsilon,N}( a_m, a_n ) $ on the training dataset, and for each state $a_m $ retain the $ k_\text{nn} $ largest values  $ K_{\epsilon,N}( a_m, a_n ) $, where $ k_\text{nn} $ is a neighborhood size parameter (different from the one used for density estimation in~\ref{appDensity}). Using these values, we form an $ N \times N $, sparse, symmetric kernel matrix $ \boldsymbol{ K } $ whose nonzero entries are $ K_{mn} = K_{\epsilon,N}( a_m, a_n ) $ if  $ a_m $ is in the $ k_\text{nn} $ neighborhood of $ a_n $ or $ a_m $ is in the $ k_\text{nn} $ neighborhood of $ a_m $. We then perform the diffusion maps normalization procedure as described in Section~\ref{secKernelOp} to construct $ \hat{\boldsymbol{P}} $ based on $ \boldsymbol{ K } $. A pseudocode for this procedure can be found in \cite{Giannakis17}. In the L96 experiments of Section~\ref{secL96} we use the value $ k_\text{nn} = \text{10,000} $ which amounts to approximately $ 8\% $ and $ 2\% $ of the training datasets for the $ F_\text{L96}=4 $ and $ F_\text{L96} = 5 $ cases, respectively. A Matlab code for computing $ \hat{\boldsymbol{P}} $ and its eigenvalues and eigenvectors is available for download at the first author's personal website (\url{http://cims.nyu.edu/~dimitris}).     
   
\subsection{\label{appKoopEig}Koopman eigenvalue problem}

As discussed in the main text, the Koopman eigenvalue problems in Definitions~\ref{defEig} and~\ref{defEigDat} are equivalent to matrix generalized eigenvalue problems of the form~\eqref{eqEig}, involving the $  \ell \times \ell $ stiffness and mass matrices $ \boldsymbol{ A } $ and $ \boldsymbol{ B } $, respectively. An obvious issue with these eigenvalue problems is that the tensor product bases of $ H^1 $ and $ H^1_{\epsilon,N} $ employed in our schemes generally require large values of the spectral truncation parameter $ \ell $ for accurate numerical solutions, leading to high computation and memory cost. Arguably the most straightforward way of overcoming this issue (at least in the context of a Matlab implementation) is to take advantage of any sparsity that  $ \boldsymbol{ A} $ might possess (note that $ \boldsymbol{ B }$ is trivially sparse since our bases are orthogonal). Such sparsity is present in both the vortical flow examples of Section~\ref{secExamples} (see~\eqref{eqWXMov}, \eqref{eqWAMov}, and~\eqref{eqWXSwitch}) and the L96-driven flows of Section~\ref{secL96} (see~\eqref{eqWXL96Trunc} and~\eqref{eqWAT2}). In such cases, so long as $ \boldsymbol{ A } $ can be stored in memory as a sparse array, \eqref{eqEig} can be solved efficiently using Matlab's \texttt{eigs} ARPACK-based solver. This solver has an option \texttt{LR} which allows one to compute a subset of the  eigenvalues $ \lambda_k $ with the largest real part, which, according to Remark~\ref{rkDirichlet} is consistent with our Dirichlet energy ordering criterion in Definition~\ref{defCoherent}. However, in practice we find that this option can be slow to converge, and faster convergence is generally achieved using option \texttt{SM} targeting eigenvalues with the smallest modulus. All of the eigenvalues and eigenfunction results reported in this paper were computed using the latter approach, which carries a risk of missing certain eigenvalues with small real part but large modulus.
   
As the problem size $ \ell $ increases, the approach described above requiring explicit formation of $ \boldsymbol{ A } $, will eventually become infeasible. For example, in the case of the L96-driven flow in Section~\ref{secL96} with $ F_\text{L6} = 5 $, we were not able to increase the spectral truncation parameters $ \ell_A $ and $ \ell_{X_1} $ beyond 500 and 20, respectively, despite the sparsity of the generator matrix. This limitation can in principle be overcome using iterative solvers such as ARPACK, since what is required by such solvers are functions for computing matrix-vector products or solving linear systems involving $ \boldsymbol{ A } $ and $ \boldsymbol{ B } $, and these functions can be coded so as not to require explicit formation of $ \boldsymbol{A} $ (see \ref{appNumLeja}). In practice, however, we found that attempting to compute via this approach in Matlab's \texttt{eigs} solver led to unstable behavior, with  ARPACK iterations frequently failing to converge. Addressing this issue is beyond the scope of this work, but with appropriate technical modifications (perhaps avoiding use of high-level functions such as \texttt{eigs}) it should be possible to overcome the problem size limitations our current code experiences.

\section{\label{appMoving}The spectrum of the generator for the moving Gaussian vortex flow}

In this Appendix, we discuss the properties of the eigenfunctions of class~2 for the moving Gaussian vortex flow in Section~\ref{secMovingVortex}, and also establish the existence of a continuous part of the spectrum for this system. First, note that the streamfunction $ \zeta $ in~\eqref{eqZetaMoving} is invariant under transformations of $ M $ that leave $ x_1 - a$ unchanged. In particular, we can define the transformation $ \Gamma : M \mapsto \tilde X= \mathbb{ T }^2 $ such that $ \Gamma( a, x_1, x_2 ) = ( \Gamma_1( a, x_1, x_2 ), \Gamma_2( a, x_1, x_2 ) )$ with $ \Gamma_1( a, x_1, x_2 ) =  x_1 - a $ and $ \Gamma_2( a, x_1, x_2 ) = x_2 $, and consider that $ \zeta $ is the pullback of the streamfunction  $ \tilde \zeta( \tilde x_1, \tilde x_2 ) = e^{\kappa ( \cos \tilde x_1 + \cos \tilde x_2 ) }  $ on $ \tilde X$, where $\tilde x_1 $ and $ \tilde x_2 $ are canonical angle coordinates.  That is, we have $ \zeta  = \tilde \zeta \circ \Gamma  $. This transformation can be thought of as a Galilean change of frame of reference to a frame comoving with the vortex center. In particular, the generating vector field $ w $ of the dynamics on $ M$ is projectible under that transformation; i.e., $ \Gamma_* w = \tilde v $ is a well-defined vector field on $ \tilde X $, and its components in the $ \{ \frac{ \partial\; }{ \partial \tilde x_1 }, \frac{ \partial\; }{ \partial \tilde x_2 } \} $ coordinate basis are given by  
\begin{equation}
  \label{eqWTilde}
  \begin{aligned}
    \tilde v^{(1)} & = \frac{ \partial \Gamma_1 }{ \partial a } u + \frac{ \partial \Gamma_1 }{ \partial x_1 } v^{(1)} + \frac{ \partial \Gamma_1 }{ \partial x_2 } v^{(2)} = - \omega - \frac{ \partial \tilde \zeta }{ \partial \tilde x_2 }, \\ 
    \tilde v^{(2)} & = \frac{ \partial \Gamma_2 }{ \partial a } u + \frac{ \partial \Gamma_2 }{ \partial x_1 } v^{(1)} + \frac{ \partial \Gamma_2 }{ \partial x_2 } v^{(2)} = \frac{ \partial \tilde \zeta }{ \partial \tilde x_1 },
  \end{aligned}
\end{equation}
respectively. Denoting the flow on $ \tilde X $ generated by $ \tilde v $ by $ \tilde \Omega_t $, one can check that $ \Gamma \circ \Omega_t = \tilde \Omega_t \circ \Gamma $. Therefore, $ \tilde \Omega_t $ preserves the pushforward measure $ \tilde \mu = \Gamma_* \mu $, and the spectrum of $ \tilde w $ includes the spectrum of (an appropriate skew-adjoint extension of) $ \tilde v $. In particular, if $ \tilde z \in L^2( \tilde X, \tilde \mu ) $ is an eigenfunction of $ \tilde v $ at eigenvalue $ \lambda $, then $ z = \tilde z \circ \Gamma $ is an eigenfunction of $ \tilde w $, also at eigenvalue $ \lambda $. Note that the pushforward measure $ \tilde \mu $ associated with the submersion $ \Gamma $ is the normalized Haar measure on the two-torus, i.e., $ d \tilde \mu = d \tilde x_1 \wedge d \tilde x_2 / ( 2 \pi ) ^2 $.  (Note that in this section we abuse notation and do not employ a distinct symbol to represent the skew-adjoint  extension of $ \tilde v $ generating the Koopman group on $ L^2( \tilde X, \tilde \mu ) $, but the distinction is important since the spectrum of a non-closed operator such as $ \tilde v $ is equal to $ \mathbb{ C } $; see, e.g., \citep[][Section~2.2]{Schmudgen12}.)
  
Inspecting~\eqref{eqWTilde}, we see that the flow on $ \tilde X $ consists of a superposition of a stationary  Gaussian vortex flow and a free stream (i.e., a harmonic velocity field; see Example~\ref{exStream}) along the negative $ \tilde x_1 $ direction associated with the Galilean transformation $ x_1 \mapsto \tilde x_1 = x_1 - \omega t $. In this domain, the equation of motion for Lagrangian tracers is $ d \tilde x / dt  = \tilde v( \tilde x( t ) ) $, where $ \tilde x( t ) = ( \tilde x_1( t ), \tilde x_2( t ) ) $. For our choice of the $ \kappa $ and $ \omega $ parameters, $ \tilde v^{(1)} $ is everywhere negative so that tracers move along the negative $ \tilde x_1 $ direction without stagnating or turning around. This, in conjunction with the fact that $ \tilde v $ possesses the reflection symmetry  $ \tilde v( -\tilde x_1, x_2 ) = \tilde v( \tilde x_1, x_2 ) $, implies that the streamlines of the flow are closed. In particular, the $ \tilde x_2 $ coordinate of the streamline passing through the point $  ( \xi_1, \xi_2 ) \in \tilde X $ can be expressed as a function  $ s_{\xi_1 \xi_2 }( \tilde x_1 ) $ by solving the ODE
\begin{displaymath}
  \frac{ ds_{\xi_1\xi_2} }{  d \tilde x_1 } = - \frac{ \partial_{\tilde x_1} \tilde \zeta( \tilde x_1, s_{\xi_1\xi_2}( \tilde x_ 1 ) ) }{ \omega + \partial_{\tilde x_2} \tilde \zeta( \tilde x_1 , s_{\xi_1\xi_2}( \tilde x_ 1))},  \quad s_{\xi_1\xi_2}( \xi_1 ) = \xi_2.
\end{displaymath}
 
The collection of all such streamlines partitions $ \tilde X $ into a one-dimensional foliation such that the vector field $ \tilde v $ is tangent to the leaves; that is,  $ \tilde v $ is a vertical vector field with respect to this foliation \cite{ONeill66}. In particular, picking the closed line $ Y = \{ ( \tilde x_1, \tilde x_2 ) \in \tilde X \mid \tilde x_1 = 0 \}  $, we can construct a submersion $ S : \tilde X \mapsto Y $ sending every point in streamline passing through $ ( \tilde x_1, \tilde x_2 ) $ to the point $ y(\tilde x_1, \tilde x_2) = s_{\tilde x_1 \tilde x_2}( 0 ) \in Y $. The preimages of this map correspond to the leaves of the foliation, and are streamlines by construction.  

Since $ \tilde v $ is vertical, we have $ S_* \tilde v = 0 $, so that $ \tilde v $ is trivially projectible and the projected ``dynamics'' on $ Y $ reduce to the identity map (i.e., $ S \circ \tilde \Omega_t = S$).  As a result, the elements of any orthonormal basis $ \{ z^Y_k \} $ of $ L^2( Y, S_* \mu_* ) $ (note that $ S_*\mu_*$ is the normalized Haar measure on $ Y $) are orthonormal eigenfunctions of $ S_* \tilde v $ at eigenvalue 0, and correspondingly the functions $ \{ z_k^Y \circ S \circ \Gamma \} $ are orthonormal eigenfunctions of $ \tilde w $ in $ L^2( M,\mu ) $, also at eigenvalue 0. Composing $ S $ with $ \Gamma $, we obtain the submersion $ \pi_Y : M \mapsto Y $, $ \pi_Y = S \circ \Gamma $, and the corresponding subspace $ \mathcal{ D }_Y = \overline{\spn\{ z_k^Y \circ \pi_Y \} } \subset H $ discussed in Section~\ref{secMovingVortex}. Note that this submersion is associated with a two-dimensional foliation of $ M $. 

Next, to establish that $ \tilde v $, and hence $ \tilde w $, also have continuous spectrum, it suffices to demonstrate that $ \tilde \Omega_t $ is conjugate by diffeomorphism to a skew rotation $ \hat \Omega_t $ on $ \hat X = \mathbb{ T }^2 $ \cite{BroerTakens93}. In canonical angle coordinates $ \{ \hat x_1, \hat x_2 \} $, the vector field $ \hat v $ generating this skew rotation takes the form 
\begin{equation}
  \label{eqSkew}
  \hat v\rvert_{(\hat x_1,\hat x_2)} = \hat \omega( \hat x_2 ) \frac{ \partial\; }{ \partial \hat x_1 },
\end{equation}
where $ \hat \omega $ is a smooth function on $ \mathbb{ T }^1 $. Note that the flow $ ( \hat X, \hat \Omega_t ) $ generated by~\eqref{eqSkew} takes place along lines  of constant $ \hat x_2 $ coordinate  with an $ \hat x_2 $-dependent frequency. To map the original flow $ ( \tilde X, \tilde \Omega_t ) $ taking place along curved streamlines to this skew rotation, fix canonical angle coordinates $ ( \tilde x_1, \tilde x_2 ) $ on $ \hat X $ with $ \tilde x_i \in [ 0, 2 \pi ) $,  and define $ \tau( \tilde x_1, \tilde x_2 ) $ to be the time taken for a Lagrangian tracer released at the point $ ( 0,  y( \tilde x_1, \tilde x_2 )  ) \in \tilde X $ (i.e., the intercept of the streamline passing through $ ( \tilde x_1, \tilde x_2 ) $ with the $ \tilde x_1 = 0 $ line) to reach the point $ ( \tilde x_1, \tilde x_2 ) $. Define also $ \bar \tau( y ) $ to be the time taken for the tracer released at $ ( 0, y ) $ to first return to this point; i.e., $ \bar \tau( y ) = \lim_{\tilde x_1 \to 2 \pi } \tau( \tilde x_1, y ) $. Note that in the domain of definition of the $ ( \tilde x_1, \tilde x_2 ) $ coordinates we have $ \tilde v( \tau ) = 1 $. With these definitions, we introduce the diffeomorphism $ R : \tilde X \mapsto \hat X $ such that $ R( \tilde x_1, \tilde x_2 ) = ( \hat x_1, \hat x_2 ) = ( R_1( \tilde x_1, \tilde x_2 ), R_2( \tilde x_1, \tilde x_2 ) ) $, where
\begin{displaymath}
  R_1( \tilde x_1, \tilde x_2 ) = 2 \pi \frac{ \tau( x_1, \tilde x_2 ) }{ \bar \tau( y( \tilde x_1, \tilde x_2 ) ) }, \quad R_2( \tilde x_1, \tilde x_2 ) = y( \tilde x_1, \tilde x_2 ).
\end{displaymath}  
Because $ \bar \tau( y( \tilde x_1, \tilde x_2 ) ) $ and $ R_2( \tilde x_1, \tilde x_2 ) $ are constant along the streamlines in $ \tilde X $, this map rectifies these streamlines into ``straight'' lines in $ \hat X $ at constant $ \hat x_2 $ coordinate. Moreover, $ \tilde v $ is projectible under this map, and because $ \tilde v( \tau ) = 1 $, it maps to the vector field $ \hat v = R_* v  $ such that $ \hat v \rvert_{R(\tilde x_1, \tilde x_2)} = \hat \omega( \hat x_2 )  \frac{ \partial \; }{ \partial \tilde x_1 } $ with $ \omega( \hat x_2 ) = 2 \pi / \bar \tau( \hat x_2) $ (note that the components $ \hat v_j $ of $ \hat v $ in the $ \{ \frac{ \partial \ }{ \partial \hat x_j } \} $ coordinate basis are given by $ \hat v_j = \tilde v( R_j ) $, and $ \tilde v( \bar \tau )  = \tilde v( R_2 ) = 0 $). It is also possible to verify that $ R \circ \tilde \Omega_t = \hat \Omega_t \circ R $, which implies (in conjunction with the fact that $ R $ is a diffeomorphism) that $ \tilde v $ and $ \hat v $ have the same spectra.  

The skew rotation generated by the vector field in~\eqref{eqSkew} is a special case of the flows studied by Broer and Takens in \cite{BroerTakens93}. There, it is shown that for nonconstant $ \hat \omega $ the generator has continuous spectrum associated with nonconstant functions along the integral curves of $ \hat v $ (i.e., the lines $ \hat x_2 = \text{const.} $). Intuitively, if $ f $ were an eigenfunction of $ \hat v $ then there should exist a time $ t \neq 0 $ such that $ f \circ \hat \Omega_t $ is mapped back into itself, but this is not possible if $ \hat \omega $ is nonconstant and $ \hat v( f ) \neq 0 $, for the the values of $ f $ at different $ \hat x_2 $ coordinates map back to themselves at different times $ 2 \pi / \hat \omega( \hat x_2 ) $. 

The subspace of functions in $L^2( \hat X, \hat \xi ) $ (where $ \hat \xi $ is the normalized Haar measure) associated with the continuous spectrum is the orthogonal complement of $ \ker \hat v $. Correspondingly, the subspace of $ L^2( \tilde X, \tilde \xi ) $ associated with the continuous spectrum of $ \tilde v $ is $ \ker \tilde v^\perp $, and pulling back this subspace to $ M $ we obtain a subspace $ \mathcal{ C } \subset H $ associated with the continuous spectrum of the generator of the full skew-product system. This subspace lies in the orthogonal complement of the discrete spectrum subspace  $ \mathcal{D} \subset H $ identified in Section~\ref{secMovingVortex}; in fact, we also have the decomposition $ \mathcal{D}^\perp = \mathcal{C} \otimes \mathcal{D} $ (note that for any $ f \in \mathcal{C} $ and $ g \in \mathcal{ D } $ the product $ f g $ lies in $ \mathcal{D}^\perp$ \cite{DasGiannakis17}). In summary, the full $ L^2 $ space of the skew-product system on $ M $ associated with the moving vortex flow admits the decomposition 
\begin{displaymath}
  H = \mathcal{ D } \oplus \mathcal{D}^\perp = \mathcal{D} \oplus \mathcal{C} \otimes \mathcal{D}.
\end{displaymath}


\section*{References}

\begin{thebibliography}{10}
\expandafter\ifx\csname url\endcsname\relax
  \def\url#1{\texttt{#1}}\fi
\expandafter\ifx\csname urlprefix\endcsname\relax\def\urlprefix{URL }\fi
\expandafter\ifx\csname href\endcsname\relax
  \def\href#1#2{#2} \def\path#1{#1}\fi

\bibitem{Ottino89}
J.~M. Ottino, The Kinematics of Mixing: Stretching, Chaos, and Transport, no.~3
  in Cambridge Texts in Applied Mathematics, Cambridge University Press,
  Cambridge, 1989.

\bibitem{SamelsonWiggins06}
R.~M. Samelson, S.~Wiggins, Lagrangian Transport in Geophysical Jets and Waves:
  The Dynamical Systems Approach, Vol.~31 of Interdisciplinary Applied
  Mathematics, Springer, New York, 2006.

\bibitem{Haller15}
G.~Haller, Lagrangian coherent structures, Annu. Rev. Fluid Mech. 47 (2015)
  137--162.
\newblock \href {http://dx.doi.org/10.1146/annurev-fluid-010313-141322}
  {\path{doi:10.1146/annurev-fluid-010313-141322}}.

\bibitem{HallerYuan00}
G.~Haller, G.~Yuan, Lagrangian coherent structures and mixing in
  two-dimensional turbulence, Phys. D 147 (2000) 352--370.
\newblock \href {http://dx.doi.org/10.1016/S0167-2789(00)00142-1}
  {\path{doi:10.1016/S0167-2789(00)00142-1}}.

\bibitem{ShaddenEtAl05}
S.~C. Shadden, F.~Lekien, J.~E. Marsden, Definition and properties of
  {L}agrangian coherent structures from finite-time {L}yapunov exponents in
  two-dimensional aperiodic flows, Phys. D 212 (2005) 271--304.
\newblock \href {http://dx.doi.org/10.1016/j.physd.2005.10.007}
  {\path{doi:10.1016/j.physd.2005.10.007}}.

\bibitem{Haller11}
G.~Haller, A variational theory of hyperbolic {L}agrangian coherent structures,
  Phys. D 240 (2011) 574--598.
\newblock \href {http://dx.doi.org/10.1016/j.physd.2010.11.010}
  {\path{doi:10.1016/j.physd.2010.11.010}}.

\bibitem{SerraHaller16}
M.~Serra, G.~Haller, Objective {E}ulerian coherent structures, Chaos 26 (2016)
  053110.
\newblock \href {http://dx.doi.org/10.1063/1.4951720}
  {\path{doi:10.1063/1.4951720}}.

\bibitem{FroylandEtAl07}
G.~Froyland, K.~Padberg, M.~H. England, A.~M. Treguier, Detection of coherent
  oceanic structures via transfer operators, Phys. Rev. Lett. 98~(22) (2007)
  224503.
\newblock \href {http://dx.doi.org/10.1103/PhysRevLett.98.224503}
  {\path{doi:10.1103/PhysRevLett.98.224503}}.

\bibitem{FroylandEtAl10b}
G.~Froyland, S.~Lloyd, N.~Santitissadeekorn, Coherent sets for nonautonomous
  dynamical systems, Phys. D 239 (2010) 1527--1541.
\newblock \href {http://dx.doi.org/doi:10.1016/j.physd.2010.03.009}
  {\path{doi:doi:10.1016/j.physd.2010.03.009}}.

\bibitem{FroylandEtAl10}
G.~Froyland, N.~Santitissadeekorn, A.~Monahan, Transport in time-dependent
  dynamical systems: {F}inite-time coherent sets, Chaos 20 (2010) 0431116.
\newblock \href {http://dx.doi.org/10.1063/1.3502450}
  {\path{doi:10.1063/1.3502450}}.

\bibitem{Froyland13}
G.~Froyland, An analytic framework for identifying finite-time coherent sets in
  time-dependent dynamical systems, Phys. D 250 (2013) 1--19.
\newblock \href {http://dx.doi.org/10.1016/j.physd.2013.01.013}
  {\path{doi:10.1016/j.physd.2013.01.013}}.

\bibitem{BudisicMezic12}
M.~Budisi\'c, I.~Mezi\'c, Geometry of the ergodic quotient reveals coherent
  structures in flows, Phys. D 241 (2012) 1255--1269.
\newblock \href {http://dx.doi.org/10.1016/j.physd.2012.04.006}
  {\path{doi:10.1016/j.physd.2012.04.006}}.

\bibitem{FroylandPadberg09}
G.~Froyland, K.~Padberg, Almost-invariant sets and invariant manifolds --
  connecting probabilistic and geometric descriptions of coherent structures in
  flows, Phys. D 238 (2009) 1507--1523.
\newblock \href {http://dx.doi.org/10.1016/j.physd.2009.03.002}
  {\path{doi:10.1016/j.physd.2009.03.002}}.

\bibitem{Froyland15}
G.~Froyland, Dynamic isoperimetry and the geometry of {L}agrangian coherent
  structures, Nonlinearity (2015) 3587--3622\href
  {http://dx.doi.org/10.1088/0951-7715/28/10/3587}
  {\path{doi:10.1088/0951-7715/28/10/3587}}.

\bibitem{KarraschKeller16}
D.~Karrasch, J.~Keller, A geometric heat-flow theory of {L}agrangian coherent
  structures.
\newblock \href {http://arxiv.org/abs/1608.05598} {\path{arXiv:1608.05598}}.

\bibitem{BanischKoltai17}
R.~Banisch, P.~Koltai, Understanding the geometry of transport: {D}iffusion
  maps for {L}agrangian trajectory data unravel coherent sets, Chaos 27 (2017)
  035804.
\newblock \href {http://dx.doi.org/10.1063/1.4971788}
  {\path{doi:10.1063/1.4971788}}.

\bibitem{FroylandJunge17}
G.~Froyland, O.~Junge, Robust {FEM}-based extraction of finite-time coherent
  sets using scattered, sparse, and incomplete trajectories (2017).
\newblock \href {http://arxiv.org/abs/1705.03640} {\path{arXiv:1705.03640}}.

\bibitem{LiuHaller04}
W.~Liu, H.~Haller, Strange eigenmodes and decay of variance in the mixing of
  diffusive tracers, Phys. D 188 (2004) 1--39.
\newblock \href {http://dx.doi.org/10.1016/S0167-2789(03)00287-2}
  {\path{doi:10.1016/S0167-2789(03)00287-2}}.

\bibitem{Pierrehumbert94}
R.~T. Pierrehumbert, Tracer microstructure in the large-eddy dominated regime,
  Chaos Soliton. Fract. 4~(6) (1994) 1091--1110.
\newblock \href {http://dx.doi.org/10.1016/0960-0779(94)90139-2}
  {\path{doi:10.1016/0960-0779(94)90139-2}}.

\bibitem{FroylandKoltai17}
G.~Froyland, P.~Koltai, Estimating long-term behavior of periodically driven
  flows without trajectory integration, Nonlinearity 30 (2017) 1948–1986.
\newblock \href {http://dx.doi.org/10.1088/1361-6544/aa6693}
  {\path{doi:10.1088/1361-6544/aa6693}}.

\bibitem{EisnerEtAl15}
T.~Eisner, B.~Farkas, M.~Haase, R.~Nagel, Operator Theoretic Aspects of Ergodic
  Theory, Vol. 272 of Graduate Texts in Mathematics, Springer, 2015.

\bibitem{GiannakisEtAl15}
D.~Gannakis, J.~Slawinska, Z.~Zhao, Spatiotemporal feature extraction with
  data-driven {K}oopman operators, J. Mach. Learn. Res. Proceedings 44 (2015)
  103--115.

\bibitem{Giannakis17}
D.~Giannakis, Data-driven spectral decomposition and forecasting of ergodic
  dynamical systems, Appl. Comput. Harmon. Anal.In minor revision.
\newblock \href {http://arxiv.org/abs/1507.02338} {\path{arXiv:1507.02338}}.

\bibitem{DasGiannakis17}
S.~Das, D.~Giannakis, Delay-coordinate maps and the spectra of {K}oopman
  operators, NonlinearityIn preparation.

\bibitem{FrankeEtAl10}
B.~Franke, C.-R. Hwang, H.-M. Pai, S.~Sheu, The behavior of the spectral gap
  under growing drift, Trans. Amer. Math. Soc. 362~(3) (2010) 1325--1350.
\newblock \href {http://dx.doi.org/10.1090/s0002-9947-09-04939-3}
  {\path{doi:10.1090/s0002-9947-09-04939-3}}.

\bibitem{ConstantinEtAl08}
P.~Constantin, A.~Kiselev, L.~Ryzhik, A.~Zlato\v{s},
  \href{http://www.jstor.org/stable/40345422}{Diffusion and mixing in fluid
  flow}, Ann. Math. 168 (2008) 643--674.
\newline\urlprefix\url{http://www.jstor.org/stable/40345422}

\bibitem{BelkinNiyogi03}
M.~Belkin, P.~Niyogi, Laplacian eigenmaps for dimensionality reduction and data
  representation, Neural Comput. 15 (2003) 1373--1396.
\newblock \href {http://dx.doi.org/10.1162/089976603321780317}
  {\path{doi:10.1162/089976603321780317}}.

\bibitem{CoifmanLafon06}
R.~R. Coifman, S.~Lafon, Diffusion maps, Appl. Comput. Harmon. Anal. 21 (2006)
  5--30.
\newblock \href {http://dx.doi.org/10.1016/j.acha.2006.04.006}
  {\path{doi:10.1016/j.acha.2006.04.006}}.

\bibitem{VonLuxburgEtAl08}
U.~von Luxburg, M.~Belkin, O.~Bousquet, Consitency of spectral clustering, Ann.
  Stat. 26~(2) (2008) 555--586.
\newblock \href {http://dx.doi.org/10.1214/009053607000000640}
  {\path{doi:10.1214/009053607000000640}}.

\bibitem{BerryHarlim16}
T.~Berry, J.~Harlim, Variable bandwidth diffusion kernels, J. Appl. Comput.
  Harmon. Anal. 40~(1) (2015) 68--96.
\newblock \href {http://dx.doi.org/10.1016/j.acha.2015.01.001}
  {\path{doi:10.1016/j.acha.2015.01.001}}.

\bibitem{BerrySauer16b}
T.~Berry, T.~Sauer, Local kernels and the geometric structure of data, J. Appl.
  Comput. Harmon. Anal. 40~(3) (2016) 439--469.
\newblock \href {http://dx.doi.org/10.1016/j.acha.2015.03.002}
  {\path{doi:10.1016/j.acha.2015.03.002}}.

\bibitem{CaliariEtAl04}
M.~Caliari, M.~Vianello, L.~Bergamaschi, Interpolating discrete
  advection–-diffusion propagators at {L}eja sequences, J. Comput. Appl.
  Math. 172~(79–-99).
\newblock \href {http://dx.doi.org/10.1016/j.cam.2003.11.015}
  {\path{doi:10.1016/j.cam.2003.11.015}}.

\bibitem{KandolfEtAl14}
P.~Kandolf, A.~Ostermann, S.~Rainer, A residual based error estimate for {L}eja
  interpolation of matrix functions, Linear Alg. Appl. 456 (2014) 157--173.

\bibitem{QiMajda16}
D.~Qi, A.~J. Majda, Predicting fat-tailed intermittent probability
  distributions in passive scalar turbulence with imperfect models through
  empirical information theory, Commun. Math. Sci. 14~(6) (2016) 1687--1722.
\newblock \href {http://dx.doi.org/10.4310/CMS.2016.v14.n6.a11}
  {\path{doi:10.4310/CMS.2016.v14.n6.a11}}.

\bibitem{AbrahamEtAl88}
R.~Abraham, J.~E. Marsden, T.~Ratiu, Manifolds, Tensor Analysis, and
  Applications, Vol.~75 of Applied Mathematical Sciences, Springer, New York,
  1988.

\bibitem{SauerEtAl91}
T.~Sauer, J.~A. Yorke, M.~Casdagli, Embedology, J. Stat. Phys. 65~(3--4) (1991)
  579--616.
\newblock \href {http://dx.doi.org/10.1007/bf01053745}
  {\path{doi:10.1007/bf01053745}}.

\bibitem{Mezic05}
I.~Mezi\'c, Spectral properties of dynamical systems, model reduction and
  decompositions, Nonlinear Dyn. 41 (2005) 309--325.
\newblock \href {http://dx.doi.org/10.1007/s11071-005-2824-x}
  {\path{doi:10.1007/s11071-005-2824-x}}.

\bibitem{BudisicEtAl12}
M.~Budisi\'c, R.~Mohr, I.~Mezi\'c, Applied {K}oopmanism, Chaos 22 (2012)
  047510.
\newblock \href {http://dx.doi.org/10.1063/1.4772195}
  {\path{doi:10.1063/1.4772195}}.

\bibitem{BabuskaOsborn91}
I.~Babu\v{s}ka, J.~Osborn, Eigenvalue problems, in: P.~G. Ciarlet, J.~L. Lions
  (Eds.), Finite Element Methods (Part 1), Vol.~II of Handbook of Numerical
  Analysis, North-Holland, Amsterdam, 1991, pp. 641--787.

\bibitem{DellnitzJunge99}
M.~Dellnitz, O.~Junge, On the approximation of complicated dynamical behavior,
  SIAM J. Numer. Anal. 36 (1999) 491.
\newblock \href {http://dx.doi.org/10.1137/S0036142996313002}
  {\path{doi:10.1137/S0036142996313002}}.

\bibitem{Rosenberg97}
S.~Rosenberg, The {L}aplacian on a {R}iemannian Manifold, Vol.~31 of London
  Mathematical Society Student Texts, Cambridge University Press, Cambridge,
  1997.

\bibitem{Phillips59}
R.~S. Phillips, Dissipative operators and hyperbolic systems of partial
  differential equations, Trans. Amer. Math. Soc. 90~(2) (1959) 193--254.

\bibitem{LehoucqEtAl98}
R.~B. Lehoucq, D.~C. Sorensen, C.~Yang, {ARPACK} Users' Guide: Solution of
  Large-Scale Eigenvalue Problems With Implictly Restarted {A}rnoldi Methods,
  SIAM, Philadelphia, 1998.

\bibitem{ReedSimon80}
M.~Reed, B.~Simon, Methods of Modern Mathematical Physics. I: Functional
  Analysis, Academic Press, San Diego, 1980.

\bibitem{Nelson59}
E.~Nelson, Analytic vectors, Ann. Math. 70~(3) (1959) 572--615.

\bibitem{StochelSzafraniec97}
F.~H. Stochel, J.~Szafraniec, {$\mathcal{C}^\infty$}-vectors and boundedness,
  Ann. Polon. Math. (1997) 223--238.

\bibitem{FujitaSuzuki91}
H.~Fujita, T.~Suzuki, Evolution problems, in: P.~G. Ciarlet, J.~L. Lions
  (Eds.), Finite Element Methods (Part 1), Vol.~II of Handbook of Numerical
  Analysis, North-Holland, Amsterdam, 1991, pp. 789--923.

\bibitem{MolerVanLoan03}
C.~Moler, C.~Van~Loan, Nineteen dubious ways to compute the exponential of a
  matrix, twenty-five years later, SIAM Rev. 45~(1) (2003) 3--49.

\bibitem{Higham08}
N.~J. Higham (Ed.), Functions of Matrices, Society for Industrial and Applied
  Mathematics (SIAM), Philadelphia, 2008.

\bibitem{CaliariEtAl14}
M.~Caliari, P.~Kandolf, A.~Ostermann, S.~Rainer, Comparison of methods for
  computing the action of the matrix exponential, BIT Numer. Math. 52~(1)
  (2014) 113–128.

\bibitem{Caliari07}
M.~Caliari, Accurate evaluation of divided differences for polynomial
  interpolation of exponential propagators, Computing 80 (2007) 189--201.

\bibitem{BroerTakens93}
H.~Broer, F.~Takens, Mixed spectra and rotational symmetry, Arch. Rational
  Mech. Anal. 124 (1993) 13--42.
\newblock \href {http://dx.doi.org/10.1007/BF00392202}
  {\path{doi:10.1007/BF00392202}}.

\bibitem{Mezic17}
I.~Mezi\'c, Koopman operator spectrum and data analysis (2017).
\newblock \href {http://arxiv.org/abs/1702.07597} {\path{arXiv:1702.07597}}.

\bibitem{Watson73}
B.~Watson, Manifold maps commuting with the {L}aplacian, J. Diff. Geom. 8
  (1973) 85--94.

\bibitem{GoldbergIshihara78}
S.~I. Goldberg, T.~Isihara, Riemannian submersions commuting with the
  {L}aplacian, J. Diff. Geom. 13 (1978) 139--144.

\bibitem{GilkeyEtAl98}
P.~B. Gilkey, J.~V. Leahy, J.~H. Park, Eigenvalues of the form valued
  {L}aplacian for {R}iemannian submersions, Proc. Amer. Math. Soc. 126~(6)
  (1998) 1845--1850.

\bibitem{Putinar97}
M.~Putinar, Generalized eigenfunction expansions and spectral decompositions,
  Banach Center Publications 38~(1) (1997) 265--286.

\bibitem{Aref84}
H.~Aref, Stirring by chaotic advection, J. Fluid. Mech. 143 (1984) 1--21.
\newblock \href {http://dx.doi.org/dx.doi.org/10.1017/S0022112084001233}
  {\path{doi:dx.doi.org/10.1017/S0022112084001233}}.

\bibitem{Giannakis15}
D.~Giannakis, Dynamics-adapted cone kernels, SIAM J. Appl. Dyn. Sys. 14~(2)
  (2015) 556--608.
\newblock \href {http://dx.doi.org/10.1137/140954544}
  {\path{doi:10.1137/140954544}}.

\bibitem{BerrySauer16}
T.~Berry, T.~Sauer, Consistent manifold representation for topological data
  analysis (2016).
\newblock \href {http://arxiv.org/abs/1606.02353} {\path{arXiv:1606.02353}}.

\bibitem{CoifmanEtAl08}
R.~R. Coifman, Y.~Shkolnisky, F.~J. Sigworth, A.~Singer, Graph {L}aplacian
  tomography from unknown random projections, IEEE Trans. Image Process.
  17~(10) (2008) 1891--1899.
\newblock \href {http://dx.doi.org/10.1109/tip.2008.2002305}
  {\path{doi:10.1109/tip.2008.2002305}}.

\bibitem{BerryEtAl15}
T.~Berry, D.~Giannakis, J.~Harlim, Nonparametric forecasting of low-dimensional
  dynamical systems, Phys. Rev. E. 91 (2015) 032915.
\newblock \href {http://dx.doi.org/10.1103/PhysRevE.91.032915}
  {\path{doi:10.1103/PhysRevE.91.032915}}.

\bibitem{MajdaKramer99}
A.~J. Majda, P.~R. Kramer, Simplified models for turbulent diffusion: Theory,
  numerical modeling, and physical phenomena, Phys. Rep. (1999) 237--574\href
  {http://dx.doi.org/10.1016/S0370-1573(98)00083-0}
  {\path{doi:10.1016/S0370-1573(98)00083-0}}.

\bibitem{MajdaGershgorin10}
A.~J. Majda, B.~Gershgorin, Quantifying uncertainty in climage change science
  through empirical information theory, Proc. Natl. Acad. Sci. 107~(34) (2010)
  14958--14963.
\newblock \href {http://dx.doi.org/10.1073/pnas.1007009107}
  {\path{doi:10.1073/pnas.1007009107}}.

\bibitem{Lorenz96}
E.~N. Lorenz, Predictability of weather and climate, in: T.~Palmer, R.~Hagedorn
  (Eds.), Predictability of Weather and Climate, Cambridge University Press,
  Cambridge, 1996, Ch.~3, pp. 40--58.

\bibitem{LuzzattoEtAl05}
S.~Luzzatto, I.~Melbourne, F.~Paccaut, The {L}orenz attractor is mixing, Comm.
  Math. Phys. 260~(2) (2005) 393--401.

\bibitem{MajdaEtAl05}
A.~J. Majda, R.~V. Abramov, M.~J. Grote, Information Theory and Stochastics for
  Multiscale Nonlinear Systems, Vol.~25 of {CRM} Monograph Series, American
  Mathematical Society, Providence, 2005.

\bibitem{KarimiPaul10}
A.~Karimi, M.~R. Paul, Extensive chaos in the {L}orenz-96 model, Chaos (2010)
  043105\href {http://dx.doi.org/10.1063/1.3496397}
  {\path{doi:10.1063/1.3496397}}.

\bibitem{VanKekemSterk17}
D.~L. van Kekem, A.~E. Sterk, Travelling waves and their bifurcations in the
  {L}orenz-96 model (2017).
\newblock \href {http://arxiv.org/abs/1704.05442} {\path{arXiv:1704.05442}}.

\bibitem{CoifmanMaggioni06}
R.~R. Coifman, M.~Maggioni, Diffusion wavelets, Appl. Comput. Harmon. Anal.
  21~(1) (2006) 53--94.
\newblock \href {http://dx.doi.org/10.1016/j.acha.2006.04.004}
  {\path{doi:10.1016/j.acha.2006.04.004}}.

\bibitem{AllardEtAl12}
W.~K. Allard, G.~Chen, M.~Maggioni, Multi-scale geometric methods for data sets
  {II}: Geometric multi-resolution analysis, Appl. Comput. Harmon. Anal. 32~(3)
  (2012) 435--462.
\newblock \href {http://dx.doi.org/10.1016/j.acha.2011.08.001}
  {\path{doi:10.1016/j.acha.2011.08.001}}.

\bibitem{BruntonEtAl16}
S.~L. Brunton, J.~L. Proctor, J.~N. Kutz, Discovering governing equations from
  data by sparse identification of nonlinear dynamical systems, Proc. Natl.
  Acad. Sci. 113~(15) (2016) 3932--3937.
\newblock \href {http://dx.doi.org/10.1073/pnas.1517384113}
  {\path{doi:10.1073/pnas.1517384113}}.

\bibitem{SzNagyEtAl10}
B.~Sz.-Nagy, C.~Foias, H.~Bercovici, L.~K\'erchy, Harmonic Analysis of
  Operators on {H}ilbert Space, 2nd Edition, Springer-Verlag, New York, 2010.
\newblock \href {http://dx.doi.org/10.1007/978-1-4419-6094-8}
  {\path{doi:10.1007/978-1-4419-6094-8}}.

\bibitem{FirschbacherEtAl16}
C.~Firschbacher, S.~Naboko, I.~Wood, The proper dissipative extensions of a
  dual pair (2016).
\newblock \href {http://arxiv.org/abs/1603.08192} {\path{arXiv:1603.08192}}.

\bibitem{Schmudgen12}
K.~Schm\"udgen, Unbounded and Self-Adjoint Operators on {H}ilbert Space, Vol.
  265 of Graduate Texts in Mathematics, Springer, Dordrecht, 2012.
\newblock \href {http://dx.doi.org/10.1007/978-94-007-4753-1_2}
  {\path{doi:10.1007/978-94-007-4753-1_2}}.

\bibitem{ONeill66}
B.~O'Neill, The fundamental equations of a submersion, Michigan Math. J. 13~(4)
  (1966) 459--469.
\newblock \href {http://dx.doi.org/10.1307/mmj/1028999604}
  {\path{doi:10.1307/mmj/1028999604}}.

\end{thebibliography}
\end{document}